\documentclass{memo-l}

\usepackage{mathtools}
\usepackage{etoolbox}
\usepackage{amsmath}
\usepackage{enumerate}
\usepackage{amssymb}
\usepackage{amscd} 
\usepackage{amsthm}
\usepackage{amsfonts}
\usepackage{cancel}
\usepackage[usenames,dvipsnames,svgnames]{xcolor}
\usepackage{tikz}
\usetikzlibrary{cd}
\usepackage{graphicx}
\usepackage[matrix, arrow, curve]{xy}
\usepackage{comment}
\usepackage[scaled]{helvet} 
\usepackage{courier} 
\usepackage[T1]{fontenc}                                                    
\usepackage{hyperref}
\usepackage{imakeidx}
\usepackage{extarrows}

\labelformat{section}{\thechapter.\thesection}

\numberwithin{equation}{chapter}

\newcommand{\U}{\operatorname{U}}

\newcommand{\Aut}{\operatorname{Aut}}

\newcommand{\im}{\operatorname{Im}}

\newcommand{\cA}{\mathcal{A}}
\newcommand{\cB}{\mathcal{B}}
\newcommand{\cC}{\mathcal{C}}

\newcommand{\cE}{\mathcal{E}}

\newcommand{\cG}{\mathcal{G}}
\newcommand{\cH}{\mathcal{H}}
\newcommand{\cI}{\mathcal{I}}

\newcommand{\cK}{\mathcal{K}}
\newcommand{\cL}{\mathcal{L}}

\newcommand{\cO}{\mathcal{O}}

\newcommand{\cQ}{\mathcal{Q}}

\newcommand{\cT}{\mathcal{T}}
\newcommand{\cU}{\mathcal{U}}
\newcommand{\cV}{\mathcal{V}}
\newcommand{\cW}{\mathcal{W}}

\newcommand{\bH}{\mathbb{H}}

\newcommand{\bN}{\mathbb{N}}

\newcommand{\bR}{\mathbb{R}}

\newcommand{\bZ}{\mathbb{Z}}

\newcommand{\R}{\mathbb{R}}
\newcommand{\C}{\mathbb{C}}

\newcommand{\qand}{\quad \textrm{and} \quad}

\newcommand\subsetsim{\mathrel{%
\ooalign{\raise0.2ex\hbox{$\subset$}\cr\hidewidth\raise-0.8ex\hbox{\scalebox{0.9}{$\sim$}}\hidewidth\cr}}}

\newcommand{\lsg}{large-scale geodesic }

\newcommand{\set}[1]{\left\{#1\right\}}
\newcommand{\setcon}[2]{\left\{#1\ \left|\ #2\right.\right\}}

\newcommand{\into}{\hookrightarrow}

\newcommand{\asdim}{\operatorname{asdim}}

\newcommand{\qker}{\operatorname{qker}}

\newcommand{\gp}[3]{(#1\mid#2)_{#3}}
\newcommand{\diam}{\operatorname{diam}}
\newcommand{\mesh}{\operatorname{mesh}}
\newcommand{\mult}{\operatorname{mult}}
\newcommand{\VR}{\operatorname{VR}}
\newcommand{\id}{\operatorname{Id}}
\newcommand{\FP}{\text{FP}}
\renewcommand{\phi}{\varphi}

\DeclareMathOperator{\dist}{dist}
\DeclareMathOperator{\proj}{proj}

\DeclareMathOperator{\ind}{ind}

\DeclareMathOperator{\Comm}{Comm}
\newcommand{\Q}{\mathbb Q}
\newcommand{\Z}{\mathbb Z}

\theoremstyle{plain}
\newtheorem{theorem}{Theorem}[chapter]
\newtheorem{corollary}[theorem]{Corollary}
\newtheorem{proposition}[theorem]{Proposition}
\newtheorem{lemma}[theorem]{Lemma}

\theoremstyle{definition}
\newtheorem{definition}[theorem]{Definition}
\newtheorem{notation}[theorem]{Notation}
\newtheorem{construction}[theorem]{Construction}

\newtheorem{remark}[theorem]{Remark}
\newtheorem{example}[theorem]{Example}

\makeindex

\patchcmd{\section}{-.5em}{.5em}{}{}
\patchcmd{\subsubsection}{-.5em}{.5em}{}{}

\makeatletter
\@namedef{subjclassname@2020}{%
  \textup{2020} Mathematics Subject Classification}
\makeatother

\makeatletter
\let\@wraptoccontribs\wraptoccontribs
\makeatother

\begin{document}
\title{Foundations of geometric approximate group theory}

\author{Matthew Cordes}
\address{Mathematics Department, Heriot-Watt University, Scotland}
\email{m.cordes@hw.ac.uk}
\thanks{}

\author{Tobias Hartnick}
\address{Institut f\"ur Algebra und Geometrie, KIT, Germany}
\curraddr{}
\email{tobias.hartnick@kit.edu}
\thanks{}

\author{Vera Toni\'c}
\address{Faculty of Mathematics, University of Rijeka, Croatia}
\curraddr{}
\email{vera.tonic@math.uniri.hr}
\thanks{}

\author{(with an appendix by Simon Machado)}
\address{Mathematics Department, ETH Z\"urich, Switzerland}
\curraddr{}
\email{smachado@ethz.ch}
\thanks{}



\date{\today}

\begin{abstract}
We develop the foundations of a geometric theory of countably-infinite approximate groups, extending work of Bj\"orklund and the second-named author. Our theory is based on the notion of a quasi-isometric quasi-action (qiqac) of an approximate group on a metric space.

More specifically, we introduce a geometric notion of finite generation for approximate group and prove that every geometrically finitely-generated approximate group admits a geometric qiqac on a proper geodesic metric space. We then show that all such spaces are quasi-isometric, hence can be used to associate a canonical QI type with every geometrically finitely-generated approximate group. This in turn allows us to define geometric invariants of approximate groups using QI invariants of metric spaces. Among the invariants we consider are asymptotic dimension, finiteness properties, numbers of ends and growth type. For geometrically finitely-generated approximate groups of polynomial growth we derive a version of Gromov's polynomial growth theorem, based on work of Hrushovski and Breuillard--Green--Tao.

A particular focus is on qiqacs on hyperbolic spaces. Our strongest results are obtained for approximate groups which admit a geometric qiqac on a proper geodesic hyperbolic space. For such ``hyperbolic approximate groups'' we establish a number of fundamental properties in analogy with the case of hyperbolic groups. For example, we show that their asymptotic dimension is one larger than the topological dimension of their Gromov boundary and that -- under some mild assumption of being ``non-elementary'' -- they have exponential growth and act minimally on their Gromov boundary. We also show that every non-elementary hyperbolic approximate group of asymptotic dimension $1$ is quasi-isometric to a finitely-generated group (in fact, a finitely-generated non-abelian free group). The case of higher asymptotic dimension remains open.
 
We also study proper qiqacs of approximate group on hyperbolic spaces, which are not necessarily cobounded. Here one focus is on convex cocompact qiqacs, for which we provide several geometric characterizations \`a la Swenson. In particular we show that all limit points of a convex cocompact qiqac are conical, whereas in general the conical limit set contains additional information. Finally, using the theory of Morse boundaries, we extend some of our results concerning qiqacs on hyperbolic spaces to qiqacs on proper geodesic metric spaces with non-trivial Morse boundary.

Throughout this book we emphasize different ways in which definitions and results from geometric group theory can be extended to approximate groups. In cases where there is more than one way to extend a given definition, the relations between different possible generalizations are carefully explored.
\end{abstract}

\maketitle

\tableofcontents

\chapter{Introduction}

The goal of this book is to extend a number of foundational results in geometric group theory from finitely-generated groups to ``finitely-generated'' approximate groups, and thus to carry out a program that was suggested by Michael Bj\"orklund and the second-named author in \cite[Chapter 3]{BH}. 
The introduction begins with
an explanation of the origins of the program, followed by some difficulties in implementation,
and finishes with a survey of a few sample results from the new theory.

\section{A historical perspective}

\begin{remark}[From algebra to quasi-algebra]
From the syntactic point of view, as developed in the early 20th century, algebra is concerned with ``sets with operations'' (e.g.\ group, rings, fields). These operations transform a number of elements of the underlying set to another element of this set. To relate different sets with operations of the same type, one then studies maps between the underlying sets which preserve the given operations.

It was already pointed out by Ulam in the 1940s, and later popularized in his book \cite{Ulam}, that this setting may be too rigid to encompass many phenomena of interest, both within mathematics and in mathematical applications. Instead of studying functions which satisfy strict functional equations such as ``$f(g+h)=f(g) + f(h)$'', Ulam suggested to investigate functions which only satisfy such functional equations ``approximately''. This idea has since manifested itself in many different areas of mathematics, from stability problems in functional analysis to the study of group quasimorphisms and bounded cohomology -- see \cite{UlamStability, FujKap, HartnickSchweitzer} for some recent examples. 

An even more radical idea is to consider structures, i.e. sets with operations, in which the operations themselves do not produce elements of the underlying set, but rather elements of a larger set which are in some sense ``close'' to the original set. It was this latter step of ``quasifying not only the morphisms, but also the objects'', which, in the early 21st century, led to the creation of a new kind of \emph{quasi-algebra}. Approximate groups are among the main protagonists of this new world, and can be defined as follows (see Definition \ref{DefTao}).
\end{remark}
\begin{definition} Let $G$ be a group and $k \in \mathbb N$. A subset $\Lambda \subset G$ containing the identity is called a \emph{$k$-approximate subgroup} of $G$ if it is symmetric (i.e. $x \in \Lambda \Longrightarrow x^{-1}\in \Lambda$)  and if there is a finite set $F \subset G$ of cardinality $k$ such that for all $x,y \in \Lambda$ there exist $z \in \Lambda$ and $f \in F$ with $xy = zf$. 

In this case we denote by $\Lambda^\infty$ the subgroup of $G$ generated by $\Lambda$ and call the pair $(\Lambda, \Lambda^\infty)$ a \emph{$k$-approximate group}.
\end{definition}

It is interesting to compare the genesis and history of approximate groups to the genesis and history of groups:

\begin{remark}[Some milestones in the history of group theory]\label{GroupHistory}
While the modern definition of a group is less than 150 years old, specific classes of groups have been studied since antiquity. Even after the definition of abstract groups, the focus of the theory has always been on specific classes of examples, often arising from other mathematical areas. Some important historical milestones in the development of modern group theory can be listed as follows\footnote{We admit that these milestones are chosen with a certain bias towards concepts relevant to approximate groups and geometric group theory.}:
\begin{itemize}
\item \emph{Finite groups} and \emph{permutation groups} became more and more important in number theory and algebra during the 18th and 19th century, due to work of Gau\ss, Abel, Galois, Kummer, Dedekind and many others.
\item \emph{Infinite abelian groups}, in particular \emph{lattices} in Euclidean spaces (of low dimensions), have been studied due to their connection with crystals; an early highlight was the theory of Bravais lattices towards the middle of the 19th century.
\item \emph{Lie groups} and more generally \emph{locally compact groups} became the basis for the modern formulation of geometry, due to Riemann, Lie, Klein, Poincar\'e and others.
\item \emph{Linear groups}, in particular \emph{discrete subgroups of Lie groups} such as Kle\-inian and Fuchsian groups, were among the first countably-infinite non-abelian groups that were studied systematically.
\item General, i.e., potentially non-linear \emph{finitely-generated groups} were studied in the first half of the 20th century by Dehn, von Neumann, Tarski, Freudenthal and others, but
became dominant in group theory only during the second half of the 20th century, following pioneering work of Gromov, Rips and others: this was the ``geometric group theory revolution''.
\item More recently, the geometric group theory revolution has also attacked \emph{compactly-generated locally compact groups}; the book \cite{CdlH} has popularized this second wave of the revolution enormously.
\end{itemize}
\end{remark}
\begin{remark}[Some milestones in the history of approximate group theory]
Just as in the group case, the theory of approximate groups is much older than the actual definition, which was coined by Terry Tao and first appeared in print in 2008 \cite{Tao}. As in the group case, approximate groups arose in at least three different areas of mathematics simultaneously:
\begin{itemize}
\item \emph{Finite approximate groups} play a central role in modern additive combinatorics; implicitly they already appear in the pioneering works of Pl\"unnecke \cite{Plunnecke}, Fre\u{\i}man \cite{Freiman}, Erd\"os and Szemer\'{e}di \cite{ErdosSzemeredi},  Ruzsa \cite{Ruzsa} and others. More recent works by Bourgain and Gamburd \cite{BourgainGamburd}, respectively Helfgott \cite{Helfgott}, relate finite approximate groups to superstrong approximation, respectively expansion in finite simple groups. The general structure theory of finite approximate groups due to Breuillard, Green and Tao \cite{BGT} is one of the cornerstones of the modern theory.
\item Independently, approximate groups occurred in the theory of \emph{mathematical quasi-crystals}. Ori\-ginally developed by Meyer \cite{Meyer1972} with a view towards applications in number theory and harmonic analysis, they rose to prominence after the discovery of materials with quasi-crystalline structures in the 1980s (see \cite{BaakeG-13} for a bibliography with hundreds of references). In the language of the present book, mathematical quasi-crystals are (translates of) \emph{uniform approximate lattices} in abelian locally compact groups.
\item The connection between approximate groups and \emph{locally compact groups} goes back even further. Compact symmetric identity neighborhoods in locally compact groups and Lie groups play a central role in connection with the solution to Hilbert's fifth problem \cite{MontgomeryZippin1952, Gleason1952}. These neighborhoods are approximate subgroups, and their study has greatly influenced the modern theory of approximate groups (see \cite{TaoHilbert5}). More recently, a structure theory of \emph{locally compact approximate groups} has been developed by Carolino \cite{Carolino}, and \emph{algebraic approximate groups} have been classified by Bj\"orklund, Hartnick and Stulemeijer \cite{BHS}. 
\item One of the main techniques in the study approximate groups which has no counterpart in group theory is the theory of \emph{good (quasi-)models} due to Hrushovski \cite{Hrushovski}. This theory connects general approximate groups to compact symmetric identity neighborhoods in locally compact groups and played a crucial role in the structure theory of Breuillard, Green and Tao. It was recently used by Hrushovski to provide a parametrization of approximate subgroups of a given group up to commensurability \cite{Hrushovski_Arithmeticity}.
\item \emph{Discrete approximate subgroups of Lie groups} appear implicitly already in work of Chifan and Ioana on relative Property (T) \cite{ChifanIoana}. The theory of \emph{model sets} in Lie groups, and more general \emph{locally compact second countable} groups, due to Bj\"orklund, Hartnick and Pogorzelski \cite{BHP1}, was a major motivation for the introduction of
\emph{approximate lattices} by Bj\"orklund and Hartnick \cite{BH}, which started a whole new line of research \cite{BHS, BH2, BH3}. \emph{Arithmetic model sets} have recently received much attention due to arithmeticity theorems of Machado \cite{Machado_Arithmeticity} and Hrushovski \cite{Hrushovski_Arithmeticity}.
\end{itemize}
Comparing this list to the list in Remark \ref{GroupHistory}, what is still missing is a \emph{general geometric theory of finitely-generated approximate groups}. The first steps towards such a theory were taken in Chapter 3 of \cite{BH}, but the focus there was mostly on the connection between isometric actions of approximate groups and approximate lattices. It was then suggested that a more general version of ``geometric approximate group theory'' should be developed. The present book is meant as a first approximation of what such a theory might look like.
\end{remark}

\section{A mantra for the reader}

Given that this book is primarily concerned with ``finitely-generated'' approximate groups, it is important to point out that even defining what a ``finitely-generated'' approximate group is, is a non-trivial task: Proposition \ref{FinGenGroup} below lists eight different characterizations of finitely-generated groups, all of which can also be formulated for approximate groups, some in several ways. In this wider context, the conditions will no longer be equivalent, and in fact there does not seem to be a single ``correct'' definition of a finitely-generated approximate group -- depending on which results one would like to prove, one or the other condition may be preferable. 

This situation is somewhat prototypical for this whole book: All of our definitions and even the majority of our theorems and their proofs specialize to well-known classical definitions, theorems and proofs in the group case. At a first glance it might thus seem that geometric group theory generalizes in a straight-forward way to approximate groups. There are, however, at least three reasons why the generalizations presented here are more involved than meets the eye.

Firstly, there are of course many classical results in geometric group theory which \emph{do not} generalize to approximate groups; the fact that these are not mentioned in this book leads to a kind of ``survivor bias'' that the reader should keep in mind. We certainly do not claim that every result in geometric group theory has a counterpart for approximate groups.

Secondly, many basic definitions in the group case admit more than one reasonable generalization to the approximate group case. It was thus one of the main tasks in writing this book to find \emph{one possible} combination of definitions that lead to a non-trivial theory. While we do believe that we succeeded in this task, as the following roughly 200 pages indicate, we do not claim that we have found the \emph{only} possible combination of definitions that works. In fact, it is quite likely that some of our definitions will have to be modified in light of new examples and phenomena that we are not yet aware of, and some of our definitions have already been modified several times during the writing of this book. To give a concrete example, the definition of a ``quasikernel'' of a quasimorphism has been modified several times, and while all of the different notions generalize the kernel of a group homomorphism, most of them did not fit into our current framework. Only time will tell whether the current definition will be the final one. Nevertheless, we find it important to describe one tentative framework for geometric approximate group theory in detail, so that this framework can serve as a basis for further investigations.

Finally, one has to keep in mind that certain basic group theoretic arguments are not available to us. A shocking example is the fact that the intersection of two approximate subgroups need not be an approximate subgroup (see Example \ref{IntersectionFail}). In many of these situations there is a work-around -- e.g.\ the intersection of the \emph{squares} of two approximate subgroups is an approximate subgroup, but these workarounds may interact with each other in unpleasant ways. Therefore, for most of our proofs, we did not start out from a known group-theoretic proof and then work by adding a number of workarounds. Rather, we would often try to find group-theoretic arguments which generalize well to approximate groups and then try to reprove classical results using only (or mostly) these arguments. In other cases we were able to find the ``correct'' group-theoretical proof, which ``generalizes easily'' to approximate groups -- usually after first considering different proofs, which do not generalize. 

For all of these reasons, the step from geometric group theory to geometric approximate group theory is larger than a first glance at our results, and even their proofs, may suggest. Throughout this book, we will choose the old words of Garret Birkhoff as our mantra:

\begin{flushright}
\textit{``In carrying out the program […], it has been found essential to proceed with great care. For at every step we have a variety of definitions to choose from, and it is only after a detailed examination that we can judge of their relative fruitfulness.''} -- \cite{Birkhoff}
\end{flushright}

\section{Three approaches to geometric group theory}

What is geometric approximate group theory? Following our, i.e.\ Birkhoff's mantra we should look at different formulations of geometric group theory and then decide which of these formulations can be adapted most ``fruitfully'' to approximate groups.

What all formulations of geometric group theory have in common is that, with every group from a certain class of groups (e.g.\ finitely-generated, countable or locally compact second countable groups), one associates a class of metric spaces. Any invariant of this class of metric spaces is then considered to be a ``geometric invariant'' of the group.

\begin{remark}[Geometric group theory via Cayley graphs]
Most elementary textbooks favor the concrete approach of starting with a \emph{finitely-generated} group $G$. With every finite generating set $S$ of $G$, one can then associate the corresponding \emph{Cayley graph} (or ``Dehn Gruppenbild'') $\mathrm{Cay}(G,S)$, which is a locally finite connected graph and hence can be considered as a proper geodesic metric space.\footnote{See Appendix \ref{AppendixLSG} for our terminology concerning metric spaces and their large-scale geometry.} Cayley graphs with respect to different finite generating sets are quasi-isometric to each other (in fact, even bi-Lipschitz), and hence quasi-isometry invariants of their Cayley graphs yield geometric invariants of finitely-generated groups.
\end{remark}
\begin{remark}[Geometric group theory via geometric actions]\label{GGT2}
While the Cayley graph approach to geometric group theory is the most elementary one, historically geometric group theory did not arise from the study of group actions on graphs, but rather from the actions of Fuchsian and Kleinian groups on the corresponding hyperbolic spaces. To include such actions into the theory, one observes that both the action of a finitely-generated group on its Cayley graph and, say, the action of a surface group on the hyperbolic plane are \emph{proper} and \emph{cobounded}; such actions are sometimes called \emph{geometric}. It is a fundamental result of geometric group theory, discovered by Schwarz and re-discovered by Milnor, that all proper geodesic metric spaces on which a given group $G$ acts geometrically are quasi-isometric to each other. Quasi-isometry invariants of such spaces thus yield geometric invariants of the acting group, and this provides a second approach to geometric group theory.

It turns out a posteriori that the class of groups which admit geometric actions on proper geodesic metric spaces coincides with the class of finitely-generated groups, hence the scope of this second approach is not larger than the scope of the first approach. However, from the approximate group point of view this has the advantage that it does not require the notion of a generating set, which is problematic for approximate groups.
\end{remark}

\begin{remark}[The pseudo-metric approach to geometric group theory]
There is actually a third approach to geometric group theory, which is less well-known, but works in larger generality. This approach is based on the observation that if $G$ is a countable group and $d$, $d'$ are any two proper left-invariant pseudo-metrics on $G$ (cf.\ Definition \ref{DefPseudoMetric} and Section \ref{SecGroupMetrics}), then $(G,d)$ and $(G,d')$ are coarsely equivalent (see Proposition \ref{CoarseGroupClassCountable}), hence any coarse invariant of $(G,d)$ is a geometric invariant of $G$. Note that if $G$ acts geometrically on a metric space $(X, d_X)$ and $o \in X$ is an arbitrary basepoint, then we can define a proper left-invariant pseudo-metric $d$ on $G$ by
\[
d(g,h) \coloneqq  d_X(g.o, h.o).
\]
Then $(G,d)$ is coarsely equivalent to $(X,d_X)$ via the orbit map $g \mapsto g.o$, and hence the coarse invariants of $(G,d)$ coincide with the coarse invariants of $(X,d_X)$. Thus, even in this third approach, geometric actions of $G$ can be used to compute geometric invariants of $G$, as in the second approach. 

The main disadvantage of this approach is that for different proper left-invariant metrics $d$, $d'$ the spaces $(G,d)$ and $(G,d')$ will only be coarsely equivalent, rather than quasi-isometric, hence the geometric invariants which can be defined for general countable groups by this approach are rather limited.

This problem can be remedied using the basic observation, first recorded by Gromov in \cite{Gromov} as a ``trivial lemma'', that every coarse equivalence between \emph{geodesic} metric spaces, or even \emph{large-scale geodesic} metric spaces in the sense of Definition \ref{DefLargeScaleGeodesic}, is automatically a quasi-isometry. Using this observation, we may thus proceed as follows: 

Denote by $[G]_c$ the class of all metric spaces which are coarsely equivalent to $(G,d)$ for some (hence any) proper left-invariant metric on $G$ and define a subclass $[G]\subset [G]_c$ by
\[
[G] \coloneqq  \{X \in [G]_c \mid X \text{ is large-scale geodesic}\}.
\]
Then all representatives of $[G]$ are mutually quasi-isometric by Gromov's trivial lemma, and hence quasi-isometry invariants of such representatives yield geometric invariants of $G$, provided $[G] \neq \emptyset$.

Somewhat surprisingly, it turns out that the condition $[G] \neq \emptyset$ is also equivalent to $G$ being finitely-generated, and in this case every Cayley graph of $G$ and every proper geodesic metric space $X$ on which $G$ acts geometrically is contained in the class $[G]$. This shows that the third approach yields the same geometric invariants for finitely-generated groups as the other two approaches. However, as we will see, it is much better adapted to possible generalizations to approximate groups.
\end{remark}

\section{Fundamental notions of geometric approximate group theory}
We now explain how the pseudo-metric approach to geometric group theory carries over to approximate groups.\footnote{We will see later that the other two approaches can be made to work, too, but this requires additional effort.}

Throughout this section let $G$ be a group and let $\Lambda \subset G$ be a countable approximate subgroup. We denote by $\Lambda^\infty < G$ the (countable) subgroup generated by $\Lambda$ so that $(\Lambda, \Lambda^\infty)$ is an approximate group. If $H$ is any countable subgroup of $G$ containing $\Lambda^\infty$, for example $\Lambda^\infty$ itself, and $d_H$ is a proper, left-invariant pseudo-metric on $H$, then we refer to the restriction of $d_H$ to $\Lambda$ as an \emph{admissible metric} on $\Lambda$. We then have the following basic observation (see Lemma \ref{ExternalQIType}):
\begin{lemma}\label{CoarseLemmaIntro} If $d$, $d'$ are any two admissible metrics on $\Lambda$, then $(\Lambda, d)$ and $(\Lambda, d')$ are coarsely equivalent.
\end{lemma}
We now fix an admissible metric $d_o$ on $\Lambda$ and define (see Definition \ref{def: canonical coarse class}):
\begin{definition}
The \emph{coarse class} of the approximate group $(\Lambda, \Lambda^\infty)$ is 
\[
[\Lambda]_c \coloneqq  \{X \text{ metric space }\mid X \text{ is coarsely-equivalent to }(\Lambda, d_o)\}.
\]
\end{definition}
\begin{remark}[Invariance properties of the coarse class]\label{InvarianceCoarseClassIntro} 
By Lemma \ref{CoarseLemmaIntro}, the coarse class $[\Lambda]_c$ is independent of the choice of $d_o$. In fact, it depends only on the isomorphism class of $(\Lambda, \Lambda^\infty)$ in a suitable sense.

Defining isomorphisms of approximate groups is actually a non-trivial task. The most naive notion is that of ``global isomorphisms'', i.e.\ isomorphisms of the ambient groups preserving the approximate subgroups, and $[\Lambda]_c$ is certainly invariant under such global isomorphisms. However, there are also much weaker notions of isomorphism for approximate groups, and $[\Lambda]_c$ is invariant under most of these. In fact, anticipating the terminology of Chapter \ref{ChapAlgebra}, we are going to show that $[\Lambda]_c$ is invariant under all $2$-local isomorphisms, in particular under all Freiman $3$-isomorphisms (see Corollary \ref{CoarseType2Local}). It is also invariant under a suitably defined notion of commensurability (see Definition \ref{def:syndetic, commensurable} and Lemma \ref{CommCoarseEquiv}).
\end{remark}
 If now $\cI$ is any coarse invariant of metric spaces, then we can define the corresponding geometric invariant of $\Lambda$ (see Remark \ref{CoarseInvariants}), by
\begin{equation}\label{InvariantIntro}
\cI(\Lambda) \coloneqq  \cI(X), \text{ where }X \in [\Lambda]_c \text{ is arbitrary}.    
\end{equation}
\begin{example}[Asymptotic dimension of approximate groups]
A numerical coarse invariant of metric spaces is Gromov's asymptotic dimension (\cite{Gromov}, see Definition \ref{def: asdim} below). Using \eqref{InvariantIntro} we define the \emph{asymptotic dimension} $\asdim \Lambda$ of an approximate group $(\Lambda, \Lambda^\infty)$ 
(see Definition \ref{def: asdim lambda}). By Remark \ref{InvarianceCoarseClassIntro}, asymptotic dimension is invariant under $2$-local isomorphisms and commensurability.
\end{example}
As in the group case, most of the finer invariants of geometric group theory require us to consider geodesic representatives of $[\Lambda]_c$. For technical reasons, it is more convenient to actually work with \emph{large-scale geodesic} representatives, i.e. representatives which are quasi-isometric to a geodesic metric space. This leads to the following definition (see Construction \ref{IntQI})
\begin{definition}\label{DefGFG} The approximate group $(\Lambda, \Lambda^\infty)$ is called \emph{geometrically fi\-ni\-te\-ly-generated} if
\[
[\Lambda]_{\mathrm{int}} \coloneqq  \{X \in [\Lambda]_c \mid X \text{ large-scale geodesic}\} \neq \emptyset.
\]
In this case, $[\Lambda]_{\mathrm{int}}$ is called the \emph{internal QI type} of $(\Lambda, \Lambda^\infty)$. 
\end{definition}
If $\Lambda = \Lambda^\infty$ happens to be a group, then it is geometrically finitely-generated if and only if it is finitely-generated as a group (cf.\ Proposition \ref{FinGenGroup}). This explains the second half of the term
``geometrically finitely-generated''. The adjective ``geometrically'' refers to the fact that, as we will see below, there are several different notions of finite generation for approximate groups, and among these the notion from Definition \ref{DefGFG} is the most geometric one. Note that the internal QI type inherits all the invariance properties of the coarse class, i.e.\ it is invariant under $2$-local isomorphisms and commensurability.


\begin{remark}[Apogees and generalized Cayley graphs]
Assume that $(\Lambda, \Lambda^\infty)$ is geometrically finitely-generated.

From Gromov's trivial lemma it follows that any two representatives of $[\Lambda]_{\mathrm{int}}$ are quasi-isometric to each other. Any quasi-isometry invariant of geodesic metric spaces can thus be used to define a geometric invariant of $(\Lambda, \Lambda^\infty)$ by applying it to a representative from $[\Lambda]_{\mathrm{int}}$. Therefore it is important to find good representatives for practical computations.

It can be shown that $[\Lambda]_{\mathrm{int}}$ can be represented by a proper geodesic metric space, and we call any such space an \emph{apogee} for $(\Lambda, \Lambda^\infty)$ (see Definition \ref{def: apogee}). In fact, every geometrically finitely-generated group admits an apogee which is a locally finite graph. Such graphs are then the natural generalizations of Cayley graphs in our context. 

Describing the class of metric spaces which arise as apogees of approximate groups is an interesting problem. We will show that any apogee must be quasi-cobounded (cf.\ Definition \ref{def: quasi cobounded}) and of coarse bounded geometry (cf.\ Definition \ref{Weinberger}), and actually many of our results concerning approximate groups hold for quasi-cobounded proper geodesic metric spaces of coarse bounded geometry.
\end{remark}
\begin{remark}[Internal distortion]
A new phenomenon in geometric approximate group theory, which has no counterpart in geometric group theory, is internal distortion. 
Let us say that an approximate group $(\Lambda, \Lambda^\infty)$ is \emph{algebraically finitely-generated} if $\Lambda^\infty$ is finitely-generated as a group. In this case, every finite-generating set $S$ for $\Lambda^\infty$ defines a word metric $d_S$ on $\Lambda^\infty$, and we refer to $d_S|_{\Lambda \times \Lambda}$ as an \emph{external metric} on $\Lambda$. Any two external metrics are quasi-isometric, and hence we can define the \emph{external QI type} of $\Lambda$ (see Definition \ref{def: Lambda ext}), as
\[
[\Lambda]_{\mathrm{ext}} = \{X \in [\Lambda]_c \mid X \text{ is quasi-isometric to }(\Lambda, d_S|_{\Lambda \times \Lambda})\}.
\]
Note that this notion depends in a very explicit way on $\Lambda^\infty$. In particular,
it is not a commensurability invariant (see Example \ref{Stark1}) and we cannot replace $\Lambda^\infty$ by a larger group in the definition of an external metric. 
If $\Lambda = \Lambda^\infty$ is a countable group, then it is geometrically finitely-generated if and only if it is (algebraically) finitely-generated, and in this case $[\Lambda]_{\mathrm{int}} = [\Lambda]_{\mathrm{ext}}$. However, while every geometrically finitely-generated approximate group is algebraically finitely-generated (see Corollary \ref{GFGAFG}), the converse does not hold (see Example \ref{AFGnotGFG}). Moreover, even if $(\Lambda, \Lambda^\infty)$ is geometrically, and hence algebraically finitely-generated, it may happen that 
$[\Lambda]_{\mathrm{int}} \neq [\Lambda]_{\mathrm{ext}}$ (see Example \ref{ExDistorted}). We then say that $(\Lambda, \Lambda^\infty)$ is \emph{distorted}; otherwise it is called \emph{undistorted} (see Definition \ref{def: undistorted}). 

It turns out that many interesting examples of approximate groups, in particular, quasikernels of many naturally occurring quasimorphisms, are  distorted. Hence it is important to develop the basic theory without the assumption that the approximate group in question is undistorted. This has drastic consequences for the theory. To begin with, if $\cI$ is a quasi-isometry invariant of metric spaces, then with every geometrically finitely-generated approximate group we can associate two corresponding invariants: The internal version $\cI_{\mathrm{int}}(\Lambda)$ is defined as $\mathcal I(X)$, where $X$ is any representative of $[\Lambda]_{\mathrm{int}}$, and the external version  $\cI_{\mathrm{ext}}(\Lambda)$ is defined as $\mathcal I(X)$, where $X$ is any representative of $[\Lambda]_{\mathrm{ext}}$. For example, whereas a finitely-generated group can be of polynomial, intermediate or exponential growth, in the approximate setting we have to distinguish carefully between internal and external growth. 

It is easy to give examples of approximate groups which have polynomial internal growth and exponential external growth (see Example \ref{PolGrowthFINec}). However, the difference between internal and external \emph{polynomial} growth disappears if one is willing to pass to a commensurable approximate group -- we do not know whether this is the case for other growth types. In fact, we have the following more precise result which
extends Gromov's polynomial growth theorem to approximate groups; see Definition \ref{def:syndetic, commensurable} for the notion of a syndetic\footnote{A syndetic subgroup of a group is just a finite index subgroup. Since there is no notion of index for approximate subgroups, we avoid this terminology in the context of approximate groups.} approximate subgroup.

\begin{theorem}[Polynomial growth theorem]\label{PGTIntro}
For a geometrically finitely-generated approximate group  $(\Lambda, \Lambda^\infty)$ the following statements are equivalent:
\begin{enumerate}[(i)]
\item  $(\Lambda, \Lambda^\infty)$ has internal polynomial growth.
\item There exists a syndetic approximate subgroup $\Lambda' \subset \Lambda^{16}$ such that $(\Lambda', (\Lambda')^\infty)$ has external polynomial growth.
\item  There exists a syndetic approximate subgroup $\Lambda' \subset \Lambda^{16}$ such that the group $(\Lambda')^\infty$ has polynomial growth.
\item There exists a syndetic approximate subgroup $\Lambda' \subset \Lambda^{16}$  such that $(\Lambda')^\infty$ is nilpotent.
\end{enumerate}
In particular, up to passing to a commensurable approximate group, internal polynomial growth, external polynomial growth and polynomial growth of the ambient group are all equivalent for a given approximate group.
\end{theorem}
The fact that in the statement of Theorem \ref{PGTIntro} we have to replace $\Lambda$ by (a syndetic subset of) a finite power should not be surprising, since we have seen that even the most basic operations for approximate groups, like taking an intersection, require us to pass from $\Lambda$ to $\Lambda^2$. Since we have to do several (but not too many) such elementary operations in our proof, the exponent is larger than $2$ (but not much larger). Our specific proof yields an exponent of $16$, but we have made no effort to optimize this exponent, hence we do not believe that the number $16$ carries any specific significance.

The proof of Theorem \ref{PGTIntro} follows Hrushovski's proof of Gromov's polynomial growth theorem, but uses the more precise results of Breuillard--Green--Tao \cite{BGT} instead of the results from  \cite{Hrushovski}.
It provides an interesting connection between geometrically finitely-generated approximate groups of polynomial growth and the extremely developed theory of finite approximate groups. Namely, the first step in the proof of Theorem \ref{PGTIntro} is the observation that if $(\Lambda, \Lambda^\infty)$ has internal polynomial growth, then there exists an exhaustion of $\Lambda^\infty$ by finite $k$-approximate subgroups for some fixed $k \in \bN$. The rest of the argument is then exclusively concerned with the structure theory of finite $k$-approximate groups. 

In general, it seems to be the case that the structure theory of approximate groups of polynomial growth is closely related to the structure theory of finite approximate groups, whereas the finite theory has only very limited impact on the theory of approximate groups of exponential growth. For this reason, the current book mostly focuses on the class of finitely-generated approximate groups of exponential growth. As we will see, this class contains, in particular, all non-elementary hyperbolic approximate groups.

\section{From isometric actions to quasi-isometric quasi-actions}
The geometric theory of undistorted approximate groups mirrors closely the geometric theory of finitely-generated groups. Examples of such approximate groups are given by uniform model sets, or more generally uniform approximate lattices. However, for distorted approximate groups this is not at all the case.
\end{remark}
\begin{remark}[Isometric actions of approximate groups]
In view of Remark \ref{GGT2}, one could say that geometric group theory is chiefly concerned with certain isometric actions on proper geodesic metric spaces. It turns out, however, that this point of view generalizes poorly to approximate groups. Let us elaborate on this point:

There is no problem in defining an isometric action of an approximate group $(\Lambda, \Lambda^\infty)$ on a metric space $X$ -- it is just an isometric action of $\Lambda^\infty$ on $X$. There is also no problem in defining what it means for such an action to be proper, cobounded or geometric: properness means that the map $\Lambda \times X \to X \times X$, $(\lambda, x) \mapsto (x, \lambda.x)$ is proper and coboundedness means that the ``quasi-orbit'' $\{\lambda.x \mid \lambda \in \Lambda, x\in X\}$ is relatively dense in $X$ for some (hence any) $x \in X$.

However, while every finitely-generated group admits a geometric isometric action on a proper geodesic space, e.g.\ on a Cayley graph, this is not the case for approximate groups. In fact, internal distortion turns out to be an obstruction (see Corollary \ref{cor:undistort}).
\end{remark}

\begin{proposition}\label{DistortionObstruction}
Every geometrically finitely-generated approximate group which admits a geometric isometric action on a proper geodesic metric space is undistorted.
\end{proposition}
For undistorted approximate groups, the usual Milnor--Schwarz lemma works (see Theorem \ref{ThmMS}):
\begin{proposition} If an undistorted approximate group $(\Lambda, \Lambda^\infty)$ acts geometrically and isometrically on a proper geodesic metric space $X$, then $X$ is an apogee for $(\Lambda, \Lambda^\infty)$.
\end{proposition}
However, in order to include distorted approximate groups into the theory, one has to leave the realm of isometric actions and consider the larger class of ``quasi-isometric quasi-actions'', or qiqacs, for short.
\begin{remark}[Quasi-isometric quasi-actions]
If $G$ is a group and $X$ is a metric space, then a qiqac of $G$ on $X$ is a map $G \times X \to X$, $(g,x) \mapsto g.x$ such that the maps $(x \mapsto g.x)_{g \in G}$ are quasi-isometries with uniform quasi-isometry constants and such that there exists a constant $C$ satisfying $d(g.(h.x), gh.x) \leq C$ for all $g , h \in G$ and $x \in X$.

Such qiqacs arise naturally in geometric group theory (see e.g.\ \cite{MSW}): If $\Gamma$ is a finitely-genera\-ted group and $X \in [\Gamma]$, then there may or may not exist a geometric action of $\Gamma$ on $X$ -- this is the famous \emph{QI rigidity problem}. But there always exists a geometric, i.e. proper and cobounded qiqac of $\Gamma$ on $X$. 

The reason that qiqacs do not feature more prominently in geometric group theory, outside of QI rigidity, is that every geometric qiqac is quasi-conjugate to an actual geometric isometric action, albeit on a different space. As Proposition \ref{DistortionObstruction} shows, this is no longer the case for approximate groups, and as a consequence, qiqacs play a central role in geometric approximate group theory.

Because of the various QI constants involved, the precise definition of a qiqac of an approximate group is somewhat involved -- see Definition \ref{def: qiqac} below. Our definition is certainly not the only possible definition of a qiqac of an approximate group -- certain uniformity conditions could be strengthened or weakened in the definition. However, as the following result shows, our definition leads to a reasonable theory (cf. Theorem \ref{MSMain}):
\end{remark}
\begin{theorem}[Approximate Milnor--Schwarz lemma]\label{MSIntro}
For every approximate group $(\Lambda, \Lambda^\infty)$ the following implications hold:
\begin{enumerate}[(i)]
\item If $(\Lambda, \Lambda^\infty)$ admits a geometric qiqac on a coarsely connected proper metric space $X$, then it is \emph{algebraically} finitely-generated.
\item Conversely, if $(\Lambda, \Lambda^\infty)$ is \emph{geometrically} finitely-generated, then it admits a geometric qiqac on a proper geodesic metric space $X$, and any such space $X$ is an apogee for $(\Lambda, \Lambda^\infty)$.
\end{enumerate}
\end{theorem}
It follows from (i) and (ii) that every geometrically finitely-generated approximate group (in fact, every syndetic approximate subgroup of a geometrically finitely-generated approximate group) is algebraically finitely-generated. 

\section{Hyperbolic approximate groups}
The results of the previous section indicate that geometric approximate group theory should be concerned with the study of quasi-isometric quasi-actions (qiqacs) of approximate groups on ``nice'' metric spaces. The precise techniques will necessarily depend on the class of metric spaces under consideration.

\begin{remark}[Hyperbolic qiqacs]
Historically, one of the early successes of geometric group theory was the theory of hyperbolic groups, i.e. groups acting geometrically on proper geodesic hyperbolic spaces \cite{Gromov, GhysdelaHarpe}. This theory has been extended in many different directions, notably to include finitely-generated groups with only very weak hyperbolicity properties (see \cite{Osin_AcylHyp} for an example). 

Therefore the remainder of this introduction will focus on qiqacs of approximate groups on hyperbolic spaces and the comparison of these results to similar results in the group case. We hasten to point out that a substantial part of this book is concerned with qiqacs on non-hyperbolic proper geodesic metric spaces. This, however, requires additional technical machinery, such as the use of Morse boundaries instead of Gromov boundaries, hence we will not mention the more general results in the introduction.
\end{remark}
We will be interested in approximate groups which admit proper qiqacs on proper geodesic hyperbolic spaces. We first consider the case of geometric actions (cf.\ Definition \ref{DefHypAG}):
\begin{definition}
A geometrically finitely-generated approximate group is cal\-led \emph{hyperbolic} if it admits a geometric qiqac on a proper geodesic hyperbolic space.
\end{definition}
We observe that if $X$ is an apogee for a hyperbolic approximate group, then there is a canonical qiqac  of $(\Lambda, \Lambda^\infty)$ on $X$, and this extends to an action of $(\Lambda, \Lambda^\infty)$ on the Gromov boundary $\partial X$ of $X$ by homeomorphism called the \emph{boundary action}. We then say that $(\Lambda, \Lambda^\infty)$ is \emph{non-elementary} if every syndetic approximate subgroup $\Lambda' \subset \Lambda$ acts without fixpoints (see Definition \ref{DefNonEl}).
\begin{example}
Examples of hyperbolic approximate groups include uniform model sets in isometry groups of symmetric spaces of rank one (i.e.\ real, complex and quaternionic hyperbolic spaces and the octonion plane) and uniform model sets in automorphism groups of trees, as well as quasi-convex approximate subgroups of hyperbolic groups (see Example \ref{ExampleHAG1} and Example \ref{ExampleHAG2}). Unfortunately, all of these examples happen to be quasi-isometric to hyperbolic groups. In fact, it is currently an open problem whether a hyperbolic approximate group exists which is not quasi-isometric to a hyperbolic group. 

There are several reasons for this lack of examples: Firstly, many of our constructions work best in hyperbolic spaces $X$ with large (in particular, non-discrete) isometry groups, but these groups often admit uniform lattices. These uniform lattices are then quasi-isometric to any approximate group which admits a uniform qiqac on $X$. Secondly, we are able to prove several strong rigidity results, which, for example, force every quasi-convex approximate subgroup of a hyperbolic group to be quasi-isometric to the ambient group (see Corollary \ref{Rigidity0}) and every non-elementary hyperbolic approximate group of asymptotic dimension $1$ to be quasi-isometric to a finitely-generated free group (see Theorem \ref{Asdim1Case}).
\end{example}
The (still ongoing) quest for hyperbolic approximate groups which are not quasi-isometric to hyperbolic groups has motivated us to take a closer look at geometric properties of hyperbolic approximate groups and to compare them to geometric properties of hyperbolic groups. It turns out that many geometric properties of hyperbolic groups extend to hyperbolic approximate groups, which makes it hard to distinguish the two classes geometrically.
The following theorem summarizes some of our results in this regard (see Corollary \ref{HyperbolicApogees}, Theorem \ref{AsdimMainConvenient}, Proposition \ref{BoundaryMinimal}, Proposition \ref{BoundaryNonel} and Theorem \ref{ExponentialGrowth} respectively).
\begin{theorem}[Hyperbolic approximate groups] Let $(\Lambda, \Lambda^\infty)$ be a hyperbolic approximate group and let $X$ be an apogee for $(\Lambda, \Lambda^\infty)$. Then:
\begin{enumerate}[(i)]
    \item $X$ is a visual proper geodesic hyperbolic space which is quasi-isometric to a closed convex subset of some real hyperbolic space $\bH^n$. 
    \item $\partial X$ is doubling and locally quasi-self-similar with respect to any visual metric. 
    \item $\asdim X = \dim \partial X+1$. 
\end{enumerate}
If $(\Lambda, \Lambda^\infty)$ is non-elementary, then moreover the following results hold:
\begin{enumerate}[(i)]
     \setcounter{enumi}{3}
     \item The boundary action of $(\Lambda, \Lambda^\infty)$ on $\partial X$ is minimal. 
    \item $\partial X$ is infinite and homeomorphic to a perfect compact subset of a sphere. 
    \item $X$ has exponential growth. 
\end{enumerate}
\end{theorem}
All of these results are classical for hyperbolic groups: (i) is the celebrated Bonk--Schramm embedding theorem \cite{Bonk-Schramm}, (iii) was conjectured by Gromov in \cite{Gromov} and established (after a lot of work by many people) by Buyalo and Lebedeva \cite{BuyaloLebedeva}. The proof of (vi) in the group case is based on the ping-pong lemma, which ensures that a non-elementary hyperbolic group contains a free semigroup; no such result is true in our general setting. Instead, if $\asdim \Lambda > 1$ then (vi) can be proved by bounding the exponential growth rate of $\Lambda$ from below by the dimension of $\partial X$, which is strictly positive by (iii). The case $\asdim \Lambda = 1$ then has to be dealt with separately (see Theorem \ref{Asdim1Case}), using an embedding theorem of Buyalo and Schroeder \cite{BuyaloSchroeder}.

\section{Convex-cocompact quasi-actions and conical limit points}

The study of hyperbolic approximate groups amounts to the study of quasi-isometric quasi-actions on hyperbolic spaces, which are both proper and cocompact. If one weakens the notions of cocompactness, then one is naturally lead to the notion of convex cocompact quasi-actions and the related notion of conical limit points.

\begin{remark}[Limit sets of proper qiqacs]
If one considers a proper qiqac of an approximate group $(\Lambda, \Lambda^\infty)$ on a proper geodesic hyperbolic space $X$ which is not cobounded, then there is still a boundary action of $(\Lambda, \Lambda^\infty)$ by homeomorphisms on the Gromov boundary $\partial X$ of $X$. In this situation, a $\Lambda^\infty$-invariant subset of the boundary is given by the \emph{limit set}
\[
\cL(\Lambda) \coloneqq  \overline{\{\lambda.o \mid \lambda \in \Lambda\}} \cap \partial X,
\]
where $o \in X$ is an arbitrary basepoint (see Definition \ref{DefGromovLimitSet}). Elements of $\cL(\Lambda)$ are called \emph{limit points} of the qiqac. Note that, by definition, every limit point $\xi$ is the limit of a sequence of the form $\lambda_n.o$ with $\lambda_n \in \Lambda$. However, it is not clear whether this sequence can be chosen to be uniformly close to a geodesic representing $\xi$. If this is the case, then we call $\xi$ a \emph{conical limit point} of the qiqac (see Definition \ref{def:conical limit point}). The subset $\cL^{\mathrm{con}}(\Lambda) \subset \cL(\Lambda)$ of conical limit points often contains additional information compared to $\cL(\Lambda)$.
\end{remark}
A striking illustration of the difference between limit points and conical limit points can be given as follows: Consider a free group $F_r$ of rank $r \geq 3$. With every cyclically reduced non-self-overlapping word $u$ of length $\geq 2$ in $F_r$ one can associate an approximate subgroup $\Lambda_u$ of $F_r$, the quasikernel of the associated cyclic counting quasimorphism. This approximate subgroup then admits an isometric action on the Cayley tree $T_{2r}$ of $F_r$, and we have (see Theorem \ref{QuasikernelsConical}):
\begin{theorem}[Reconstruction from conical limit points]
Let $u,v \in F_r$ be two cyclically reduced non-self-overlapping words of length $\geq 2$. Then $\cL(\Lambda_u) = \cL(\Lambda_v)$, whereas we have $\cL^{\mathrm{con}}(\Lambda_u) = \cL^{\mathrm{con}}(\Lambda_v)$ if and only if $u = v^{\pm 1}$.
\end{theorem} %
On the other hand, there are also interesting qiqacs for which every limit point is conical. 
\begin{remark}[Weak hulls and convex cocompact qiqacs]
Let $X$ be a proper geodesic hyperbolic space. Given a closed subset $Z \subset \partial X$ of the Gromov boundary, we define its \emph{weak hull} $\mathfrak{H}(Z)$ as the set of all geodesics which connect points in $Z$ (see Definition \ref{def:weak hull}). 

In particular, given a proper qiqac of an approximate group $(\Lambda, \Lambda^\infty)$ with unbounded orbits, we can consider the weak hull $\mathfrak{H}(\cL(\Lambda))$ of the limit set. This weak hull is not quite invariant under $\Lambda^\infty$, but it is ``quasi-invariant'' under $\Lambda$. This is enough to define a ``restricted qiqac'' of $\Lambda$ on $\mathfrak{H}(\cL(\Lambda))$, and we say that the qiqac is \emph{convex cocompact} if this restricted qiqac is cobounded.
\end{remark}
We then have the following characterization (see Corollary \ref{thm:quasi-convex equiv to cocompact}); in the group case, this reduces to a classical result of Swenson \cite{Swenson}.
\begin{theorem}[Convex cocompact qiqacs] A proper qiqac of an infinite approximate group $(\Lambda, \Lambda^\infty)$ on a proper geodesic hyperbolic space $X$ is convex cocompact if and only if it has quasi-convex quasi-orbits. In this case, all limit points of the qiqac are conical.
\end{theorem} 
This theorem readily implies that quasi-convex approximate subgroups of hyperbolic groups are syndetic in, hence quasi-isometric to the ambient group (see Corollary \ref{Rigidity0}), which makes it hard to use hyperbolicity of the ambient group to establish hyperbolicity of an approximate group.
\newpage
    \section*{Acknowledgements}

Just like its authors, this book has been growing slowly but steadily over the last seven years. The authors would thus like to express their gratitude to all the people and institutions which have supported the project over this long period of time. In particular, the authors are indebted to the following people:
\begin{itemize}
    \item To Michael Bj\"orklund, without whom this whole theory would not exist.
    \item To Simon Machado for numerous explanations concerning the work of Breuillard--Green--Tao, and in particular for contributing Appendix \ref{AppendixBGT}.
    \item To Nina Lebedeva for providing a list of corrections for
 the  proof of \cite[Thm.\ 1.1]{BuyaloLebedeva}, on which the proof of Theorem \ref{l-dim-leq-dim} is based.
 \item To Alexey Talambutsa for help with the proof of Theorem \ref{BlockersExist}.
 \item To Stefan Witzel, who suggested the definition of finiteness properties for countable approximate groups used in this book; the second half of Section 3.2 paraphrases joint work with him \cite{HaWi}.
 \item To Gabi Ben Simon, Laura Bonn, Emmanuel Breuillard, Sergei Buyalo, James Farre, Francesco Fournier-Facio, Elia Fioravanti, Slava Grigorchuk, Anton Hase, Udi Hrushovski, Nir Lazarovich, Waltraud Lederle, Leonard Rubin, Michah Sageev, Viktor Schroeder, Alessandro Sisto and Emily Stark and the anonymous referees for various comments, corrections and suggestions.
 \item To the members of the Geometry and Topology Seminar at the Technion for providing a stimulating and collaborative atmosphere over the years; this book can be seen as a continuation of the discussions which arose from various talks in this seminar.
\end{itemize}
Moreover, the authors would like to thank the following institutions for providing them with excellent working conditions during various stages of their collaboration: the Technion Faculty of Mathematics, Haifa, the Faculty of Mathematics at the University of Rijeka, the Institut f\"ur Algebra und Geometrie at the Karlsruher Institut f\"ur Technologie, the Mathematische Forschungsinstitut Oberwolfach and the University of Sulaymaniyah.

M.C.\ was supported in part by a Zuckerman STEM Leadership Fellowship, an ETH Postdoctoral Fellowship co-funded by a Marie Curie Actions for People COFUND Program, and an SNSF Ambizione Fellowship. T.H. and V.T. were supported by the Erasmus+ KA107 program through the 2016-2018 inter-institutional agreement between the Technion and the University of Rijeka. T.H. was moreover supported through the program ``Research in Pairs'' by the Mathematische Forschungsinstitut Oberwolfach in 2020. 

Finally, the authors would like to dedicate this book to Daniel and Sara and to the memory of Edward Cordes, Ivan Ivan\v si\' c and Ljudmila Toni\' c.\\

\begin{flushright}
M.\ Cordes, T.\ Hartnick, V.\ Toni\'c 
\end{flushright}

\chapter{Basic constructions of approximate groups}\label{ChapAlgebra}

In this chapter we introduce the notion of an approximate (sub)group and some slight variants thereof. We first give some general background concerning subsets and filtrations of groups and then focus on constructions of approximate subgroups. Generalizing the basic fact that intersections of subgroups and images and preimages of subgroups under group homomorphisms are again subgroups, we show that \emph{(local) direct images} (Proposition \ref{ImagesExist} and Corollary \ref{QuasiImages}), \emph{thick intersections} (Lemma \ref{Intersection}), \emph{(local) partial kernels} (Lemma \ref{Lem: partial kernel ag}), \emph{thick preimages} (Lemma \ref{ThickPreim} and Proposition \ref{TopHru}) and \emph{(partial) quasikernels} (Example \ref{PartialQuasikernel}) are approximate subgroups. We also give some constructions of approximate subgroups which have no counterpart in group theory. These include compact symmetric \emph{identity neighborhoods} (Example \ref{FundCommClass}) \emph{powers} of approximate subgroups (Proposition \ref{PropShuffling}), approximate subgroups from \emph{syndetic coarse filtrations} (Proposition \ref{PropCoarseFiltrationMain}), and \emph{model sets} (Example \ref{CuPModel}). As usual in geometric group theory, we want to consider approximate subgroups up to \emph{commensurability}, and all of our constructions can be carried out on the level of commensurability classes. In fact, by Theorem \ref{HruBaby} (due to Hrushovski), our constructions are sufficiently general to produce \emph{all} commensurability classes of approximate subgroups of a given group. 

\section{Subsets and filtrations of groups}
Before we begin our discussion of approximate groups, we discuss some general algebra of subsets of groups and describe a general framework of filtered groups. We will see later that approximate groups fall into this framework. Throughout this section, let $\Gamma$ be a group and let $A,B,C \subset \Gamma$ be subsets.

\begin{notation}[Subsets of groups]
The \emph{product set} of $A$ and $B$ and the \emph{inverse set} of $A$ are respectively defined as
\[
AB \coloneqq  \{ab\mid a \in A, b \in B\} \qand A^{-1} \coloneqq  \{a^{-1} \mid a \in A\}.
\]
We say that $A$ is \emph{symmetric} if $A = A^{-1}$. Given $k \geq 2$ we denote by
$A^k \coloneqq  AA^{k-1}$ the \emph{$k$th power} of $A$ and by $A^{-k} \coloneqq (A^{-1})^k = (A^k)^{-1}$ its inverse set. We also set $A^{0} = \{e\}$ and
say that $A$ is \emph{unital} if $e \in A$; we write $A^+ \coloneqq A \cup \{e\}$ for the \emph{unitalization} of $A$.
\end{notation}
\begin{remark}\label{Rem: properties AB}  We will use basic identities and inclusions between product sets and inverse sets without further comments. For example, we have $A^kA^l = A^{k+l}$ for all $k,l \in \bN_0$ (but \emph{not} for all $k,l \in \Z$), and if $A$ is unital, then $A \subset A^2 \subset A^3 \subset \dots$. Moreover, $(AB)^{-1} = B^{-1}A^{-1}$, $(A \cup B)C = AC \cup BC$ and $(A \cap B)C  \subset AC \cap BC$, and similarly
\[
(A\cap B)^k \subset A^k \cap B^k \qand A^k \cup B^k \subset (A \cup B \cup \{e\})^k \quad \text{for all } k \in \Z.
\]
We can combine these to show e.g.\ that if $C$ is symmetric, then
\[
(C^k A)\cap (C^l B) \subset C^l ((C^{l+k}A)\cap B) \quad \text{for all }l,k \in \bN_0.
\]
\end{remark}
We mention two less obvious inclusions:
\begin{lemma}[Rusza's Covering Lemma]\label{Rusza}
If $F \subset A$ is a maximal subset such that the sets $fB$ are pairwise disjoint as $f$ ranges over $F$, then
$A \subset FBB^{-1}$.
\end{lemma}
\begin{proof} Let $x \in A$. We then have $xB \cap fB \neq \emptyset$ for some $f\in F$, since otherwise we could enlarge $F$ by $x$, contradicting maximality. This implies that $x \in fBB^{-1} \subset FBB^{-1}$, and since $x \in A$ was arbitrary, the lemma follows.
\end{proof}
\begin{lemma}\label{LemmaAdditiveComb} If $F \subset \Gamma$ is finite, then there exists $F' \subset \Gamma$ with $|F'| \leq |F|$ such that $A \cap FB \subset F'(A^{-1}A \cap B^{-1}B)$. 
\end{lemma}
\begin{proof} Let $F_0 \coloneqq  \{f \in F \mid A \cap fB \neq \emptyset\}$; for every $f \in F_0$ we pick an element $x_f \in A \cap fB$ and set $F' \coloneqq  \{x_f \mid f \in F_0\}$. Then for $f \in F_0$ and $x \in A \cap fB$, 
\[
x = x_f (x_f^{-1}x) \in  F'(A^{-1}A \cap (fB)^{-1}fB)= F'(A^{-1}A \cap B^{-1}B),
\]
and hence $A \cap fB \subset F'(A^{-1}A \cap B^{-1}B)$. We deduce that
\[
A \cap FB  = A \cap F_0B  = \bigcup_{f \in F_0} A \cap fB \subset F'(A^{-1}A \cap B^{-1}B).
\]
Since $|F'| \leq |F_0| \leq |F|$ this proves the lemma.
\end{proof}
\begin{corollary}\label{lemmasymmtrick} If $A \subset FB$ for some finite $F \subset \Gamma$, then
\begin{equation}\label{SymmTrick}
A \subset F'(A^{-1}A \cap B^{-1}B) \text{ for some }F' \subset \Gamma \text{ with }|F'| \leq |F|.
\end{equation}
\end{corollary}
\begin{definition} \label{def:syndetic, commensurable}
We say that $A$ is \emph{left-syndetic}\index{left-syndetic} (respectively \emph{right-syndetic}\index{right-syndetic}) in $B$ if there exists a finite subset $F \subset \Gamma$ such that $A \subset B \subset AF$ (respectively $A \subset B \subset FA$). We say that $A$ and $B$ are \emph{left-commensurable}\index{left-commensurable} (respectively \emph{right-commensurable}\index{right-commensurable}) if there exist finite subsets $F_1, F_2 \subset \Gamma$ such that $A \subset BF_1$ and $B \subset A F_2$ (respectively $A \subset F_1B$ and $B \subset F_2A$). If $\Gamma$ is a topological group, then we modify the definitions above by requiring the sets $F$, $F_1$, $F_2$ in question be compact (rather than finite).
\end{definition}
Note that left-/right-commensurability are precisely the equivalence relations generated by left-/right-syndeticity.  
\begin{definition} The \emph{left-commensurator}\index{left-commensurator} of a subset $A \subset \Gamma$ is defined as 
\[
\mathrm{LComm}_\Gamma(A) = \{g \in \Gamma \mid \exists\, F \subset \Gamma \text{ finite}: \; gAg^{-1} \subset AF\}.
\]
\end{definition}
Dually, one can define the \emph{right-commensurator} $\mathrm{RComm}_\Gamma(A)$ so that
\begin{equation}\label{LeftRightCommInverse}
\mathrm{RComm}_\Gamma(A) = \mathrm{LComm}_\Gamma(A^{-1})^{-1}.
\end{equation}
Note that $\mathrm{LComm}_\Gamma(A)$ is a unital semigroup and that
\begin{equation}\label{LCommAlt}
g \in \mathrm{LComm}_\Gamma(A) \iff \exists \, F' \subset \Gamma \text{ finite}: \; gA \subset AgF'.
\end{equation}
From \eqref{LCommAlt} and \eqref{LeftRightCommInverse} one deduces immediately:
\begin{lemma}[Shuffling of finite sets]\label{LeftRightComm} If $B \subset \mathrm{LComm}_\Gamma(A)$ (respectively  $B \subset \mathrm{RComm}_\Gamma(A)$) is finite, then there exists a finite subset $B^* \subset \Gamma$  such that
\[
BA \subset AB^* \quad (\text{respectively } AB \subset B^*A).
\]
\end{lemma}
\begin{lemma}\label{CommDepOnClass} If $A$ and $B$ are left-commensurable, then \[\mathrm{LComm}_\Gamma(A)= \mathrm{LComm}_\Gamma(B).\]
\end{lemma}
\begin{proof} Let $F \subset \Gamma$ be finite with $B \subset AF$ and $A \subset BF$. If $gA \subset AgF'$ for some $F' \subset \Gamma$, then $gB \subset gAF \subset AgF'F \subset Bg (g^{-1}FgF'F)$, hence the lemma follows from \eqref{LCommAlt}.
\end{proof}
\begin{remark}[Commensurability of symmetric subsets] Every non-amenable group contains a left-syndetic set which is not right-syndetic (\cite{Paulsen}), but for symmetric sets the notions coincide, and we say $A$ is \emph{syndetic} in $B$. Similarly, two symmetric subsets $A$ and $B$ are left-commensurable if and only if they are right-commensurable; in this case we write $A\sim_{\mathrm{co}} B$ and say that $A$ and $B$ are \emph{commensurable}\index{commensurable}. We then denote by $[A]_{\mathrm co}$ the \emph{commensurability class}\index{commensurability class} of $A$.

If $A$ is symmetric, then $\mathrm{RComm}_\Gamma(A) = \mathrm{LComm}_\Gamma(A)^{-1}$ by \eqref{LeftRightCommInverse} and hence \[
\mathrm{Comm}_\Gamma(A) \coloneqq \mathrm{LComm}_\Gamma(A) \cap  \mathrm{RComm}_\Gamma(A)
\]
is a group, called the \emph{commensurator}\index{commensurator} of $A$ in $\Gamma$. One checks that, for $g \in \Gamma$,
\[
g \in \mathrm{Comm}_\Gamma(A)  \iff gAg^{-1} \sim_{\mathrm co} A.
\]
\end{remark}

\begin{remark}[Independence of the ambient group]\label{DefinedOverSubgroup} Suppose that $A$ and $B$ are left-commensurable and contained in a subgroup $\Gamma_0 < \Gamma$. We can then find a finite set $F_1$ such that $A \subset BF_1$ and such that $A \cap Bf_1 \neq \emptyset$ for all $f_1 \in F_1$. For any $f_1 \in F_1$ we then have $f_1 \in B^{-1}A \subset \Gamma_0$, and thus $F_1 \subset \Gamma_0$. By a symmetric argument we find $F_2 \subset \Gamma_0$ such that $B \subset AF_2$, and hence $A$ and $B$ are also left-commensurable as subsets of $\Gamma_0$. The same remark applies to right-commensurability.

In particular, every (left-/right-)commensurability class of subsets of $\Gamma_0$ extends to a unique (left-/right-)commensurability class of subsets of $\Gamma$. The latter class is then said to be \emph{supported on} $\Gamma_0$.
\end{remark}

\begin{definition} An ascending sequence $\{e\} \subset \Gamma_1 \subset \Gamma_2 \subset \Gamma_3 \subset \dots$ of subsets of $\Gamma$ is called a \emph{coarse filtration} if there exists a function $\rho: \bN \times \bN \to \bN$ such that
$\Gamma_k\Gamma_l \subset \Gamma_{\rho(k,l)}$.
It is called a \emph{filtration} if we can choose $\rho(k,l)= k+l$. The pair  $(\Gamma, (\Gamma_k)_{k \in \mathbb N})$ is then called a \emph{(coarsely) filtered} group\index{filtered group}\index{coarsely filtered group}, the group $\Gamma$ is called the \emph{ambient group}\index{ambient group} and $\Gamma_k$ is called the \emph{$k$th filtration step}\index{$k$th filtration step}.
\end{definition}
\begin{example}
 If $\Gamma$ is a group and $\Lambda \subset \Gamma$ is a unital subset, then we obtain a filtered group $(\Gamma, (\Gamma_k)_{k \in \mathbb N})$ by setting $\Gamma_k \coloneqq  \Lambda^k$. This filtered group is called the \emph{filtered group associated with the pair $(\Lambda, \Gamma)$}.
\end{example}
\begin{example}[Filtered commensurator]\label{FilteredCommensurator1}
Let $G$ be a group and $H < G$ be a subgroup. Given $g \in G$ we write 
\[ {}^gH \coloneqq gH g^{-1}\qand H(g) \coloneqq {}^gH \cap H\]
and define the \emph{complexity} of $g$ as  $c(g) \coloneqq \max\{[H: H(g)], [H:H(g^{-1})]\}$. Then the commensurator $\Gamma \coloneqq \Comm_G(H)$ consists precisely of the elements of finite complexity, and we denote by $\Gamma_n = \Comm_G^{(n)}(H)$ the elements of complexity at most $n$.

To see that this defines a coarse filtration on $\Gamma$, let $g_1 \in \Gamma_{d_1}$ and $g_2 \in \Gamma_{d_2}$ and observe that $[H:H(g_1^{-1}) \cap H(g_2)] \leq d_1d_2$.
For $\Lambda \coloneqq H(g_1g_2) \cap {}^{g_1}H$ this implies 
\begin{eqnarray*}
[{}^{g_1}H: \Lambda] &=& [{}^{g_1}H: {}^{g_1g_2}H \cap  H \cap {}^{g_1}H] \quad = \quad  [H: {}^{g_2}H \cap  {}^{g_1^{-1}}\!\!H \cap {}H]\\  & = & [H: H(g_1^{-1}) \cap H(g_2)] \quad \leq \quad d_1d_2.
\end{eqnarray*}
Since $\Lambda =  {}^{g_1g_2}H \cap  H \cap {}^{g_1}H =  {}^{g_1g_2}H \cap H(g_1) \subset H(g_1)$ we have
\begin{eqnarray*}
[H:\Lambda]  &=& [H:H(g_1)][H(g_1) : \Lambda]  \quad \leq \quad d_1 [H(g_1):\Lambda]\\ &\leq&  d_1 [{}^{g_1}H:H(g_1)] [H(g_1):\Lambda]
\quad=\quad d_1 [{}^{g_1}H: \Lambda] \quad \leq \quad d_1^2d_2.
\end{eqnarray*}
Since $[H: \Lambda] = [H:H(g_1g_2)][H(g_1g_2):\Lambda]$ we deduce that
\[
[H: H(g_1g_2)] \quad = \quad \frac{[H:\Lambda]}{[H(g_1g_2):\Lambda]} \quad \leq \quad [H: \Lambda]  \quad \leq \quad d_1^2d_2.
\]
Reversing roles yields $[H:H(g_2^{-1}g_1^{-1})] \leq d_1d_2^2$ and hence $(\Gamma, (\Gamma_n)_{n \in \mathbb N})$ is a coarsely filtered group (called the \emph{filtered commensurator} of $H$ in $G$) with
\[
\Gamma_{d_1}\Gamma_{d_2} \subset \Gamma_{d_1d_2\max\{d_1, d_2\}} \quad (d_1, d_2 \in \mathbb N).
\]
\end{example}
Similar coarse filtrations also exist on abstract commensurators, almost automorphism groups of trees, quasi-isometry groups of metric spaces and many other groups of geometric origin, where there is a natural notion of complexity similar to the one used in Example \ref{FilteredCommensurator1}.
\begin{remark}[Syndetic filtrations]
We say that a (coarse) filtration $(\Gamma_n)_{n \in \mathbb N}$ of $\Gamma$ is \emph{left-syndetic}\index{filtration!syndetic}\index{coarse filtration!syndetic} if $\Gamma_n$ is left-syndetic 
in $\Gamma_{n+1}$ for all $n \in \bN$. Equivalently, $\Gamma_1$ is left-syndetic in each $\Gamma_n$. If all $\Gamma_n$ are symmetric, we simply speak of a \emph{syndetic} coarse filtration. Right-syndetic filtrations are defined dually. 
\end{remark}
We conclude this section by discussing various notions of morphism between subsets of groups and filtered groups.
\begin{remark}[Maps between subsets of groups]
Let $A'$ be a subset of a group $\Gamma'$ and let $f: A \to A'$ be a map. If $A$ is symmetric (respectively unital), then we say that
$f$ is \emph{symmetric}\index{symmetric map} (respectively \emph{unital}\index{unital map}) if $f(a^{-1}) = f(a)^{-1}$ for all $a \in A$ (respectively $f(e) = e$). We say that $f$ is a \emph{partial homomorphism}\index{partial homomorphism} if for all $a_1, a_2 \in A$ with $a_1a_2 \in A$ we have $f(a_1a_2) = f(a_1)f(a_2)$. If $A$ is unital (and symmetric), this implies that $f$ is unital (and symmetric).

If $(\Gamma, (\Gamma_n)_{n \in \bN})$ and $(\Gamma', (\Gamma'_n)_{n \in \mathbb N})$ are coarsely filtered groups then we refer to a group homomorphism $\rho: \Gamma \to \Gamma'$ as a \emph{coarsely filtered morphism}\index{coarsely filtered morphism}\index{coarsely filtered morphism} if there exists a function $\tau: \bN \to \bN$ such that $\rho(\Gamma_n) \subset \Gamma'_{\tau(n)}$; it is called a \emph{filtered morphism}\index{filtered morphism} if $\tau$ can be chosen as the identity. 

The filtered groups $ (\rho(\Gamma), (\rho(\Gamma_n))_{n \in \mathbb N})$ and $(\ker(\rho), (\ker(\rho) \cap \Gamma_n)_{n \in \mathbb N})$ are then called its \emph{image}\index{filtered morphism!image} and \emph{kernel}\index{filtered morphism!kernel} respectively.
The restriction $\rho|_{\Gamma_n}\coloneqq\rho_n: \Gamma_n \to \Gamma'_{\tau(n)}$ is then a partial homomorphism, referred to as the  \emph{$n$-component\index{coarsely filtered morphism!component of, filtered morphism, component of}} of $\rho$. Similarly, if $N \in \mathbb N$, then a partial homomorphism $\rho: \Gamma_N \to \Gamma'$ is called an  \emph{$N$-local coarsely filtered morphism}\index{local coarsely filtered morphism} if there exists $\tau: \{1, \dots, N\} \to \mathbb N$ such that $\rho(\Gamma_n) \subset \Gamma'_{\tau(n)}$. If $\tau$ can be chosen to be the inclusion, then $\rho$ is called an \emph{$N$-local filtered morphism}. For $n \leq N$ the restriction of $\rho$ to $\Gamma_n$ is called the $n$-component of $\rho$, and we define the \emph{partial images}\index{filtered morphism!partial image} $\rho_n(\Gamma_n)$ and \emph{partial kernels} $\{\gamma \in \Gamma_n \mid \rho_n(\gamma) = e\}$.
\end{remark}
\begin{remark}
The following categories will feature prominently in the sequel:
\begin{enumerate}[(i)]
\item The category of unital subsets of groups and partial homomorphisms and its full subcategory of symmetric unital subsets.
\item The category of coarsely filtered groups and coarsely filtered morphisms and its subcategories of filtered groups and filtered morphisms, respectively filtered groups and $N$-local filtered morphisms.
\end{enumerate}
\end{remark}

\section{Commensurability classes of approximate subgroups}
Before we recall the definition of an approximate subgroup, due to Tao \cite{Tao}, we introduce two closely related asymmetric versions. Throughout this section, $\Gamma$ denotes a group.
\begin{definition}\label{DefQG}
Let $K \in \bN$.  A subset $\Lambda \subset \Gamma$ is called a $K$-\emph{left-quasi-subgroup}\index{left-quasi-subgroup} of $\Gamma$ if it satisfies the following conditions:
\begin{enumerate}[({QG}1)]
\item There exists a subset $F_1 \subset \Gamma$ such that $|F_1| \leq K$ and $\Lambda^{-1} \subset \Lambda F_1$.
\item There exists a subset $F_2 \subset  \Gamma$ such that $|F_2| \leq K$ and $\Lambda^2 \subset \Lambda F_2$. 
\end{enumerate}
It is called a $K$-\emph{right-quasi-subgroup}\index{right-quasi-subgroup} of $\Gamma$ if it satisfies the dual conditions:
\begin{enumerate}[({QG}1{${}^{\rm op}$})]
\item There exists a subset $F_1 \subset \Gamma$ such that $|F_1| \leq K$ and $\Lambda^{-1} \subset F_1\Lambda$.
\item There exists a subset $F_2 \subset  \Gamma$ such that $|F_2| \leq K$ and $\Lambda^2 \subset F_2\Lambda$. 
\end{enumerate}
\end{definition}
\begin{remark}\label{LGSAx} Let us comment on the axioms and their implications:
\begin{enumerate}[(i)]
\item In the theory of \emph{finite} quasi-subgroups it is essential to keep track of the constant $K$; in the infinite case this is usually less essential, and we will often not mention $K$ and simply speak of \emph{left-/right-quasi-subgroups}.
\item If $\Lambda \subset \Gamma$ is a left-quasi-subgroup, then $\Lambda^{-1}$ is a right-quasi-subgroup and vice versa. We thus focus on left-quasi-subgroups in the sequel.
\item If $\Lambda$ is a left-quasi-subgroup of $\Gamma$, then $\Lambda^+ \coloneqq  \Lambda\cup \{e\}$ is a \emph{unital} left-quasi-subgroup. We will thus restrict attention to unital left-quasi-subgroups. Given a unital left-quasi-subgroup $\Lambda$ we denote by $\Lambda^\infty \coloneqq  \bigcup_{k=1}^\infty \Lambda^k$ the \emph{enveloping semigroup}\index{enveloping semigroup} of $\Lambda$ and refer to $(\Lambda, \Lambda^\infty)$ as a \emph{left-quasi-group}.
\item  A subset $\Lambda \subset \Gamma$ is a left-quasi-subgroup if and only if there exist finite subsets $F_1 \subset \Lambda^{-2}$ and $F_2 \subset \Lambda^{-1}\Lambda^2$ such that $\Lambda^{-1} \subset \Lambda F_1$ and $\Lambda^{2} \subset \Lambda F_2$. This implies the existence of a finite subset $F_2' \subset \Lambda^{-1}\Lambda\Lambda^{-1}$ such that $\Lambda \Lambda^{-1} \subset \Lambda F_2'$.
\item If $d$ is a left-invariant metric on $\Gamma$ and $\Lambda \subset \Gamma$ is a left-quasi-subgroup, then there exists a constant $C \geq 0$ such that for all $x, y \in \Lambda$ we have
\begin{equation}\label{LQSGMetric}
d(x^{-1}, \Lambda) \leq C \quad \text{and} \quad d(xy, \Lambda) \leq C.
\end{equation}
In this sense, left-quasi-subgroups are subsets of groups on which multiplication and inversion are defined up to a bounded error. They thus seem to be the natural objects for a generalization of geometric group theory. We will see later that stronger symmetry assumptions are sometimes needed in applications; this will lead us to the notion of an approximate subgroup.
\end{enumerate}
\end{remark}
The methods of geometric group theory do not allow us to distinguish commensurable subgroups of a given group. Similarly we will see that the methods of geometric approximate group theory only allow us to distinguish (left-)quasi-subgroups up to (left-)commensurability. Given a left-quasi-subgroup $\Lambda$ we denote by $[\Lambda]_{\mathrm{lco}}$ its left-commensurability class.
\begin{lemma}\label{ShufflingLemma} Let $\Lambda \subset \Gamma$ and let $k > l \geq 1$.
\begin{enumerate}[(i)]
\item If $\Lambda \subset \Gamma$ satisfies (QG2), then $\Lambda^k \subset \Lambda^{k-l} F_2^l$. In particular, $(\Lambda^k)_{k \in \mathbb N}$ is a left-syndetic filtration of $\Gamma$ and $[\Lambda]_{\mathrm{lco}} = [\Lambda^k]_{\mathrm{lco}}$ for all $k \geq 1$.
\item If $\Lambda$ is a unital left-quasi-subgroup of $\Gamma$, then $\Lambda^\infty \subset \mathrm{LComm}_\Gamma(\Lambda)$. 
\end{enumerate}
\end{lemma}
\begin{proof} (i) For every $k \geq 2$, (QG2) implies $\Lambda^k = (\Lambda^{k-2})\Lambda^2 \subset \Lambda^{k-2}\Lambda F_2 = \Lambda^{k-1} F_2$, hence
the statement follows by induction.
\item (ii) If $g \in \Lambda^\infty$, then $g \in \Lambda^k$ for some $k \in \mathbb N$. Then by (i) we have $g\Lambda \subset \Lambda^{k+1} \subset \Lambda g (g^{-1} F_2^k)$, hence $g \in \mathrm{LComm}_\Gamma(\Lambda)$ by \eqref{LCommAlt}.
\end{proof}
\begin{corollary}\label{qsgpowers}
Let $\Lambda \subset \Gamma$ be a unital left-quasi-subgroup and assume that $\Lambda^{-1} \subset \mathrm{LComm}_\Gamma(\Lambda)$. Then $\Lambda^k$ is a unital left-quasi-subgroup for all $k \in \bN$. In particular, this holds if $\Lambda^{-1} \subset \Lambda^\infty$.
\end{corollary}
\begin{proof}  Choose $F_1 \subset \Lambda^{-2}$ such that  $\Lambda^{-1} \subset \Lambda F_1$.  Then $(\Lambda^k)^{-1} = (\Lambda^{-1})^k \subset (\Lambda F_1)^{k}$, and $F_1 \subset \Lambda^{-2} \subset \mathrm{LComm}(\Lambda)$ since the latter is a semigroup containing $\Lambda^{-1}$. Property (QG1) of $\Lambda^k$ then follows by applying Lemma \ref{LeftRightComm} repeatedly, and (QG2) follows from Lemma \ref{ShufflingLemma}.(i).
\end{proof}
Passing to powers will be essential for us in the sequel, hence we need to assume that $\Lambda^{-1} \subset \mathrm{LComm}_\Gamma(\Lambda)$ in most of what follows. We then say that $\Lambda$ is \emph{quasi-symmetric}\index{left-quasi-subgroup!quasi-symmetric}.
\begin{proposition}\label{Symmetrization}
If $\Lambda \subset \Gamma$ is a unital quasi-symmetric left-quasi-subgroup, then $\Lambda' \coloneqq \Lambda\Lambda^{-1}$ is a unital symmetric left-quasi-subgroup with $[\Lambda']_{lco} = [\Lambda]_{lco}$.
\end{proposition}
\begin{proof} $\Lambda'$ is unital and symmetric, hence satisfies (QG1). Let $F_1, F_2$ be as in Remark \ref{LGSAx}.(iv); then
\[
\Lambda \subset \Lambda\Lambda^{-1} \subset \Lambda \Lambda F_1 \subset \Lambda F_2F_1,
\]
hence $[\Lambda']_{lco} = [\Lambda]_{lco}$. Moreover, since $\Lambda$ and $\Lambda^{-1}$ are contained in the semigroup $\mathrm{LComm}_\Gamma(\Lambda)$, so are $F_1 \subset \Lambda^{-2}$ and $F_2 \subset \Lambda^{-1}\Lambda^2$. We deduce with Lemma \ref{LeftRightComm} that there exists a finite set $F$ such that
\[
(\Lambda')^2 = (\Lambda\Lambda^{-1})^2 \subset \Lambda \Lambda F_1 \Lambda \Lambda F_1 \subset \Lambda F_2F_1\Lambda F_2F_1 \subset \Lambda \Lambda F \subset \Lambda F_2F \subset \Lambda' F_2F.\qedhere
\]
\end{proof}
This motivates the study of unital \emph{symmetric} left-quasi-subgroups, which can be characterized as follows:
\begin{definition}[Tao]\label{DefTao} Let $K \in \bN$. A subset $\Lambda \subset \Gamma$ is called a \emph{$K$-approximate subgroup} if the following hold:
\begin{enumerate}[({AG}1)]
\item $\Lambda = \Lambda^{-1}$ and $e \in \Lambda$.
\item There exists $F_\Lambda \subset  \Gamma$ such that $|F_\Lambda| \leq K$ and $\Lambda^2\subset \Lambda  F_\Lambda$.
\end{enumerate}
\end{definition}
As before, we will usually drop $K$ from the name. If $\Lambda$ is an approximate subgroup we can find a finite set $F_\Lambda$ such that
\begin{equation}\label{FLambda}
\Lambda^2 \subset F_\Lambda \Lambda \cap \Lambda F_\Lambda \quad \text{and} \quad F_\Lambda \subset \Lambda^3.
\end{equation}
We will usually reserve the letter $F_\Lambda$ to denote such a set. 
\begin{remark}[Symmetrization]
Since a $K$-approximate subgroup is a unital symmetric $K$-left-quasi-subgroup (or equivalently, a  unital symmetric $K$-right-quasi-subgroup), Proposition \ref{Symmetrization} yields a bijection
\begin{align*}
& \{[\Lambda]_{lco} \mid \Lambda \subset \Gamma \text{ unital quasi-symmetric left-quasi-subgroup}\}\\
& \longrightarrow \{[\Lambda']_{\mathrm{co}} \mid \Lambda' \subset \Gamma \text{ approximate subgroup}\},
\end{align*}
which maps the class of $\Lambda$ to the class of $\Lambda \Lambda^{-1}$.
\end{remark}
First examples of approximate subgroups arise from syndetic subsets:
\begin{proposition}\label{WrongFiniteIndex}
If $\Xi \subset \Gamma$ is an approximate subgroup and $\Lambda \subset \Xi$ is unital, symmetric and syndetic, then $\Lambda$ is an approximate subgroup of $\Gamma$. In particular, every unital, symmetric, syndetic subset of $\Gamma$ is an approximate subgroup.
\end{proposition}
\begin{proof}  Let $F_1, F_2$ be finite sets such that $\Xi \subset \Lambda F_1$ and $\Xi^2 \subset \Xi F_2$, and let $F\coloneqq  F_1F_2$. Then $\Lambda^2 \subset \Xi^2 \subset \Xi F_2 \subset \Lambda F_1 F_2 = \Lambda F$.
\end{proof}
However, these do not define interesting commensurability classes. The following class of examples is more interesting.
\begin{example}[Fundamental commensurability class]\label{FundCommClass} Let $G$ be a locally compact group; then all relatively compact symmetric\footnote{Non-symmetric relatively compact identity neighborhoods are still unital quasi-symmetric left-quasi-subgroups, which illustrates that this wider notion is often more natural.} identity neighborhoods of $G$ are approximate subgroups of $G$. Moreover, they all define the same commensurability class $\chi(G)$ of approximate subgroups of $G$, called the \emph{fundamental commensurability class}\index{fundamental commensurability class} of $G$. 
\end{example}
By a recent breakthrough of Hrushovski \cite{Hrushovski_Arithmeticity}, all commensurability classes of approximate subgroups arise from fundamental commensurability classes of connected Lie groups by means of various constructions discussed below. This explains the fundamental importance of Example \ref{FundCommClass}. To discuss further examples, we record the symmetric cases of Lemma \ref{ShufflingLemma} and Corollary \ref{qsgpowers} for ease of reference:
\begin{proposition}\label{PropShuffling} Let $\Lambda \subset \Gamma$ be an approximate subgroup.
\begin{enumerate}[(i)]
\item If $F_\Lambda$ is as in \eqref{FLambda}, then 
for all $k>l \geq 1$ we have $\Lambda^k \subset \Lambda^{k-l} F_\Lambda^l \cap F_\Lambda^l \Lambda^{k-l}$.
\item $\Lambda^k$ is an approximate subgroup of $\Comm_\Gamma(\Lambda)$ for every $k \in \mathbb N$.
\item $(\Lambda^k)$ is a syndetic filtration of $\Comm_\Gamma(\Lambda)$, in particular the approximate subgroups $\Lambda^k$ are mutually commensurable.
\end{enumerate}
\end{proposition}
 Proposition \ref{PropShuffling} connects approximate subgroups to syndetic coarse filtrations.
\begin{remark}[Approximate subgroups and coarse filtrations]\label{FiltrationExamples}
By Proposition \ref{PropShuffling} every approximate subgroup gives rise to a syndetic filtration of its commensurator. Conversely, let $(\Gamma_n)_{n \in \mathbb N}$ be a left-syndetic coarse filtration and let $\Lambda$ be a \emph{symmetric} set with $\Gamma_k \subset \Lambda \subset \Gamma_n$ for some $n \geq k \geq 1$. Then there exists a finite $F \subset \Gamma$ with
\[
\Lambda^2 \subset \Gamma_{\rho(n,n)} \subset \Gamma_k F \subset \Lambda F,
\]
hence $\Lambda$ is an approximate subgroup of $\Gamma$; moreover, the commensurability class of $\Lambda$ depends only on the given coarse filtration. 

In particular, commensurability classes of approximate subgroups can be identified with equivalence classes of syndetic coarse filtrations, where syndetic coarse filtrations $(\Gamma_k^{(1)})$ and $(\Gamma_k^{(2)})$ are considered equivalent if for all $k \in \mathbb N$ and $i,j \in \{1,2\}$ there exists $k' \in \mathbb N$ such that $\Gamma^{(i)}_k \subset \Gamma^{(j)}_{k'}$.

In applications it is often more convenient to work with a  left-syndetic coarse filtration $(\Gamma_n)_{n \in \mathbb N}$ which is only \emph{quasi-symmetric}\index{coarse filtration!quasi-symmetric} in the sense that there exist $k,l \in \mathbb N$ such that $e \in \Gamma_k^{-1} \subset \Gamma_l$. In this case $\Lambda \coloneqq \Gamma_k\Gamma_k^{-1}$ is symmetric and satisfies $\Gamma_k \subset \Lambda \subset \Gamma_{\rho(k,l)}$. To summarize:
\end{remark}
\begin{proposition}\label{PropCoarseFiltrationMain} The following are equivalent for a subset $\Lambda \subset \Gamma$.
\begin{enumerate}[(i)]
\item $\Lambda$ is left-commensurable to an approximate subgroup of $\Gamma$.
\item $\Lambda$ is left-commensurable to a unital quasi-symmetric left-quasi-subgroup of $\Gamma$.
\item $\Lambda$ is left-commensurable to some (hence any) filtration step in a quasi-symmetric left-syndetic coarse filtration of $\Gamma$.
\item $\Lambda$ is left-commensurable to some (hence any) filtration step in a syndetic coarse filtration of $\Gamma$
\end{enumerate}
Moreover, there are canonical bijections between commensurability classes of approximate subgroups, left-commensurability classes of unital quasi-symmetric left-quasi-subgroups and equivalence classes of quasi-symmetric left-syndetic coarse filtrations (equivalently, syndetic coarse filtrations) of $\Gamma$.
\end{proposition}
We have already mentioned that many groups of geometric origin admit natural coarse filtrations by symmetric subsets. If these are syndetic, then they thus provide examples of approximate subgroups. Let us give a concrete example:
\begin{example}[A syndetic filtration]\label{AAutMain} Let $(\cT_d, o)$ denote the $d$-regular rooted tree\index{$d$-regular rooted tree} as in Definition \ref{regularrootedtree}. We denote by $\mathrm{Aut}(\cT_d, o)$ the group of rooted automorphisms of $\cT_d$ and by $\mathrm{AAut}(\cT_d, o)$ the group of rooted almost automorphisms of $\cT_d$. We refer the reader to Appendix \ref{trees} for background on these groups. 

By Lemma \ref{AAutFiltration} the group $\mathrm{AAut}(\cT_d, o)$  admits a natural symmetric filtration given by
 \begin{equation}\label{AAutm}
 \mathrm{AAut}^m(\mathcal{T}_{d}, o) \coloneqq \{ [\varphi] \in \mathrm{AAut}(\mathcal{T}_{d}, o) \mid \mathrm{depth} [\varphi] \leq m\},
 \end{equation}
where $\mathrm{Aut}(\cT_d, o) = \mathrm{AAut}^0(\mathcal{T}_{d}, o)$. It follows from work of Le Boudec \cite{AdrienLeBoudec} that if $V_d \subset \mathrm{AAut}(\cT_d, o)$ denotes the Higman--Thompson group\footnote{For the definition see \cite[Sec.\ 2 and 3]{AdrienLeBoudec}, where $\cT_d$ is denoted $\cT_{d,d}$} and ${V}^{m}_d \coloneqq V_{d} \cap \mathrm{AAut}^{m}(\mathcal{T}_{d})$, then for all $m \in \bN_0$ we have
 \begin{equation}\label{SyndeticExample}
\mathrm{AAut}^{m+1}(\mathcal{T}_{d}) \subset \mathrm{AAut}^{m}(\mathcal{T}_{d}) V_d^{m+1} \qand |V_d^{m+1}|< \infty.
\end{equation}
Indeed, by Remark \ref{LeBoudecConvenient} we have $|V_d^{m+1}|< \infty$ and $\mathrm{AAut}^{m+1}(\mathcal{T}_{d}, o) \subset W V_d^{m+1}$, where $W <  \mathrm{Aut}(\cT_d, o) = \mathrm{AAut}^0(\mathcal{T}_{d}, o)$ denotes the pointwise stablizer of the first level of $\cT_d$. This shows that the filtration \eqref{AAutm} is syndetic, hence the subsets $\mathrm{AAut}^m(\mathcal{T}_{d}, o)$ are mutually commensurable approximate subgroups of $\mathrm{AAut}(\cT_d, o)$.
\end{example}

We close this section by discussing commensurators of approximate subgroups; the following is contained in \cite[Remark 4.4 and Lemma 5.1]{Hrushovski_Arithmeticity}: 
\begin{proposition}[Commensurator]\label{CommAG} Let $\Lambda \subset \Gamma$ be an approximate subgroup.
\begin{enumerate}[(i)]
\item If $F_0 \subset  \mathrm{Comm}_\Gamma(\Lambda)$ is finite, then $\Xi_{F_0} \coloneqq F_0\Lambda \cup \Lambda \cup \Lambda F_0^{-1}$ is an approximate subgroup of $\Gamma$ which is commensurable to $\Lambda$. If $F_0$ is unital, then the same holds for $F_0\Lambda F_0^{-1}$.
\item  $\mathrm{Comm}_\Gamma(\Lambda)$ equals the union of all approximate subgroups $\Xi$ with $\Xi\sim_{\mathrm{co}} \Lambda$.
\end{enumerate}
\end{proposition}
\begin{proof} (i) Assume without loss of generality that $F_0$ is unital and let $\Xi = \Xi_{F_0}$ or $\Xi = F_0 \Lambda F_0^{-1}$. Then
$\Xi$ satisfies (AG1) by construction and by Lemma \ref{LeftRightComm} there exist finite sets $F, F'$ with
\[
\Lambda \subset \Xi \subset \Xi^2 \subset F_0\Lambda F_0^{-1} F_0 \Lambda F_0^{-1} \subset \Lambda^2 F \subset \Lambda F' \subset \Xi F'.
\]
(ii) If $g \in  \mathrm{Comm}_\Gamma(\Lambda)$, then $g \in \Xi_{\{g\}}$, where $\Xi_{\{g\}}$ is as in (i). Conversely assume that $\Xi \sim_{co} \Lambda$ is an approximate subgroup. Then by Proposition \ref{PropShuffling} and Lemma \ref{CommDepOnClass} we have $\Xi \subset \mathrm{Comm}_\Gamma(\Xi) = \mathrm{Comm}_\Gamma(\Lambda)$.
\end{proof}
 Note that if $\Lambda$ and $\Xi$ are commensurable approximate subgroups, then (AG2) and Corollary \ref{lemmasymmtrick} yield
\begin{equation}\label{CommElements}
[\Lambda]_{\mathrm{co}} = [\Lambda^k]_{\mathrm{co}} = [\Xi]_{\mathrm{co}} = [\Xi^k]_{\mathrm{co}} = [\Lambda^k \cap \Xi^l]_{\mathrm{co}} \text{ for all }k,l \geq 2.
\end{equation}
\section{Categories of approximate groups}
We introduce the following parlance:
\begin{definition}\label{DefApGr} If $\Gamma$ is a group and $\Lambda \subset \Gamma$ is an approximate subgroup, then the pair $(\Lambda, \Lambda^\infty)$ is called an  \emph{approximate group}\index{approximate group} and the associated filtered group $(\Lambda^\infty, (\Lambda^k)_{k \in \bN})$ is called a \emph{filtered approximate group}\index{approximate group!filtered}. If $(\Xi, \Xi^\infty)$ and $(\Lambda, \Lambda^\infty)$ are approximate groups, then $(\Xi, \Xi^\infty)$ is called an \emph{approximate subgroup}\index{approximate group!approximate subgroup of} of $(\Lambda, \Lambda^\infty)$  provided $\Xi \subset \Lambda$ and $\Xi^\infty < \Lambda^\infty$.
\end{definition}
\begin{remark}
We will consider groups as approximate groups by identifying $\Gamma$ with $(\Gamma, \Gamma)$. Then $(\Xi, \Xi^\infty)$ is an approximate subgroup of $(\Gamma, \Gamma)$ in the sense of Definition \ref{DefApGr} if and only if $\Xi$ is an approximate subgroup of $\Gamma$ in the sense of Definition \ref{DefTao}. We say that an approximate group $(\Lambda, \Lambda^\infty)$ is \emph{finite}\index{approximate group!finite} (\emph{countable}\index{approximate group!countable}) if $\Lambda$ is finite (countable). Note that the former does not imply that $\Lambda^\infty$ is finite. Similarly we say that an approximate subgroup $(\Xi, \Xi^\infty)$ of an approximate group $(\Lambda, \Lambda^\infty)$ is \emph{syndetic} if $\Xi$ is a syndetic subset of $\Lambda$\index{approximate group!syndetic approximate subgroup of} (which does not imply that $\Xi^\infty \subset \Lambda^\infty$ is syndetic). If $\Lambda$ is syndetic in $\Lambda^\infty$, then we refer to $(\Lambda, \Lambda^\infty)$ as an \emph{almost group}\index{almost group}. 
\end{remark}
We discuss several notions of morphisms between approximate groups, starting from the most obvious generalizations of a (partial) group homomorphism:
\begin{definition}\label{def: global morphism} A \emph{global morphism}\index{global morphism}\index{approximate group!morphism of}\index{morphism} between approximate groups is a filtered morphism between the associated filtered groups. Given $N \in \mathbb N$, an \emph{$N$-local morphism}\footnote{If $\rho_N: \Xi^N \to \Lambda^N$ is an $N$-local morphism in this sense, then $\rho_1 = \rho_N|_{\Xi}$ is a  \emph{Freiman $N$-homomorphism}\index{Freiman $k$-homomorphism} in the sense of \cite[Def. 1.4]{Breuillard}. Conversely, one can check that if $\rho: \Xi \to \Lambda$ is a Freiman $m$-homomorphism for some $m \geq \frac 3 2 k$, then there is a unique $k$-local morphism $\rho_k: \Xi^k \to \Lambda^k$ such that $\rho_k(x_1 \cdots x_k) \coloneqq  \rho(x_1) \cdots \rho(x_k)$. Up to a change of constant the two concepts are thus equivalent.}\index{approximate group!$k$-local morphism of} between approximate groups is an $N$-local filtered morphism between the associated filtered groups.
\end{definition}
More explicitly, a global morphism $\rho: (\Xi, \Xi^\infty) \to (\Lambda, \Lambda^\infty)$ between approximate groups is given by a group homomorphism $\rho: \Xi^\infty \to \Lambda^\infty$ which restricts to a partial homomorphism $\rho_1: \Xi \to \Lambda$. Similarly, an $N$-local morphism between approximate groups is given by a partial homomorphism $\rho_N: \Xi^N \to \Lambda^N$ which restricts to a partial homomorphism $\rho_1: \Xi \to \Lambda$. We often write it as $\rho_N: (\Xi, \Xi^N) \to (\Lambda, \Lambda^N)$.
Approximate groups and global or $N$-local morphisms form a category; in particular we can talk about isomorphic (or locally isomorphic) approximate groups.
\begin{remark}[Kernels and Images]\label{KernelsAndImages} Let $\rho: (\Xi, \Xi^\infty) \to (\Lambda, \Lambda^\infty)$ be a global morphism of approximate groups. We define the \emph{$k$th partial kernel}\index{morphism!partial kernel} and the \emph{$k$th partial image}\index{morphism!partial image} as $\ker_k(\rho) \coloneqq  \ker(\rho) \cap \Xi^k$, respectively $\im_k(\rho) \coloneqq  \rho(\Xi^k) = \rho(\Xi)^k$. By definition these are the partial kernels and images of the induced map between the associated filtered groups. Partial kernels and images can also be defined for local morphisms in the obvious way. In Example \ref{ExPI} we will show that partial images are pairwise commensurable approximate subgroups. In Corollary \ref{PartialKernelGlobal} and Example \ref{NotPKG} we will show that the partial kernels $\ker_k(\rho)$ are pairwise commensurable approximate subgroups for $k \geq 2$, whereas $\ker_1(\rho)$ need not be an approximate subgroup.
\end{remark}
It turns out that, as far as approximate groups are concerned, the notion of global or local morphism is not as natural as may seem at first sight. To explain the problem, let $\Gamma$ and $H$ be groups and let $\rho: \Gamma \to H$ be a function. Then $\rho$ is a group  homomorphism if and only if its graph $\mathrm{gr}(\rho) = \{(g, \rho(g)) \mid g \in \Gamma\}$ is a subgroup of $\Gamma \times H$. In the context of approximate groups it seems more natural to consider maps whose graph $\mathrm{gr}(\rho)$ is an \emph{approximate subgroup} of $\Gamma \times H$. We now work towards a characterization of such functions. 
\begin{remark}[Left-defect set]\label{Drho} We define the \emph{left-defect set}\index{left-defect set} of $\rho: \Gamma \to H$ as 
\[
D(\rho) = \{\rho(y)^{-1}\rho(x)^{-1}\rho(xy) \mid x,y \in \Gamma\}.
\]
Since $\rho(xy) = \rho(x)\rho(y)\rho(y)^{-1}\rho(x)^{-1}\rho(xy)$, this $D(\rho)$ is the smallest subset $D \subset H$ such that
\begin{equation}\label{Drho1}
\rho(xy) \in \rho(x)\rho(y)D \qand \rho(x)\rho(y) \in \rho(xy)D^{-1} \quad \text{ for all }x,y \in \Gamma.
\end{equation}
If $D \subset H$ is \emph{any} set satisfying \eqref{Drho1} (for example, any set containing $D(\rho)$), then $\{\rho(x)\rho(x^{-1}), \rho(x^{-1})\rho(x)\} \subset \rho(e)D^{-1}$ and hence
\begin{equation}\label{Drho2}
\rho(x^{-1}) \in \rho(x)^{-1}\rho(e)D^{-1} \qand \rho(x)^{-1} \in \rho(x^{-1}) D \rho(e)^{-1} \quad \text{ for all }x \in \Gamma.
\end{equation}
Moreover, if $A, B \subset H$ and $x \in \rho^{-1}(A)$ and $y \in \rho^{-1}(B)$, then
$\rho(xy) \in \rho(x)\rho(y)D \subset ABD$, and hence
\begin{equation}\label{EvenWeirder}
\rho^{-1}(A) \rho^{-1}(B) \subset \rho^{-1}(ABD) \quad \text{ for all }A,B \subset H.
\end{equation}
Combining \eqref{Drho1} and \eqref{Drho2} we also deduce that
\begin{equation}\label{Drho3}
\rho(xy^{-1})\in \rho(x)\rho(y)^{-1}D' \quad \text{ for all }x,y \in \Gamma,
\end{equation}
where $D' \coloneqq \rho(e)D^{-1}D$. If $D'$ is any set satisfying \eqref{Drho3}, then the same argument as in the proof of \eqref{EvenWeirder} shows that
\begin{equation}\label{ReallyWeird}
\rho^{-1}(A) \rho^{-1}(B)^{-1} \subset \rho^{-1}(AB^{-1}D') \quad \text{ for all }A,B \subset H.
\end{equation}
\end{remark}
The following is an asymmetric version of \cite[Lemma 2.1.18]{Machado23}:
\begin{lemma}\label{MacAsym} Let $\rho: \Gamma \to H$ be a function. Then the graph $\mathrm{gr}(\rho)$ is a left-quasi-subgroup of $\Gamma \to H$ if and only if $D(\rho)$ is finite. 
\end{lemma}
\begin{proof} Assume first that $D(\rho)$ is finite and let $x, y \in \Gamma$. By \eqref{Drho2} we have
\[
(x, \rho(x))^{-1} \in (x^{-1}, \rho(x^{-1}))(\{e\} \times D(\rho)\rho(e)^{-1}),
\]
and by \eqref{Drho1} we have
\[
(x,\rho(x))(y, \rho(y)) = (xy, \rho(x)\rho(y)) \in (xy, \rho(xy))(\{e\} \times D(\rho)^{-1}),
\]
hence $\mathrm{gr}(\rho)$ is a left-quasi-subgroup.

Conversely assume that $\mathrm{gr}(\rho)$ is a left-quasi-subgroup and let $F=F_\Gamma \times F_H$ $\subset \Gamma \times H$ be finite and symmetric with $\mathrm{gr}(\rho)^2 \subset \mathrm{gr}(\rho)F$.
For all $x,y \in \Gamma$ we then find 
$z \in \Gamma$ and $u \in F_\Gamma$, $v \in F_H$ with $(x,\rho(x))(y, \rho(y)) = (z,\rho(z))(u,v)$. The first coordinate yields $z=xyu^{-1}$ and the second coordinate yields
\begin{equation}\label{MacAsym1}
\rho(x)\rho(y) \in \rho(xy u^{-1})F_H.
\end{equation}
Now $F^* \coloneqq (\rho(F_\Gamma)F_H)^{-1}$ is finite, and choosing $z = x^{-1} = y$ in \eqref{MacAsym1} yields
\begin{equation}\label{MacAsym2}
\rho(z^{-1})\rho(z) \in (F^*)^{-1} \implies \rho(z^{-1})^{-1} \in \rho(z)F^* \quad \text{for all }z \in \Gamma.
\end{equation}
Now let $u \in F_\Gamma$ as in \eqref{MacAsym1} and observe that
$ 
(u^{-1}, \rho(y^{-1}x^{-1})\rho(xyu^{-1})) =$\\
$ \left(y^{-1}x^{-1}, \rho(y^{-1}x^{-1})\right)\left(xyu^{-1}, \rho(xyu^{-1})\right) \in (F_\Gamma \times H) \cap \mathrm{gr}(\rho)^2.
$
By definition of $F$ we have
\[
(F_\Gamma \times H) \cap \mathrm{gr}(\rho)^2 \subset (F_\Gamma \times H) \cap \mathrm{gr}(\rho)F \subset  ((F_\Gamma^2 \times H) \cap \mathrm{gr}(\rho))F,
 \]
 and since the projection of $(F_\Gamma^2 \times H) \cap \mathrm{gr}(\rho)$ to $\Gamma$ is injective with finite image, we deduce that $(F_\Gamma \times H) \cap \mathrm{gr}(\rho)^2$ is contained in a finite set, which we may assume to be symmetric and of the form $F_\Gamma' \times F_H'$. This then implies that
 \begin{equation}\label{MacAsym3}
  \rho(y^{-1}x^{-1})\rho(xyu^{-1}) \in F_H' \implies \rho(xyu^{-1}) \subset \rho(y^{-1}x^{-1})^{-1}F_H'.
 \end{equation}
Applying \eqref{MacAsym1}, \eqref{MacAsym3} and  \eqref{MacAsym2} (in this order) then yields
 \[
 \rho(x)\rho(y) \in   \rho(xy u^{-1})F_H \subset \rho(y^{-1}x^{-1})^{-1}F_H'F_H \subset \rho(xy) F^*F_H'F_H,
 \]
 and hence $D(\rho) \subset (F^*F_H'F_H)^{-1}$ is finite. 
\end{proof}
\begin{definition}\label{DefQM} A map $\rho: \Gamma \to H$ is called a \emph{quasimorphism}\index{quasimorphism} if $D(\rho)$ is finite.
\end{definition}
\begin{remark}[Left vs.\ right] A priori what we call a quasimorphism in Definition \ref{DefQM} should be called a \emph{left-quasimorphism}\index{left-quasimorphism}, and one should define a \emph{right-quasimorphism}\index{right-quasimorphism} by demanding that the \emph{right-defect set}\index{right-defect set}
\begin{equation}\label{RDefectSet}
D^*(f) \coloneqq  \{f(x)f(y)f(xy)^{-1} \mid x, y \in G\}
\end{equation}
be finite. However, it was proved by Heuer in \cite[Prop.\ 2.3]{Heuer1} that the two notions coincide. This has the surprising consequence that a graph of a function between two groups is a left-quasi-subgroup if and only if it is a right-quasi-subgroup. 
\end{remark}
We record the following symmetric version of Lemma \ref{MacAsym} (\cite[Lemma 2.1.18]{Machado23}):
\begin{corollary} The graph of $\rho: \Gamma \to H$ is an approximate subgroup of $\Gamma \times H$ if and only if $\rho$ is a symmetric unital quasimorphism.
\end{corollary}
This motivates us to consider the following classes of morphisms:
\begin{definition} Let $(\Xi, \Xi^\infty)$ and $(\Lambda, \Lambda^\infty)$ be approximate groups.
\begin{enumerate}[(i)]
\item A map of pairs $\rho: (\Xi, \Xi^\infty) \to (\Lambda, \Lambda^\infty)$ is called a \emph{global quasimorphism}\index{global quasimorphism} (or simply a \emph{quasimorphism})\index{quasimorphism} if $\rho: \Xi^\infty \to \Lambda^\infty$ is a quasimorphism in the sense of Definition \ref{DefQM}.
\item A map of pairs $\rho_N: (\Xi, \Xi^N) \to (\Lambda, \Lambda^\infty)$ is called an \emph{$N$-local quasimorphism}\index{quasimorphism!local} (or simply an \emph{$N$-quasimorph\-ism}) if the (left-)$N$-defect set
\begin{equation}\label{DefectSet}
D_N(\rho_N) \coloneqq  \{\rho_N(y)^{-1}\rho_N(x)^{-1}\rho_N(xy) \mid x, y, xy \in \Xi^N\}
\end{equation}
is finite.
\end{enumerate}
\end{definition}
For the moment we do not insist that quasimorphisms be symmetric, although some applications may require this.
\begin{remark}[Quasimorphisms and filtrations]\label{Quasifilt}
If $\rho: (\Xi, \Xi^\infty) \to (\Lambda, \Lambda^\infty)$ is a quasimorphism between approximate groups, then $D(\rho) \subset \Lambda^\infty$ is finite, hence there exists $M \in \mathbb N$ such that $D(\rho) \subset \Lambda^M$. We then have 
\[
\rho(\Xi^k) \subset \rho(\Xi)^k D(\rho)^{k-1} \subset \Lambda^{k+(k-1)M} \subset (\Lambda^{2M})^k,
\]
hence $\rho$ induces a filtered morphism $(\Xi, \Xi^\infty)\to (\Lambda^{2M}, \Lambda^\infty)$, and thus a coarsely filtered morphism between $(\Xi, \Xi^\infty)$ and $(\Lambda, \Lambda^\infty)$.
\end{remark}
We conclude this section by discussing two generalizations of the notion of a quasimorphism which we will need in the sequel.
\begin{remark}[Extension to locally compact targets]\label{TopQM}
If $H$ is a locally compact group, rather than an abstract (discrete) group, then it is natural to call a map $\rho: \Gamma \to H$ a \emph{quasimorphism} provided $D(\rho)$ is relatively compact. For distinction, we sometimes say that $\rho$ is a \emph{topological quasimorphism}\index{quasimorphism!topological}. Remark \ref{Drho} applies mutatis mutandis to the topological case. In particular, there exists a relatively compact subset $K'\subset H$ such that
\begin{equation}\label{HruWeirdK}
\rho(xy^{-1}) \subset \rho(x)\rho(y)^{-1}K' \quad \text{ for all }x,y \in \Gamma.
\end{equation}
If $K$ is such a set and contains $D(\rho)$, then for all $A, B \subset H$ we have
\begin{equation}\label{WeirdStuff}
\rho^{-1}(A) \rho^{-1}(B) \subset \rho^{-1}(ABK) \qand \rho^{-1}(A) \rho^{-1}(B)^{-1} \subset \rho^{-1}(AB^{-1}K).
\end{equation}
\end{remark}
\begin{example}[Real-valued quasimorphisms]\label{RQM}
A real-valued function $\rho: \Gamma \to \R$ is a (topological) quasimorphism if and only if $D(\rho)$ is contained in a compact symmetric interval. The smallest such interval is then given by $[-d(\rho), d(\rho)]$, where
\begin{equation}\label{DefectNew}
d(\rho)\coloneqq \sup|\rho(gh)-\rho(g) - \rho(h)|
\end{equation}
denotes the \emph{defect}\index{defect} of $\rho$. See Section \ref{SecQuasimorphisms} for details about real-valued (topological) quasimorphisms.
\end{example}
\begin{remark}[Topological quasimorphisms with non-commutative targets] 
It was po\-in\-ted out by Thurston long ago that ``constructing non-trivial quasimorphisms with non-commutative targets is difficult''. Recently, a far reaching structure theory for general quasimorphisms with countable non-commutative target was developed by Fujiwara and Kapovich \cite{FujKap}, which confirms this statement. It is remarkable that, on the other hand, there are plenty of non-trivial topological quasimorphisms into non-discrete Lie groups, see
\cite{Brandenbursky}.
\end{remark}
The following is adapted from the proof of \cite[Prop.\ 5.12]{Hrushovski}
\begin{lemma}\label{HruMagic} Let $\rho: \Gamma \to H$ be a topological quasimorphism and let $K' \subset H$ be as in \eqref{HruWeirdK}. Then for all $A, B, F_1 \subset H$ with $A \subset BF_1$ there exists $F_2 \subset \Gamma$ with $|F_2| \leq |F_1|$ such that
\[
\rho^{-1}(A) \subset \rho^{-1}(BB^{-1}K')F_2.
\]
\end{lemma}
\begin{proof} Let $M \subset H$ and $x \in H$. We claim that
\begin{equation}\label{HruLemma1}
\exists\, x' = x'(M,x) \in \Gamma\ \  \text{ such that }\ \  \rho^{-1}(Mx) \subset \rho^{-1}(MM^{-1}K')x'.
\end{equation}
If $\rho^{-1}(Mx) = \emptyset$, then there is nothing to show. Otherwise pick $z, x' \in \rho^{-1}(Mx)$ and observe that
\[
\rho(z(x')^{-1}) \in \rho(z)\rho(x')^{-1}K' \subset Mx(Mx)^{-1}K' \subset MM^{-1}K'.
\]
This implies that $z(x')^{-1} \in \rho^{-1}(MM^{-1}K')$. Since $z \in \rho^{-1}(Mx)$ was arbitrary this shows \eqref{HruLemma1}.
Now, if we set $F_2 \coloneqq \{x'=x'(B, x) \mid x \in F_1\}$, then $|F_2| \leq |F_1|$ and
\[
\rho^{-1}(A) \subset \bigcup_{x \in F_1} \rho^{-1}(Bx) \subset \bigcup_{x' \in F_2} \rho^{-1}(BB^{-1}K')x' = \rho^{-1}(BB^{-1}K')F_2.\qedhere
\]
\end{proof}
Historically, quasimorphisms originate from work of Ulam on stability of linear functional equations in Banach spaces \cite{Ulam}. To describe Ulam's characterization of quasimorphisms we recall from Proposition \ref{ExistLeftAdmissible} that if $H$ is a countable group, then it admits a left-invariant \emph{proper} metric $d$.
\begin{proposition} If $\Gamma$ is a group and $H$ is a countable group with proper left-invariant metric $d$, then a function $f: \Gamma \to H$ is a quasimorphism if and only if there exists $C \geq 0$ such that 
\begin{equation}\label{dCQuasiHom}
d(f(xy), f(x)f(y)) \leq C \quad \text{for all }x,y \in \Gamma.
\end{equation}
\end{proposition}
\begin{proof} Since \eqref{dCQuasiHom} is equivalent to $d(f(y)^{-1}f(x)^{-1}f(xy), e) \leq C$ this follows from the fact that, by properness of $d$, a subset of $H$ is finite if and only if it is contained in a ball around $e$.
\end{proof}
This motivates the following definition:
\begin{definition} If $\Gamma, H$ are groups and $d$ is a metric on $H$, then a function $f: \Gamma \to H$ satisfying \eqref{dCQuasiHom} is called \emph{$(d, C)$-quasi-homomorphism}. It is called a \emph{$d$-quasi-homomorphism} if it is a $(d,C)$-quasi-homomorphism\index{quasi-homomorphism} for some $C \geq 0$. 
\end{definition}
For proper left-invariant metrics on locally compact groups we recover the notion of a topological quasimorphism. The definition of a quasi-homomorphism is particularly useful when dealing with non-locally-compact targets such as additive groups of Banach spaces, which are the classical context of Ulam stability (cf.\ \cite{Ulam}, \cite{UlamStability}).

\section{Direct and inverse images}
The fact that direct and inverse images of groups under group homomorphisms are again groups is of fundamental importance in group theory.
In this section we study images and pre-images of (thickened) approximate subgroups under various types of morphisms. Unless otherwise mentioned $\Gamma$ and $H$ denote discrete groups. As far as direct images are concerned, everything is straightforward:
\begin{proposition}[Direct images]\label{ImagesExist} Let $\rho: \Gamma \to H$ be a quasimorphism. If $\Xi \subset \Gamma$ is a left-quasi-subgroup, then $\Lambda \coloneqq  \rho(\Xi) \subset H$ is a left-quasi-subgroup. In particular, if $\Xi$ is an approximate subgroup and $\rho$ is symmetric, then $\Lambda$ is an approximate subgroup.
\end{proposition}
\begin{proof} Let $g \in \Xi$. Using \eqref{Drho2} and (QG1) we have
\[
\rho(g)^{-1} \in \rho(g^{-1})D(\rho)\rho(e)^{-1} \subset \rho(\Xi^{-1})D(\rho)\rho(e)^{-1} \subset \rho(\Xi F_1)D(\rho)\rho(e)^{-1}.
\] 
By \eqref{Drho1} we have 
\[
\rho(\Xi F_1) \subset \rho(\Xi) \rho(F_1)D(\rho) = \Lambda \rho(F_1)D(\rho),
\]
and thus $\Lambda^{-1} \subset \Lambda \rho(F_1)D(\rho)^2\rho(e)^{-1}$, which establishes (QG1) for $\Lambda$. As for (QG2) a similar computation shows
\[
\Lambda^2 = \rho(\Xi)^2 \subset \rho(\Xi^2)(D(\rho))^{-1} \subset \rho(\Xi F_2)D(\rho)^{-1} \subset \Lambda \rho(F_2)D(\rho)D(\rho)^{-1}.
\]
This proves the first statement, and the second statement then follows from the fact that the image of a symmetric unital set under a symmetric map is symmetric and unital.
\end{proof}
If $\rho$ is symmetric, then the proof simplifies further since we can choose $F_1 = \{e\}$, and $F_2$ to be contained in $\Xi^3$. A closer inspection of the proof then yields the following result:
\begin{corollary}[Local direct images]\label{QuasiImages} Let $\rho: (\Xi, \Xi^\infty) \to (\Lambda, \Lambda^\infty)$ be a symmetric $4$-local quasimorphism of approximate groups. Then $\rho(\Xi)$ is an approximate subgroup of $\Lambda^\infty$. 
\end{corollary}
\begin{example}[Partial images]\label{ExPI} If  $\rho: (\Xi, \Xi^\infty) \to (\Lambda, \Lambda^\infty)$ is a global or $4$-local morphism, then its partial images are approximate subgroups of $\Lambda^\infty$.
\end{example}
\begin{remark}[Direct image of commensurability class] If $\rho: \Gamma \to H$ is a quasimorphism of groups and $\omega$ is a commensurability class of approximate subgroups of $\Gamma$, then $\rho(\Lambda)$ is commensurable to $\rho(\Lambda')$ for all $\Lambda, \Lambda' \in \omega$. Indeed, if $\Lambda \subset \Lambda' F$ for a finite set $F$, then $\rho(\Lambda) \subset \rho(\Lambda') \rho(F)D(\rho)$. We thus define the \emph{direct image}\index{commensurability class!direct image} of $\omega$ under $\rho$ as $\rho(\omega)\coloneqq [\rho(\Lambda)]_{\mathrm{co}}$.
\end{remark}
The situation for pre-images is more complicated. A particular simple case concerns the inclusion $\rho: \Gamma \hookrightarrow H$ of a group $\Gamma$ into some ambient group $H$. If $\Lambda \subset H$ is an approximate subgroup, then $\rho^{-1}(\Lambda) = \Lambda \cap \Gamma$, and this intersection need not be an approximate subgroup of $\Gamma$, as the following example shows:
\begin{example}[Bad intersections]\label{IntersectionFail} Let $M \subset 2\Z$ be an arbitrary symmetric subset containing $0$, which is \emph{not} an approximate subgroup (for example, the set of all even squares and their negatives). Then $\Lambda \coloneqq  (2\Z+1) \cup M$ is an approximate subgroup of $\Z$, since it is even syndetic in $\Z$, and $\Gamma \coloneqq  2\Z$ is a subgroup of $\Z$. Nevertheless the intersection $\Lambda \cap \Gamma = M$ is not an approximate subgroup of $\Z$. This shows that intersections of approximate subgroups can be completely arbitrary symmetric unital sets.
\end{example}
It is clear from this example that nothing can be said about general intersections of approximate subgroups. However, as is often the case in the theory of approximate groups, this problem can be fixed by passing to suitable powers:
\begin{lemma}[Thick intersections]\label{Intersection} Let $\Lambda, \Xi \subset \Gamma$ be approximate subgroups and $k,l \geq 2$. Then $\Lambda^k \cap \Xi^l$ is an approximate subgroup of $\Gamma$ and $[\Lambda^k \cap \Xi^l]_{\mathrm{co}} = [\Lambda^2 \cap \Xi^2]_{\mathrm{co}}$ depends only on the commensurability classes of $\Lambda$ and $\Xi$.

\end{lemma}
\begin{proof} By symmetry we may assume that $l \leq k$. 
Let $F\subset \Gamma$ be finite and such that $\Lambda^2\subset F\Lambda$, $\Xi^2\subset F\Xi$.
Then, by Remark \ref{Rem: properties AB} and Proposition \ref{PropShuffling},
\[
(\Lambda^k \cap \Xi^l)^2 \subset \Lambda^{2k}\cap \Xi^{2l} \ \subset \ \left( F^{2k-1}\Lambda \right)\cap \left(F^{2l-1}\Xi\right) \ \subset \  F^{2l-1}\left( (F^{2k+2l-2}\Lambda) \cap \Xi\right).
\]
By Lemma \ref{LemmaAdditiveComb} applied to $\Xi \cap F^{2k+2l-2}\Lambda$
we find a set $F'$ with $|F'|\leq  |F^{2k+2l-2}|$ such that
\[
(\Lambda^k \cap \Xi^l)^2 \subset F^{2l-1}F'(\Lambda^2 \cap \Xi^2) \subset F^{2l-1}F' (\Lambda^k \cap \Xi^l),
\]
which shows that $\Lambda^k \cap \Xi^l$ satisfies (AG2). Since it also satisfies (AG1), it is an approximate subgroup.

\item Now let  $\Lambda_0$ and $\Xi_0$ be approximate subgroups commensurable to $\Lambda$ and $\Xi$ respectively. Choose
 $F \subset \Gamma$ symmetric such that $\Lambda_0^k \subset F \Lambda$ and $\Xi_0^l \subset F\Xi$. Then
 \[
 \Lambda_0^k \cap \Xi_0^l \subset F\Lambda \cap F \Xi \subset F(F^2 \Lambda \cap \Xi),
 \]
 and by Lemma \ref{LemmaAdditiveComb} applied to $\Xi\cap F^2\Lambda$ we find a set $F'$ with $|F'| \leq |F^2|$ such that 
 \[
  \Lambda_0^2 \cap \Xi_0^2 \subset  \Lambda_0^k \cap \Xi_0^l \subset F(F^2 \Lambda \cap \Xi) \subset FF' (\Lambda^2 \cap \Xi^2).
 \]
Reversing the roles of $\Lambda$ and $\Lambda_0$ and $\Xi$ and $\Xi_0$ respectively we deduce that
\[
 [\Lambda_0^2 \cap \Xi_0^2]_{\mathrm{co}} =  [\Lambda_0^k \cap \Xi_0^l]_{\mathrm{co}}  =  [\Lambda^2 \cap \Xi^2]_{\mathrm{co}} =  [\Lambda^k \cap \Xi^l]_{\mathrm{co}}. \qedhere
\]
\end{proof}
\begin{definition}
If $[\Lambda]_{\mathrm{co}}$ and $[\Xi]_{\mathrm{co}}$ are commensurability classes of approximate subgroups of $\Gamma$, we denote by $[\Lambda]_{\mathrm{co}} \cap [\Xi]_{\mathrm{co}} \coloneqq [\Lambda^2 \cap \Xi^2]_{\mathrm{co}}$ their \emph{intersection class}. 
\end{definition}
We can now give a first general construction of countable approximate groups which are not almost groups.
\begin{example} Let $G$ be a non-compact connected Lie group and let $\Gamma<G$ be a finitely-generated dense subgroup. 
Fix a finite generating set $S$ of $\Gamma$ and a compact symmetric identity neighborhood $W$ of $G$ with dense interior $U$ containing $S$. By Lemma \ref{Intersection} the intersection $\Lambda = \Gamma \cap W^2$ is an approximate subgroup of $\Gamma$ which is dense in $W$ and in particular infinite. Since $S \subset W$ we have $\Gamma = \Lambda^\infty$. Now the countable approximate group $(\Lambda, \Lambda^\infty)$ is not an almost group, for otherwise we could find a finite subset $F \subset \Gamma$ such that $\Gamma \subset \Lambda F$, and since $\Gamma W$ contains a dense open subgroup of $G$ we would then have
$G = \Gamma W \subset \Lambda F W \subset W^2 FW$, contradicting the assumption that $G$ is non-compact.
\end{example}
We will later refine this construction by taking pre-images of compact identity neighborhoods under more general (quasi-)morphisms. Here is another immediate consequence of  Lemma \ref{Intersection}:
\begin{corollary}[Higher partial kernels]\label{PartialKernelGlobal} If $\rho: (\Xi, \Xi^\infty) \to (\Lambda, \Lambda^\infty)$ is a global morphism, then the partial kernels $\ker_k(\rho) =   \ker(\rho) \cap \Xi^k$ are mutually commensurable approximate subgroups of $\Xi^\infty$ for all $k \geq 2$.
\end{corollary}
\begin{example}[Failure of first partial kernel]\label{NotPKG}
Let $M \subset 2\Z$ be as in Example \ref{IntersectionFail} and let $\Lambda \coloneqq  (2\Z+1) \cup M$. Then the canonical projection $\Z \to \Z/2\Z$ induces a global morphism
$
\rho: (\Lambda, \Z) \to (\Z/2\Z, \Z/2\Z) 
$ of approximate groups and we have $\ker_1(\rho) = \Lambda \cap 2\Z = M$, which is not an approximate subgroup of $\Z$.
\end{example}
With some additional effort we can extend Corollary \ref{PartialKernelGlobal}.
\begin{lemma}[Local kernels] \label{Lem: partial kernel ag}
Let $\rho_k: (\Xi, \Xi^{k}) \to (\Lambda, \Lambda^{k})$ be a $k$-local morphism of approximate groups for some $k \geq 10$. Then $
\ker_2(\rho_{k}),\ \dots , \ \ker_{\lfloor \frac{k+2}{6} \rfloor}(\rho_{k})$ are approximate subgroups, and if $\Xi$ is countable, then these are mutually commensurable.
\end{lemma}
The countability assumption on $\Xi$ is probably not necessary, but it helps us to deal with the overwhelming bookkeeping. The proof contains no new ideas, but is notationally rather heavy.
\begin{small}
\begin{proof} To keep the notation at bay we first prove that $\ker_2(\rho_{10})$ is an approximate subgroup of $\Xi^{\infty}$; we abbreviate $\rho \coloneqq \rho_{10}$ and note that $\ker_2(\rho) = \Xi^2 \cap \rho^{-1}(e)$ is symmetric and unital. It remains to  find a finite subset $ F^\ast$ of $\Xi^\infty$ such that $(\ker_2(\rho))^2\subset  F^\ast \ker_2(\rho)$. We will construct this set as in the proof of Lemma \ref{Intersection}.  

By Proposition \ref{PropShuffling} there is a finite set  $F\subset\Xi^3$  such that $\Xi^k\subset F^{k-1}\Xi$ for all $k \geq 2$, and we claim that
\begin{equation}\label{ker22}
\ker_2(\rho)^2 \subset  F^{3}\Xi  \cap \rho^{-1}(e).
\end{equation}
Indeed, if $a,b \in \ker_2(\rho)$, then $ab \in \Xi^4$ and hence $a, b, ab \in \Xi^{10}$. Since $\rho$ is a $10$-local morphism, this implies $\rho(ab)=\rho(a)\rho(b)=e$, hence $ab \in \Xi^4 \cap \rho^{-1}(e) \subset F^3 \Xi \cap \rho^{-1}(e)$.

\item We now consider the finite set $T \coloneqq  \{t \in F^{3} \mid t\Xi \cap \rho^{-1}(e) \neq \emptyset\}$ and fix $t \in T$. Observe that $t\Xi\subset F^3\Xi\subset \Xi^{10}$. We claim that for every $z_t \in t\Xi \cap \rho^{-1}(e)$ we have
\begin{equation}\label{eqn: partial ker}
t\Xi \cap \rho^{-1}(e) \subset z_t  \ker_2(\rho).
\end{equation}
Indeed, given $z \in t\Xi \cap \rho^{-1}(e)$, we have $z_t^{-1}z \in\Xi^2$. Since $\rho$ is a $10$-local morphism and $z_t^{-1}$, $z$, $z_t^{-1}z \in \Xi^{10}$, we have
\[
\rho (z_t^{-1} z)=\rho(z_t^{-1})\rho(z)=e \ \Rightarrow \  z_t^{-1}z \in \rho^{-1}(e).
\]
Therefore
$z_t^{-1}z\in \Xi^2 \cap \rho^{-1}(e) = \ker_2(\rho)$, hence $z \in z_t\ker_2(\rho)$, which proves \eqref{eqn: partial ker}.

Combining \eqref{ker22} and \eqref{eqn: partial ker} we deduce that that
\[
\ker_2(\rho)^2 \subset F^{3}\Xi  \cap \rho^{-1}(e) 
 = \bigcup_{t \in T} \left(t\Xi \cap \rho^{-1}(e)\right) \subset  \bigcup_{t \in T} z_t \ker_2(\rho) =  \left( \bigcup_{t \in T} z_t \right) \ker_2(\rho).
\]
Since $T$ is finite, this shows that $\ker_2(\rho)$ is an approximate subgroup of $\Xi^\infty$.

Essentially the same proof (with increased notation) shows that if $k\geq 2$ and if  $\rho_{(6k-2)}: (\Xi, \Xi^{6k-2}) \to (\Lambda, \Lambda^{6k-2})$ is a $(6k-2)$-local morphism, then the partial kernels $\ker_2(\rho_{(6k-2)}), \dots, \ker_k(\rho_{(6k-2)})$ are approximate subgroups of $\Xi^\infty$. Indeed, if we abbreviate $\rho\coloneqq\rho_{j}$ and choose $j \leq k$, then $\ker_j(\rho)^2 \subset  F^{2k-1}\Xi  \cap \rho^{-1}(e)$, and if we define $T \coloneqq  \{t \in F^{2k-1} \mid t\Xi \cap \rho^{-1}(e) \neq \emptyset\}$, then for all $t \in T$ we find 
$t\Xi \cap \rho^{-1}(e) \subset z_t  \ker_j(\rho)$. This then shows that
\[
\ker_j(\rho)^2 \subset F^{2k-1}\Xi  \cap \rho^{-1}(e) 
 = \bigcup_{t \in T} \left(t\Xi \cap \rho^{-1}(e)\right) \subset  \bigcup_{t \in T} z_t \ker_j(\rho) =  \left( \bigcup_{t \in T} z_t \right) \ker_j(\rho),
\]
hence $\ker_j(\rho)$ is an approximate subgroup.

As for the commensurability statement, we will prove the stronger claim that if $\rho_{k+1}: (\Xi, \Xi^{k+1}) \to (\Lambda, \Lambda^{k+1})$ is a $(k+1)$-local morphism for some $k \geq 2$, then $\ker_2(\rho_{k+1})$ is left-syndetic in $\ker_j(\rho_{k+1})$, for all $2 \leq j \leq k$. We assume without loss of generality that $\rho_1$ is surjective; since $\Xi^{\infty}$ is countable, it admits a proper left-invariant metric (see Lemma \ref{ExistLeftAdmissible}), and we fix such a metric $d$ once and for all. 

We are going to show that there exists an $N \in \bR$ such that for all 
$x \in \ker_k(\rho_{k+1})$ and $y \in \ker_2(\rho_{k+1})$ we have $d(x,y) < N$. This will then imply $d(y^{-1}x, e) < N$ by left-invariance, and hence $\ker_k(\rho_{k+1}) \subset \ker_2(\rho_{k+1})B(e, N)$, where $B(e, N)$ denotes the ball of radius $N$ around $e$. Since $d$ is proper this will prove the lemma.

Choose again a finite set $F \subset \Xi^{k+1}$ such that $\Xi^k \subset \Xi F$ and set $C_1 \coloneqq  \max_{f \in F}d(f,e)$. 
Given $\xi \in \ker_k(\rho_{k+1})=\Xi^k\cap \rho_{k+1}^{-1}(e)$, we choose $\xi_o \in \Xi$ and $f \in F$ such that $\xi = \xi_o f$. Then $d(\xi, \xi_o) = d(\xi_of, \xi_o) = d(f, e) \leq C_1$. 
Moreover, since $\rho_{k+1}$ is a $(k+1)$-local morphism, $\xi_o$, $f$, $\xi \in \Xi^{k+1}$, and
$\rho_{k+1}(\xi) = e$, we have $e = \rho_{k+1}(\xi) = \rho_{k+1}(\xi_o)\rho_{k+1}(f)$, and hence
\[
\rho_1(\xi_o)  = \rho_{k+1}(\xi_o)= \rho_{k+1}(f)^{-1} \in (\rho_{k+1}(F))^{-1} \cap \Lambda.
\]
Thus if we define $F' \coloneqq  (\rho_{k+1}(F))^{-1}$, then $\xi_o \in \rho_1^{-1}(F' \cap \Lambda)$ and hence
\begin{equation}\label{CoarseKernel1}
\ker_k(\rho_{k+1}) \subset N_{C_1}(\rho_1^{-1}(F' \cap \Lambda)),
\end{equation}
where $N_{C_1}$ denotes the $C_1$-neighborhood with respect to $d$. Since $F$ is finite, so is $F' \cap \Lambda$, say $F' \cap \Lambda  = \{\lambda_1, \dots, \lambda_N\}$. Since $\rho_1$ is surjective, we find $\xi_1, \dots, \xi_N \in \Xi$ with $\rho_1(\xi_i)=\lambda_i$ for all $i=1, \dots, N$. Now let $\xi_i' \in \rho_1^{-1}(\lambda_i)$ and set $\xi_i'' \coloneqq  \xi_i'\xi_i^{-1}$. Since $\xi_i', \xi_i^{-1}\in \Xi$, we have $\xi_i'' \in \Xi^2$, in particular, $\xi_i, \xi_i', \xi_i'' \in \Xi^{k+1}$. We deduce that $\rho_{k+1}(\xi_i'')=\rho_{k+1}(\xi_i')\rho_{k+1}(\xi_i^{-1})=\lambda_i\lambda_i^{-1}=e$, and hence 
$\xi_i'' \in \Xi^2 \cap \rho_{k+1}^{-1}(e) = \ker_2(\rho_{k+1})$. On the other hand, if we set $C_2 \coloneqq  \max_{i=1, \dots, N} d(\xi_i, e)$, then $d(\xi_i', \xi_i'') = d(\xi_i', \xi_i'\xi_i^{-1}) = d(e, \xi_i^{-1}) \leq C_2$. This shows that
\begin{equation}\label{CoarseKernel2}
\rho_1^{-1}(F' \cap \Lambda)= \bigcup_{i=1}^N \rho_1^{-1}(\lambda_i) \subset N_{C_2} (\ker_2(\rho_{k+1}) ),
\end{equation}
and combining \eqref{CoarseKernel1} and \eqref{CoarseKernel2} we obtain
$\ker_k(\rho_{k+1}) \subset N_{C_1+C_2}(\ker_2(\rho_{k+1}) )$, which finishes the proof.
\end{proof}
\end{small}
The following notion will become important in our study of quasi-isometry types of finitely-generated approximate groups.
\begin{definition}\label{DefFiniteKernel} Let $\rho: (\Xi, \Xi^\infty) \to (\Lambda, \Lambda^\infty)$ be a global morphism of approximate groups. We say that $\rho$ has \emph{finite kernels}\index{finite kernels} if for every $k \in \mathbb N$ the partial kernel $\ker_k(\rho)$ is finite. We then say that $(\Xi, \Xi^\infty)$ is a \emph{finite extension}\index{finite extension} of $(\rho(\Xi), \rho(\Xi)^\infty)$.
\end{definition}
Note that this does \emph{not} mean that the group homomorphism $\rho: \Xi^\infty \to \Lambda^\infty$ has finite kernel: if $\Gamma_1, \Gamma_2$ are countably infinite groups and $S$ is a finite generating set of $\Gamma_2$, then the projection $p: (\Gamma_1 \times S, \Gamma_1 \times \Gamma_2) \to \Gamma_1$ has finite kernel, since $\ker_k(p) = \{e\} \times S^k$ is finite, for all $k\in \bN$, but the underlying group homomorphism has infinite kernel. As a special case of Corollary \ref{PartialKernelGlobal} we may record:
\begin{corollary}\label{FiniteKernels} A global morphism $\rho: (\Xi, \Xi^\infty) \to (\Lambda, \Lambda^\infty)$ of countable approximate groups has finite kernels if and only if the second partial kernel $\ker_2(\rho)$ is finite.
\end{corollary}
\begin{definition} If $\Gamma_0 < \Gamma$ is a subgroup, then for every commensurability class $\omega$ of approximate subgroups of $\Gamma$ we define its \emph{restriction to $\Gamma_0$}\index{commensurability class!restriction} as $\omega \cap [\Gamma_0]_{\mathrm{co}}$.
\end{definition}
\begin{remark}[Restriction and support] 
By Remark \ref{DefinedOverSubgroup} every commensurability class of approximate subgroups in a subgroup $\Gamma_0 \subset \Gamma$ extends uniquely to a commensurability class in $\Gamma$, and as before we say that the latter class is \emph{supported on} $\Gamma_0$. With our new notation this happens if and only if $\omega \cap [\Gamma_0]_{\mathrm{co}} = \omega$. 
\end{remark}
The following was established by Hrushovski as a consequence of his classification theorem; no elementary proof seems to be known.
\begin{proposition}[{\cite[Prop.\ 5.10]{Hrushovski_Arithmeticity}}]\label{Minim} For any commensurability class $\omega$ of approximate subgroups of $\Gamma$ there exists a unique commensurability class $\langle \omega \rangle_{\min}$ with the following properties:
\begin{enumerate}[(i)]
\item Every $\Lambda \in \omega$ is supported on some $\Gamma_0 \in \langle \omega \rangle_{\min}$.
\item If $\Lambda \in \omega$ is supported on some subgroup $\Gamma_1 < \Gamma$, then $\Gamma_1$ contains some $\Gamma_0 \in \langle \omega \rangle_{\min}$.
\end{enumerate}
\end{proposition}
\begin{definition} $\langle \omega \rangle_{\min}$ is called the \emph{minimal commensurability class of supporting subgroups} of $\omega$.
We say that an approximate group $(\Lambda, \Lambda^\infty)$ is \emph{minimal} if $\Lambda^\infty \in \langle [\Lambda]_{\mathrm{co}} \rangle_{\mathrm{min}}$.\index{minimal approximate subgroup}
\end{definition}
\begin{corollary}\label{MinimalAGTrick} Let $(\Lambda, \Lambda^\infty)$ be an approximate group. Then:
\begin{enumerate}[(i)]
\item There exists a syndetic approximate subgroup $(\Lambda_o, \Lambda_o^\infty)$ of $(\Lambda^2, \Lambda^\infty)$ which is minimal.
\item  If  $(\Lambda, \Lambda^\infty)$ is minimal and $(\Xi, \Xi^\infty)$ is a syndetic approximate subgroup of $(\Lambda^k, \Lambda^\infty)$ for some $k \in \bN$, then $\Xi^\infty$ is syndetic in $\Lambda^\infty$.
\end{enumerate}
\end{corollary}
\begin{proof} (i) Let $\omega = [\Lambda]_{\mathrm{co}}$ and $H \in \langle \omega \rangle_{\min}$. If we set $\Lambda_o \coloneqq \Lambda^2 \cap H$, then $\Lambda_o \in \omega$ and
$\Lambda_o^\infty \subset H$, and thus $\Lambda_o^\infty \in \langle \omega \rangle_{\mathrm{min}}$ by minimality.
(ii) is immediate from the definition of $ \langle \omega \rangle_{\min}$ and the fact that $\Lambda \sim_{\mathrm{co}} \Lambda^k$ for all $k \in \bN$.
\end{proof}

We now turn to general pre-images. From now on let $\rho: \Gamma \to H$ be a quasimorphism. By Remark \ref{Drho} there then exists a finite symmetric subset $D \subset H$ such that
\begin{equation}\label{DSym}\rho(xy^{-1}) \in \rho(x)\rho(y)^{-1}D \qand \rho(xy) \in \rho(x)\rho(y)D \quad \text{ for all }x,y \in \Gamma.\end{equation}
We refer to any finite symmetric set $D$ satisfying \ref{DSym} as a \emph{symmetric defect set}. Since the intersection of symmetric defect sets is a symmetric defect set, there is a unique minimal such set $D_{\mathrm{min}}$. The key observation is as follows:
\begin{lemma}[Thick pre-images]\label{ThickPreim} Let $\rho: \Gamma \to H$ be a quasimorphism, $\Lambda \subset H$ an approximate subgroup and assume that there exists a symmetric defect set $D\subset H$ which satisfies $D \subset \mathrm{Comm}_H(\Lambda)$.
\begin{enumerate}[(i)]
\item For $j \geq 2$ we have $[\rho^{-1}(\Lambda^j D)]_{\mathrm{lco}} = [\rho^{-1}(\Lambda^j D) \rho^{-1}(\Lambda^j D)^{-1}]_{\mathrm{lco}}$, and these left-commensurability classes coincide for all $j \geq 2$ and all possible $D$.
\item If $\Xi \subset \Gamma$ is symmetric and satisfies $\rho^{-1}(\Lambda^2 D) \subset \Xi \subset  \rho^{-1}(\Lambda^j D) \rho^{-1}(\Lambda^j D)^{-1}$ for some $j \geq 2$, then $\Xi$ is an approximate subgroup. 
\item The sets $\rho^{-1}(\Lambda^j D) \rho^{-1}(\Lambda^j D)^{-1}$ are mutually commensurable approximate subgroups.
\end{enumerate}
\end{lemma}
\begin{proof} Fix $D$ and abbreviate $\Gamma_j \coloneqq  \rho^{-1}(\Lambda^j D)$.

(i) Let $j \geq 2$; using  \eqref{ReallyWeird} and symmetry of $\Lambda^j$ and $D$, we have
\[
\Gamma_j\Gamma_j^{-1} \subset \rho^{-1}(\Lambda^j DD \Lambda^j D).
\]
Since $\Lambda \sim_{co} \Lambda^j$ by Proposition \ref{PropShuffling}.(iii), we have $D \subset \mathrm{Comm}_H(\Lambda^j)$ by Lemma \ref{CommDepOnClass}, and thus Lemma \ref{LeftRightComm} yields a finite set $F_1\subset H$ such that $\Lambda^j DD \Lambda^j D \subset \Lambda^{2j}F_1$. By Proposition \ref{PropShuffling}.(i), the set $\Lambda^{2j}F_1$ is covered by finitely many right-translates of $\Lambda$, and hence Lemma \ref{HruMagic} yields a finite set $F_2$ such that $\Gamma_j\Gamma_j^{-1} \subset \rho^{-1}(\Lambda \Lambda^{-1} D)F_2 = \Gamma_2F_2$. This shows that $\Gamma_2 \subset \Gamma_j \subset \Gamma_j\Gamma_j^{-1} \subset \Gamma_2F_2$, hence all of these sets are left-commensurable for fixed $D$. To see independence of $D$, note that $\rho^{-1}(\Lambda^2 D_{\mathrm{min}}) \subset \Gamma_j$ and, by the argument above applied to $D_{\mathrm{min}}$, 
\[
\Gamma_j \Gamma_j^{-1} \subset \rho^{-1}(\Lambda^{2j}F_1) \subset \rho^{-1}(\Lambda\Lambda^{-1}D_{\mathrm{min}})F_2' = \rho^{-1}(\Lambda^2 D_{\mathrm{min}})F_2'
\]
for some finite set $F_2'\subset \Gamma$.

(ii) By assumption, $\Xi$ satisfies (AG1). Arguing as in (i) and using \eqref{EvenWeirder} we find finite sets $F_1 \subset H$ and $F_2 \subset \Gamma$ such that
\begin{align*}
\Xi^2 & \; \subset \; \rho^{-1}(\Lambda^j D D\Lambda^jD) \rho^{-1}(\Lambda^j D D\Lambda^jD) \; \subset \; \rho^{-1}(\Lambda^j D D\Lambda^jD\Lambda^j D D\Lambda^jDD)\\
& \; \subset\; \rho^{-1}(\Lambda^{4j}F_1) \; \subset \; \rho^{-1}(\Lambda \Lambda^{-1} D)F_2 \;= \; \Gamma_2F_2 \; \subset \; \Xi F_2.
\end{align*}
(iii) This follows from (i) and (ii).
\end{proof}
\begin{remark}[Inverse image of a commensurability class] 
If $\rho: \Gamma \to H$ is a quasimorphism and $\Lambda \subset H$ is an approximate subgroup, then we call $(\rho, \Lambda)$ an \emph{admissible pair} if there exists a symmetric defect set $D$ of $\rho$ such that $D \subset \mathrm{Comm}_H(\Lambda)$. In this case 
\[
\chi(\rho, \Lambda) \coloneqq [\rho^{-1}(\Lambda^2 D) \rho^{-1}(\Lambda^2D)^{-1}]_{\mathrm{co}} \subset [\rho^{-1}(\Lambda^2D)]_{\mathrm{lco}}
\]
is a commensurability class of approximate subgroups, which depends only on $\rho$ and $\Lambda$. We claim that it actually depends only on $\rho$ and the commensurability class of $\Lambda$. Indeed, if $\Lambda_1 \sim_{co}\Lambda$, then $\mathrm{Comm}_H(\Lambda_1) = \mathrm{Comm}_H(\Lambda)$ and we find a finite set $F$ with $\Lambda_1 \subset \Lambda F$. As in the proof of Lemma \ref{ThickPreim} we find finite sets $F_1 \subset H$ and $F_2 \subset \Gamma$ such that
\[
\rho^{-1}(\Lambda_1^2 D) \subset \rho^{-1}(\Lambda F \Lambda F D) \subset \rho^{-1}(\Lambda^2 F_1) \subset \rho(\Lambda^2D) F_2.
\]
A symmetric argument then shows that $\rho^{-1}(\Lambda^2D) \sim_{lco} \rho^{-1}(\Lambda_1^2D)$, and hence $\chi(\rho, \Lambda) = \chi(\rho, \Lambda_1)$. In the sequel we refer to $\chi(\rho, [\Lambda]_{\mathrm{co}}) \coloneqq \chi(\rho, \Lambda)$ as the \emph{inverse image}\index{commensurability class!inverse image} of $[\Lambda]_{\mathrm{co}}$ under $\rho$.
\end{remark}
We now investigate various special cases; we first consider homomorphisms.
\begin{example}[Inverse images under homomorphisms] If $\rho: \Gamma \to H$ is a homomorphism, then we can choose $D = \{e\}$ and hence
every approximate subgroup $\Lambda$ is compatible with $\rho$. Moreover, $\rho^{-1}(\Lambda^j)$ is symmetric for every $j \geq 2$, hence 
\[\chi(\rho, [\Lambda]_{\mathrm{co}}) = [\rho^{-1}(\Lambda^2)]_{\mathrm{co}} = [\rho^{-1}(\Lambda^3)]_{\mathrm{co}} = \dots\]
Example \ref{IntersectionFail} shows that this class is not represented by $\rho^{-1}(\Lambda)$ in general. However, if $\rho$ is surjective, then the square can be avoided by the following proposition.
\end{example}
\begin{proposition}[Surjective case]\label{SurjHom}\label{Preimage1} If $\rho: \Gamma \to H$ is an epimorphism and $\Lambda \subset H$ is an approximate subgroup, then $\Xi \coloneqq \rho^{-1}(\Lambda)$ is an approximate subgroup which represents $\chi(\rho, \Lambda)$. Moreover, if $\Lambda$ generates $H$, then $\Xi$ generates $\Gamma$.
\end{proposition}
\begin{proof} We fix a section $\sigma: H \to \Gamma$ of $\rho$ so that $\rho \circ \sigma=\id_H$ and set $F \coloneqq \sigma(F_\Lambda)$, where $F_\Lambda$ is as in \eqref{FLambda}. Given $\xi_1, \xi_2 \in \Xi$ we have $\rho(\xi_1\xi_2) \in \Lambda^2 \subset \Lambda F_\Lambda$, hence there exist $x \in \Lambda$ and $y \in F_\Lambda$ such that $\rho(\xi_1\xi_2) = xy$. Then
\[
\rho(\xi_1\xi_2 \sigma(y)^{-1}) = xy\rho(\sigma(y)^{-1}) = x \in \Lambda,
\]
hence $\xi_1\xi_2\sigma(y)^{-1} \in \Xi$. This shows that $\Xi^2 \subset \Xi F$. The first statement now follows from the inclusions  $\Xi \subset \rho^{-1}(\Lambda^2) \subset \Xi^2 \subset \Xi F$ and the fact that $\Xi$ is unital and symmetric.

For the second statement assume that $\Lambda$ generates $H$ and let $g \in \Gamma$. Then $\rho(g) = \lambda_1 \cdots \lambda_k$ for certain $\lambda_i \in \Lambda$, and if we set $\xi_i \coloneqq \sigma(\lambda_i)$, then $\xi_i \in \Xi$ and $\xi_0 \coloneqq g \xi_k^{-1}\cdots \xi_1^{-1} \in \ker(\rho) = \rho^{-1}(\{e\}) \subset \Xi$, hence  $g = \xi_0\xi_1 \cdots \xi_k \in \Xi^\infty$.
\end{proof}

\begin{example}[Fundamental commensurability class of a homomorphism]\label{FundHom} Let $\Gamma$ be a (discrete) group, $H$ be a locally compact group and $\rho: \Gamma \to H$ be a group homomorphism. We recall from Remark \ref{FundCommClass} that $\chi(H)$ denotes the fundamental commensurability class of $H$ and define the \emph{fundamental commensurability class} of $\rho$\index{fundamental commensurability class!of a homomorphism} by
\[
\chi(\rho) \coloneqq \chi(\rho, \chi(H)).
\]
If $W \subset H$ is a relatively compact symmetric identity neighborhood, then there exist compact symmetric identity neighborhoods $W_0$ and $W_1$ such that $W_0^2 \subset W \subset W_1^2$, and hence $\chi(\rho)$ is represented by $\rho^{-1}(W)$. Note that if $H$ is discrete, then we may choose $W = \{e\}$; thus for discrete targets $H$, $\chi(\rho)$ is just the kernel of $\rho$. To get interesting examples, one should consider a non-discrete target $H$.
\end{example}
\begin{definition}[Hrushovski]\label{DefLaminar} We say that a commensurability class $\omega$ of approximate subgroups of a group $\Gamma$ is \emph{laminar}\index{laminar}\index{approximate subgroup!laminar} if $\omega = \chi(\rho)$, for some homomorphism $\rho: \Gamma \to H$ into a locally compact group.
\end{definition}
We now return to the case where $H$ is discrete, but allow $\rho$ to be a non-trivial quasimorphism.
\begin{example}[Quasikernel]\label{qkerAbstract} If $\Lambda = \{e\}$, i.e. $\Lambda$ is trivial, then every quasimorphism $\rho: \Gamma \to H$ is compatible with $\Lambda$. If $D$ is any symmetric defect set of $\rho$, then
\[
\qker(\rho) \coloneqq \chi(\rho, \{e\}) = [\rho^{-1}(D)\rho^{-1}(D)^{-1}]_{\mathrm{co}}.
\]
If $\rho$ happens to be symmetric, then $\rho^{-1}(D)$ itself is already symmetric and hence 
\[
\qker(\rho)  = [\rho^{-1}(D)]_{\mathrm{co}}.
\]
We refer to any representative of $\qker(\rho)$ as a \emph{quasikernel}\index{quasikernel} of $\rho$. Technically, this is the pullback of the canonical commensurability class of $H$ under $\rho$. 
\end{example}
\begin{example}[Partial quasikernels]\label{PartialQuasikernel}
Let $\rho: (\Xi, \Xi^\infty) \to (\Lambda, \Lambda^\infty)$ be a global quasimorphism of approximate groups. We then define $\mathrm{pqker}(\rho)$ as the commensurability class of the restriction of $\qker(\rho)$ to $[\Xi]_{\mathrm{co}}$. Thus, with $D$ as in Example \ref{qkerAbstract}, we have
\[
\mathrm{pqker}(\rho) = [\Xi^2 \cap \rho^{-1}(D)]_{\mathrm{co}} =  [\Xi^3 \cap \rho^{-1}(D)]_{\mathrm{co}} = \dots
\]
Any representative of this class is called a \emph{partial quasikernel}\index{partial quasikernel}\index{quasikernel!partial} of $\rho$. This generalizes the notion of a partial kernel of a global morphism.
\end{example}
\begin{example}[Quasikernels of $\Z$-valued quasimorphisms]\label{QkerZ}
If $\rho: \Gamma \to \Z$ is a symmetric $\Z$-valued quasimorphism, then a symmetric defect set of $\rho$ is given by $[-d(\rho), d(\rho)]$ with $d(\rho)$ as in \eqref{DefectNew}.
We deduce that
\[
\qker(\rho, C) \coloneqq\{x \in \Gamma \mid |\rho(x)| \leq C\}
\]
is a quasikernel for all $C \geq d(\rho)$.
\end{example}

\begin{example}[Global quasimorphisms]\label{GlobQuaMo} Let $\rho: (\Xi, \Xi^\infty) \to (\Lambda, \Lambda^\infty)$ be a global quasimorphism of approximate groups with symmetric defect set $D$. Then by Proposition 2.27 $\Lambda$ and $\rho$ are compatible and hence $\chi(\rho, \Lambda)$ is defined.

Assume for simplicity that $\rho$ is symmetric so that $\Theta \coloneqq \rho^{-1}(D\Lambda^2 D)$ is symmetric. Since $D$ commensurates $\Lambda^2$, the set $D \Lambda^2 D$ is contained in finitely many right-translates of $\Lambda$, and hence $\Theta \subset \rho^{-1}(\Lambda^2 D) F$ for some finite set $F$ by Lemma \ref{HruMagic}. This shows that $\Theta$ is left-commensurable to $\rho^{-1}(\Lambda^2 D)$. In view of Lemma \ref{ThickPreim} this implies the first part of the following result. 
\end{example}
\begin{corollary} If $\rho: (\Xi, \Xi^\infty) \to (\Lambda, \Lambda^\infty)$ is a symmetric global quasimorphism of approximate groups with symmetric defect set $D$, then
$\rho^{-1}(D\Lambda^2 D)$ is an approximate subgroup of $\Xi^\infty$ and 
\begin{equation}\label{chirhoLambdasymm}
\chi(\rho, \Lambda) = [\rho^{-1}(D\Lambda^2 D)]_{\mathrm{co}}.
\end{equation}
If $\rho$ is surjective, then we can replace $\rho^{-1}(D\Lambda^2 D)$ by $\rho^{-1}(D\Lambda D)$.
\end{corollary}
\begin{proof} The statement about $D\Lambda^2 D$ is contained in Example \ref{GlobQuaMo}. Now assume that $\rho$ is surjective; we fix a section $\sigma: H \to \Gamma$ of $\rho$ and argue as in the proof of Proposition \ref{SurjHom}.

Thus let $\Theta \coloneqq \rho^{-1}(D\Lambda D)$ and $\Xi \coloneqq \rho^{-1}(D\Lambda^2 D)$ which are symmetric and unital. By \eqref{EvenWeirder} we have $\Theta^2 \subset \Xi^2 \subset \rho^{-1}(A)$, where $A \coloneqq D\Lambda^2 D D\Lambda^2 D D$. Since $D \subset \Lambda^{\infty} \subset \mathrm{Comm}_H(\Lambda)$, the set $A$ is covered by finitely many translates of $\Lambda^4$, hence of $\Lambda$, i.e.\ we find a finite set $F^*$ such that $\rho(\Theta^2) \subset \rho(\Xi^2) \subset \Lambda F^*$. Given $\xi_1, \xi_2 \in \Xi$ we thus find $x \in \Lambda$ and $y \in F_\Lambda$ such that $\rho(\xi_1\xi_2) = xy$.
Then 
\[\rho(\xi_1\xi_2 \sigma(y)^{-1}) \in \rho(\xi_1 \xi_2)\rho(\sigma(y))^{-1}D = xy y^{-1}D \subset \Lambda D,\]
and hence $\xi_1\xi_2\sigma(y)^{-1} \in \rho^{-1}(\Lambda D) \subset \Theta$. This shows that 
\[\Theta^2 \subset \Xi^2 \subset \Theta \widehat{F} \subset \Xi \widehat{F}, \quad \text{ where }\widehat{F} = \sigma(F^*).\]  This proves that $\Xi$ and $\Theta$ are commensurable approximate subgroups.
\end{proof}
\begin{example}[The abelian case] Assume that $H$ is abelian (written additively) and that $\rho: \Gamma \to H$ is a symmetric quasimorphism. Then
$\rho$ is compatible with every approximate subgroup $\Lambda \subset H$, and for any symmetric defect set $D$ we have
\[\chi(\rho, \Lambda) = [\rho^{-1}(\Lambda + \Lambda +D)]_{\mathrm{co}}.\]
If $\rho$ is surjective, then we can replace $\Lambda + \Lambda$ by $\Lambda$.
\end{example}
It is important to also consider inverse images of approximate subgroups under \emph{topological} quasimorphisms into locally compact groups in the sense of Remark \ref{TopQM}. We will not do this systematically here and instead confine ourselves to two results; the first concerns quasikernels of real-valued topological quasimorphisms and the second concerns inverse images of fundamental commensurability classes and generalizes Example \ref{FundHom}.

To state the first result, let $\rho: \Gamma \to \R$ be a symmetric (topological) quasimorphism. As in the integer-valued case we denote
\[
    \qker(\rho, C) \coloneqq  \{g \in \Gamma \mid |\rho(g)| \leq C\} \quad (C \geq 0).
\]
\begin{proposition}\label{QKReal} For all $C \geq d(\rho) + 2$ the subsets $\qker(\rho, C)$ are mutually commensurable approximate subgroups of $\Gamma$.
\end{proposition}
\begin{proof} By Remark \ref{quasiremark}.(vi) there exists an integer valued symmetric quasimorphism $\bar \rho$ such that $\|\rho- \bar \rho\|_\infty \leq 1/2$ and $d(\bar \rho) \leq d(\rho) + 3/2$. In particular, for all $C > 1/2$ we have
\[
\qker(\bar \rho, C-1/2) \subset \qker(\rho, C) \subset \qker(\bar \rho, C+1/2). 
\]
If $C \geq d(\rho) + 2$, then $C - 1/2 \geq d(\bar \rho)$, and hence $\qker(\rho, C)$ is a quasikernel for $\bar \rho$ by Example \ref{QkerZ}.
\end{proof}
\begin{notation}
As in the integral case we set $\qker(\rho) \coloneqq [\qker(\rho, d(\rho)+2)]_{\mathrm{co}}$ and refer to approximate subgroups in $\qker(\rho)$ as \emph{quasikernels}\index{quasikernel!of a real-valued quasimorphism} of $\rho$. 

By the same argument as in the proof of Proposition \ref{QKReal}, $\qker(\rho)$ does not change if we replace $\rho$ by a symmetric quasimorphism at bounded distance from $\rho$.
\end{notation}
We now turn to our second result concerning pre-images under topological quasimorphisms, {cf.\ \cite[Prop.\ 5.12]{Hrushovski_Arithmeticity}}.
\begin{proposition}\label{TopHru} Let $\rho: \Gamma \to H$ be a topological quasimorphism and let $K \subset H$ be symmetric and relatively compact satisfying \eqref{WeirdStuff}.
\begin{enumerate}[(i)]
\item If $\Lambda \subset \Gamma$ is symmetric and $\rho^{-1}(U_1K) \subset \Lambda \subset \rho^{-1}(U_2K)$ for relatively compact identity neighborhoods $U_1, U_2 \subset H$, then $\Lambda$ is an approximate subgroup.
\item The approximate subgroups in (i) define a common commensurability class $\chi(\rho)$, which is represented by $\rho^{-1}(UK) \rho^{-1}(UK)^{-1}$ for any relatively compact identity neighborhood $U$.
\end{enumerate}
\end{proposition}
\begin{proof} (i) Let $B \subset H$ be an identity neighborhood with $BB^{-1} \subset U_1$. Then $\Lambda^2 \subset \rho^{-1}(A)$, where $A \coloneqq U_2KU_2KK$. Since the closure of $A$ is compact, it is contained in $BF_1$ for some finite set $F_1$, hence by Lemma \ref{HruMagic} there exists a finite set $F_2$ with 
\[
\Lambda^2 \subset \rho^{-1}(BB^{-1}K)F_2 \subset \rho^{-1}(U_1K)F_2 \subset \Lambda F_2,
\]
hence $\Lambda$ is an approximate subgroup. 

(ii) Concerning the first statement, it is enough to show that if $U_1 \subset U_2$ are relatively compact identity neighborhoods, then $\rho^{-1}(U_1K) \subset \rho^{-1}(U_2K)$ is left-syndetic. For this let $V$ be an identity neighborhood with $VV^{-1} \subset U_1$. Since $A = U_2K$ is covered by finitely many translates of $V$, Lemma \ref{HruMagic} yields a finite set $F$ such that
\[
\rho^{-1}(U_2K) \subset \rho^{-1}(VV^{-1}K)F \subset \rho^{-1}(U_1K)F.
\]
For the second statement we observe that by \eqref{WeirdStuff} we have
\[
\rho^{-1}(UK) \subset \rho^{-1}(UK) \rho^{-1}(UK)^{-1} \subset \rho^{-1}(UKKU^{-1}K) = \rho^{-1}(VK),
\]
where $V\coloneqq UKKU^{-1}$ is an open identity neighborhood.
\end{proof}
\begin{definition} The commensurability class $\chi(\rho)$ is called the \emph{fundamental commensurability class}\index{fundamental commensurability class!of a quasimorphism} of the topological quasimorphism $\rho$, and any approximate subgroup in $\chi(\rho)$ is called \emph{quasi-laminar}\index{quasi-laminar}\index{approximate subgroup!quasi-laminar}.
\end{definition}
The following is a very weak version of \cite[Thm.\ 5.16]{Hrushovski_Arithmeticity}:
\begin{theorem}[Hrushovski]\label{HruBaby}
Let $\Lambda \subset \Gamma$ be an approximate subgroup. Then there exists a subgroup $\Gamma_0 < \Gamma$, a connected Lie group $H$ and a topological quasimorphism $\rho: \Gamma_0 \to H$ such that $\Lambda$ is supported on $\Gamma_0$ and $[\Lambda]_{\mathrm{co}} = \chi(\rho)$. In particular, every approximate subgroup is quasi-laminar.
\end{theorem}
The following result \cite[Prop.\ 5.29]{Hrushovski_Arithmeticity} shows that it is not possible to replace ``quasi-laminar'' by ``laminar'' in Theorem \ref{HruBaby}:
\begin{proposition}[Hrushovski]\label{NonLaminar} Let $f: \Gamma \to \R$ be a homogeneous real-valued quasimorphisms. Then the quasikernel of $f$ is not laminar.
\end{proposition}
Theorem \ref{HruBaby} is the starting point of a more precise classification of approximate subgroups (cf. \cite[Cor.\ 5.26]{Hrushovski_Arithmeticity}). While this provides, at least in principle, a parametrization of all commensurability classes of approximate subgroups for a given group, it brings very little geometric information about the resulting approximate groups.

\section{Good models, model sets and approximate lattices}
Proposition \ref{NonLaminar} shows that laminarity is a rather special property of commensurability classes of approximate subgroups. We can provide a characterization of this property using the following notion.
\begin{definition} Let $\Gamma$ be a group and $\Lambda \subset \Gamma$ be an approximate subgroup. A \emph{good model}\index{good model}  for $(\Lambda, \Gamma)$ is a pair $(H,\rho)$, where $H$ is a locally compact group and $\rho: \Gamma \to H$ is a group homomorphism such that 
\begin{enumerate}[(M1)]
\item $\rho(\Lambda)$ is relatively compact in $H$;
\item there exists an identity neighborhood $U$ in $H$ such that $\rho^{-1}(U) \subset \Lambda$.
\end{enumerate}
\end{definition}
The definition is in close analogy with the notion of a good model for an ultra approximate group as introduced by Breuillard, Green and Tao in their study of finite approximate groups \cite[Def. 3.5]{BGT}.  Note that if $(H, \rho)$ is a good model for $(\Lambda, \Gamma)$, then it is also a good model for $(\Lambda, \Lambda^\infty)$ and for $(\Lambda^k, \Lambda^\infty)$ for every $k \geq 1$.
\begin{theorem}[Hrushovski]\label{GoodModelTheorem}
For an approximate group $(\Lambda, \Lambda^\infty)$ the following are equivalent:
\begin{enumerate}[(i)]
\item $[\Lambda]_{\mathrm co}$ is laminar.
\item $(\Lambda^k, \Lambda^\infty)$ admits a good model $(H, \rho)$ for some $k \in \bN$.
\end{enumerate}
Moreover, in the situation of (ii) we have $\Lambda \in \chi(\rho)$.
\end{theorem}
\begin{proof} (i)$\implies$(ii) Let us explain how this follows from \cite{Hrushovski_Arithmeticity}.  Firstly, by Proposition \ref{PropShuffling} we have $\Lambda^\infty = \bigcup \Lambda^k  \subset \Comm_{\Lambda^\infty}(\Lambda)$. Then \cite[Lemma 5.4]{Hrushovski_Arithmeticity} says that there exists a homomorphism $\rho: \Lambda^\infty \to H$ such that $\Lambda \in \chi(\rho)$. There thus exists a compact symmetric identity neighborhood $W \subset H$ such that
\[
\rho^{-1}(W) \subset \Lambda F_1 \qand \Lambda \subset \rho^{-1}(W)F_2
\]
for some finite $F_1, F_2 \in \Lambda^\infty$. The latter implies that $\rho(\Lambda) \subset W \rho(F_2)$, which shows (M1). Since $F_1 \subset \Lambda^\infty$ is finite, it is contained in $\Lambda^{k-1}$ for some $k \in \bN$, and hence we can choose $U$ as the interior of $W$ in (M2).

(ii)$\implies$(i) Let $(H,\rho)$ be a good model for $(\Lambda^k, \Lambda^\infty)$, and let
$\Gamma \coloneqq \Lambda^\infty$ and $\Gamma_H \coloneqq \rho(\Gamma)$. We may replace $H$ by the closure of $\Gamma_H$ and thereby assume that $\Gamma_H$ is dense in $H$. Let $U \subset H$ be an open identity neighborhood with $\rho^{-1}(U) \subset \Lambda^k$ and set $W \coloneqq \overline{\rho(\Lambda)^{2k}}$. Then $\Lambda^{2k} \subset \rho^{-1}(W)$, and we claim that this subset is syndetic.

Indeed, since $\Gamma_H$ is dense we have $H = \Gamma_H U$. Since $W$ is compact we thus find a finite subset $F_1 \subset \Gamma$ such that
$W \subset \rho(F_1) U$. If $\gamma \in \rho^{-1}(W)$, then 
\[\rho(\gamma) \in W \cap \rho(\Gamma) \subset \rho(F_1)U \cap \rho(\Gamma) \subset \rho(F_1)(U \cap \rho(\Gamma)) \subset  \rho(F_1)\rho(\rho^{-1}(U)) \subset \rho(F_1)\rho(\Lambda^k).\] We thus find $f_1 \in F_1$ and $\lambda \in \Lambda^k$ with
\[
\rho(\gamma) = \rho(f_1)\rho(\lambda) \implies \rho(\lambda^{-1}f_1^{-1}\gamma) = e \implies \lambda^{-1}f_1^{-1}\gamma \in \rho^{-1}(U) \subset \Lambda^k.
\]
This shows that $\rho^{-1}(W) \subset F_1 \Lambda^{2k}$, and hence $\rho^{-1}(W)$ and $\Lambda^{2k}$ are commensurable. We deduce that $[\Lambda]_{\mathrm{co}} = [\Lambda^2]_{\mathrm{co}} = \chi(\rho)$, hence $[\Lambda]_{\mathrm co}$ is laminar. 
\end{proof}
There is a close relation between laminarity, good models and so-called model sets, which are certain Delone sets in locally compact groups (in the sense of Definition \ref{DefDelone}) with good discreteness properties. The following construction was introduced by Meyer \cite{Meyer1972} for abelian groups around 40 years before good models were introduced; the extension to the non-abelian case is fairly recent \cite{BH, BHP1}. 
\begin{example}[Cut-and-project construction]\label{CuPModel}
Let $G, H$ be locally compact groups, and denote by $\proj_G: G \times H \to G$ and $\proj_H: G \times H \to H$ the canonical projections. Given a subgroup $\Gamma<G \times H$ we denote its projections  by $\Gamma_G \coloneqq  \proj_G(\Gamma)$ and $\Gamma_H \coloneqq  \proj_H(\Gamma)$. Assume that the restriction $p: \Gamma \to \Gamma_G$ of $\proj_G$ is injective, and define
$\tau \coloneqq \proj_H \circ p^{-1}: \Gamma_G \to H$.
Then $\tau$ is a homomorphism, and thus for every relatively compact symmetric identity neighborhood $W$ in $H$ the subset
$\Lambda(\Gamma, W) \coloneqq  \tau^{-1}(W)$ is an approximate subgroup of $G$. Explicitly, we have
\[
\Lambda(\Gamma, W) = \proj_G(\Gamma \cap (G \times W)),
\]
i.e.\ $\Lambda(\Gamma, W)$ arises from $\Gamma$ by first cutting it with the ``strip'' $(G \times W)$ in $G \times H$, and then projecting down to $G$. Sets of the form $\Lambda(\Gamma, W)$ are thus sometimes referred to as \emph{cut-and-project sets}\index{cut-and-project set}; in the modern literature on aperiodic order the more commonly used term is \emph{model sets}\index{model set}. Note that according to our definition the lattice $\Gamma$ used to define a model set $\Lambda(\Gamma, W)$ does not have to be cocompact; if it is, then we say that $\Lambda(\Gamma, W)$ is a \emph{uniform model set}. One can show that every uniform model set is a Delone set in the ambient locally compact group (see \cite{BH}).

If $\Lambda = \Lambda(\Gamma, W)$ is a model set, then $\Lambda$ is an approximate subgroup of $G$ and $(\Lambda, G)$ has a good model, namely $(H, \tau)$. This is linguistically convenient, though an epistemological accident.  
\end{example}
Model sets play a central role in aperiodic order and, more specifically, in the theory of approximate lattices. In \cite{BH}, Bj\"orklund and the second named author introduced the notion of a uniform approximate lattice and discussed various tentative definitions of a non-uniform approximate lattice. After the appearance of \cite{Hrushovski_Arithmeticity} and \cite{MachadoDef}, the relations between these various definitions have been clarified. The following definition\footnote{The definition is stronger than the original definition from \cite{BH}, but weaker than the notion of a \emph{strong} approximate lattice.} has become widely accepted:
\begin{definition}\label{DefUniformApproximateLattice} An approximate subgroup $\Lambda$ of a locally compact second countable (lcsc) 
group $G$ is called an \emph{approximate lattice}\index{approximate lattice} if it is discrete and there exists a Borel subset $\mathcal F \subset G$ of finite (left-)Haar measure such that $G = \Gamma \mathcal F$. An approximate lattice is called \emph{uniform}\index{approximate lattice!uniform} if it is a Delone set in $G$, i.e.\ if $\mathcal F$ can be chosen to be compact.
\end{definition}
We will see later (see Theorem \ref{ThmIsometricActions} below) that uniform approximate lattices are closely related to geometric isometric actions of approximate groups. The following theorem summarizes our current knowledge concerning approximate lattices; here, an approximate lattice $\Lambda \subset G$ is called \emph{strong} if its orbit closure in the space of closed subsets of $G$ (with the Chabauty--Fell topology) admits a $G$-invariant probability measure $\mu$ with $\mu(\{\emptyset\}) = 0$, see \cite{BH}.
\begin{theorem}[Hrushovski, Machado, Bj\"orklund--Hartnick--Pogorzelski]\label{HuMo}
For an approximate lattice $\Lambda \subset G$ the following statements are equivalent:
\begin{enumerate}[(i)]
\item $\Lambda^6$ contains a model set as a syndetic subset.
\item $\Lambda$ is commensurable to a model set.
\item $\Lambda$ is commensurable to a strong approximate lattice.
\item $[\Lambda]_{\mathrm{co}}$ is laminar.
\end{enumerate}
Moreover, if $G$ is amenable or semisimple, then every approximate lattice in $G$ satisfies these conditions.
\end{theorem}
The implication (i)$\Rightarrow$(ii) is clear and (ii)$\Rightarrow$(iii) was established in \cite{BHP1}. The implications (iii)$\Rightarrow$(iv)$\Rightarrow$(i) were established by Machado \cite{MachadoDef} based on work of Hrushovski \cite{Hrushovski_Arithmeticity}. The fact that every approximate lattice in a semisimple group is strong is due to Hrushovski \cite{Hrushovski_Arithmeticity}, whereas the amenable case is due to Machado \cite{Machado_GoodModels}. In the abelian case the theorem reduces to a celebrated theorem of Meyer \cite{Meyer1972} which can be seen as the starting point of the modern theory of aperiodic order.

\chapter{Large-scale geometry of approximate groups}\label{ChapLargeScale}

As recalled in Appendix \ref{AppendixGGT}, every countable group $\Gamma$ gives rise to a canonical coarse equivalence class $[\Gamma]_c$ of metric spaces, and if $\Gamma$ is finitely-generated, then this coarse equivalence class can be refined into a canonical QI type $[\Gamma] \subset [\Gamma]_c$ of metric spaces. In this chapter we generalize these notions to countable approximate groups. 

The definition of a \emph{canonical coarse class of a countable approximate group} is in complete analogy with the group case. Just as the coarse class of a group is invariant under group isomorphisms and commensurability, the coarse class of an approximate group is invariant under $2$-local isomorphisms (Corollary \ref{CoarseType2Local}) and commensurability (Lemma \ref{CommCoarseEquiv}).

The canonical QI type of a finitely-generated group admits two different generalizations to the approximate setting, namely the \emph{external QI type} of an \emph{algebraically finitely-generated approximate group} and the \emph{internal QI type} of a \emph{geometrically finitely-generated approximate group}. We will see that the interplay between these two competing definitions is at the heart of geometric approximate group theory.

Since every countable approximate group gives rise to a canonical coarse class, we can use coarse invariants of metric spaces to study countable approximate groups. Similarly, we can use QI invariants to study countable approximate groups which are finitely-generated in a suitable sense. We discuss some examples of such invariants, most of which will reappear in later chapters.

We will freely use the language of (metric) coarse geometry; our terminology concerning various classes of coarse maps is summarized in Appendix \ref{AppendixLSG}. We will also assume some basic facts concerning geometric group theory, not just for finitely-generated groups, but more generally for countable and locally compact second countable (lcsc)\index{group!lcsc}\index{lcsc group} groups. See Appendix \ref{AppendixGGT} for background and notation.

\section{The coarse space associated with a countable approximate group}

Our first goal in this chapter is to associate with each countable approximate group $(\Lambda, \Lambda^\infty)$ a canonical coarse equivalence class $[\Lambda]_c$ of metric spaces, generalizing the construction in the group case summarized in Section \ref{SecCoarseClass} of Appendix \ref{AppendixGGT}. Our starting point is the following easy observation; see Definition \ref{DefLeftAdmissible} and Corollary \ref{AutomaticAdmissibleDiscrete} for the notion of a left-admissible metric on a countable group.
\begin{lemma}\label{ExternalQIType} Let $\Gamma$ be a countable group, $\Lambda \subset \Gamma$ be a subset and $d$ and $d'$ be left-admissible metrics on $\Gamma$. Then:
\begin{enumerate}[(i)]
\item The identity map $(\Lambda, d|_{\Lambda \times \Lambda}) \to (\Lambda, d'|_{\Lambda \times \Lambda})$ is a coarse equivalence.
\item If $d$ and $d'$ are moreover large-scale geodesic (e.g.\ word metrics with respect to finite symmetric generating sets), then the identity $(\Lambda, d|_{\Lambda \times \Lambda}) \to (\Lambda, d'|_{\Lambda \times \Lambda})$ is a quasi-isometry.
\end{enumerate}
\end{lemma}
\begin{proof} The identity map $(\Gamma, d) \to (\Gamma, d')$ is a coarse equivalence by Corollary \ref{CoarseGroupClassCountable}, and if $d$ and $d'$ are large-scale geodesic, then it is a quasi-isometry by Corollary \ref{GromovTrivialSymmetric}. Restricting to $(\Lambda, d)$ we obtain a coarse embedding, respectively a quasi-isometric embedding. Since the image of $\Lambda$ is given by $\Lambda$, the result follows.
\end{proof}
By Lemma \ref{ExternalQIType}.(i) the following notion is well-defined, i.e.\ independent of the choice of a left-admis\-sible metric $d$:
\begin{definition} \label{def: canonical coarse class}
Let $(\Lambda, \Lambda^\infty)$ be a countable approximate group and let $d$ be a left-admissible metric on $\Lambda^\infty$. Then
\[
[\Lambda]_c \coloneqq  [(\Lambda, d|_{\Lambda \times \Lambda})]_c
\]
is called the \emph{canonical coarse class of $\Lambda$}\index{canonical coarse class}.
\end{definition}
\begin{remark}[Independence of the ambient group] Assume that $(\Lambda, \Lambda^\infty)$ is an approximate subgroup of a countable group $\Gamma$. If $d$ is a left-admissible metric on $\Gamma$, then $d|_{\Lambda^\infty \times \Lambda^\infty}$ is a left-admissible metric on $\Lambda^\infty$, and hence
\[
[\Lambda]_c \coloneqq  [(\Lambda, d|_{\Lambda \times \Lambda})]_c.
\]

Thus the canonical coarse class of $\Lambda$ is intrinsic in the sense that it is independent of the ambient group used to define it. 

Note that, since the restriction of a left-admissible metric is proper, the class $[\Lambda]_c$ admits a representative which is a proper metric space. Also note that for the definition of the canonical coarse class we do not use that $\Lambda$ is an approximate subgroup, so one could as well associate a canonical coarse class of (pseudo-)metric spaces with an arbitrary subset of a countable group. 
\end{remark}
The following three lemmas and their corollaries record basic properties of the coarse equivalence classes of approximate groups.
\begin{lemma}[Commensurability invariance]\label{CommCoarseEquiv}\label{LambdaKCoarse}  Let $\Gamma$ be a countable group and $\Lambda, \Lambda' \subset \Gamma$ be approximate subgroups. Then 
\[
[\Lambda]_{\mathrm{co}} = [\Lambda']_{\mathrm{co}} \implies [\Lambda]_c = [\Lambda']_c.
\]
In particular, $[\Lambda]_c = [\Lambda^k]_c$ for every $k \geq 1$.
\end{lemma}
\begin{proof} Let $d$ be a left-admissible metric on $\Gamma$ and let $F_1, F_2 \subset \Gamma$ be finite such that $\Lambda \subset \Lambda'F_1$ and $\Lambda' \subset \Lambda F_2$. Then every $\lambda \in \Lambda$ can be written as $\lambda = \lambda'f$ for some $\lambda' \in \Lambda'$ and $f \in F_1$ and hence $d(\lambda, \lambda') = d(f,e)$, which is uniformly bounded. This shows that $\Lambda$ is contained in a bounded $d$-neighborhood of $\Lambda'$, and by a symmetric argument $\Lambda'$ is contained in a bounded $d$-neighborhood of $\Lambda$. This implies that $(\Lambda, d|_{\Lambda \times \Lambda})$ and $(\Lambda', d|_{\Lambda' \times \Lambda'})$ are coarsely equivalent, and the second statement then follows from Proposition \ref{PropShuffling}.
\end{proof}
If $\omega = [\Lambda]_{\mathrm{co}}$ is a commensurability class of countable approximate groups, we may thus define its \emph{coarse class} as $[\omega]_c \coloneqq  [\Lambda]_c$.
\begin{example}[Coarse kernel of a local morphism]\label{CoarseKernel} Let $(\Xi, \Xi^\infty)$, $(\Lambda, \Lambda^\infty)$ be approximate groups with $(\Xi,\Xi^\infty)$ countable, and let $\rho_k: (\Xi, \Xi^k) \to (\Lambda, \Lambda^k)$ be a $k$-local morphism, for some $k \geq 10$. Then by Lemma \ref{Lem: partial kernel ag}
and Lemma \ref{CommCoarseEquiv} we have
\[
[\ker_2(\rho_k)]_c = [\ker_3(\rho_k)]_c = \dots = [\ker_{\lfloor \frac{k+2}{6} \rfloor}(\rho_{k})]_c.
\]
This common coarse equivalence class will be called the \emph{coarse kernel}\index{coarse kernel} of $\rho_k$ and denoted $[\ker(\rho_k)]_c$. The coarse kernel $[\ker(\rho)]_c$ of a global morphism $\rho $ is defined similarly.
\end{example}
\begin{example}[Coarse quasikernel] 
If $\rho: \Gamma \to H$ is a quasimorphism between countable groups (or a real-valued quasimorphism), then all quasikernels of $\rho$ are commensurable, hence define a common coarse class $[\qker]_c$ called the \emph{coarse quasikernel} of $\rho$.
\end{example}
\begin{lemma}[Coarse embeddings]\label{SubgroupCoarse} Let $(\Xi, \Xi^\infty)$ be an approximate subgroup of a countable approximate group $(\Lambda, \Lambda^\infty)$. Then the inclusion $\Xi\hookrightarrow \Lambda$ induces a coarse embedding $[\Xi]_c \to [\Lambda]_c$.
\end{lemma}
\begin{proof} Let $d$ be a left-admissible metric on $\Lambda^\infty$. Then the restriction of $d$ to $\Xi^\infty$ is a left-admissible metric, and hence the restrictions of $d$ to $\Lambda$ and $\Xi$ respectively are left-admissible metrics. It follows that $[\Lambda]_c$ and $[\Xi]_c$ are represented by $(\Lambda, d|_\Lambda)$ and $(\Xi, d|_\Xi)$ respectively, and  since the inclusion $(\Xi, d|_\Xi) \to (\Lambda, d|_\Lambda)$ is isometric, the lemma follows.
\end{proof}
Recall from Lemma \ref{characterisation-CL,CE} that a map $h: X \to Y$ between pseudo-metric spaces is \emph{coarsely Lipschitz} if and only if for every $t \geq0$ there exists $s\geq0$ such that if  $x, x' \in X$ satisfy $d_X(x, x')\leq t$, then $d_Y (h(x),h(x'))\leq s$. In the language of coarse geometry (see Remark \ref{coarse structure def}), this means that $h$ is bornologous with respect to the pseudo-metric coarse structures on $X$ and $Y$. In particular, this notion depends only on the coarse equivalence classes of $X$ and $Y$, respectively.

\begin{lemma}\label{Bornologous} Let $(\Xi, \Xi^\infty)$, $(\Lambda, \Lambda^\infty)$ be countable approximate groups and let $f_2: (\Xi, \Xi^2) \to (\Lambda, \Lambda^\infty)$ be a $2$-local quasimorphism. Then the restriction $f_1: \Xi \to \Lambda$ is coarsely Lipschitz with respect to the canonical coarse classes of $\Xi$ and $\Lambda$.
\end{lemma}
\begin{proof} Fix left-admissible metrics $d$ on $\Xi^\infty$ and $d'$ on $\Lambda^\infty$ and denote by $\|\cdot \|$ and $\|\cdot \|'$ the corresponding norms. By Remark \ref{Quasifilt} there exists $N \in \bN$ such that $\rho(\Xi^2) \subset \Lambda^N$. Given $t,s > 0$ we define
\[
\bar B(t) \coloneqq  \{\xi \in \Xi^2 \mid \|\xi\| \leq t\} \quad \text{and} \quad \bar B'(s) \coloneqq  \{\lambda \in \Lambda^N \mid \|\lambda\|' \leq s\}.
\]
Since $f_2$ is a $2$-local quasimorphism, the local defect set $D \coloneqq D_2(f_2)$ is finite. For all $g,h \in \Xi$ we have $f_2(g^{-1}h)^{-1}f_1(g^{-1})f_1(h) \in D^{-1}$, and consequently
\[
d'(f_1(g^{-1})f_1(h), f_2(g^{-1}h ))   =  d'(f_2(g^{-1}h)^{-1}f_1(g^{-1})f_1(h) , e) \leq C,
\]
where $C\coloneqq  \max_{x \in D^{-1}} d'(x,e)$. Specializing to the case $h = g$ we obtain
\[ d'(f_1(g), f_1(g^{-1})^{-1}) \leq d'(f_1(g^{-1})f_1(g), f_2(g^{-1}g)) + d'(f_1(e), e)\leq 2C.
\]

Since $d$ is proper, the set $\bar B(t)$ and thus also $f_2(\bar B(t))$ is finite, for every $t >0$. We deduce that for every $t>0$ there exists $s > 0$ such that \begin{eqnarray*}
f_2(\bar B(t)) \subset \bar B'(s) &\Rightarrow& \forall \xi \in \Xi^2: \; \|\xi\| \leq t \; \Rightarrow \; \|f_2(\xi)\|' \leq s\\
&\Rightarrow& \forall g,h \in \Xi: \; \|g^{-1}h\| \leq t \; \Rightarrow\;   \|f_2(g^{-1}h)\|' \leq s\\
&\Rightarrow& \forall g,h \in \Xi: \; \|g^{-1}h\| \leq t \; \Rightarrow\; \|f_1(g^{-1})f_1(h)\|' \leq s + C\\
&\Rightarrow& \forall g,h \in \Xi: \; d(g,h) \leq t \; \Rightarrow \; d'(f_1(g^{-1})^{-1}, f_1(h)) \leq s +C.
\end{eqnarray*}
Thus, for all $g, h \in \Xi$ with $d(g,h)\leq t $ we have
\[
d'(f_1(g), f_1(h)) \leq  d'(f_1(g), f_1(g^{-1})^{-1}) +  d'(f_1(g^{-1})^{-1}, f_1(h)) \leq  s+3C. \qedhere
\]
\end{proof}
Combining this lemma with Lemma \ref{CoarseMorphisms} we deduce:

\begin{corollary}\label{CoarseType2Local} Let $(\Xi, \Xi^\infty)$, $(\Lambda, \Lambda^\infty)$ be countable approximate groups and let $f_2: (\Xi, \Xi^2) \to (\Lambda, \Lambda^2)$ be a $2$-local isomorphism. Then $[\Xi]_c = [\Lambda]_c$.
\end{corollary}
We express this corollary by saying that the canonical coarse class of a countable approximate group depends only on its $2$-local isomorphism type\footnote{Since every $3$-Freiman isomorphism is a $2$-local isomorphism, the coarse class only depends on the $3$-Freiman isomorphism type in the same sense.}.

\section{Coarse invariants of countable approximate groups}
By a \emph{coarse invariant}\index{coarse invariant} of (proper) pseudo-metric spaces we mean an assignment $\cI$ which assigns to each (proper) pseudo-metric space $(X,d)$ some object $\cI(X,d)$ such that $\cI(X,d) = \cI(Y,d)$ whenever $(X,d)$ and $(Y,d)$ are coarsely equivalent. For example, every functor on the category of (metrizable) coarse spaces defines a coarse invariant, but there are also interesting numerical coarse invariants.
\begin{remark}[Coarse invariants of countable approximate groups]\label{CoarseInvariants}
If $\cI$ is a coarse invariant of proper pseudo-metric spaces and $(\Lambda, \Lambda^\infty)$ is a countable approximate group, then we define
\[
\cI(\Lambda) \coloneqq  \cI(X,d),
\]
where $(X,d)$ is any proper representative of $[\Lambda]_c$. Since $\cI$ is coarsely invariant, this is well-defined, and by Lemma \ref{CommCoarseEquiv} and Corollary \ref{CoarseType2Local} the assignment $\Lambda \mapsto \cI(\Lambda)$ is invariant under commensurability  and $2$-local isomorphism (in particular, under global isomorphism). We refer to such an invariant as a \emph{coarse invariant}\index{approximate group!coarse invariant of} of countable approximate groups.
\end{remark}
We now discuss some examples of coarse invariants of approximate groups.
\begin{definition}\label{def: asdim} Let $n \in \mathbb N_0$.
A metric space $(X,d)$ has \emph{asymptotic dimension}\index{asymptotic dimension}\index{$\asdim$} $\asdim X \leq n$ if for every $R >0$ there exist collections $\mathcal U^{(0)}, \dots, \mathcal U^{(n)}$ of subsets of $X$ with the following properties:
\begin{enumerate}[({ASD}1)]
\item $\bigcup_{i=0}^n \mathcal U^{(i)}$ is a cover of $X$;
\item For every $i=0, \dots, n$ the collection $\mathcal U^{(i)}$ is \emph{uniformly bounded}\index{uniformly bounded family}, i.e. there exists $D>0$ such that for all $U \in \mathcal U^{(i)}$ we have $\diam U \leq D$;
\item For every $i=0, \dots, n$ the collection $\mathcal U^{(i)}$ is \emph{$R$-disjoint}\index{$R$-disjoint}, i.e. for all $U,V \in \mathcal U^{(i)}$ we have either $U=V$ or $d(U, V) \geq R$.
\end{enumerate}
If $n\in \bN_0$ is the smallest number so that these properties are satisfied for every $R>0$, then
we say that $\asdim X = n$. If there is no such $n\in \bN_0$, we say that $\asdim X = \infty$.
\end{definition}
It evidently suffices for $(X,d)$ to have the above properties for all sufficiently large $R$. Our definition of asymptotic dimension is one of many equivalent definitions, see \cite{BellDran1}. It is sometimes referred to as the \emph{coloring definition}\index{asymptotic dimension!coloring definition of} of $\asdim$, because the union of collections $\mathcal U^{(0)}, \dots, \mathcal U^{(n)}$ can be regarded as one cover $\mathcal{U}$ of $X$, which is uniformly bounded, and can be separated into $(n+1)$ subfamilies $\mathcal{U}^{(i)}$, each of which has its own color $i$, and is $R$-disjoint. Coarse invariance of asymptotic dimension is established in the following lemma:
\begin{lemma}\label{AsdimInvariant}
\begin{enumerate}[(i)]
\item If $X$ is a subspace of $Y$, then $\asdim X\leq \asdim Y$.
\item If $X$ and $Y$ are coarsely equivalent, then  $\asdim X = \asdim Y$.
\item If $X$ embeds coarsely into $Y$, then $\asdim X \leq \asdim Y$.
\end{enumerate}
\end{lemma}
\begin{proof} (i) Let $\asdim Y = n$. If $R > 0$, and $\mathcal U_Y^{(0)}, \dots, \mathcal U_Y^{(n)}$ satisfy (ASD1)--(ASD3) above, then the collections
\[
\mathcal U_X^{(i)} \coloneqq  \{U \cap X \mid U \in \mathcal U_Y^{(i)}\}, i=0,\cdots , n
\]
satisfy (ASD1)--(ASD3) for $X$, so $\asdim X\leq n$.

(ii) Assume that $\asdim X=n$, and that $f: X \to Y$ is a coarse equivalence whose image is $K$-relatively dense in $Y$. Given $\mathcal U_X^{(0)}, \dots, \mathcal U_X^{(n)}$ which satisfy (ASD1)--(ASD3) above for constants $D$ and $R$, one can check that 
\[
\mathcal U_Y^{(i)} \coloneqq  \{N_K(f(U)) \mid U \in \mathcal U_X^{(i)}\}, i=0,\cdots , n
\]
satisfiy (ASD1)--(ASD3) with constants $D' = \Phi_+(D)+2K$ and $R'  = {\Phi_-}(D) - 2K$, where $\Phi_-$ and $\Phi_+$ are lower and upper controls for $f$. This implies that $\asdim Y \leq n = \asdim X$, and the opposite inequality follows by a symmetric argument.

(iii) follows from (i) and (ii).
\end{proof}
\begin{definition}\label{def: asdim lambda} The \emph{asymptotic dimension}\index{approximate group!asdim of} of a countable approximate group $(\Lambda, \Lambda^\infty)$ is the coarse invariant defined as $\asdim \Lambda \coloneqq  \asdim X$ for some (hence any) $X \in [\Lambda]_c$.
\end{definition}
We will say more about asymptotic dimension of countable approximate groups in Chapter \ref{ChapAsdim} below. For later reference we record the following consequences of Lemmas \ref{LambdaKCoarse}, \ref{SubgroupCoarse} and \ref{AsdimInvariant}:
\begin{corollary}[Monotonicity of asymptotic dimension for countable approximate groups]\label{CorMonotonicity} Let $(\Lambda,\Lambda^\infty)$ be a countable approximate group, and let $(\Xi, \Xi^\infty)  \subset (\Lambda,\Lambda^\infty)$ be an approximate subgroup. Then the following hold:
\begin{enumerate}[(i)]
\item $\asdim\Lambda \leq \asdim \Lambda^\infty$.
\item $\asdim\Lambda^k = \asdim\Lambda$ for all $k \geq 1$.
\item $\asdim\Xi  \leq \asdim\Lambda$.
\end{enumerate}
\end{corollary}
The numerical coarse invariant $\asdim$ is closely related to a more geometric coarse invariant which takes values in homeomorphism types of compact spaces:
\begin{example}[Higson corona] Let $(X,d)$ be a proper metric space; a bounded continuous function $f: X \to \C$ is called a \emph{Higson function}\index{Higson function} if for all $x \in X$ and $r>0$ we have 
\[
\lim_{x \to \infty} \sup\{|f(x) - f(y)|\mid d(x,y) \leq r\} = 0.
\]
The Higson functions form a unital $C^*$-subalgebra $C_h(X)$ of $C_b(X)$, hence its Gelfand spectrum $hX \coloneqq  \mathrm{spec}(C_h(X))$ defines a compactification of $X$, called the \emph{Higson compactification}\index{Higson compactification} and the boundary $\nu X \coloneqq  hX \setminus X$ is called the \emph{Higson corona}\index{Higson corona} of $X$. By \cite[Cor.\ 2.42]{Roe} the assignment $\cI$ which assigns to a proper metric space $X$ the homeomorphism type $[\nu X]$ of its Higson corona is a coarse invariant. We thus define the \emph{Higson corona}\index{approximate group!Higson corona of} of a countable approximate group $(\Lambda, \Lambda^\infty)$ as $\nu \Lambda \coloneqq  [\nu X]$, where $(X,d)$ is any proper representative of $[\Lambda]_c$. As a special case of \cite[Thm. 6.2]{Dra00} we see that if $\asdim \Lambda < \infty$, then \[\asdim \Lambda = \dim \nu \Lambda,\] where $\dim$ denotes topological (or covering) dimension. Thus for countable approximate groups of finite asymptotic dimension, the asymptotic dimension can be recovered from the Higson corona.
\end{example}

Here is a related coarse invariant:
\begin{example}[Space of coarse ends] With every metric space $(X,d)$ one can associate via Construction \ref{CoarseEnds} a topological space $\cE_\infty(X)$, whose elements are called the \emph{coarse ends}\index{coarse ends} of $X$. By Theorem \ref{CoarseEndsInvariant} the homeomorphism type of $\cE_\infty(X)$ is a coarse invariant of $X$. In particular, we may define $\cE_\infty(\Lambda)$ as the homeomorphism type of $\cE_\infty(X)$ where $X$ is some (hence any) representative of $[\Lambda]_c$.
\end{example}
It turns out that the possible homeomorphism types of spaces of coarse ends of approximate groups are severely restricted. To give a flavor of the results to come, we mention the following result; see Definition \ref{def: gfg} below for the notion of a geometrically finitely-generated approximate group.
\begin{corollary}\label{EndsApp} If $(\Lambda, \Lambda^\infty)$ is a geometrically finitely-generated approximate group and $\cE_\infty(\Lambda)$ contains at least $3$ points, then it is uncountable and homeomorphic to a Cantor space. In particular, a geometrically finitely-generated approximate group has either $0$, $1$, $2$ or uncountably many coarse ends.
\end{corollary}
\begin{proof} We will later see that if $(\Lambda, \Lambda^\infty)$ is geometrically finitely-generated in the sense of Definition \ref{def: gfg}, then $[\Lambda]_c$ admits a representative which is quasi-cobounded, large-scale geodesic and of coarse bounded geometry (see Proposition \ref{ApogeeQuasiCobounded} and Corollary \ref{CoarseBoundedGeometry}). The corollary then follows from Theorem \ref{ACCriterion}.
\end{proof}
Other coarse invariants can be defined using the Vietoris--Rips complex. The following presentation follows \cite{HaWi}.
\begin{construction}[Vietoris--Rips complex]
Let $(X,d)$ be a metric space. For $r \geq 0$, the \emph{Vietoris--Rips complex}\index{Vietoris--Rips complex} $\VR_r(X)$ is the following abstract simplicial complex: The vertex set is $\VR_r(X)^{(0)} \coloneqq X$ and if $x_0, \dots, x_n \in \VR_r(X)^{(0)}$ are vertices, then $(x_0,\ldots,x_n)$ is an $n$-simplex of $\VR_r(X)$ if  $d(x_i,x_j) \le r$ for all $i,j \in \{0, \dots, n\}$. We denote by $|\VR_r(X)|$ the geometric realization of $\VR_r(X)$.
\end{construction}
Given $s>r>0$ we denote by $\iota_{r}^s: |\VR_r(X)| \into |\VR_s(X)|$ the natural inclusion map, which is the identity on the set of vertices.
\begin{lemma}\label{VR1} Let $(X,d)$, $(Y,d)$ be metric spaces and let $\varphi, \psi: X \to Y$ be coarse Lipschitz maps which are at bounded distance. Then for every $p>0$ there exists $q = q(p)>0$ with the following properties:
\begin{enumerate}[(i)]
\item The maps $\varphi, \psi$ extend to simplicial maps $\varphi_*, \psi_*: |\VR_p(X)| \to |\VR_q(Y)|$.
\item The maps $\varphi_*$ and $\psi_*$ are homotopic.
\end{enumerate}
\end{lemma}
\begin{proof} Let $\Phi$ be an upper control function for both $\varphi$ and $\psi$ and let $C \geq 0$ be such that $\varphi$ and $\psi$ are $C$-close. We then define $q(p) \coloneqq  \Phi(p)+C$.

\item (i) Let $(x_0,\ldots,x_n)$ be a simplex in $\VR_p(X)$. The for all $i,j \in \{0, \dots, n\}$ we have $d(x_i, x_j)\leq p$, and hence
\[
d(\varphi(x_i),\varphi(x_j)) \leq \Phi(p) \leq q(p).
\]
This shows that $\varphi_*(x_0,\ldots,x_n) \coloneqq  (\varphi(x_0), \dots, \varphi(x_n))$ is a simplex in $\VR_{q(p)}(Y)$, and hence we obtain the desired extension $\varphi_*$. The same argument also applies to $\psi$.

(ii) If $(x_0,\ldots,x_n)$ is a simplex in $\VR_p(X)$, then for all $i,j \in \{0, \dots, n\}$ we have
\[
d(\varphi(x_i),\psi(x_j)) \le d(\varphi(x_i),\varphi(x_j)) + d(\varphi(x_j),\psi(x_j)) \leq \Phi(p) +C = q(p),
\]
hence $(\varphi(x_0),\ldots,\varphi(x_n),\psi(x_0),\ldots,\psi(x_n))$ is a simplex in $\VR_{q(p)}(Y)$. This shows that if $\sigma$ is the realization of the simplex $(x_0, \dots, x_n)$, then
 $\varphi_*(\sigma)$ and $\psi_*(\sigma)$ are contained in a common simplex. The lemma follows.
\end{proof}

\begin{remark}[$K$-acyclicity]
In the sequel, by a \emph{homotopy functor}\index{homotopy functor} $K$ we mean a functor from the homotopy category of triangulizable topological spaces (i.e.\ spaces which arise as geometric realizations of abstract simplicial complexes) into some category $\cC$. We fix such a homotopy functor $K$ and assume that $\cC$ admits a final object $0$.

As in  Remark \ref{Coa0} we define a category $\mathbf{Coa}_0$ of coarse metric spaces as follows: Objects in $\mathbf{Coa}_0$ are non-empty pseudo-metric spaces and morphisms are closeness classes of coarse Lipschitz maps. By Lemma \ref{CoarseMorphisms} two objects in $\mathbf{Coa}_0$ are isomorphic if and only if they are coarsely equivalent.

We say that an object $X$ in $\mathbf{Coa}_0$ is \emph{$K$-acyclic}\index{$K$-acyclic} if the following holds: For every $r>0$ there exists $s>0$ such that the morphism
\[
K_r^s \coloneqq  K(\iota_r^s): K(|\VR_r(X)|) \to K(|\VR_s(X)|).
\]
induced by the inclusion map $\iota_{r}^s: |\VR_r(X)| \into |\VR_s(X)|$ factors through $0$.

It follows from Lemma \ref{VR1} that if $X$ and $Y$ are isomorphic in $\mathbf{Coa}_0$, then $X$ is $K$-acyclic if and only if $Y$ is. In fact, we have the following more general statement.
\end{remark}
\begin{lemma} Let $[f]: X \to Y$ be a morphism in $\mathbf{Coa}_0$ which is represented by a coarse Lipschitz map $f$. If $[f]$ admits a left-inverse $[g]$ and $Y$ is $K$-acyclic, then so is $X$.
\end{lemma}
\begin{proof} Let $p>0$ and using Lemma \ref{VR1}.(i) choose $r>0$ so large that $f$ induces a map $f_*: |\VR_p(X)| \to |\VR_r(Y)|$. By $K$-acyclicity of $Y$ we then find $s>0$ such that  $\iota_{r}^s: |\VR_r(Y)| \into |\VR_s(Y)|$ induces a morphism $K_r^s:K(|\VR_r(Y)|) \to K(|\VR_s(Y)|)$ factoring through $0$. Finally, choose $t>0$ such that $g$ induces a morphism $ |\VR_s(Y)| \to  |\VR_t(X)|$. Then $g \circ f$ induces the trivial morphism
\[
K((f\circ g)_*) = K(g_*) \circ K_r^s \circ K(f_*): K( |\VR_p(X)|)\to K( |\VR_t(X)|).
\]
Since $g \circ f \sim \id_X$, it follows from Lemma \ref{VR1}.(ii) (after enlarging $t$ if necessary), that $(f\circ g)_* \simeq (\id_X)_* = \iota_p^t$. The lemma follows.
\end{proof}

We may thus define:
\begin{definition} A countable approximate group $(\Lambda, \Lambda^\infty)$ is called \emph{$K$-acyclic}\index{approximate group!$K$-acyclic} with respect to a homotopy functor $K$ if some (hence any) representative of $[\Lambda]_c$ is $K$-acyclic.
\end{definition}
\begin{remark}[Basepoints] The following modification of $K$-acyclicity is sometimes useful. Assume that $K$ is a functor from the homotopy category of \emph{pointed} topological spaces. We then say that an object $X$ in $\mathbf{Coa}_0$ is $K$-acyclic, if the morphism
\[
K_r^s \coloneqq  K(\iota_r^s): K(|\VR_r(X)|, o) \to K(|\VR_s(X)|, o).
\]
induced by the inclusion map $\iota_{r}^s: |\VR_r(X)| \into |\VR_s(X)|$ factors through $0$ for every choice of basepoint $o \in X = |V_r{X}|^{(0)} =  |V_s{X}|^{(0)} $. $K$-acyclicity of approximate groups is then defined accordingly.
\end{remark}
The following invariants are studied in \cite{HaWi}:
\begin{example}[Finiteness properties of approximate groups]\label{FPn}
We say that a countable approximate group $(\Lambda, \Lambda^\infty)$ has the finiteness property $(F_n)$ if it is $\pi_k$-acyclic for all $k \in \{0, \dots, n-1\}$, where $\pi_k$ denotes the $k$th homotopy group (note that $k$ goes from $0$ to $n-1$, and not to $n$). It has finiteness property $(\FP_n)$ if it is $H_k$-acyclic for all $k \in \{0, \dots, n-1\}$, where $H_k$ denotes the $k$th simplicial (or singular) homology group. It is established in \cite{HaWi} that in the case where $\Lambda = \Lambda^\infty$ is a group, this definition is equivalent to the usual definition of finiteness properties of countable groups. 

By the Hurewicz theorem, $\pi_k$-acyclicity implies $H_k$-acyclicity for every $k$, hence Property $(F_n)$ implies Property $(\FP_n)$ for every $n$. Similarly, Properties $(F_2)$ and $(\FP_n)$ together imply Property $(F_n)$.
\end{example}

\section{External and internal QI type}
For a finitely-generated group $\Gamma$ the canonical coarse class $[\Gamma]_c$ contains a canonical QI type $[\Gamma]$ which is represented, for example, by any Cayley graph of $\Gamma$. As mentioned before, there is more than one way to define finitely-generated approximate groups\footnote{This corresponds to the fact that there are many different characterizations of finitely-generated groups, cf.\ Corollary \ref{FinGenGroup}} and there are several possible QI types for such approximate groups that one may want to consider.

The simplest and most algebraic definition of being finitely-generated for a countable approximate group $(\Lambda, \Lambda^\infty)$ is the one introduced in \cite{BH}; while it has many deficits from the geometric point of view, it still turns out to be useful.
\begin{definition} An approximate group $(\Lambda, \Lambda^\infty)$ is called \emph{algebraically finitely-generated}\index{algebraically finitely-generated approx.\ group} if $\Lambda^\infty$ is finitely-generated as a group.
\end{definition}
If $(\Lambda, \Lambda^\infty)$ is an algebraically finitely-generated approximate group, then $\Lambda^\infty$ admits a large-scale geodesic left-admissible metric $d$; for example, every word metric with respect to a finite generating set is large-scale geodesic and left-admissible. We then refer to the metric $d|_{\Lambda \times \Lambda}$ as an \emph{external metric}\index{metric!external} on $\Lambda$. If $d'$ is any other large-scale geodesic left-admissible metric on $\Lambda^\infty$, then by Corollary \ref{FinGenGroup} the identity map $(\Lambda^\infty, d) \to (\Lambda^\infty, d')$ is a quasi-isometry. It thus restricts to a quasi-isometry $(\Lambda, d|_{\Lambda \times \Lambda}) \to (\Lambda, d'|_{\Lambda \times \Lambda})$. To summarize:
\begin{proposition} If $(\Lambda, \Lambda^\infty)$ is an algebraically finitely-generated approximate group, then any two external metrics on $\Lambda$ define the same QI type.
\end{proposition}
\begin{definition}\label{def: Lambda ext} If $(\Lambda, \Lambda^\infty)$ is algebraically finitely-generated and $d$ is an external metric on $\Lambda$, then the QI type of $(\Lambda, d)$ is denoted $[\Lambda]_{\rm ext}$ and referred to as the \emph{external QI type}\index{external QI type} of $(\Lambda, \Lambda^\infty)$.
\end{definition}
Note that if $d$ is an external metric on $\Lambda^k$ for some $k \geq 2$, then $d|_{\Lambda \times \Lambda}$ is an external metric on $\Lambda$, and since $\Lambda$ is syndetic in $\Lambda^k$ we thus have
\begin{equation}
[\Lambda]_{\mathrm{ext}} = [\Lambda^k]_{\mathrm{ext}} \quad \text{ for all }k \in \bN.
\end{equation}
Moreover we have $[\Lambda]_{\rm ext} \subset [\Lambda]_c$, but unlike $[\Lambda]_c$ the external QI type is \emph{not} invariant under commensurability:
\begin{example}[Failure of commensurability invariance]\label{Stark1} We consider\footnote{This is \cite[Example 3.26]{BH}, due to Emily Stark.} $\Lambda \coloneqq   \langle a \rangle \cup \{b,b^{-1}\}\subset \Lambda^\infty =  {\rm BS}(1,2) \coloneqq   \langle a, b \mid bab^{-1} = a^2 \rangle$. By definition, $\Lambda$ is symmetric, contains the identity and generates $\Lambda^\infty$. A short calculation involving the defining relation (and using that $(b^{-1}ab)^2 = a$) shows that 
$\Lambda^2 \subset \Lambda \{e, b, b^{-1}, b^{-1}a\}$, hence $(\Lambda, \Lambda^\infty)$ is an approximate group. We choose the finite generating set $S\coloneqq \{a^{\pm 1}, b^{\pm 1}\}$ and denote by $d_\Lambda$ the external metric on $\Lambda$ induced by the word metric $d_S$. Since $|a^{2^n}|_S \leq 2n+1$, the distance $d_\Lambda(a^k, a^l)$ is logarithmic in $|k-l|$. 

On the other hand, let $\Xi  \coloneqq  \langle a \rangle \subset \Lambda$. Then $\Xi$ is commensurable to $\Lambda$ and an external metric on $\Xi = \Xi^\infty$ is given by $d_\Xi(a^k, a^l) = |k-l|$ and hence $d_\Xi$ is not quasi-isometric to the restriction of $d_\Lambda$. Thus, 
\[[\Lambda]_{\mathrm{co}} = [\Xi]_{\mathrm{co}}, \quad \text{but} \quad [\Lambda]_{\mathrm{ext}} \neq [\Xi]_{\mathrm{ext}}.\]
\end{example}
One can fix the lack of commensurability invariance of the external QI type using the following consequence of Corollary \ref{MinimalAGTrick}.
\begin{lemma}\label{MinimalExtLemma} 
For $j \in \{1,2\}$ let $(\Lambda_j, \Lambda_j^\infty)$ be algebraically finitely-generated minimal approximate groups and $\Gamma$ be a group containing both $\Lambda_1$ and $\Lambda_2$.
\begin{enumerate}[(i)]
\item If $\Xi$ is syndetic in $\Lambda_1^k$ for some $k \in \bN$, then $\Xi$ is algebraically finitely-generated and $[\Xi]_{\mathrm{ext}} = [\Lambda_1]_{\mathrm{ext}}$.
\item If $\Lambda_1 \sim_{\mathrm{co}} \Lambda_2$, then $[\Lambda_1]_{\mathrm{ext}} = [\Lambda_2]_{\mathrm{ext}}$.
\end{enumerate}
\end{lemma}
\begin{proof} (i) By Corollary \ref{MinimalAGTrick}.(ii) the subgroup $\Xi^\infty \subset \Lambda_1^\infty$ has finite index, hence is finitely-generated. If $d$ is a  large-scale geodesic and left-admissible metric on $\Lambda_1^\infty$, then so is its restriction to $\Xi^\infty$, hence $d|_{\Xi \times \Xi}$ and $d|_{\Lambda_1^k \times \Lambda_1^k}$ are external metrics. For these metrics, the inclusion $\Xi \hookrightarrow \Lambda_1^k$ is isometric and coarsely onto, hence $[\Xi]_{\mathrm{ext}} = [\Lambda_1^k]_{\mathrm{ext}} = [\Lambda_1]_{\mathrm{ext}}$.

\item (ii) By \eqref{CommElements} the approximate subgroup $\Xi  \coloneqq  \Lambda_1^2 \cap \Lambda_2^2$ is syndetic in $\Lambda_1^2$ and $\Lambda_2^2$. Applying (i) twice thus yields $ [\Lambda_1]_{\mathrm{ext}} = [\Xi]_{\mathrm{ext}} = [\Lambda_2]_{\mathrm{ext}}$.
\end{proof}
\begin{construction}[Minimal external QI type]\label{MinextQI}
If $(\Lambda, \Lambda^\infty)$ is an algebraically finitely-generated approximate group, then by Corollary \ref{MinimalAGTrick} there exists a minimal syndetic approximate subgroup $(\Lambda_{\mathrm{min}}, \Lambda_{\mathrm{min}}^\infty)$ of $(\Lambda^2, \Lambda^\infty)$. By Lemma \ref{MinimalExtLemma}.(i) this approximate group is algebraically finitely-generated, and we define the \emph{minimal external QI type} of $\Lambda$ as
\[
[\Lambda]_{\mathrm{minext}}  \coloneqq  [\Lambda_{\mathrm{min}}]_{\mathrm{ext}}.
\]
By Lemma \ref{MinimalExtLemma}.(ii) this QI type depends only on $(\Lambda, \Lambda^\infty)$, and it is commen\-sur\-abi\-li\-ty-invariant by construction. Moreover, $[\Lambda]_{\mathrm{minext}} \subset [\Lambda_{\mathrm{min}}]_c = [\Lambda]_c$. 
\end{construction}
The problem with this construction is that there is no known algorithm to find a minimal syndetic approximate subgroup of a given approximate group. Indeed, the proof of Corollary \ref{MinimalAGTrick} rests on structure theory and is non-constructive. We thus turn to a different approach, which is based on finiteness properties (see Example \ref{FPn}) and the following observation:
\begin{theorem}[Large-scale geodesic representatives]\label{GFG} Let $(\Lambda, \Lambda^\infty)$ be a countable approximate group, let $\widehat{d}$ be a left-admissible metric on $\Lambda^\infty$ and let $d = \widehat{d}|_{\Lambda\times \Lambda}$. Then the following are equivalent:
\begin{enumerate}[(i)]
\item $(\Lambda, \Lambda^\infty)$ has finiteness property $(F_1)$.
\item There exists a proper metric $d'$ on $\Lambda$ which is coarsely equivalent to $d$ and such that $(\Lambda, d')$ is large-scale geodesic.
\item There exists a representative $X \in [\Lambda]_c$ which is large-scale geodesic.
\item Every representative $X \in [\Lambda]_c$ is coarsely connected.
\item $(\Lambda, d)$ is coarsely connected.
\end{enumerate}
\end{theorem}
\begin{proof} (i)$\iff$(v): By definition, $(\Lambda, \Lambda^\infty)$ has finiteness property $(F_1)$ iff for every $r >0$ there exists $s>0$ such that the map
\[
(\pi_0)_r^s: \pi_0(|\VR_{r}|) \to \pi_0(|\VR_s|)
\]
is trivial. This property only depends on the $1$-skeleton of the respective complexes, and it means that any two points $x,y \in \Lambda = |\VR_{r}|^{(0)}$ can be connected by a path in $|\VR_s|^{(1)}$. The latter means that there exist $x_0, \dots, x_n$ with $x=x_0$ and $x_n = y$ such that $(x_i, x_{i+1})$ is an edge in $|\VR_s|$, for $i \in \{0, \dots, n-1\}$ or equivalently $d(x_i, x_{i+1}) \leq s$. This, however, means precisely that $(\Lambda, d)$ is $s$-coarsely connected.

\item (ii)$\implies$(iii)$\implies$(iv)$\implies$(v): These follow from Lemma \ref{CharLSG} and the fact that every space which is quasi-isometric to $(\Lambda, d)$ is contained in $[\Lambda]_c$.

\item (v)$\implies$(ii) We fix $C_0 >0$ such that $(\Lambda, d)$ is coarsely $C_0$-connected. For every $C\geq C_0$ we then obtain a large-scale geodesic metric $d_C$ on $\Lambda$ by setting
\begin{eqnarray*}
d_C(x,y) \coloneqq  \min \{n \in \mathbb N_0 \mid \exists\, x_0, \dots, x_n \in \Lambda : \; x=x_0,\ y=x_n, & &\hspace{-5mm}d(x_{i-1}, x_{i}) \leq C,\\
& &\hspace{-5mm}\forall i \in \{1, \dots, n\}\}.
\end{eqnarray*}
We claim that this metric is coarsely equivalent to $d$ for all sufficiently large $C$. Since $d(x,y) \leq \frac 1 C d_C(x,y)$ it suffices to show that
\[
\forall\, r>0\  \exists\, t>0:\; d(x,y) \leq r \implies d_C(x,y) \leq t.
\]
For this we choose a finite symmetric set $F$ with $\Lambda^2 \subset \Lambda F$ and given $r >0$ we abbreviate 
\[
s\coloneqq \max\{ d(e,f) \mid f \in F\} \qand L(r) \coloneqq  \max \{d_{C_0}(e, \widehat{z}) \mid \widehat{z} \in \Lambda \cap B(e,r+s)\}.
\]
We choose $C$ so large that $C \geq C_0+2s$. Now let $r \geq 0$ and $x,y \in \Lambda$ with $d(x,y) \leq r$. Since $x^{-1}y \in \Lambda^2$, there exist $\widehat{z} \in \Lambda$ and $f \in F$ such that $x^{-1}y = \widehat{z} f$. By definition,
\[
d(e,\widehat{z}) \leq \widehat{d}(e,x^{-1}y) + \widehat{d}(\widehat{z}f, \widehat{z}) \leq r+s \implies d_{C_0}(e, \widehat{z}) \leq L(r).
\]
There thus exists an $n \leq L(r)$ and elements $\widehat{x}_0,\dots, \widehat{x}_n \in \Lambda$ such that
\[
\widehat{x}_0 = e, \quad \widehat{x}_n = \widehat{z} \qand d(\widehat{x}_{i-1}, \widehat{x}_i) \leq C_0, \text{ for all }i \in \{1, \dots, n\}.
\]
Since $\Lambda$ is $s$-relatively dense in $\Lambda^2$, we can find elements $x_1, \dots, x_n \in \Lambda$ such that
\[
x_0 = x \qand \widehat{d}(x_{i}, x\widehat{x}_i) \leq s,  \text{ for all }i \in \{1, \dots, n\}.
\]
We then have, for all $i \in \{1, \dots, n\}$,
\[
d(x_{i-1}, x_i) \leq  \widehat{d}(x_{i-1}, x\widehat{x}_{i-1})+ \widehat{d}((x\widehat{x}_{i-1},  x\widehat{x}_i) + \widehat{d}(x\widehat{x}_i, x_i) \leq s+C_0+s \leq C.
\]
If we set $x_{n+1} \coloneqq  y$, then we also have
\[
d(x_n, x_{n+1})=d(x_n,y) \leq \widehat{d}(x_n, x\widehat{x}_n) + \widehat{d}(x\widehat{x}_n, y) \leq s+\widehat{d}(yf,y) \leq 2s \leq C.
\]
This shows that $d_C(x,y) \leq n+1\leq L(r)+1$. Also note that $\widehat{d}$  being proper implies $d$ is proper, and therefore so is $d_C$,  which finishes the proof.
\end{proof}
\begin{definition}\label{def: gfg} A countable approximate group  $(\Lambda, \Lambda^\infty)$ is called \emph{geometrically finitely-generated}\index{geometrically finitely-generated approx.\ group} if it satisfies the equivalent conditions of Theorem \ref{GFG}.
\end{definition}
Note that, by definition, being geometrically finitely-generated is a coarse invariant among countable approximate groups.
\begin{construction}[Internal QI type]\label{IntQI}\label{CanoicalQITypeGood}
A basic result in geometric group theory, known as ``Gromov's trivial lemma'' (see Corollary \ref{GromovTrivialSymmetric}), ensures that two large-scale geodesic metric spaces are quasi-isometric if and only if they are coarsely equivalent. In particular, if $(\Lambda, \Lambda^\infty)$ is a geometrically finitely-generated approximate group, then by Theorem \ref{GFG} the set 
\begin{equation}\label{int}
[\Lambda]_{\mathrm{int}} \coloneqq  \{X \in [\Lambda]_c \mid X \text{ is large-scale geodesic}\}\end{equation}
is non-empty, hence defines a QI-type which is called the \emph{internal QI type}\index{internal QI type} of $(\Lambda, \Lambda^\infty)$ and refines the coarse class $[\Lambda]_c$. If $\Xi$ is another geometrically finitely-generated approximate group, then by definition we have
\[
[\Lambda]_{\mathrm{int}} = [\Xi]_{\mathrm{int}} \iff [\Lambda]_c = [\Xi]_c.
\]
In particular, among geometrically finitely-generated approximate groups, the internal QI type is invariant both under commensurability and $2$-local isomorphism by Lemma \ref{CommCoarseEquiv} and Corollary \ref{CoarseType2Local}. Therefore, if $(\Lambda, \Lambda^\infty)$ is geometrically finitely-generated, then $[\Lambda]_{\mathrm{int}} = [\Lambda^k]_{\mathrm{int}}$ for all $k \in \bN$.
\end{construction}
\begin{example}[Internal QI kernel] If $\rho_k: (\Xi, \Xi^k) \to (\Lambda, \Lambda^k)$ is a $k$-local morphism of approximate groups for some $k \geq 10$ and $\ker_2(\rho_k)$ is geometrically finitely-generated, then all partial kernels up to $\ker_{\lfloor \frac{k+2}{6} \rfloor}(\rho_{k})$ are geometrically finitely-generated and  
 $[\ker(\rho_k)]_{\mathrm{int}} \coloneqq [\ker_2(\rho_k)]_{\mathrm{int}} = [\ker_3(\rho_k)]_{\mathrm{int}} = \dots = [\ker_{\lfloor \frac{k+2}{6} \rfloor}(\rho_{k})]_{\mathrm{int}}$. We refer to this class as the \emph{internal QI kernel}\index{QI kernel} of $\rho_k$; we apply the same terminology to global morphisms. 
\end{example}
\begin{definition}\label{def: apogee} If $(\Lambda, \Lambda^\infty)$ is a geometrically finitely-generated approximate group, then a proper geodesic metric space $(X,d) \in [\Lambda]_{\mathrm{int}}$ is called an \emph{apogee}\index{apogee} for $(\Lambda, \Lambda^\infty)$. If moreover $X$ is a locally finite graph and $d$ is the canonical metric on $X$ (cf.\ Example \ref{CanonicalMetric}), then $(X,d)$ is called a \emph{generalized Cayley graph}\index{generalized Cayley graph} for $(\Lambda, \Lambda^\infty)$.
\end{definition}
Cayley graphs of finitely-generated groups with respect to finite generating sets are obvious examples of generalized Cayley graphs, hence the name. The following remark shows, that every geometrically finitely-generated approximate group admits a generalized Cayley graph (hence an apogee).
\begin{remark}[Representing the internal QI type]\label{RepresentingQIType}
It follows from Theorem \ref{GFG} that the internal QI type $[\Lambda]_{\mathrm{int}}$ of a geometrically finitely-generated approximate group $(\Lambda, \Lambda^\infty)$ can always be represented by a proper metric $d$ on $\Lambda$. We refer to any such metric as an \emph{internal metric}\index{metric!internal} on $\Lambda$. The proof of Theorem \ref{GFG} provides an explicit construction of such an internal metric on $\Lambda$ as a ``large-scale path metric''. By definition, any internal metric on $\Lambda$ is coarsely equivalent (but in general not quasi-isometric) to any external metric.

If $d$ is an internal metric on $\Lambda$, then $(\Lambda, d)$ is proper and large-scale geodesic, and by Lemma \ref{CharLSG} every proper large-scale geodesic metric space is quasi-isometric to a locally finite graph. This shows that every approximate group admits a generalized Cayley graph.
\end{remark}
It is an interesting question which metric spaces can appear as apogees of geometrically finitely-generated approximate groups. To discuss this, we introduce the following notion:
\begin{definition}\label{def: quasi cobounded} Let $(X,d)$ be a pseudo-metric space, let $K \geq 1$, $C \geq 0$ and $r>0$. We say that $X$ is $(K, C, r)$-\emph{quasi cobounded} if for all $x,y \in X$ there exists a $(K, C, C)$-quasi-isometry $f: X \to X$ such that $d(f(x), y)<r$. It is called \emph{quasi-cobounded}\index{quasi-cobounded!space} if it is $(K, C, r)$-quasi cobounded for some $(K, C,r)$ and \emph{cobounded}\index{cobounded space} if it is $(1,0,0)$-quasi-cobounded.
\end{definition}
\begin{remark} There are various equivalent formulations of quasi-coboundedness:
\begin{enumerate}[(i)]
\item A given space $(X,d)$ is quasi-cobounded for some constants $(K', C', r')$ if there exist $K \geq 1$, $C \geq 0$, $r>0$ and $o \in X$ with the following property:  for all $x \in X$ there exists a $(K, C, C)$-quasi-isometry $f: X \to X$ such that $d(f(x), o)<r$. Moreover, in this case the constants $(K', C', r')$ depend only on the constants $(K, C,r)$.
\item Assume that $(X,d)$ is $(K, C, r)$-quasi-cobounded and let $x,y \in X$. Pick a quasi-isometry $f: X \to X$ such that $d(f(x), y) < r$ and define
\[
\widetilde{f}: X \to X, \quad \widetilde{f}(z) \coloneqq  \left\{\begin{array}{ll} f(z), & \text{if }z \neq x\\ y, & \text{if }z = x.   \end{array}\right.
\]
This $\widetilde{f}$ is at distance at most $r$ from $f$ (hence a $(K, C+3r, 3r)$-quasi-isometry) with $\widetilde{f}(x) = y$. It follows that a space is quasi-cobounded if and only if there exist constants $K' \geq 1$ and $C' \geq 0$ such that for all $x,y \in X$ there exists a $(K', C', C')$-quasi-isometry with $f(x) = y$. Spaces with this property are called \emph{coarsely quasi-symmetric}\index{coarsely quasi-symmetric} spaces in \cite{AlvarezCandel}.
\end{enumerate}
\end{remark}
\begin{proposition}[Quasi-coboundedness of apogees]\label{ApogeeQuasiCobounded} If $(\Lambda, \Lambda^\infty)$ is a geometrically finitely-generated approximate group, then every $(X,d) \in [\Lambda]_{\mathrm{int}}$ is quasi-cobounded. In particular, every apogee for $(\Lambda, \Lambda^\infty)$ is a proper, geodesic, quasi-cobounded space.
\end{proposition}
\begin{proof} Let $d_\infty$ be a left-admissible metric on $\Lambda^\infty$. Since $\Lambda$ is an approximate group, it is of finite index in $\Lambda^2$, and hence both $\Lambda$ and $\Lambda^2$ are geometrically finitely-generated. Let $d$ and $d'$ be internal metrics on $\Lambda$ and $\Lambda^2$ respectively. The inclusion $\iota: \Lambda \to \Lambda^2$ is a coarse equivalence with respect to the restrictions of $d_\infty$, and hence also with respect to $d$ and $d'$. By Corollary \ref{GromovTrivialSymmetric} it is thus a quasi-isometry and we denote by $p: (\Lambda^2, d') \to (\Lambda, d)$ a quasi-inverse of $\iota$ with $p|_\Lambda = \mathrm{Id}$. Now for every $\lambda \in \Lambda$ the map $L(\lambda): (\Lambda, d) \to (\Lambda, d)$ given by $L(\lambda)(x) \coloneqq  p(\lambda x)$ is a quasi-isometry with $L(\lambda)(e) = \lambda$, and these quasi-isometries have uniform QI constants.
\end{proof}
Most of the results that we will establish about apogees of  geometrically finitely-generated approximate groups will actually be true for general proper geodesic quasi-cobounded spaces.

\section{Distortion}
In the previous section we have defined two notions of being finitely-generated for countable approximate groups, which are equivalent when restricted to countable groups by Corollary \ref{FinGenGroup}. The following example shows that they are not equivalent for general countable approximate groups.
\begin{example}[Algebraically finitely-generated does not imply geometrically finitely-generated]\label{AFGnotGFG} Let $\Gamma \coloneqq  F_2$ be the free group on two generators $a$ and $b$ and let $\rho \coloneqq  \phi_{ab}: \Gamma \to \Z$ be the $ab$-counting quasimorphism (see Example \ref{CountingQuasimorphism}). By Example \ref{QkerZ} the quasikernel
\[
\Lambda \coloneqq  \qker(\phi_{ab}) = \{w \in F_2\mid |\rho(ab)|\leq 1\}
\]
is an approximate subgroup of $F_2$, and since $a, b \in \Lambda$ we have $\Lambda^\infty = F_2$, hence $(\Lambda, \Lambda^\infty)$ is algebraically finitely-generated. One can show that $\Lambda$ considered as a subset of the Cayley tree of $F_2$ is not coarsely connected, so $(\Lambda, \Lambda^\infty)$ is not geometrically finitely-generated by Theorem \ref{GFG}. This is a special case of a very general phenomenon: While every approximate subgroup of a finitely-generated free group which contains a generating set is algebraically finitely-generated, such an approximate subgroup can only be geometrically finitely-generated if it is an almost group (see Corollary \ref{Rigidity0} below).
\end{example}
On the other hand we will see in Corollary \ref{GFGAFG} below that every geometrically finitely-generated approximate group is algebraically finitely-generated. However, even if $(\Lambda, \Lambda^\infty)$ is geometrically and (hence) algebraically finitely-generated, then $[\Lambda]_{\mathrm{int}}$ and $[\Lambda]_{\mathrm{ext}}$ need not coincide.
\begin{definition} \label{def: undistorted} An approximate group $(\Lambda, \Lambda^\infty)$ is called \emph{undistorted}\index{approximate group!undistorted}\index{undistorted approximate group} if it is both geometrically and algebraically finitely-generated and if  $[\Lambda]_{\mathrm{int}}=[\Lambda]_{\rm ext}$. It is called \emph{distorted}\index{approximate group!distorted}\index{distorted approximate group} if it is both geometrically and algebraically finitely-generated, but $[\Lambda]_{\mathrm{int}} \neq [\Lambda]_{\rm ext}$.
\end{definition}
\begin{example}[A distorted approximate group]\label{ExDistorted} 
Let $(\Lambda, \Lambda^\infty)$ and $\Xi \subset \Lambda$ as in Example \ref{Stark1}. Since $\Xi$ is a finitely-generated group we have $[\Xi]_{\mathrm{int}} = [\Xi]_{\mathrm{ext}}$; from Construction \ref{IntQI} and Example \ref{Stark1} we thus deduce that
\[
[\Lambda]_{\mathrm{int}} = [\Xi]_{\mathrm{int}} = [\Xi]_{\mathrm{ext}} \neq [\Lambda]_{\mathrm{ext}}.
\]
Thus $(\Lambda, \Lambda^\infty)$ is distorted.
\end{example}
\begin{remark}\label{InfoDistorted} Concerning (un)distorted approximate groups we observe:
\begin{enumerate}[(1)]
\item By definition, a geometrically and (hence) algebraically finitely-generated approximate group $(\Lambda, \Lambda^\infty)$ is undistorted if and only if every external metric is also an internal metric. Note that this condition (which amounts to equality $[\Lambda]_{\mathrm{int}}=[\Lambda]_{\rm ext}$ of subsets of $[\Lambda]_c$) is a priori stronger than just the existence of a quasi-isometry between a representative of $[\Lambda]_{\mathrm{int}}$ and a representative of $[\Lambda]_{\rm ext}$.
\item Let $(\Lambda, \Lambda^\infty)$ be an approximate group and $k \in \bN$. If $(\Lambda, \Lambda^\infty)$ is geometrically and (hence) algebraically finitely-generated, then so is $\Lambda^k$, and since the internal and external QI types of $\Lambda$ and $\Lambda^k$ coincide, we deduce that $(\Lambda, \Lambda^\infty)$ is undistorted if and only if $(\Lambda^k, \Lambda^\infty)$ is.
\item The approximate group $(\Lambda, \Lambda^\infty)$ from Example \ref{ExDistorted} is not minimal. We do not know whether a minimal approximate group can be distorted or, in other words, whether the minimal external QI type coincides with the internal QI type. This lack of knowledge is mostly due to the fact that computing minimal external QI types is hard.
\end{enumerate}
\end{remark}
We will later see that distortion is an obstruction to the existence of geometric isometric actions of approximate groups. We now give a general criterion to ensure that a given approximate group is undistorted; this criterion reformulates a result from \cite{BH}.
\begin{definition} Let $(X, d)$ be a proper metric space. A subset $A \subset X$ is called \emph{weakly $(R, K, C)$-quasi-convex}\index{weakly quasi-convex!subset} in $X$ provided that for all $x,y \in A$ there is  a $(K,C)$-quasi-geodesic segment between $x$ and $y$ contained in $N_R(A)$.
\end{definition}
By definition, weak quasi-convexity is invariant under quasi-isometries of pairs, hence we can define:
\begin{definition} An approximate group $(\Lambda, \Lambda^\infty)$ is called \emph{weakly quasi-convex}\index{weakly quasi-convex!group} if it is algebraically finitely-generated and $\Lambda$ is a weakly quasi-convex subset of $\Lambda^\infty$ with respect to some (hence any) word metric on $\Lambda^\infty$.
\end{definition}
Note that $(\Lambda, \Lambda^\infty)$ is weakly quasi-convex in $\Lambda^\infty$ if and only if there exists $k = k(S) \in \mathbb N$ such that for all $x, y \in \Lambda$ there exists a geodesic from $x$ to $y$ in the Cayley graph of $\Lambda^\infty$ with respect to $S$ whose vertices are contained in $\Lambda^k$. We then call $k$ a \emph{convexity parameter}\index{convexity parameter} of $(\Lambda, \Lambda^\infty)$ with respect to $S$. We can now state the following criterion, which is slightly adapted from \cite{BH}.
\begin{proposition} If $(\Lambda, \Lambda^\infty)$ is a weakly quasi-convex approximate group, then it is geometrically finitely-generated and undistorted.
\end{proposition}
\begin{proof} It suffices to show that $(\Lambda, d_S|_{\Lambda\times \Lambda})$ is quasi-isometric to a connected graph. For this let $k$ be a convexity parameter for $(\Lambda, \Lambda^\infty)$ with respect to $S$ and denote by $\Gamma_k$ the full subgraph of the Cayley graph of $\Lambda^\infty$ with respect to $S$ on the vertex set $\Lambda^k$, and by $\Gamma \subset \Gamma_k$ the connected component of $\Lambda$. Then the assumption implies that the inclusion $(\Lambda, d_S|_{\Lambda \times \Lambda}) \to \Gamma$ is a quasi-isometry, and $\Gamma$ is a connected graph, hence geodesic.
\end{proof}
Note that it was established in \cite[Prop. 3.27]{BH} that every uniform model set is weakly-quasi convex. Thus uniform model sets provide examples of undistorted approximate groups. We will later extend this result to arbitrary uniform approximate lattices.

\begin{remark}[Concerning terminology]
In the paper \cite{BH}, which introduced geometric approximate group theory, the authors referred to algebraically finitely-generated approximate groups simply as \emph{finitely-generated approximate} \emph{gro\-ups}, and to the external QI type $[\Lambda]_{\mathrm{ext}}$ as the \emph{canonical QI type}. They then wrote $[\Lambda]$ instead of $[\Lambda]_{\mathrm{ext}}$. However, the focus of the relevant parts of \cite{BH} was mostly on uniform approximate lattices, and we will see that uniform approximate lattices are always undistorted.

Since the writing of \cite{BH} it has become increasingly clear that there are many interesting examples of distorted approximate groups.
To deal with these examples, it is necessary to distinguish carefully between internal and external QI types.
\end{remark}

\section[Growth of approximate groups and the polynomial growth thm]{Growth of approximate groups and the polynomial growth theorem}

The purpose of this section is to define growth of approximate groups in analogy with growth of groups. We first recall a very general notion of growth for metric spaces of coarse bounded geometry in the sense of Block and Weinberger \cite{BlockWeinberger1992}. The following definition is taken from \cite{AlvarezCandel}; for the terminology see Definition \ref{Weinberger} and Remark \ref{GrowthEq}.
\begin{definition} Let $(X,d)$ be a pseudo-metric space of coarse bounded geometry with basepoint $o \in X$ and let $\Lambda \subset X$ be a quasi-lattice.
\begin{enumerate}[(i)]
\item The \emph{growth function}\index{growth function} of $\Lambda$ is defined as $\gamma_{\Lambda, o} (r) \coloneqq  |\Lambda \cap \overline{B}(o,r)|$.
\item The \emph{growth}\index{growth} of $X$, denoted $\mathrm{gr}(X)$, is the growth equivalence class of the function $\gamma_{\Lambda, o}$.
\end{enumerate}
\end{definition}
By Lemma \ref{GrowthQuasiLattices} the growth of $X$ is independent of the choice of $o$ and $\Lambda$.
\begin{proposition}[QI invariance of growth]\label{QIGrowth} If $X$ is of coarse bounded geometry and $Y$ is quasi-isometric to $X$, then $Y$ is of coarse bounded geometry and $\mathrm{gr}(X) = \mathrm{gr}(Y)$.
\end{proposition}
\begin{proof} Let $f: X \to Y$ be a $(K, C,C')$-quasi-isometry and let $\Lambda_1 \subset X$ be a quasi-lattice, so that $\Lambda_2 \coloneqq  f(\Lambda_1)$ is a quasi-lattice in $Y$ by Lemma \ref{Quasilattice}. Let $x_0 \in X$ be a basepoint and $y_0 \coloneqq  f(x_0) \in Y$. Assume that $y \in \Lambda_2 \cap \overline{B}(y_0, r)$ and let $x \in \Lambda_1$ such that $y = f(x)$. We then have
\[
d(f(x), f(x_0)) \leq r \Rightarrow d(x, x_0) \leq  Kr+KC \Rightarrow \Lambda_2 \cap \overline{B}(y_0, r) \subset f(\Lambda_1 \cap B(x_0, Kr+Kc)).
\]
This shows that $\gamma_{\Lambda_2, y_0} \preceq \gamma_{\Lambda_1, x_0}$, and hence the proposition follows by symmetry.
\end{proof}
In view of this proposition we can thus talk unambiguously about a QI class of coarse bounded geometry (meaning that one, hence any of its representatives has coarse bounded geometry), and of its growth.
 In order to apply this to approximate groups we need the following observation:
\begin{lemma}\label{CoarseBG1} Let $(\Lambda, \Lambda^\infty)$ be a countable approximate group and let $d$ be the restriction of a left-admissible metric $\widehat{d}$ on $\Lambda^\infty$ to $\Lambda$. Then $(\Lambda, d)$ is of coarse bounded geometry and $\Lambda$ is a quasi-lattice in $(\Lambda, d)$.
\end{lemma}
\begin{proof} Given $z \in \Lambda^\infty$ we denote by $\overline{B}(z,r)$ the closed ball of radius $r$ around $z$ with respect to $\widehat{d}$ in $\Lambda^\infty$. Fix $y \in \Lambda$ and set $\Phi_+(r) \coloneqq  |\Lambda^3\cap \overline{B}(y,r)|$. Then for every $x \in \Lambda$ we have
\[
|\Lambda \cap \overline{B}(x,r)| = |yx^{-1}(\Lambda \cap \overline{B}(x,r))| \leq |\Lambda^3 \cap \overline{B}(y,r)| = \Phi_+(r).
\]
We deduce that $\Lambda$ is a quasi-lattice in itself with respect to $d$, and hence $(\Lambda, d)$ is of coarse bounded geometry.
\end{proof}
Since quasi-lattices are preserved under coarse equivalences (see Lemma \ref{Quasilattice}) we deduce:
\begin{corollary}\label{CoarseBoundedGeometry} Let  $(\Lambda, \Lambda^\infty)$ be a countable approximate group.
\begin{enumerate}[(i)]
\item  If $(\Lambda, \Lambda^\infty)$ is algebraically finitely-generated, then every $X \in [\Lambda]_{\mathrm{ext}}$ is of coarse bounded geometry and if $d$ is any external metric on $\Lambda$, then $\mathrm{gr}(X)$ is represented by the function
\[
\gamma(r) \coloneqq  |\{x \in \Lambda \mid d(x,e) \leq r\}|.
\]
\item If $(\Lambda, \Lambda^\infty)$ is geometrically finitely-generated, then every $X \in [\Lambda]_{\mathrm{int}}$ is of coarse bounded geometry and if $d$ is any internal metric on $\Lambda$, then $\mathrm{gr}(X)$ is represented by the function
\[
\gamma(r) \coloneqq  |\{x \in \Lambda \mid d(x,e) \leq r\}|.
\]
\end{enumerate}
\end{corollary}
\begin{definition} Let  $(\Lambda, \Lambda^\infty)$ be a countable approximate group. If $(\Lambda, \Lambda^\infty)$ is algebraically (respectively geometrically) finitely-generated and $X \in [\Lambda]_{\mathrm{ext}}$ (respectively $X \in [\Lambda]_{\mathrm{int}}$), then $\mathrm{gr}(X)$ is called the \emph{external growth}\index{growth!external} (respectively \emph{internal growth}\index{growth!internal}) of $(\Lambda, \Lambda^\infty)$ and denoted $\mathrm{gr}_{\mathrm{ext}}(\Lambda)$ (respectively $\mathrm{gr}_{\mathrm{int}}(\Lambda)$).
\end{definition}
If $(\Lambda, \Lambda^\infty)$ is undistorted, then we simply write $\mathrm{gr}(\Lambda) \coloneqq  \mathrm{gr}_{\mathrm{int}}(\Lambda) = \mathrm{gr}_{\mathrm{ext}}(\Lambda)$ and refer to $\mathrm{gr}(\Lambda)$ as the growth of $\Lambda$. If $\Lambda$ is a countable discrete group then we recover the usual notion of growth of a group.

\begin{proposition}[Growth inequalities]\label{GrowthInequalities} For every geometrically and (hence) algebraically finitely-generated approximate group $(\Lambda, \Lambda^\infty)$ we have
\[
\mathrm{gr}_{\mathrm{int}}(\Lambda) \preceq \mathrm{gr}_{\mathrm{ext}}(\Lambda) \preceq \mathrm{gr}(\Lambda^\infty).
\]
\end{proposition}
\begin{proof} Let $d$ be an external metric on $\Lambda$. As in the proof of Theorem \ref{GFG} we can then find $C>0$ and an internal metric of the form
\begin{eqnarray*}
d_C(x,y) \coloneqq  \inf \{n \mid \exists\, x = x_0, \dots, x_n = y \in \Lambda:\; d(x_{i-1}, x_{i})<C \text{ for all } i \in \{1, \dots, n\}\}.
\end{eqnarray*}
For all $x,y \in \Lambda$ we then have $d_C(x,y) \geq C \cdot d(x,y)$, and hence 
\[
|\{x \in \Lambda \mid d_C(x,e) \leq r\}| \leq |\{x \in \Lambda \mid d(x,e) \leq C r\}|.
\]
This implies the first inequality and the second inequality is obvious.
\end{proof}
\begin{definition} We say that a pseudo-metric space $(X,d)$ of coarse bounded geometry has
\begin{itemize}
\item \emph{exponential growth}\index{growth!exponential} if $0< \varliminf_{r \to \infty} {\log(\gamma(r))}/{r} < \infty$ for some $\gamma \in \mathrm{gr}(X)$;
\item \emph{subexponential growth}\index{growth!subexponential} if $\sup_{r>0} {\log(\gamma(r))}/{r} \leq 0$ for some $\gamma \in \mathrm{gr}(X)$;
\item \emph{polynomial growth}\index{growth!polynomial} if $\sup_{r>0} {\log (\gamma(r))}/{\log r} < \infty$ for some $\gamma \in \mathrm{gr}(X)$.
\end{itemize}
\end{definition}
It is easy to check that these properties actually do not depend on the choice of representative $\gamma$. Note in particular that if $X$ has polynomial growth, then there exists $d \in \bN$ such that for every $\gamma \in \mathrm{gr}(X)$ there exists $C>0$ such that for all $r \in \R$,
\[ 
\gamma(r) \leq C r^d.
\]
We say that an approximate group $(\Lambda, \Lambda^\infty)$ has internal (respectively external) exponential/subexponential/polynomial growth if $(\Lambda, d)$ has the corresponding growth type with respect to some (hence any) internal (respectively external) metric.
\begin{corollary}\label{GrowthImpl} In the situation of Proposition \ref{GrowthInequalities} we have:
\begin{enumerate}[(i)]
\item $\Lambda^\infty$ has pol.\ growth $\implies$ $\Lambda$ has ext.\ pol.\ growth $\implies$ $\Lambda$ has int.\ pol.\ growth.
\item $\Lambda$ has int.\ exp.\ growth $\implies$ $\Lambda$ has ext.\ exp.\ growth $\implies$ $\Lambda^\infty$ has exp.\ growth.
\end{enumerate}
\end{corollary}
Both inequalities in Proposition \ref{GrowthInequalities} can be strict, as can be seen from the following two examples.
\begin{example} Let $\Lambda^\infty \coloneqq  \Z^2$ and $\Lambda \coloneqq  \{(a,b) \in \Z^2 \mid |b| \leq 1\}$. Then $\Lambda$ has linear external growth in the sense that $\mathrm{gr}_{\mathrm{ext}}(\Lambda)$ is represented by $\gamma_{\mathrm{ext}}(r) \coloneqq  r$, whereas $\mathrm{gr}(\Lambda^\infty)$ is quadratic.
\end{example}
\begin{example}\label{PolGrowthFINec} If $(\Lambda, \Lambda^\infty)$ is as in Example \ref{ExDistorted}, then  $(\Lambda, \Lambda^\infty)$ has linear internal growth in the sense that $\mathrm{gr}_{\mathrm{int}}(\Lambda) = \mathrm{gr}_{\mathrm{int}}(\bZ)$ is represented by $\gamma_{\mathrm{int}}(r) \coloneqq  r$. In particular,  $(\Lambda, \Lambda^\infty)$ has internal polynomial growth. On the contrary, $\mathrm{gr}_{\mathrm{ext}}(\Lambda)$ is represented by $\gamma_{\mathrm{ext}} = 2^r$ and hence $(\Lambda, \Lambda^\infty)$ has external exponential growth. 
\end{example}
Again, this example is not minimal, and we do not know whether for minimal approximate groups the internal and external growth types coincide. If this was the case, then for every approximate group $(\Lambda, \Lambda^\infty)$ we could find $\Lambda' \subset \Lambda^\infty$ in the commensurability class of $\Lambda$ whose external growth coincides with the internal growth of $\Lambda$ (or, equivalently, $\Lambda'$). While we do not know how to prove the latter statement in full generality, we can at least say that if $(\Lambda, \Lambda^\infty)$ has internal polynomial growth, then as in Example \ref{PolGrowthFINec} we can always find a commensurable approximate subgroup $\Lambda' \subset \Lambda^\infty$ which has external polynomial growth (albeit possibly with a different growth rate). However, this fact is far from obvious and requires an approximate version of Gromov's celebrated polynomial growth theorem.

Recall that the classical version of Gromov's polynomial growth theorem \cite{GromovPolGrowth} states that a finitely-generated group has polynomial growth if and only if it is virtually nilpotent. Recently a new proof of Gromov's theorem was given by Hrushovski \cite{Hrushovski}; this proof is based on the fact that in every group of polynomial growth one can find an exhaustive sequence of balls which are finite $k$-approximate subgroups for some fixed constant $k$. It then analyzes ultralimits of such sequences of finite $k$-approximate subgroups using models in connected Lie groups. In fact, Hrushovski's method can also be used to derive various strengthenings of Gromov's theorem; 
Breuillard, Green and Tao have developed a quantitative refinement of this idea which yields for example a strong quantitative version of Gromov's theorem \cite[Cor.\ 11.2]{BGT}. The same machinery also applies to approximate groups and yields the following polynomial growth theorem:
\begin{theorem}[Polynomial growth theorem]\label{PGT} Let $(\Lambda_o, \Lambda_o^\infty)$ be an approximate group which is geometrically and (hence) algebraically finitely-generated. Then the following statements are equivalent:
\begin{enumerate}[(i)]
\item  $(\Lambda_o, \Lambda_o^\infty)$ has internal polynomial growth.
\item Some $\Lambda \subset \Lambda_o^\infty$ which is commensurable to $\Lambda_o$ has internal polynomial growth
\item There exists a finite index approximate subgroup $\Lambda \subset \Lambda_o^{16}$ such that  $(\Lambda, \Lambda^\infty)$ has internal polynomial growth.
\item There exists a finite index approximate subgroup $\Lambda \subset \Lambda_o^{16}$ such that $\Lambda^\infty$ is virtually nilpotent.
\item  There exists a finite index approximate subgroup $\Lambda \subset \Lambda_o^{16}$ such that $\Lambda^\infty$ has polynomial growth.
\item There exists a finite index approximate subgroup $\Lambda \subset \Lambda_o^{16}$ such that $(\Lambda, \Lambda^\infty)$ has external polynomial growth.
\end{enumerate}
\end{theorem}
Thus, up to passing to a commensurable approximate group, internal polynomial growth, external polynomial growth and polynomial growth of the ambient group are all equivalent, and hence equivalent to virtual nilpotency of the ambient group. Example \ref{PolGrowthFINec} shows that this statement is not true without passing to a commensurable approximate group. It also shows that while $\Lambda$ and $\Lambda_o$ are commensurable, this need not be the case for $\Lambda^\infty$ and $\Lambda_o^\infty$. We will explain the proof of Theorem \ref{PGT} in Section \ref{SubsecPGT1} below, based on a reformulation of results from \cite{BGT} explained to us by Simon Machado and presented in Appendix \ref{AppendixBGT}.

\chapter{Geometric quasi-actions of approximate groups}

The main idea of geometric group theory is to study finitely-generated groups via their isometric actions on metric spaces. A particular role is played by actions which are \emph{geometric}, since these provide models for the canonical QI class of the group by the Milnor--Schwarz lemma. Isometric actions of approximate groups can be studied in complete analogy with the group case (see Section \ref{SecIsometric}). However, contrary to the group case, not every algebraically and geometrically finitely-generated approximate group admits a geometric isometric action on a proper geodesic metric space -- distortion turns out to be an important obstruction. To resolve this issue we introduce in Section \ref{SecQiqac} the notion of a quasi-isometric quasi-action (qiqac); we then establish in Sections \ref{SecGeomQiqac} and \ref{SecLeftRegular} that every geometrically finitely-generated approximate group admits a geometric qiqac on a proper geodesic metric space. For this reason, qiqacs take over the role of isometric actions in the geometric theory of approximate groups. In Section \ref{SecMilnorSchwarz} we establish the main result of this chapter, the Milnor--Schwarz lemma for approximate groups (see Theorem \ref{ThmMS} for its most general form). We then discuss various applications. Among other things, we show that geometrically finitely-generated approximate groups are algebraically finitely-generated (Corollary \ref{GFGAFG}) and that approximate groups acting geometrically by isometries on proper geodesic metric spaces must be undistorted (Corollary \ref{cor:undistort}), and we explain how to derive a version of Gromov's polynomial growth theorem for approximate groups from work of Hrushovski and Breuillard--Green--Tao (Section \ref{SubsecPGT1}).

\section{Isometric actions of approximate groups}\label{SecIsometric}
Throughout this section, $(\Lambda, \Lambda^\infty)$ denotes an approximate group and $(X,d)$ denotes a pseudo-metric space with basepoint $o \in X$. 
\begin{definition}\label{DefCobProp} Let $\cA$ be a set and $\rho: \cA \to {\rm Map}(X, X)$ be a map. 
\begin{enumerate}[(i)]
\item $\rho$ is called \emph{cobounded}\index{cobounded map}\index{map!cobounded} if for some $x \in X$ the set $\rho(\cA)(x) \coloneqq  \{\rho(a)(x) \mid a \in \cA\}$ is relatively dense in $X$.
\item $\rho$ is called \emph{proper}\index{proper!map}\index{map!proper} if the map $\cA \times X \to X \times X$, $(a,x)\mapsto (x, \rho(a)(x))$ is proper, where $\cA$ carries the discrete topology.
\item A subset $Y \subset X$ is called \emph{quasi-invariant}\index{quasi-invariant!subset} under $\rho(\cA)$ if there exists $D \geq 0$ such that for every $a \in \cA$ we have $d_{\mathrm{Haus}}(Y, \rho(a)(Y)) \leq D$, where $d_{\mathrm{Haus}}$ denotes Hausdorff distance.
\end{enumerate}
\end{definition}
\begin{definition} 
\begin{enumerate}[(i)]
\item An \emph{isometric action}\index{isometric action}\index{approximate group!isometric action of} of $(\Lambda, \Lambda^\infty)$ on $X$ is a global morphism $\rho: (\Lambda, \Lambda^\infty) \to {\rm Is}(X)$.
\item If $\rho: (\Lambda, \Lambda^\infty) \to {\rm Is}(X)$ is an isometric action, then the \emph{$\Lambda$-orbit}\index{approximate group action!orbit} of $o$ under $\rho$ is
\[
\rho(\Lambda)(o) \coloneqq  \{\rho(\lambda)(o)\mid \lambda \in \Lambda\} \subset X.
\]
\end{enumerate}
\end{definition}
\begin{proposition}\label{OrbitsQuasiinvariant} Let $\rho: (\Lambda, \Lambda^\infty) \to {\rm Is}(X)$ be an isometric action on a pseudo-metric space $(X,d)$.
\begin{enumerate}[(i)]
\item Every $\Lambda$-orbit is quasi-invariant under $\rho(\Lambda)$.
\item Any two $\Lambda$-orbits are at bounded distance.
\end{enumerate}
\end{proposition}
\begin{proof} (i) Let $o \in X$ and let $\cO \coloneqq  \rho(\Lambda)(o)$. Choose a finite set $F \subset \Lambda^\infty$ such that $\Lambda^2 \subset F\Lambda$ and set $R \coloneqq  \sup\{d(o, f.o) \mid f \in F\}$; then for every $\lambda \in \Lambda$ we have $d_{\mathrm{Haus}}(\cO, \rho(\lambda)(\cO)) \leq R$.

(ii) This follows from the fact that $d(\rho(\lambda)(x), \rho(\lambda)(y)) = d(x,y)$ for all $\lambda \in \Lambda$ and $x,y \in X$.
\end{proof}
We recall that an action of a discrete group on a proper metric space is geometric if and only if it is cobounded and proper (see the discussion in Remark \ref{ProperAct} concerning various equivalent definitions). By analogy we define:
\begin{definition} Let $(X,d)$ be a proper pseudo-metric space. An isometric action  $\rho:  (\Lambda, \Lambda^\infty) \to {\rm Is}(X)$ is called 
\begin{enumerate}
\item \emph{cobounded}\index{cobounded action}\index{action!group action!cobounded} if $\rho_1: \Lambda \to {\rm Is}(X) \subset {\rm Map}(X, X)$ is cobounded in the sense of Definition \ref{DefCobProp};
\item \emph{proper}\index{proper!action}\index{action!group action!proper} if $\rho_1: \Lambda \to {\rm Is}(X) \subset {\rm Map}(X, X)$ is proper in the sense of Definition \ref{DefCobProp};
\item \emph{geometric}\index{geometric action}\index{action!group action!geometric} if it is cobounded and proper. 
\end{enumerate}
\end{definition}
Every infinite countable approximate group admits an isometric action with infinite orbits, which is, however, usually not cobounded:
\begin{example}
If $(\Lambda, \Lambda^\infty)$ is an arbitrary countable approximate group and $d$ is a left-admissible metric on $\Lambda^\infty$, then $(\Lambda, \Lambda^\infty)$ acts isometrically on $(\Lambda^\infty, d)$ via $\rho(\lambda)(x) \coloneqq  \lambda x$. Thus every countable approximate group admits an isometric action. If $(\Lambda, \Lambda^\infty)$ is algebraically finitely-generated, then it even admits an isometric action on a locally finite graph (e.g.\ any Cayley graph of $\Lambda^\infty$). However, these actions are not very much of interest to us since the $\Lambda$-orbits are not cobounded unless $(\Lambda, \Lambda^\infty)$ is an almost group.\end{example}
The goal of the remainder of this section is to characterize those approximate groups which admit geometric isometric actions on large-scale geodesic proper metric spaces. The final result will be as follows (see Definition \ref{DefUniformApproximateLattice} and the discussion thereafter for the notion of a uniform approximate lattice):
\begin{theorem}[Approximate groups with geometric isometric actions]\label{ThmIsometricActions} The following are equivalent:
\begin{enumerate}[(i)]
\item $(\Lambda, \Lambda^\infty)$ admits a geometric isometric action on a large-scale geodesic proper metric space.
\item $(\Lambda, \Lambda^\infty)$ is a finite extension of a uniform approximate lattice in a compactly-generated lcsc group.
\end{enumerate}
In this case, $(\Lambda, \Lambda^\infty)$ is algebraically and geometrically finitely-generated and undistorted.
\end{theorem}
We now turn to the proof of the equivalence (i)$\iff$(ii) of Theorem \ref{ThmIsometricActions}; the final statement will be proved in Section \ref{SecMilnorSchwarz}. The implication (ii)$\Rightarrow$(i) in Theorem \ref{ThmIsometricActions} comes from the following simple observation:
\begin{proposition}\label{IsometricActionsTrivialDirection} Let $\Lambda$ be a uniform approximate lattice in a lcsc group $G$, let $\Lambda^\infty < G$ be its enveloping group and let $d$ be a left-admissible metric on $G$. Denote by $\lambda_G: G \to  {\rm Is}(G, d)$ the isometric action of $G$ on itself by left-multiplication. Then the isometric action of $(\Lambda, \Lambda^\infty)$ on $(G,d)$ given by 
\[
(\Lambda, \Lambda^\infty) \hookrightarrow G \xrightarrow{\lambda_G} {\rm Is}(G, d)
\]
is geometric.
\end{proposition}
\begin{proof} Since $\Lambda = \Lambda.e$ is relatively dense in $G$, the action is cobounded. Properness follows from properness of the left-regular action of $G$ on itself together with discreteness of $\Lambda$.
\end{proof}
\begin{proof}[Proof of Theorem \ref{ThmIsometricActions}.] (ii) $\Rightarrow$ (i) Assume that $\pi: (\Lambda, \Lambda^\infty) \to (\Xi, \Xi^\infty)$ is a finite extension and that $(\Xi, \Xi^\infty)$ is a uniform approximate lattice in a compactly-generated lcsc group $G$. Since $G$ is compactly generated, we can find a left-admissible metric $d$ on $G$ such that $(G, d)$ is (proper and) large-scale geodesic (Proposition \ref{CompactlyGeneratedConsequences}). We fix such a metric once and for all. We consider the composition
\[
\rho: (\Lambda, \Lambda^\infty) \xrightarrow{\pi} (\Xi, \Xi^\infty) \hookrightarrow G \xrightarrow{\lambda_G} {\rm Is}(G, d).
\]
Since $\pi$ is onto, we have $\rho(\Lambda).e = \Xi$, hence $\rho$ is cobounded. As for properness, consider
\[
A_R \coloneqq  \{\lambda \in \Lambda \mid \rho(\lambda) B_R(e) \cap B_R(e) \neq \emptyset \}.
\]
From the properness part of Proposition \ref{IsometricActionsTrivialDirection} we know that $F_R \coloneqq  \pi(A_R) \subset \Lambda$ is finite. Now assume that $\lambda_1, \lambda_2 \in A_R$ such that $\pi(\lambda_1) = \pi(\lambda_2)$. Then $\pi(\lambda_1^{-1}\lambda_2) = e$ and hence $\lambda_1^{-1}\lambda_2 \in \ker(\pi_2: \Lambda^2 \to \Xi^2)$, which is finite by assumption. This shows that $A_R$ fibers over the finite set $F_R$ with finite fibers, hence it is finite as well, showing properness of the action $\rho$. This finishes the proof.
\end{proof}
We now turn to the converse implication. Note that if an approximate group $(\Lambda, \Lambda^\infty)$ acts geometrically by isometries on a proper metric space $X$, then in particular $\Lambda^\infty$ and hence ${\rm Is}(X)$ acts coboundedly on $X$. For this reason we now restrict attention to metric spaces which admit a cobounded action of the isometry group. We will call such metric spaces \emph{cobounded}\index{metric space!cobounded}\index{cobounded} for short. We recall from Proposition \ref{MSTop} that the isometry group of a \lsg cobounded proper metric space is a compactly-generated lcsc group.
\begin{proposition}\label{IsometricActionProperness} Let $X$ be a \lsg cobounded proper metric space and let $\Lambda$ be an approximate subgroup of $G \coloneqq  {\rm Is}(X)$. Then the following are equivalent:
\begin{enumerate}[(i)]
\item The isometric action of $(\Lambda, \Lambda^\infty)$ on $X$ is geometric.
\item $\Lambda$ is relatively dense in $G$ and the map $p: \Lambda \times G \to G \times G$, $(\lambda, g) \mapsto (g, \lambda g)$ is proper. 
\end{enumerate}
\end{proposition}
\begin{proof} Since $G$ is a compactly-generated lcsc group it admits a large-scale geodesic left-admissible metric $d_G$ (Proposition \ref{CompactlyGeneratedConsequences}). Fix a basepoint $o \in X$ and let $\iota = \iota_o: (G, d_G) \to (X, d)$ be the corresponding orbit map. By Proposition \ref{MSTop} the map $\iota$ is a quasi-isometry. Denote by  
\[
\mu: \Lambda \times G \to G, \   (\lambda, g) \mapsto \lambda g \quad \text{and} \quad \rho: \Lambda \times X \to X, \   (\lambda, x) \mapsto \lambda(x),
\]
the natural actions of $\Lambda$ on $G$ and $X$ respectively. Then we have a commuting diagram
\[
\begin{xy}\xymatrix{
 \Lambda \times G  \ar[d]_{{\rm id} \times \iota} \ar[rr]^{\mu}&& G \ar[d]^{\iota}\\
 \Lambda \times X  \ar[rr]_{\rho} && X.
}\end{xy}\]
Since $\iota$ is a quasi-isometry, it follows that $\Lambda$ is relatively dense in $G$ if and only if $\Lambda.o = \iota(\Lambda)$ is relatively dense in $X$, which means that the action is cobounded. Similarly, the map $p:\Lambda \times G \to G \times G$, $(\lambda, g) \mapsto (x, \lambda g)$ is proper if and only if the map $p_X: \Lambda \times X \to X \times X, (\lambda, x) \mapsto (x, \lambda.x)$ is proper. This establishes the desired equivalence. 
\end{proof}
We can reformulate condition (ii) of Proposition \ref{IsometricActionProperness}
using the following proposition:
\begin{proposition}\label{AppLatticeFromAction} Let $G$ be a compactly-generated lcsc group and $\Lambda \subset G$ be an approximate subgroup. Then the following are equivalent:
\begin{enumerate}[(i)]
\item $\Lambda$ is relatively dense in $G$ and the map $p: \Lambda \times G \to G \times G$, $(\lambda, g) \mapsto (g, \lambda g)$ is proper. 
\item $\Lambda$ is an approximate lattice in $G$.
\end{enumerate}
\end{proposition}
\begin{proof} (ii) $\Rightarrow$ (i):  If $\Lambda$ is an approximate lattice in $G$, then $\Lambda$ is relatively dense in $G$. Moreover, it is uniformly discrete and hence locally finite with respect to any left-admissible metric on $G$. In particular, if we fix such a metric, then for every $R > 0$ the intersection $B_{2R}(e) \cap \Lambda$ is finite. Since
\[
\{\lambda \in \Lambda\mid B_{R}(e) \cap \lambda B_R(e)\} \subset B_{2R}(e) \cap \Lambda,
\]
this implies that the map $p: \Lambda \times G \to G \times G$ is proper.

(i) $\Rightarrow$ (ii) Since $\Lambda$ is relatively dense by assumption, we only need to show that $\Lambda$ is uniformly discrete. Assume for contradiction that this is not the case. Then there exist elements $g_n, h_n\in G$ with $g_n \neq h_n$ such that $x_n \coloneqq  g_n^{-1}h_n$ converges to $e$. Note that $x_n \in \Lambda^2 \setminus \{e\}$ and choose a finite $F_\Lambda \subset \Lambda^3$ such that $\Lambda^2 = \Lambda F_\Lambda$. Passing to a subsequence we may assume that $x_n \in \Lambda\{f\}$ for some fixed $f \in F_\Lambda$. This means that there exist $\lambda_n \in \Lambda \setminus \{f^{-1}\}$ such that $\lambda_n \to f^{-1}$, i.e., $g \coloneqq  f^{-1}\in G$ is an accumulation point of $\Lambda$. Fix a left-admissible metric $d$ on $G$ and let $R_0 \coloneqq  d(g, g^2)$ and $R \coloneqq  2R_0$. Now for every $\epsilon \in (0, R_0]$ there exist infinitely many $n \in \mathbb N$ with $\lambda_n \in B_\epsilon(g)$. For such $n$ we then have $d(g\lambda_n, g^2) = d(\lambda_n, g) < \epsilon$. In particular,
\[ 
g \lambda_n \in B_{R}(g)\lambda_n \quad \text{and} \quad g \lambda_n \in B_\epsilon(g^2) \subset B_{R_0 + \epsilon}(g) \subset B_{R}(g),
\]
hence $B_R(g)\lambda_n \cap B_R(g) \neq \emptyset$. This shows that the map $(\lambda, g) \mapsto (g, g\lambda)$ is not proper. Applying inverses to both sides and using that $G$ and $\Lambda$ are symmetric it follows that $p$ is not proper, which is a contradiction.
\end{proof}
\begin{proof}[Proof of Theorem \ref{ThmIsometricActions}] (i)$\implies$(ii): Assume that $(\Lambda, \Lambda^\infty)$ admits a geometric isometric action on a large-scale geodesic proper metric space $(X,d)$, say $\rho: (\Lambda, \Lambda^\infty) \to {\rm Is}(X)$. Recall from Proposition \ref{CompactlyGeneratedConsequences} that $G \coloneqq  {\rm Is}(X)$ is a compactly-generated lcsc group and denote by $(\Xi, \Xi^\infty)$ the image of $\rho$. Then by the implication (i) $\Rightarrow$ (ii) of Proposition \ref{IsometricActionProperness} and the implication (i) $\Rightarrow$ (ii) of Proposition \ref{AppLatticeFromAction} we deduce that $\Xi$ is an approximate lattice in $G$, and thus it remains to show only that the extension $(\Lambda, \Lambda^\infty) \to (\Xi, \Xi^\infty)$ is finite. We first observe that since $\rho: (\Lambda, \Lambda^\infty) \to {\rm Is}(X)$ is a proper isometric action, then by Lemma \ref{GeometricSurvives2} the induced isometric action $\rho_k: (\Lambda^k, \Lambda^\infty) \to {\rm Is}(X)$ is also proper. Now if $k \in \mathbb N$ and $\lambda \in \ker_k(\rho)$, then for any basepoint $o \in X$ and every $R > 0$ we have
\[
\lambda B_R(o) \cap B_R(o) \neq \emptyset.
\]
From properness of $\rho_k$ we thus deduce that $\ker_k(\rho)$ is finite, hence $\rho$ has finite kernels. This finishes the proof of the implication (i)$\Rightarrow$(ii).
\end{proof}
Let us point out that Theorem \ref{ThmIsometricActions} and its proof seriously restrict the class of geodesic proper metric spaces on which an approximate group, which is not an almost group, can act:
\begin{corollary}
If $X$ is a large-scale geodesic proper metric space and $\rho: (\Lambda, \Lambda^\infty)\to \mathrm{Is}(X)$ is a geometric isometric action, then
the following are equivalent:
\begin{enumerate}[(i)]
\item  The kernel of $\rho: \Lambda^\infty \to \mathrm{Is}(X)$ is finite and $\rho(\Lambda^\infty) \subset \mathrm{Is}(X)$ is a lattice.
\item The kernel of $\rho: \Lambda^\infty \to \mathrm{Is}(X)$ is finite and $\rho(\Lambda^\infty) \subset \mathrm{Is}(X)$ is discrete.
\item $\Lambda^\infty$ acts properly on $X$.
\item $(\Lambda, \Lambda^\infty)$ is an almost group.
\end{enumerate}
\end{corollary}
\begin{proof}  (i)$\implies$(ii) Clear.

(ii)$\implies$(iii) Since $\mathrm{Is}(X)$ (as a locally compact group) acts properly on $X$, so does every discrete subgroup, and hence every finite extension of a discrete subgroup.

(iii)$\implies$(iv) We first observe that by Theorem \ref{ThmIsometricActions} $(\Lambda, \Lambda^\infty)$ is geometrically and algebraically finitely-generated. We fix a basepoint $o \in X$, a finite generating set $S$ and associated word metric $d_S$ for $\Lambda^\infty$ and denote by $d= d_S|_{\Lambda \times \Lambda}$ the corresponding internal metric. By (ii) and Lemma \ref{MSClassic} the orbit map $(\Lambda^\infty, d_S) \to X$, $g \mapsto g.o$ is a quasi-isometry. Since $\Lambda$ acts coboundedly on $X$, the restriction $(\Lambda, d) \to X$ is also a quasi-isometry. This implies that the inclusion $(\Lambda, d) \hookrightarrow (\Lambda^\infty, d_S)$ is a quasi-isometry, which means that $\Lambda$ is syndetic in $\Lambda^\infty$.

(iv)$\implies$(i) Write $\Lambda^\infty = \Lambda F$ for a finite subset $F \subset \Lambda^\infty$ and assume that $g \in \Lambda^\infty$ is in the kernel of $\rho$. Write $g = \lambda f$ with $\lambda \in \Lambda$ and $f \in F$. Then for any $o \in X$ we have $o = \rho(g).o = \rho(\lambda) \rho(f).o$, hence $\rho(\lambda^{-1}).o = \rho(f).o$. By properness, this forces $\lambda^{-1}$ to be contained in a finite subset of $\Lambda$, hence $g$ is contained in a finite subset of $\Lambda^\infty$. This establishes the first statement.

Since $\Lambda$ acts cocompactly on $X$, so does $G \coloneqq \mathrm{Is}(X)$, hence by Proposition \ref{MSTop} it is compactly-generated and the orbit maps $G \hookrightarrow X$ are quasi-isometries. Since $\Lambda$ acts geometrically on $X$, it follows that $\rho(\Lambda)$ is relatively dense in $G$ and that the map $\rho(\Lambda) \times G \to G \times G$, $(\lambda, g) \mapsto (g, \lambda g)$ is proper. By Proposition \ref{AppLatticeFromAction} $\rho(\Lambda)$ is then a uniform approximate lattice in $G$, i.e.\ uniformly discrete and syndetic. Since $\rho(\Lambda)$ is syndetic in $\rho(\Lambda^\infty)$, it follows that the latter is also discrete and syndetic, i.e.\ a uniform lattice.
\end{proof}
\begin{corollary}\label{DiscreteIsometry} Assume that $(\Lambda, \Lambda^\infty)$ is not an almost group. Then $(\Lambda, \Lambda^\infty)$ does not admit any
geometric isometric action on a large-scale geodesic proper metric space with discrete isometry group.
\end{corollary}

\section{Quasi-isometric quasi-actions (qiqacs)}\label{SecQiqac}

Since not all approximate groups are finite extensions of uniform approximate lattices in a compactly generated lcsc group (e.g.\ distorted approximate groups are not), it follows from Theorem \ref{ThmIsometricActions} that not every approximate group admits a geometric isometric action. In order to develop a far-reaching geometric theory of approximate groups one thus has to relax the definition of an isometric action.
\begin{definition}\label{def: qiqac} Let $(\Lambda, \Lambda^\infty)$ be an approximate group, $(X, d)$ be a pseudo-metric space and $o \in X$. Given $k \in \mathbb N$, let $K_k \geq 1$ and $C_k, C_k' \geq 0$ be constants.
\begin{enumerate}[(i)]
\item A map of pairs $\rho: (\Lambda, \Lambda^\infty) \to (\widetilde{{\rm QI}}(X), \widetilde{{\rm QI}}(X))$ is called a \emph{$(K_k, C_k, C_k')$-quasi-isometric quasi-action}\index{qiqac!quasi-isometric-quasi-action}\index{quasi-action!quasi-isometric}\index{quasi-action!qiqac} (or \emph{$(K_k, C_k, C_k')$-qiqac}\index{qiqac} for short) of $(\Lambda, \Lambda^\infty)$ on $X$ if for every $k \geq 1$ we have $\rho(\Lambda^k) \subset \widetilde{\rm QI}_{K_k, C_k, C_k'}(X)$ and
\begin{equation}\label{qac}
d(\rho(\lambda_1)\rho(\lambda_2).x, \rho(\lambda_1\lambda_2).x) < C_k' \quad (x \in X, \lambda_1, \lambda_2 \in \Lambda^k).
\end{equation}
\item If there exists $K>0$ such that $K_k \leq K$ for all $k \in \mathbb N$, then a $(K_k, C_k, C_k')$-qiqac is called a \emph{$(K, C_k, C_k')$-uniform qiqac}\index{qiqac!uniform}. A $(1, C_k, C_k')$-uniform qiqac is called a \emph{$(C_k, C_k')$-rough  quasi-action}\index{qiqac!rough quasi-action} (or \emph{roqac}\index{qiqac!roqac} for short).
\item A map of pairs $\rho: (\Lambda, \Lambda^\infty) \to (\widetilde{{\rm QI}}(X), \widetilde{{\rm QI}}(X))$ is called a \emph{weak qiqac}\index{qiqac!weak} if it factors through a morphism  $\overline{\rho}: (\Lambda, \Lambda^\infty) \to ({\rm QI}(X), {\rm QI}(X))$ (called the \emph{associated morphism}\index{qiqac!associated morphism}). 
\item If  $\rho: (\Lambda, \Lambda^\infty) \to (\widetilde{{\rm QI}}(X), \widetilde{{\rm QI}}(X))$ is a  $(K_k, C_k, C_k')$-qiqac, then the set 
\[
{\rho}(\Lambda).o \coloneqq  \{{\rho}(\lambda).o\mid \lambda \in \Lambda\}
\]
is called the \emph{$\Lambda$-quasi-orbit}\index{quasi-orbit} of $o$ under $\rho$.
\end{enumerate}
\end{definition}
In the sequel, when given a group $G$ and an approximate group $ (\Lambda, \Lambda^\infty)$, we will often denote a map of pairs $ (\Lambda, \Lambda^\infty) \to (G,G)$ simply by $(\Lambda, \Lambda^\infty) \to G$. This is in accordance with our previous abuse of notation for morphisms and applies in particular to qiqacs.
\begin{remark}\label{QuasiActionRemarks}
\begin{enumerate}[(i)]
\item By definition, every isometric action is a uniform qiqac and every uniform qiqac is a qiqac. Moreover, a map $\rho:(\Lambda, \Lambda^\infty) \to \widetilde{{\rm QI}}(X)$ is a weak qiqac in the sense of the previous definition if and only if 
\[
C_{\lambda_1, \lambda_2}\coloneqq  \sup_{x \in X} d(\rho(\lambda_1)\rho(\lambda_2).x, \rho(\lambda_1\lambda_2).x) < \infty \quad \text{for all $\lambda_1, \lambda_2 \in \Lambda$.}
\]
This formula shows that every qiqac is a weak qiqac, and that a weak qiqac needs to satisfy two additional requirements to be called a qiqac: The constants $C_{\lambda_1, \lambda_2}$ have to be bounded \emph{uniformly} as $\lambda_1$ and $\lambda_2$ vary inside some fixed $\Lambda^k$, and the QI-constants of $\rho(\lambda)$ have to be bounded \emph{uniformly} as $\lambda$ varies over some $\Lambda^k$.
\item $\Lambda$-quasi-orbits are quasi-invariant under $\rho(\Lambda)$ and $\Lambda$-quasi-orbits through different base-points for a given qiqac are at mutually bounded distance (by a quasified version of the proof of Proposition \ref{OrbitsQuasiinvariant}).
\item It is immediate from the above definition that every qiqac $\rho:(\Lambda, \Lambda^\infty) \to (\widetilde{{\rm QI}}(X), \widetilde{{\rm QI}}(X))$ restricts to a qiqac $\rho:(\Lambda^k, \Lambda^\infty) \to (\widetilde{{\rm QI}}(X), \widetilde{{\rm QI}}(X))$, called the \emph{induced} qiqac\index{qiqac!induced}. 
\item The conditions in the definition of a qiqac can be weakened. It suffices to assume that $\rho(\Lambda)$ is uniform and that \eqref{qac} holds for all $\lambda_1 \in \Lambda$ and $\lambda_2 \in \Lambda^k$. Using the latter assumption one can then prove by induction that \eqref{qac} holds for all $\lambda_1, \lambda_2 \in \Lambda^k$ (for a larger uniform constant), and conclude that $\rho(\Lambda^k)$ is uniformly close to $(\rho(\Lambda))^k$, hence uniform if $\rho(\Lambda)$ is. We chose the current formulation in order to make the existence of the induced qiqac more obvious.
\item Our definition of a qiqac is of course not the only possible relaxation of the definition of an isometric action. As far as the involved QI constants are concerned, 
one might as well ask for stronger or weaker uniformity conditions and it may seem at this point that our definition is rather arbitrary. However, in reality the wiggle room is quite limited: If one asks for more uniformity then the left-regular quasi-action defined in Section \ref{SecLeftRegular} below will no longer be a quasi-action, and if one relaxes the uniformity too much, then one can no longer establish basic results such as the Milnor--Schwarz lemma.
\item The definition of a qiqac does not require any relation between $\rho(g)$ and $\rho(g)^{-1}$ a priori. However, such a relation does in fact always hold by the following lemma.
\end{enumerate}
\end{remark}
\begin{lemma}\label{Inverses} If $\rho: (\Lambda, \Lambda^\infty) \to \widetilde{{\rm QI}}(X)$ is a qiqac, then there exist constants $D_k$ such that for all $x \in X$ and $g \in \Lambda^k$,
\[
d(\rho(g)^{-1}x, \rho(g^{-1})x) < D_k \quad \text{and}\quad d(\rho(e)x, x) < D_1.
\]
\end{lemma}
\begin{proof} Assume that $\rho$ is a $(K_k, C_k, C_k)$-qiqac. Then for all $x \in X$ we have
\[
C_1 > d(\rho(e^2)x, \rho(e)\rho(e)x) = d(\rho(e)x, \rho(e)\rho(e)x)  \geq K_1^{-1} d(x, \rho(e)x)-C_1,
\]
and hence $d(\rho(e)x, x) < 2K_1C_1$ independently of $x$. We deduce that for $g \in \Lambda^k$ we have
\begin{eqnarray*}
d(\rho(g)^{-1}x, \rho(g^{-1})x) &\leq& K_kd(x, \rho(g)\rho(g^{-1})x)+C_k \\&\leq& K_k(d(x, \rho(gg^{-1})x)+ d(\rho(gg^{-1})x, \rho(g)\rho(g^{-1})x))) +C_k\\
&\leq & K_k(d(x, \rho(e)x) + C_k)+C_k\\
&\leq& K_k(2K_1C_1+C_k)+C_k,
\end{eqnarray*}
which is again independent of $x$.
\end{proof}
Important examples of qiqacs arise from isometric actions by a process called \emph{quasi-conjugation}\index{quasi-conjugation}:
\begin{construction}[Quasi-conjugation] Let $\rho: (\Lambda, \Lambda^\infty) \to \widetilde{\rm QI}(X) $ be a qiqac, let $\varphi: X \to Y$ be a quasi-isometry and let $\overline{\varphi}: Y \to X$ be a quasi-inverse of $\varphi$. Then we obtain a qiqac of $(\Lambda, \Lambda^\infty)$ on $Y$ by 
\[
(\varphi, \overline{\varphi})_*\rho: (\Lambda, \Lambda^\infty) \to {\rm QI}(Y), \quad \lambda \mapsto \varphi \circ \rho(\lambda) \circ \overline{\varphi}.
\]
We refer to $(\varphi, \overline{\varphi})_*\rho$ as the \emph{quasi-conjugation} of $\rho$ by $(\varphi, \overline{\varphi})$.
This construction applies in particular to the case where $\rho$ is an isometric action to begin with. In this case, $(\varphi, \overline{\varphi})_*\rho$ is actually a uniform qiqac. More generally, the class of uniform qiqacs is invariant under quasi-conjugation, whereas the class of roqacs is not.
\end{construction}

\begin{definition}  Let $(\Lambda, \Lambda^\infty)$ be an approximate group, $(X, d)$, $(Y,d')$ be  pseudo-metric spaces, and let $\rho_0, \rho_1: (\Lambda, \Lambda^\infty) \to \widetilde{{\rm QI}}(X)$ and $\rho_2:  (\Lambda, \Lambda^\infty) \to \widetilde{{\rm QI}}(Y)$ be qiqacs. For every $k \in \mathbb N$ let $K_k \geq 1$ and $C_k, C_k' \geq 0$ be constants.
\begin{enumerate}[(i)]
\item We say that $\rho_0$ and $\rho_1$ are $(C_k)$-\emph{equivalent} if $d(\rho_0(\gamma).x, \rho_1(\gamma).x) < C_k$ for all $\gamma \in \Lambda^k$ and $x \in X$.
\item We say that $\rho_1$ and $\rho_2$ are $(K_k, C_k, C_k')$-\emph{quasi-conjugate}\index{quasi-conjugate} if there exist $(K_k, C_k, C_k')$-quasi-isometries $\varphi: X \to Y$ and $\overline{\varphi}: Y \to X$ such that $\varphi \circ \overline{\varphi}$ and $\overline{\varphi} \circ \varphi$ are $C_k'$-close to the respective identities and such that $(\varphi, \overline{\varphi})_*\rho_1$ and $\rho_2$ are $(C_k)$-equivalent.
\end{enumerate}
\end{definition}
\begin{remark}
Note that if $\rho_0$ and $\rho_1$ are equivalent qiqacs, then the associated morphisms into ${\rm QI}(X)$ coincide. More generally, if they are quasi-conjugate qiqacs on the same space $X$, then the associated morphisms are conjugate by an element of ${\rm QI}(X)$, hence the name. 

However, being quasi-conjugate is stronger than just being conjugate inside ${\rm QI}(X)$: The definitions require an additional uniformity, which is needed, for example, to ensure that quasi-orbits of equivalent quasi-actions are at bounded Hausdorff distance from each other.
\end{remark}
An important construction in group theory is to restrict a given action of a group $\Gamma$ to a $\Gamma$-invariant set, for example a single $\Gamma$-orbit. Similarly we will need to be able to ``restrict'' qiqacs of an approximate group $(\Lambda, \Lambda^\infty)$ to $\Lambda$-quasi-orbits; see e.g.\ Construction \ref{PreLeftRegular} and Proposition \ref {BoundaryStableOrbitHyp} below for two different applications of this technique. While such $\Lambda$-quasi-orbits are typically not invariant under $\Lambda^\infty$, they are always quasi-invariant under $\Lambda$ by Remark \ref{QuasiActionRemarks}.(ii). We thus need to figure out how to ``restrict'' a qiqac of $(\Lambda, \Lambda^\infty)$ to a $\Lambda$-quasi-invariant set. We start with the technically easier special case in which the set in question is actually invariant under $\Lambda^\infty$.

\begin{construction}[Quasi-restrictions of qiqacs, I]\label{Quasirestriction1} Assume that $\rho$ is a qiqac of an approximate group $(\Lambda, \Lambda^\infty)$ on a pseudo-metric space $(X,d)$ and let $Y \subset X$ be a subset which is quasi-invariant under $\rho(\Lambda^\infty)$. By definition, this means that the inclusion
\[
\iota: Y \hookrightarrow \widetilde{Y} \coloneqq  Y \cup \bigcup_{\lambda \in \Lambda^\infty} \rho(\lambda)(Y)
\]
is coarsely onto. Thus there exists a quasi-isometry $p:X \to X$ with the following properties:
\begin{itemize}
\item $p$ is of finite distance from the identity and $p(x) = x$ for all $x \in X \setminus \widetilde{Y}$.
\item $p|_{\widetilde{Y}}$ is a quasi-inverse to $\iota$.
\end{itemize}
Then $\widetilde{\rho} \coloneqq  p \circ \rho$ defines a qiqac which is equivalent to $\rho$ and such that $Y$ is invariant under $\widetilde{\rho}(\Lambda^\infty)$. We may thus turn every $\rho(\Lambda^\infty)$-quasi-invariant subset into a $\rho(\Lambda^\infty)$-\emph{invariant} subset at the cost of replacing $\rho$ by an equivalent qiqac. We then obtain a qiqac
\[
\widetilde{\rho}|_Y : (\Lambda, \Lambda^\infty)\to \widetilde{\mathrm{QI}}(Y), \quad \lambda \mapsto \widetilde{\rho}(\lambda)|_Y,
\]
whose equivalence class is uniquely determined by $\rho$. We will refer to any qiqac in this equivalence class as a \emph{quasi-restriction}\index{quasi-restriction} of $\rho$ to $Y$.
\end{construction}
A variant of this construction still applies if $Y$ is merely assumed to be quasi-invariant under $\Lambda$, albeit the details become more technical:
\begin{construction}[Quasi-restrictions of qiqacs, II]\label{Quasirestriction2} Assume that $\rho$ is a qiqac of an approximate group $(\Lambda, \Lambda^\infty)$ on a pseudo-metric space $(X,d)$ and let $Y \subset X$ be a subset which is quasi-invariant under $\rho(\Lambda)$. Set $Y_1 \coloneqq  Y$ and for $n \in \bN$ define
\[
 Y_{n+1} \coloneqq  Y_1 \cup Y_2 \cup \dots \cup Y_{n} \cup \bigcup_{\lambda \in \Lambda^{n}} \rho(\lambda)(Y).
\]
If $F \subset \Lambda^\infty$ is finite with $\Lambda^2 \subset F \Lambda$, then $\Lambda^n \subset F^{n-1} \Lambda$, and hence $\rho(\Lambda^n)(Y)$ is at bounded distance from $\rho(\Lambda^{n-1})(Y)$. If $\rho$ is an isometric action, then this bound can in fact be chosen independently of $n$, but in general it will depend on $n$. In any case, we may deduce that the isometric inclusions $\iota_n^{n+1}: Y_{n} \to Y_{n+1}$ are quasi-isometries. We choose quasi-inverses $p^{n+1}_n: Y_{n+1} \to Y_n$ such that \[
p_n^{n+1} \iota^{n+1}_n(y) = y \text{ for all } y \in Y_n  \quad \text{ and } \quad d(\iota_n^{n+1}(p^{n+1}_n(x)), x) \leq \delta_n \text{ for all }x \in Y_{n+1},\]
where $\delta_n$ is a constant depending on the QI constants of $\rho$. Again, if $\rho$ was actually an isometric action, then the sequence $(\delta_n)$ can be chosen to be bounded. 

In the sequel, for all $k<n$ we denote by $\iota_k^n: \Lambda^k \to \Lambda^n$ the isometric embedding and define
\[
 p_{k}^n: \Lambda^n \to \Lambda^k, \quad y \mapsto p^{k+1}_{k} \circ \dots \circ p^{n}_{n-1}(y).
\]
Then for all $k<n$ the maps $p_k^n$ are $(1, \sigma_n, 0)$-quasi-isometries for some constant $\sigma_n \geq 0$, and satisfy $p^n_k(\iota_k^n(y)) = y$. Note that even in the case of an isometric action the sequence $(\sigma_n)$ may grow linearly in $n$ and need not be bounded.

Now let $g \in \Lambda^\infty$ and choose $k$ such that $g \in \Lambda^k \setminus \Lambda^{k-1}$. Then for all $y \in Y$ we have $\rho(g)(y) \in Y_{k+1}$, and we define $\widetilde{\rho}_Y(g) \coloneqq  p^{k+1}_1 \circ \rho(g)$, i.e.,
\[\begin{xy}\xymatrix{
\widetilde{\rho}_Y(g): &Y_1 \ar@/^1pc/[rrrr]^{\rho(g)}& \ar@/^/[l]^{p_1^2}Y_2& \ar@/^/[l]^{p_2^3}Y_3&  \ar@/^/[l]^{p_3^4}
 \dots & \ar@/^/[l]^{p_{k}^{k+1}} Y_{k+1}
}\end{xy}.
\]
If $\rho$ is a $(K_k, C_k, C_k')$-qiqac, then $\widetilde{\rho}_Y(g)$ will be a $(K_k, C_k + \sigma_k, C_k' + \sigma_k)$-quasi-isometry of $Y= Y_1$.
Moreover, if $h \in \Lambda^l$ and $y \in Y$, then we have
\begin{eqnarray*}
&& d\left(\widetilde{\rho}_Y({gh})(y), \widetilde{\rho}_Y(g)\widetilde{\rho}_Y(h)(y)\right) \\
&=&  d\left( p_1^{k+l+1}(\rho(gh)(y)), p_1^{k+1}(\rho(g)(p_1^{l+1}(\rho(h)(y))))\right)\\
&\leq& d\left( p_1^{k+l+1}(\rho(gh)(y)), \rho(gh)(y)\right)\ + \ d\left(\rho(gh)(y), \rho(g)(p_1^{l+1}(\rho(h)(y)))\right)\\ &&  + \; d\left( \rho(g)(p_1^{l+1}(\rho(h)(y)), p_1^{k+1}(\rho(g) (p_1^{l+1}(\rho(h)(y))))\right),
\end{eqnarray*}
which is bounded by some constant depending on $k$ and $l$. This shows that $\widetilde{\rho}_Y$ is a qiqac. Since this construction generalizes Construction \ref{Quasirestriction2} we still refer to $\widetilde{\rho}_Y$ as a \emph{quasi-restriction}\index{quasi-restriction} of $\rho$ to $Y$.
\end{construction}
In view of Remark \ref{QuasiActionRemarks}.(ii) we are now able to restrict any given qiqac of $(\Lambda, \Lambda^\infty)$ to an arbitrary $\Lambda$-quasi-orbit. This construction will become important very soon (see Section \ref{SecLeftRegular}, in particular Construction \ref{PreLeftRegular}).

\section{Geometric quasi-actions}\label{SecGeomQiqac}
Imitating the definition in the case of isometric actions we make the following definition:
\begin{definition} Let  $(\Lambda, \Lambda^\infty)$ be an approximate group and let $X$ be a proper metric space. A qiqac $\rho:  (\Lambda, \Lambda^\infty) \to \widetilde{{\rm QI}}(X)$ is called 
\begin{enumerate}
\item \emph{cobounded}\index{qiqac!cobounded} if $\rho_1: \Lambda \to  \widetilde{{\rm QI}}(X) \subset {\rm Map}(X, X)$ is cobounded in the sense of Definition \ref{DefCobProp};
\item \emph{proper}\index{qiqac!proper} if $\rho_1: \Lambda \to  \widetilde{{\rm QI}}(X) \subset {\rm Map}(X, X)$ is proper in the sense of Definition \ref{DefCobProp};
\item \emph{geometric}\index{qiqac!geometric} if it is cobounded and proper. 
\end{enumerate}
\end{definition}
\begin{remark}\label{RemarkProperness}
\begin{enumerate}
\item By definition, a qiqac $\rho$ is cobounded if some of its quasi-orbits are relatively dense. It then follows from Remark \ref{QuasiActionRemarks} that all quasi-orbits are relatively dense.
\item Properness of $\rho$ means that for every $R >0$ and some $x \in X$ the set
\[
\{\lambda \in \Lambda\mid B_R(x) \cap \rho(\lambda).B_R(x) \neq \emptyset\}
\]
is finite. This then holds for \emph{every} $x \in X$. In fact, from Lemma \ref{GeometricSurvives2} we can draw the stronger conclusion that for every $k \geq 1$, $R>0$ and $x \in X$ the set
\[
\{\mu \in \Lambda^k\mid B_R(x) \cap \rho(\mu).B_R(x) \neq \emptyset\}
\]
is finite.
\item Every isometric action of an approximate group is in particular a qiqac, and hence we obtain notions of a proper, cobounded or geometric isometric action. For isometric actions of groups we recover the usual definitions.
\end{enumerate}
\end{remark}
The property of being geometric is invariant under quasi-conjugation:
\begin{proposition} \label{GeometricSurvives1} 
If two qiqacs are quasi-conjugate, then one is proper, co\-boun\-ded or geometric if and only if the other one has the corresponding property.
\end{proposition}
The proof of Proposition \ref{GeometricSurvives1} has nothing to do with approximate groups. In fact, the proposition is just a special case of the following two lemmas, which are formulated in the generality of Definition \ref{DefCobProp}.
\begin{lemma}\label{ProperQCInv} Let $X, Y$ be proper metric spaces, let $\varphi: X \to Y$ be a quasi-isometry and $\bar \varphi: Y \to X$ be a quasi-inverse to $\varphi$. Let $A$ be a set, $\rho: A \to {\widetilde{\rm QI}}(X)$ be a map with uniform image and define $\bar \rho:  A \to \widetilde{\rm QI}(Y)$ by $\bar \rho(a) \coloneqq  \varphi \circ \rho(a) \circ \bar \varphi$. Then the following hold:
\begin{enumerate}[(i)]
\item $\bar \rho$ is cobounded if and only if $\rho$ is cobounded.
\item $\bar \rho$ is proper if and only if $\rho$ is proper.
\end{enumerate}
\end{lemma}
\begin{proof} We choose $K>1$ and $C>0$ such that the following hold: $\varphi$ and all $\rho(a)$ are $(K, C, C)$-quasi-isometries for all $a\in A$. Moreover, there exists a $(K, C, C)$-quasi-isometry $\bar \varphi$ such that $\varphi \bar \varphi$ and $\bar\varphi \varphi$ are at distance at most $C$ from the respective identities. Let $x_o \in X$ and $y_o \coloneqq  \varphi(x_o)$.

(i) Assume that $\rho(A).x_o$ is $R$-relatively dense in $X$. Then for every $y \in Y$ we find $a = a(y) \in A$ such that $d(\rho(a)(x_o), \bar \varphi(y)) \leq R$. Then
\begin{eqnarray*}
d(\bar \rho(a) y_o, y) &=& d(\varphi \rho(a) \bar \varphi\varphi(x_o), y) \ \leq \  d(\varphi \rho(a) \bar \varphi\varphi(x_o), \varphi \rho(a)(x_o)) +  d(\varphi \rho(a)(x_o), y)\\
&\leq & K d( \rho(a) \bar \varphi\varphi(x_o), \rho(a)(x_o))+ C + K d(\bar \varphi \varphi \rho(a)(x_o), \bar \varphi(y)) +C\\
&\leq& K^2 d(\bar \varphi \varphi(x_0),x_0) + KC + C + Kd( \rho(a)(x_o), \bar\varphi(y))  +KC +C\\
&\leq& K^2 C+ KC +C + KR + KC + C \ = \ KR + K^2 C + 2KC + 2C.
\end{eqnarray*}

Since $y \in Y$ was arbitrary, the set $\bar \rho(A).y_o$ is $(KR + K^2 C + 2KC + 2C)$-relatively dense in $Y$, and hence $\bar \rho$ is cobounded. The converse follows by reversing the roles of $X$ and $Y$.

(ii) Assume that $\rho$ is proper, i.e. for every $R > 0$ and every $x \in X$ the set 
\[
A^\rho_{x, R} \coloneqq  \{a \in A\mid B(x, R) \cap \rho(a).B(x, R) \neq \emptyset\}
\]
is finite. Now let $y \in Y$ and assume that 
\begin{equation}\label{ProperQCInvToShow}B(y, R) \cap \bar \rho(a).B(y, R) \neq \emptyset \end{equation} for some  $R \gg 0$ and $a \in A$. We then find $y_1, y_2 \in B(y, R)$ such that $\bar \rho(a)(y_1) = y_2$, and we define $x\coloneqq  \bar \varphi(y)$, $x_1 \coloneqq  \bar\varphi(y_1)$ and $x_2 \coloneqq  \bar\varphi(y_2)$. Then $x_1, x_2 \in B_{KR+ C}(x)$ and
\[
 x_2 = \bar \varphi(y_2) = \bar \varphi \bar\rho(a)(y_1) =  \bar \varphi \varphi \rho(a) \bar \varphi(y_1) = \bar\varphi\varphi \rho(a) x_1,
\]
hence $d( \rho(a)x_1, x_2) < C$ and thus
\[
d(\rho(a)x_1, x) \leq d( \rho(a)x_1, x_2)  + d(x_2, x) < C+KR+ C = KR + 2C =: R'.
\]
We deduce that $\{x_1, \rho(a)x_1\} \subset B_{R'}(x)$ and hence $a \in A^\rho_{x, R'}$. This shows that for every $R \gg 0$ the set of $a \in A$ satisfying \eqref{ProperQCInvToShow} is finite, whence $\bar \rho$ is proper. The converse follows again by reversing the roles of $X$ and $Y$.
\end{proof} 
Similarly we have:
\begin{lemma}\label{GeometricInvUnderEquiv} Let $(X,d)$ be a proper pseudo-metric space, $C>0$, let $A$ be a set and let $\rho_0, \rho_1: A \to {\rm Map}(X, X)$ be such that $d(\rho_0(a)(x), \rho_1(a)(x))<C$ for all $a \in A$ and $x \in X$. Then $\rho_0$ is cobounded, respectively proper, if and only if $\rho_1$ has the corresponding property.
\end{lemma}
\begin{proof} The statement about coboundedness follows immediately from the fact that $\rho_0(A).x$ and $\rho_1(A).x$ are at bounded Hausdorff distance for each $x \in X$. As for properness, assume that $\rho_0$ is proper and assume that for some $\gamma \in A$, $o \in X$ and $R>0$ we have $B(o, R) \cap \rho_1(\gamma).B(o, R) \neq \emptyset$, say $x' = \rho_1(\gamma)(x)$ is in this intersection, for some $x, x' \in B(o, R)$. Then $x'' \coloneqq  \rho_0(\gamma)x \in B(o, R + C)$ and hence $B(o, R+C) \cap \rho_0(\gamma).B(o, R+C) \neq \emptyset$. This shows that $\gamma$ was contained in a finite set, hence $\rho_1$ is proper.
\end{proof}
At this point we have established Proposition \ref{GeometricSurvives1}. We also mention the following closure property of geometric qiqacs.
\begin{lemma} \label{GeometricSurvives2} If a qiqac $\rho:  (\Lambda, \Lambda^\infty) \to \widetilde{{\rm QI}}(X)$ is proper, cobounded or geometric, then the induced quasi-isometric quasi-action $\rho:  (\Lambda^k, \Lambda^\infty) \to \widetilde{{\rm QI}}(X)$ has the same property.
\end{lemma}

\begin{proof} The statement about coboundedness is immediate since $\rho(\Lambda^k).x \supset \rho(\Lambda).x$ for every $x \in X$. As for properness, let $F_k \subset \Lambda^\infty$ be finite such that $\Lambda^k \subset \Lambda F_k$. Since $F_k$ is finite, we can choose $k'>k$ such that $F_k \subset \Lambda^{k'}$. Choose constants $K_k > 1$ and $C_k >0$ such that $\rho$ is a $(K_k, C_k, C_k)$-qiqac.

From now on we assume that $B(x, R) \cap \rho(g).B(x, R) \neq \emptyset$ for some $g \in \Lambda^k$, $x \in X$ and $R>0$. This means that there exist $x', x'' \in B(x, R)$ such that
\begin{equation}\label{rhogxProperInduces}
x'' = \rho(g)(x').
\end{equation}
We need to show that $g$ is confined to a finite subset of $\Lambda^k$. We first claim that
\begin{equation}
M(x, R) \coloneqq  \max_{f \in F_k} \sup_{x' \in B(x, R)} d(\rho(f)x', x') < \infty.
\end{equation}
Indeed, if $f \in F_k$ and $x' \in B(x, R)$, then
\[
 d(\rho(f)x', x') \leq d(\rho(f)x', \rho(f)x) + d(\rho(f)x, x) + d(x, x') \leq K_{k'} R + C_{k'} + d(\rho(f)x, x) + R,
\]
which is bounded uniformly in $x'$. Secondly, we can write $g = \lambda f$ for some $\lambda \in \Lambda$ and $f \in F_k \subset \Lambda^{k'}$. For all $x' \in B(x, R)$ we then have
\begin{eqnarray*}
d(\rho(g)x', \rho(\lambda)x') &=& d(\rho(\lambda f)x', \rho(\lambda)x') \quad \leq \quad d(\rho(\lambda) \rho(f)x', \rho(\lambda)x') + C_{k'}\\
 &\leq& K_k d( \rho(f)x', x') + C_k + C_{k'} \quad \leq \quad K_kM(x, R) + C_k+C_{k'}.
\end{eqnarray*}
Let us abbreviate $C(x) \coloneqq   K_kM(x,R) + C_k+C_{k'}$ an define $x''' \coloneqq  \rho(\lambda)x'$. Then
\[
d(x'', x''') = d(\rho(g)x', \rho(\lambda)x') \leq C(x),
\]
hence $x', x''' \in B(x, R+C(x))$. We thus deduce that $B(x, R+C(x)) \cap \rho(\lambda).B(x, R+C(x)) \neq \emptyset$. By properness of the original qiqac we conclude that $\lambda$ is confined to a finite subset of $\Lambda$. Since $g = \lambda f$ for some $f \in F_k$, this implies that $g$ is confined to a finite subset of $\Lambda^k$. This finishes the proof.
\end{proof}
We have seen in Proposition \ref{ApogeeQuasiCobounded} that every apogee for a geometrically finitely-generated approximate group is necessarily quasi-cobounded. The same conclusion also holds for pseudo-metric spaces which admit a cobounded qiqac of an approximate group. We can actually make a more precise statement using the following terminology:
\begin{definition} Let $X$ be a pseudo-metric space. We say that $X$ is \emph{semi-uniformly quasi-cobounded}\index{semi-uniformly quasi-cobounded}\index{quasi-cobounded!semi-uniformly} if there exist constants $K \geq 1$ and $C, R \geq 0$ and a set $\cA$ of $(K, C)$-quasi-isometries of $X$ with the following properties: For all $x, y \in X$ there exists $g \in A$ such that $d(g(x), y) < R$ and there exist constants $C_k \geq 0$ such that $\cA^k$ consists of $(K, C_k)$-quasi-isometries for all $k \in \mathbb N$. We say that $X$ is \emph{uniformly quasi-cobounded}\index{uniformly quasi-cobounded}\index{quasi-cobounded!uniformly} if moreover the sequence $C_k$ can be chosen to be bounded.
\end{definition} 
We record the following immediate consequences of this definition for later reference:
\begin{proposition}\label{QuasiCobounded} If a metric space $(X,d)$ admits a geometric qiqac of an approximate group $(\Lambda, \Lambda^\infty)$, then it is quasi-cobounded. If this qiqac can be chosen to be uniform, then it is even semi-uniformly quasi-cobounded.
\end{proposition}

\section{Left-regular quasi-actions}\label{SecLeftRegular}
The purpose of this section is to show that every algebraically or geometrically finitely-generated approximate group admits a geometric qiqac on a metric space $(X,d)$, which in the case of a geometrically finitely-generated group can be chosen to be a proper geodesic metric space. The case of an algebraically finitely-generated algebraic group was already treated in \cite{BH}, but we recall it here, since the argument in the geometrically finitely-generated case is based on similar ideas. Our inspiration comes again from the group case:
\begin{example}[The group case]
If $\Gamma$ is a finitely-generated group with finite generating set $S$ and associated word metric $d_S$, then left-regular action of $\Gamma$ on itself defines a geometric isometric action on $(\Gamma, d_S)$. If $X$ is any representative of the canonical QI type of $[\Gamma]$, then we can quasi-conjugate the left-regular action to obtain a geometric qiqac of $\Gamma$ on $X$. Note that the resulting quasi-conjugacy class of qiqacs does not depend on the generating set $S$.
\end{example}
We will give similar constructions both for the internal and external QI type. For the moment let $(\Lambda, \Lambda^\infty)$ be an arbitrary approximate group (without any finiteness assumptions). We fix once and for all a left-admissible metric $d$ on $\Lambda^\infty$. Following \cite{BH} we are going to construct a geometric quasi-action of $(\Lambda, \Lambda^\infty)$ on $(\Lambda, d|_{\Lambda \times \Lambda})$.
\begin{construction}\label{PreLeftRegular}
    Let us denote by $L: (\Lambda, \Lambda^\infty) \to \mathrm{Is}(\Lambda^\infty, d)$ the left-regular action of $\Lambda^\infty$ and note that $\Lambda = L(\Lambda)(e)$ is just the quasi-orbit of $e$ under $L$; in particular, $\Lambda$ is quasi-invariant under $L(\Lambda)$. We may thus form the quasi-restriction 
\[
\lambda \coloneqq  \widetilde{L}_{\Lambda}:  (\Lambda, \Lambda^\infty) \to \widetilde{\mathrm{QI}}(\Lambda, d|_{\Lambda \times \Lambda})
\]
as in Construction \ref{Quasirestriction2}.
\end{construction}
Let us describe $\lambda$ more explicitly to get an estimate of the QI-constants of this qiqac: 

Let a finite $F\subset \Lambda^3$ be such that $\Lambda^2 \subset \Lambda F$ and set $\delta \coloneqq  \max_{f \in F} d(f,e)$. Then, for all $n\in\bN$, the inclusions $\iota_n^{n+1}: \Lambda^{n} \to \Lambda^{n+1}$ are $(1,0, \delta)$-quasi-isometries, so we can find quasi-inverses $p^{n+1}_n: \Lambda^{n+1} \to \Lambda^n$ such that \[
p_n^{n+1} \iota^{n+1}_n(x) = x \text{ for all } x \in \Lambda^n  \quad \text{ and } \quad d(\iota_n^{n+1}(p^{n+1}_n(x)), x) \leq \delta \text{ for all }x \in \Lambda^{n+1},\]
and such that each $p_n^{n+1}$ is a $(1, 2\delta, 0)$-quasi-isometry. Thus, if  for all $k<n$ we denote by $\iota_k^n: \Lambda^k \to \Lambda^n$ the isometric embedding and define
\[
 p_{k}^n: \Lambda^n \to \Lambda^k, \quad x \mapsto p^{k+1}_{k} \circ \dots \circ p^n_{n-1}(x), 
\]
then for all $k<n$ the maps $p_k^n$ are $(1, 2(n-k)\delta, 0)$-quasi-isometries and satisfy $p^n_k(\iota_k^n(x)) = x$. If $g \in \Lambda^k \setminus \Lambda^{k-1}$, then $\lambda(g) = p^{k+1}_1 \circ L(g)$, i.e.,
\[\begin{xy}\xymatrix{
\lambda(g): &\Lambda \ar@/^1pc/[rrrr]^{L(g)}& \ar@/^/[l]^{p_1^2}\Lambda^2& \ar@/^/[l]^{p_2^3}\Lambda^3&  \ar@/^/[l]^{p_3^4}
 \dots & \ar@/^/[l]^{p_k^{k+1}} \Lambda^{k+1}
}\end{xy}.
\]
We deduce that $\lambda(g) : \Lambda \to \Lambda$ is a $(1, 2k\delta)$-quasi-isometric embedding. Moreover, if $g \in \Lambda^k$, $h \in \Lambda^l$ and $x \in \Lambda$, then we have
\begin{eqnarray*}
d(\lambda({gh})(x), \lambda(g)\lambda(h)(x)) 
&\leq& 4(k+l)\delta.
\end{eqnarray*}
Therefore, if $k=l$ we get $d(\lambda({gh})(x), \lambda(g)\lambda(h)(x)) \leq 8k\delta$. Since $\lambda(e) = {\rm Id}_\Lambda$, note that for all $x \in X$ and $g \in \Lambda^k$ we have
\[
d(\lambda(g)\lambda(g^{-1})(x), x) \leq 8k\delta \qand d(\lambda(g^{-1})\lambda(g)(x), x) \leq 8k\delta.
\]
In particular, $\lambda(g)(\Lambda)$ is $8k\delta$-relatively dense in $\Lambda$, hence $\lambda(g)$ is a $(1, 2k\delta, 8k\delta)$-quasi-isometry with quasi-inverse $\lambda(g^{-1})$. We have thus established:
\begin{proposition}\label{GeneralLeftRegularQA} For every left-admissible metric $d$ on $\Lambda^\infty$, the map  $\lambda: (\Lambda, \Lambda^\infty) \to \widetilde{{\rm QI}}(\Lambda, d|_{\Lambda \times \Lambda})$, $g \mapsto \lambda(g)$ is a cobounded quasi-isometric quasi-action, in fact a cobounded $(2k\delta, 8k\delta)$-rough quasi-action. 
\end{proposition}
At this point the arguments for the external and internal QI type start to diverge. Let us first deal with the external case. Thus assume that $\Lambda^\infty$ is finitely-generated and fix a finite-generating set $S$ of $\Lambda^\infty$ and associated word metric $d_S$. From Proposition \ref{GeneralLeftRegularQA} we then obtain a cobounded rough quasi-action
\begin{equation}\label{ExternalQA}
\lambda: (\Lambda, \Lambda^\infty) \to \widetilde{{\rm QI}}(\Lambda, d_S|_{\Lambda \times \Lambda}), \quad g \mapsto \lambda(g).
\end{equation}
We recall that this quasi-action depends on a choice of projections $(p_k^{k+1})_{k \geq 1}$.
\begin{lemma}  \label{ExternalQA1}
\begin{enumerate}[(i)]
\item For any choice of finite generating set $S$ and projections $(p_k^{k+1})_{k \geq 1}$, the quasi-action \eqref{ExternalQA} is proper, hence geometric.
\item For any two choices the resulting quasi-actions are quasi-conjugate.
\end{enumerate}
\end{lemma}
\begin{proof}  (i) Denote by $B(o, R)$ the ball of radius $R$ around $o$ in $\Lambda^\infty$ and set $B_\Lambda(o, R) \coloneqq  B(o, R)\cap \Lambda$.
Assume that $\lambda(g).B_\Lambda(o, R) \cap B_\Lambda(o, R) \neq \emptyset$ for some $g \in \Lambda$. Then there exist $x, x' \in B_\Lambda(o, R)$ such that $\lambda(g) x = x'$. Set $x'' \coloneqq  gx$. Then $d(x'', \lambda(g)x) \leq 2\delta$, hence $x, x'' \in B(o, R + 2\delta)$. We deduce that $g.B(o, R) \cap B(o, R) \neq \emptyset$, which by properness of the action of $\Lambda^\infty$ on $(\Lambda^\infty, d_S)$ confines $g$ to a finite subset of $\Lambda^\infty$.

(ii) We first fix a generating set $S$ and define $\delta$ as above. By construction we then have $d(\lambda(g)(h), gh)$
$< 2k\delta$ for all $g \in \Lambda^\infty$ and $h \in \Lambda$, i.e., we have $d(\lambda(g)(h), gh) < 2k\delta$ independently of the choice of projections $p_k^{k+1}$. We deduce that any two left-regular quasi-actions with respect to the same generating set are equivalent.

Now let $S$ and $S'$ be two different finite generating sets for $\Lambda^\infty$. We can then find a quasi-isometry $\varphi: (\Lambda^\infty, d_S) \to (\Lambda^\infty, d_{S'})$ with quasi-inverse $\bar \varphi$. Then quasi-conjugation by $(\varphi, \varphi')$ turns left-regular quasi-actions with respect to $S$ into left-regular quasi-actions with respect to $S'$, where the corresponding projections are also quasi-conjugate. This finishes the proof.
\end{proof}
\begin{remark}[Uniqueness of external left-regular quasi-action]
Note that if $(X, d)$ is an arbitrary representative of $[\Lambda]_{\rm ext}$, then it is quasi-isometric to $(\Lambda, d)$ for some (hence any) external metric $d$ on $\Lambda$, hence we can quasi-conjugate the geometric quasi-action on $(\Lambda, d)$ defined above to a geometric quasi-action on $X$. We refer to any quasi-action obtained in this way as an \emph{external left-regular quasi-action}\index{quasi-action!external left-regular}. By Lemma \ref{ExternalQA1} the external left-regular quasi-action is unique up to quasi-conjugacy.
\end{remark}
The previous construction shows that if $(\Lambda, \Lambda^\infty)$ is algebraically finitely-ge\-ne\-ra\-ted, then it admits a geometric quasi-action on any chosen representative of $[\Lambda]_{\rm ext}$. We now want to establish a similar result for representatives of the internal QI type. This is more subtle and requires us to appeal to the uniform version of Gromov's trivial lemma (Proposition \ref{GromovUniform}). 

Thus let $(\Lambda, \Lambda^\infty)$ be a geometrically finitely-generated approximate group and let $(X,d)$ be a representative of $[\Lambda]_{\mathrm{int}}$. Also, let $\widehat{d}$ be the restriction of a left-admissible metric on $\Lambda^\infty$ to $\Lambda$. By Proposition \ref{GeneralLeftRegularQA} we have a rough quasi-action
\begin{equation}\label{ExternalQA2}
\widehat{\lambda}: (\Lambda, \Lambda^\infty) \to \widetilde{{\rm QI}}(\Lambda, \widehat{d}), \quad g \mapsto \lambda(g),
\end{equation}
and by the choice of $(X,d)$ we have mutually coarsely inverse coarse equivalences
\[
\psi: (\Lambda, \widehat{d}) \to (X, d) \quad \text{and} \quad \overline{\psi}:(X, d) \to  (\Lambda, \widehat{d}).
\]
Now if $f \in \widetilde{{\rm QI}}(\Lambda, \widehat{d})$, then $\psi \circ f \circ \overline{\psi}$ is a coarse self-equivalence of $(X, d)$. Since $(X, d)$ is large-scale geodesic, it is thus a quasi-isometry by the symmetric version of Gromov's trivial lemma (Corollary \ref{GromovTrivialSymmetric}). We thus obtain a well-defined map
\[
\psi_*: \widetilde{{\rm QI}}(\Lambda, \widehat{d}) \to \widetilde{{\rm QI}}(X, d), \quad f \mapsto \psi \circ f \circ \overline{\psi}.
\]
Moreover, by Proposition \ref{GromovUniform} this map preserves uniform subsets of quasi-isometries, i.e.\ for all $K \geq 1$, $C, C' \geq 0$ there exist $k\geq 1$, $c, c' \geq 0$ such that
\[
\psi_*\left(\widetilde{{\rm QI}}_{K, C, C'}(\Lambda, \widehat{d})\right) \subset \widetilde{{\rm QI}}_{k, c, c'}(X,d).
\]
We deduce:
\begin{proposition}\label{UQAExistence} The map $\lambda \coloneqq  \psi_* \circ \widehat{\lambda}: (\Lambda, \Lambda^\infty) \to \widetilde{{\rm QI}}(X, d)$ is a uniform quasi-action.
\end{proposition}
\begin{proof} Choose $C_n, C_n' > 0$ for all $n \in \mathbb N$ such that $\widehat{\lambda}$ is a $(1, C_n, C_n')$-qiqac. Then for every $n \in \bN$ we have $\rho(\Lambda^n) \subset \widetilde{\rm QI}_{K_n, C_n, C_n'}(\Lambda, \widehat{d}|_{\Lambda \times \Lambda})$ and hence by the previous remark we find constants $k \geq 1$ and $c_n, c_n' \geq 0$ such that $\lambda(\Lambda^n) \subset \widetilde{\rm QI}_{k, c_n, c_n'}(\Lambda, d)$. On the other hand if $\lambda_1, \lambda_2 \in \Lambda^n$, then for all $x \in \Lambda$ we have
\[
\widehat{d}(\widehat{\lambda}(\lambda_1)\widehat{\lambda}(\lambda_2).x, \widehat{\lambda}(\lambda_1\lambda_2).x) < C_n',
\]
and thus 
\[
{d}({\lambda}(\lambda_1){\lambda}(\lambda_2).x, {\lambda}(\lambda_1\lambda_2).x) < \Phi_+(C_n'),
\]
where $\Phi_+$ is an upper control for $\psi$. This finishes the proof.
\end{proof}
In the sequel we refer to a geometric quasi-action of $(\Lambda, \Lambda^\infty)$ which arises in this way as an \emph{internal left-regular quasi-action}\index{quasi-action!internal left-regular} of $(\Lambda, \Lambda^\infty)$ on $(X,d)$. Appealing to Proposition \ref{GromovUniform} again, one sees that any two internal left-regular quasi-actions are quasi-conjugate. Finally one deduces:
\begin{proposition} If $(\Lambda, \Lambda^\infty)$ is a geometrically finitely-generated approximate group, then every left-regular quasi-action of $(\Lambda, \Lambda^\infty)$ is geometric.
\end{proposition}
\begin{proof} We consider a left-regular quasi-action $\lambda: (\Lambda, \Lambda^\infty) \to \widetilde{{\rm QI}}(X, d)$ for some $(X,d) \in [\Lambda]$ with notation as above. Coboundedness of $\lambda$ is immediate from the fact that $\psi$ is coarsely onto. As for properness, one first observes as in the external case that $\widehat{\lambda}$ is proper. Since coarse equivalences are proper maps, one deduces that $\lambda$ is also proper.
\end{proof}
To summarize:
\begin{proposition}\label{ExUniformQiqac} If $(\Lambda, \Lambda^\infty)$ is a geometrically finitely-generated approximate group and $X \in [\Lambda]_{\mathrm{int}}$, then $(\Lambda, \Lambda^\infty)$ admits a uniform qiqac on $X$.
\end{proposition}
We record a number of useful consequences of Proposition \ref{ExUniformQiqac}. With Proposition \ref{QuasiCobounded} and Theorem \ref{GFG} we obtain:
\begin{corollary}\label{suqc} If $(\Lambda, \Lambda^\infty)$ is a geometrically finitely-generated approximate group, then every $X \in [\Lambda]_{\mathrm{int}}$ is a semi-uniformly quasi-cobounded coarsely connected space.
\end{corollary}
Moreover, we can choose $X$ to be a generalized Cayley graph for $(\Lambda, \Lambda^\infty)$; this then shows:
\begin{corollary}\label{EGeometricQA} If $(\Lambda, \Lambda^\infty)$ is a geometrically finitely-generated approximate group, then it admits a uniform geometric qiqac on a quasi-cobounded proper metric space, and in fact on a quasi-cobounded locally finite graph.
\end{corollary}

We conclude this section with a few remarks concerning products of balls with respect to external and internal metrics, which will be used in the proof of the polynomial growth theorem. For external metrics, the results are obvious:
\begin{remark} Let $(\Lambda, \Lambda^\infty)$ be an algebraically finitely-generated approximate group. We fix a left-admissible metric $d$ on $\Lambda^\infty$ and given $j \in \mathbb N$ and $r > 0$ we define
\begin{equation}\label{ExternalBalls}
B_r^{(j)} \coloneqq  \{\lambda \in \Lambda^j \mid d(\lambda, e) < r\} \qand B_r \coloneqq  B_r^{(1)} \subset \Lambda.
\end{equation}
Now if $\lambda_1 \in B_{r_1}^{(j_1)}$ and $\lambda_2 \in B_{r_2}^{(j_2)}$, then
\[
d(\lambda_1\lambda_2, e) \leq d(\lambda_1\lambda_2, \lambda_1) + d(\lambda_1, e) = d(\lambda_2, e) + d(\lambda_1, e) \leq r_1 + r_2
\]
and since $\lambda_1\lambda_2 \in \Lambda^{j_1+j_2}$ we obtain
\begin{equation}
B_{r_1}^{(j_1)}B_{r_2}^{(j_2)} \subset B_{r_1+r_2}^{(j_1+j_2)}, \quad \text{in particular} \quad B_{r}^k = (B_r^{(1)})^k \subset B^{(k)}_{kr} \; \text{for all }k \in \mathbb N, r>0.
\end{equation}
\end{remark}
For balls with respect to internal metrics a slightly weaker conclusion still holds:
\begin{proposition} Let $N \in \mathbb N$ and let $(\Lambda, \Lambda^\infty)$ be a geometrically finitely-generated approximate group. Moreover, let $d$ be an internal metric on $\Lambda^N$ and given $r > 0$ and $j \in \{1, \dots, N\}$ let $B_r^{(j)}$ and $B_r$ be as in \eqref{ExternalBalls}. Then there exist constants $C\geq 1$, $K \geq 0$ such that for all $r_1, r_2>0$ and $j_1, j_2 \in \mathbb N$ with $j_1+j_2 \leq N$ we have
\[
B_{r_1}^{(j_1)}B_{r_2}^{(j_2)} \subset B_{C(r_1+r_2) + K}^{(j_1+j_2)}.
\]
In particular, there exist $C' \geq 1$ and $K' \geq 0$ such that for all $r>0$,
\[
B_r^N =(B_r^{(1)})^N \subset B_{C'r+K'}^{(N)}.
\]
\end{proposition}
\begin{proof} The key observation, which is contained in the proof of Proposition \ref{UQAExistence}, is that there exist constants $C \geq 1$ and $K \geq 0$ such that  for all $j_1, j_2 \in \mathbb N$ with $j_1+j_2 \leq N$ and for all $\lambda \in \Lambda^{j_1}$ the map 
\[
\Lambda^{j_2} \to \Lambda^{j_1+j_2}, \quad x \mapsto \lambda x
\]
is a $(C,K)$-quasi-isometric embedding. This then implies that for $\lambda_1 \in B_{r_1}^{(j_1)}$ and $\lambda_2 \in B_{r_2}^{(j_2)}$, we have
\begin{eqnarray*}
d(\lambda_1\lambda_2, e) &\leq& d(\lambda_1\lambda_2, \lambda_1) + d(\lambda_1, e) = C d(\lambda_2, e)+K + d(\lambda_1, e)\\
&\leq& r_1 + Cr_2 + K \leq C(r_1+r_2) + K.
\end{eqnarray*}
The first statement then follows as in the external case, and the second statement follows by induction.
\end{proof}
This implies in particular the following small tripling result for internal balls:
\begin{corollary}\label{TriplingInternalBalls} Let $(\Lambda, \Lambda^\infty)$ be a geometrically finitely-generated approximate group and let $d$ be an internal metric on $\Lambda$ which is the restriction of an internal metric from $\Lambda^3$. Denote by $B_r \coloneqq  B(e,r) \subset \Lambda$ the $r$-ball around $e$ with respect to $d$. Then there exist $C, C' \geq 1$ and $K \geq 0$ such that for all $r > 0$,
\[
|B_r^3| \leq C' |B_{Cr+K}|.
\]
\end{corollary}
\begin{proof} 

By the previous proposition we find constants $C_0, K_0$ such that $B_r^3 \subset B^{(3)}_{C_0r + K_0}$. It thus remains to show only that there exists a finite set $F \subset \Lambda^\infty$ such that for all $r>0$ we have $B^{(3)}_r \subset B_{C_1r+K_1}F$. We will choose $F \subset \Lambda^\infty$ such that $\Lambda^3 \subset \Lambda F$; then for every $g \in B^{(3)}_r$ we find $\lambda \in \Lambda$ and $f \in F$ such that $g = \lambda f$ and hence if we set $K_1 \coloneqq  \max\{d(f^{-1}, e) \ | f \in F\}$, then
\[
d(\lambda, e) = d(gf^{-1}, e) \leq d(gf^{-1}, g) +  d(g,e) = r+d(f^{-1}, e) \leq r + K_1,
\]
which shows that $B^{(3)}_r \subset B_{r+K_1}F$ and finishes the proof.
\end{proof}

\section{The Milnor--Schwarz lemma and its variants}\label{SecMilnorSchwarz}
 The classical Milnor--Schwarz lemma (Lemma \ref{MSClassic}) says that orbit maps of geometric actions of finitely-genera\-ted groups on proper large-scale geodesic spaces are quasi-isometries. In fact, finite generation of the group in question does not have to be assumed but follows automatically from the existence of such a geometric action. The conclusions of the lemma remain valid if one replaces actions by quasi-actions and geodesicity by large-scale geodesicity; these generalizations are straight-forward. If one weakens the geodesicity assumption further to coarse connectedness, then one still obtains the weaker conclusion that the orbit maps are coarse equivalences. All of these statements remain correct in the wider setting of approximate groups, except for one subtle point: The existence of a geometric qiqac on a proper geodesic metric space only implies \emph{algebraic} finite generation, whereas the other parts of the lemma work with \emph{geometric} finite generation. Taking this subtlety into account, one arrives at Theorem \ref{MSIntro} from the introduction, which we restate here for the convenience of the reader:
\begin{theorem}[Approximate Milnor--Schwarz lemma]\label{MSMain}
For every approximate group $(\Lambda, \Lambda^\infty)$ the following implications hold:
\begin{enumerate}[(i)]
\item If $(\Lambda, \Lambda^\infty)$ admits a geometric qiqac on a coarsely connected proper metric space $X$, then it is algebraically finitely-generated.
\item Conversely, if $(\Lambda, \Lambda^\infty)$ is geometrically finitely-generated, then it admits a geometric qiqac on a proper geodesic metric space $X$, and any such space $X$ is an apogee for $(\Lambda, \Lambda^\infty)$.
\end{enumerate}
\end{theorem}
Part (i) was already established (in a slightly stronger form) in Proposition \ref{ExUniformQiqac}. Slightly stronger forms of Parts (ii) and (iii) are established as Parts (ii) and (vii) of the following theorem, which is the main result of this section:
\begin{theorem}[General Milnor--Schwarz lemma for approximate groups]\label{ThmMS}
Let $(\Lambda, \Lambda^\infty)$ be an approximate group, $(X,d)$ be a coarsely-connected proper pseudo-metric space and let $\rho:(\Lambda, \Lambda^\infty) \to \widetilde{\rm QI}(X)$ be a $(K_k, C_k, C_k')$-qiqac of $(\Lambda, \Lambda^\infty)$ on $X$. Let $x_0 \in X$ be a basepoint and denote by $\iota: \Lambda \to X$ the associated orbit map given by $\lambda \mapsto \rho(\lambda).x_0$.
\begin{enumerate}[(i)]
\item If $\rho$ is cobounded, then there exists $R > 0$ such that $F \coloneqq   \{\mu \in \Lambda^2 \mid B_R(x_0) \cap \rho(\mu).B_R(x_0) \neq \emptyset\}$ generates $\Lambda^\infty$.
\end{enumerate}
If moreover $\rho$ is geometric, then the following hold.
\begin{enumerate}[(i)]
  \setcounter{enumi}{1}
\item $(\Lambda, \Lambda^\infty)$ is algebraically finitely-generated.
\item The orbit map  $\iota: (\Lambda, d_F|_{\Lambda \times \Lambda}) \to (X, d)$, $\lambda \mapsto \rho(\lambda).x_0$ is a coarse equivalence for some (hence any) finite generating set $F$ of $\Lambda^\infty$.
\item If all of the sequences $(K_k)$, $(C_k)$ and $(C_k')$ are bounded, then the upper control of $\iota$ can be chosen to be affine linear, i.e.\ there exist constants $K\geq 1$ and $C \geq 0$ such that $d(\iota(\mu), \iota(\lambda)) \leq K d_F(\mu, \lambda) + C$. In particular, this happens if $\rho$ is an isometric action or $\Lambda = \Lambda^{\infty}$ is a group.
\end{enumerate}
If moreover $X$ is large-scale geodesic, then the following hold.
\begin{enumerate}[(i)]
  \setcounter{enumi}{4}
\item The lower control of $\iota$ can be chosen to be affine linear, i.e.\ there exist constants $K\geq 1$ and $C \geq 0$ such that $d(\iota(\mu), \iota(\lambda)) \geq K^{-1} d_F(\mu, \lambda) - C$.
\item If  all of the sequences $(K_k)$, $(C_k)$ and $(C_k')$ are bounded, then $\iota: (\Lambda, d_F) \to (X, d)$ is a quasi-isometry and hence $(X,d)$ represents $[\Lambda]_{\rm ext}$. In particular, this happens if $\rho$ is an isometric action or $\Lambda = \Lambda^{\infty}$ is a group.
\end{enumerate}
If moreover $(\Lambda, \Lambda^\infty)$ is geometrically finitely-generated, then
\begin{enumerate}[(i)]
  \setcounter{enumi}{6}
\item $(X, d)$ represents $[\Lambda]_{\mathrm{int}}$.
\end{enumerate}
In particular, if a geometrically finitely-generated approximate group quasi-acts geometrically by quasi-isometries on a large-scale geodesic proper pseudo-metric space, then this space represents its internal and external QI type.
\end{theorem}
\begin{proof}

Throughout the proof we will assume without loss of generality that $(K_k)$ and $(C_k)$ are increasing sequences and that $C_k'<C_k$. We also choose $(D_k)$ as in Lemma \ref{Inverses}. 

(i) Fix a basepoint $x_0 \in X$ and a constant $C \geq 0 $ such that $X$ is $C$-coarsely connected and the quasi-orbit $\mathcal O \coloneqq  \rho(\Lambda).x_0$ is $C$-relatively dense in $X$. Then for every $\lambda \in \Lambda$ there exists a $C$-path from $x_0$ to $x_n = \rho(\lambda).x_0$, i.e., there exist points $x_0, x_1, \dots, x_n \in X$ such that $d(x_i, x_{i+1}) < C$. Since $\mathcal O$ is $C$-relatively dense, there exist $\lambda_0, \dots, \lambda_{n} \in \Lambda$ such that $\lambda_0 = e$, $\lambda_n = \lambda$ and $d( \rho(\lambda_i).x_0, x_i) < 2C$. We then have
\[d( \rho(\lambda_i).x_0,  \rho(\lambda_{i+1}).x_0) < 3C,\]
hence there exists $R > 0$ (depending on $C$ and the implied quasi-isometry constants) such that 
\[
d(\rho(\lambda_i^{-1}\lambda_{i+1}).x_0, x_0) < R
\]
and thus $\{\lambda_0^{-1}\lambda_1, \dots, \lambda_{n-1}^{-1}\lambda_n\} \subset F$. On the other hand,
\[
\lambda = \lambda_0^{-1}\lambda_n = (\lambda_0^{-1}\lambda_1) (\lambda_1^{-1}\lambda_2) \cdots (\lambda_{n-1}^{-1}\lambda_n),
\]
and since $\lambda \in \Lambda$ was arbitrary we conclude that $\langle F \rangle = \langle \Lambda \rangle = \Lambda^\infty$.

(ii) This follows from (i) and Remark \ref{RemarkProperness}.

(iii), (iv) Since the statement does not depend on the choice of generating system, we may assume that $F \coloneqq   \{\mu \in \Lambda^2 \mid B_R(x_0) \cap \rho(\mu).B_R(x_0) \neq \emptyset\}$ where $R$ is chosen as in (i). (Note again that $F$ is finite by Remark \ref{RemarkProperness}.) Since $\rho$ is cobounded, the orbit map $\iota$ has relatively dense image. It thus remains to find upper and lower controls for $\iota$, and to show that the upper control can be chosen affine linear if $(K_n)$ and $(C_n)$ (and hence $(C_n')$ by our convention $C_n' < C_n$) are bounded.

Since $F$ is finite, the maximum 
\[
C_F \coloneqq  \max_{f \in F} d(x_0, \rho(f)x_0)
\]
exists, and we define an upper control $\Phi_+$ by $\Phi_+(n) \coloneqq  K_1( K_{2n}C_F + 2C_{2n})n + (K_1+1)C_1 + D_1$ and note that this upper control is affine linear if and only if $(K_{n})$ and $(C_{n})$ are bounded. Now denote by $d_{\rho}$ the pseudo-metric on $\Lambda^\infty$ given by
\[
d_{\rho}(\mu, \lambda) \coloneqq  d(\rho(\mu).x_0, \rho(\lambda).x_0) = d(\iota(\mu), \iota(\lambda)),
\]
and observe that for $\mu, \lambda \in \Lambda$ we have
\begin{eqnarray*}
d_{\rho}(\mu, \lambda) &=& d(\rho(\mu).x_0, \rho(\lambda).x_0) \leq K_1 d(x_0, \rho(\mu)^{-1}\rho(\lambda).x_0) + C_1\\ &\leq&  K_1 d(x_0, \rho(\mu^{-1})\rho(\lambda).x_0) + C_1 + D_1 \\
&\leq&  K_1 d(x_0, \rho(\mu^{-1}\lambda).x_0) + K_1C_1 + C_1 + D_1.
\end{eqnarray*}
Now if $\mu, \lambda \in \Lambda$ with $d_F(\mu, \lambda) = n$, then there exist $f_1, \dots, f_n \in F$ such that  $\mu^{-1}\lambda = f_1 \cdots f_n$, and hence
\begin{eqnarray*}
d_{\rho}(\mu, \lambda) &\leq&  K_1 d(x_0, \rho(\mu^{-1}\lambda).x_0) + (K_1+1)C_1 + D_1 \\
&=&  K_1 d(x_0, \rho(f_1 \cdots f_n).x_0) + (K_1+1)C_1 + D_1\\
&\leq& K_1 \left(\sum_{k=0}^{n-1} d(\rho(f_1 \cdots f_{k}).x_0, \rho(f_1 \cdots f_{k+1}).x_0) \right)+ (K_1+1)C_1 + D_1,
\end{eqnarray*}
where we will take that, for $k=0$, the expression $\rho(f_1\ldots f_k).x_0$ stands for (just) $x_0$.
So let $0 \leq k \leq n-1$. Since $F \subset \Lambda^2$, we have $f_1 \cdots f_k \in \Lambda^{2k}$, leading to the estimate
\begin{eqnarray*}
&&d(\rho(f_1 \cdots f_k).x_0, \rho(f_1 \cdots f_{k+1}).x_0)\\ &\leq& d(\rho(f_1 \cdots f_k).x_0, \rho(f_1 \cdots f_{k})\rho(f_{k+1}).x_0) + d(\rho(f_1 \cdots f_{k})\rho(f_{k+1}).x_0,  \rho(f_1 \cdots f_{k+1}).x_0)\\
&\leq& (K_{2k}d(x_0, \rho(f_{k+1}).x_0) + C_{2k}) + C_{2k} \leq K_{2k}C_F + 2C_{2k} \leq  K_{2n}C_F + 2C_{2n}.
\end{eqnarray*}
Plugging this into our previous estimate for $d_{\rho}(\mu, \lambda)$ we deduce that
\begin{eqnarray*}
d_{\rho}(\mu, \lambda) &\leq& K_1 \left(\sum_{k=0}^{n-1}K_{2n}C_F + 2C_{2n}  \right)+ (K_1+1)C_1 + D_1\\
&= &K_1( K_{2n}C_F + 2C_{2n})n + (K_1+1)C_1 + D_1 \quad = \quad \Phi_+(n).
\end{eqnarray*}
Using that $n = d_F(\mu, \lambda)$ and unraveling the definition of $d_\rho$ we thus obtain
\begin{equation} d(\iota(\mu), \iota(\lambda)) \leq \Phi_+(d_F(\mu, \lambda)).\end{equation}

To obtain a lower control we introduce the following notation: Given $n \in \mathbb N$ we define
\[
B_n \coloneqq  \{\lambda \in \Lambda \mid d_\rho(e, \lambda) < n\}.
\]
By properness of the quasi-action, each of the sets $B_n$ is finite. We deduce that
\[ B_1 \subset B_2 \subset \dots \quad \text{and} \quad F \cap \Lambda \subset F^2 \cap \Lambda \subset F^3 \cap \Lambda \subset \dots \]
 are finite exhaustions of $\Lambda$. In particular, for every $n \in \mathbb N$ we can find $m \in \mathbb N$ such that $B_{n+1} \subset F^m \cap \Lambda$. In fact, we can
 find a continuous strictly increasing function $\sigma : [0, \infty) \to [0, \infty)$ which satisfies $\sigma(t) \to \infty$ as $t \to \infty$, $\sigma(\mathbb N) = \mathbb N$ and 
  \[
 B_{n+1} \subset F^{\sigma(n)} \cap \Lambda
 \]
 for all $n \in \mathbb N$. Observe that such a function is necessarily invertible, and that its inverse $\Phi_- \coloneqq  \sigma^{-1}$ is a lower control function. Now let $\mu, \lambda \in \Lambda$ and assume that $n \leq d_\rho(\mu, \lambda) < n+1$. Then
\begin{eqnarray*}
& & d_\rho(e ,\mu^{-1}\lambda) < n+1 \Rightarrow\mu^{-1}\lambda \in B_{n+1} \subset F^{\sigma(n)} \\
&\Rightarrow& d_F(\mu, \lambda) = d_F(e, \mu^{-1}\lambda) \leq \sigma(n) \leq \sigma(d_\rho(\mu, \lambda)) \leq \sigma(d(\iota(\mu), \iota(\lambda)))\\
&\Rightarrow& \Phi_-(d_F(\mu, \lambda)) \leq (\Phi_- \circ \sigma)(d(\iota(\mu), \iota(\lambda))) = d(\iota(\mu), \iota(\lambda)).
\end{eqnarray*}
This shows that $\Phi_-$ is a lower control for $\iota$ and finishes the proof.

(v) This follows from (iii) and Lemma \ref{GromovTrivial}.

(vi) This follows from (iv) and (v).

(vii) This follows from (v) and another application of Lemma \ref{GromovTrivial}.
\end{proof}
\begin{remark} Assume that $(\Lambda, \Lambda^\infty)$ is distorted and let $d$ be an external metric on $\Lambda$. Then $(\Lambda, \Lambda^\infty)$ quasi-acts geometrically by quasi-isometries on $(\Lambda, d)$, but by assumption $(\Lambda, d)$ does not represent $[\Lambda]_{\mathrm{int}}$. This shows in particular that (vii) does not hold in general without the assumption that $(X,d)$ be large-scale geodesic.
\end{remark}
We now derive a number of consequences of Theorem \ref{ThmMS}. Firstly, we obtain the following implication which was already mentioned several times before:
\begin{corollary}[Geometric finite generation implies algebraic finite generation]\label{GFGAFG} If $(\Lambda, \Lambda^\infty)$ is a geometrically finitely-generated group, then it is also algebraically finitely-generated.
\end{corollary}
\begin{proof} By Corollary \ref{EGeometricQA}, $(\Lambda, \Lambda^\infty)$ admits a geometric qiqac on a proper geodesic metric space. The corollary then follows from Theorem \ref{ThmMS}.(ii).
\end{proof}

Secondly, we obtain a new geometric criterion for undistortedness:
\begin{corollary}[Undistortedness from isometric actions] \label{cor:undistort} 
Let $(\Lambda, \Lambda^\infty)$ be a geometrically finitely-generated approximate group which acts geometrically by isometries on a large-scale geodesic proper pseudo-metric space $(X,d)$. Then $(\Lambda, \Lambda^\infty)$ is algebraically finitely-generated and undistorted.
\end{corollary}
\begin{proof} Finite generation follows from Theorem \ref{ThmMS}.(ii). On the other hand, by Theorem \ref{ThmMS}.(vi) and (vii) we have
\[
[\Lambda]_{\rm ext} = [(X, d)] = [\Lambda]_{\mathrm{int}},
\]
which shows that $(\Lambda, \Lambda^\infty)$ is undistorted.
\end{proof}
Thirdly, we can now finish the proof of Theorem \ref{ThmIsometricActions}:
\begin{proof}[Proof of Theorem \ref{ThmIsometricActions}, concluded] If $(\Lambda, \Lambda^\infty)$ is as in Theorem \ref{ThmIsometricActions}, then by Theorem \ref{ThmMS}.(ii) the group $\Lambda^\infty$ is finitely-generated. Moreover, if $S$ is any finite generating set of $\Lambda^\infty$, then by Theorem \ref{ThmMS}.(vi) there is a quasi-isometry between $(X, d)$ and $(\Lambda, d_S|_{\Lambda \times \Lambda})$, and consequently $(\Lambda, d_S|_{\Lambda \times \Lambda})$ is large-scale geodesic. This shows that $(\Lambda, \Lambda^\infty)$ is geometrically finitely-generated. Moreover, some (hence any) external metric on $\Lambda$ is an internal metric, and thus $(\Lambda, \Lambda^\infty)$ is undistorted. Since the equivalence (i)$\iff$(ii) was already established in Section \ref{SecIsometric}, this finishes the proof.
\end{proof}
Finally, it turns out -- somewhat surprisingly -- that under some mild hypothesis every geometric qiqac can be replaced by a \emph{uniform} geometric qiqac on the same space. This is a consequence of the implication (iii)$\Rightarrow$(ii) of the following corollary.
\begin{corollary}[Making geometric quasi-actions uniform]\label{MSConverse} Let $(\Lambda, \Lambda^\infty)$ be a geometrically finitely-generated approximate group and let $(X, d)$ be a large-scale geodesic proper pseudo-metric space. Then the following are equivalent:
\begin{enumerate}[(i)]
\item $(X,d)$ represents the internal QI type of $(\Lambda, \Lambda^\infty)$, i.e.\ $[(X, d)] = [\Lambda]_{\mathrm{int}}$.
\item There exists a geometric uniform qiqac of $(\Lambda, \Lambda^\infty)$ on $X$.
\item There exists a geometric qiqac of $(\Lambda, \Lambda^\infty)$ on $X$.
\end{enumerate}
\end{corollary}
\begin{proof} The implication (i) $\Rightarrow$ (ii) was established in Proposition \ref{ExUniformQiqac}, the implication  (ii) $\Rightarrow$ (iii) is obvious and (iii)$\Rightarrow$(i) is part of the Milnor--Schwarz lemma (Theorem \ref{ThmMS}).
\end{proof}
A closer inspection of the proof actually shows that every geometric qiqac of $(\Lambda, \Lambda^\infty)$ on $X$ is quasi-conjugate to a uniform one.

\section{Application to polynomial growth}\label{SubsecPGT1}
We now discuss the proof of Theorem \ref{PGT}. As mentioned earlier, the lion's share of the proof is given by the argument of Hrushovski and its refinement due to Breuillard-Green-Tao. Our task here is twofold: Firstly, we need to translate Theorem \ref{PGT} into the correct language to apply the results from \cite{BGT}. Secondly, we need to deduce from the small tripling property of internal balls (Corollary \ref{TriplingInternalBalls}) that symmetrizations of such balls are $k$-approximate groups for a suitable $k$ and a carefully chosen sequence of radii; this argument is basically the same as in the group case, once the small tripling property has been established.

Concerning the easy implications of Theorem \ref{PGT} we note that the equivalences (i)$\Leftrightarrow$(ii)$\Leftrightarrow$(iii) are immediate from Lemma \ref{CanoicalQITypeGood} and Proposition \ref{QIGrowth}, and that (v)$\Rightarrow$(vi)$\Rightarrow$(iii) hold by Corollary \ref{GrowthImpl}. Finally, (iv)$\Rightarrow$(v) holds by the easy direction of Gromov's polynomial growth theorem. The non-obvious implication is thus the implication (i)$\Rightarrow$(iv), which seems to require some version of the Hrushovski-Breuillard-Green-Tao machine. The precise statement we will use is as follows:
\begin{theorem}\label{BGTConvenient}
Let $(\Lambda, \Lambda^\infty)$ be a geometrically-finitely-generated approximate group. Assume that there exist $k \in \bN$ and finite $k$-approximate subgroups $\Lambda_1 \subset \Lambda_2 \subset \dots \subset \Lambda$ such that $\Lambda = \bigcup \Lambda_n$. Then there exists a finite index approximate subgroup $\Lambda' \subset \Lambda^{16}$ which generates a nilpotent group.
\end{theorem}
In the case where $\Lambda$ is an actual group, this result was established by Hrushovski in \cite[Thm.\ 7.1]{Hrushovski}. The approximate version, however, requires the additional machinery developed by Breuillard, Green and Tao in \cite{BGT}. We explain how to derive it from results in \cite{BGT} in Appendix \ref{AppendixBGT}. 

In order to deduce Theorem \ref{PGT} from Theorem \ref{BGTConvenient} we need to construct a sequence $(\Lambda_n)$ as in the theorem in any given approximate group of polynomial growth. For this we are going to use the following result, which is a weak form of \cite[Cor. 3.11]{Tao}:
\begin{lemma}[Tao's quantitative tripling lemma]\label{TaoTaoTao} Let $C>0$ and let $(A_n)_{n \in \mathbb N}$ be a sequence of finite subsets of a group $G$ such that $|A_n^3| \leq C |A_n|$ holds for all $n\in \mathbb N$. Then there exists $k \in \mathbb N$ (depending on $C$) such that all of the sets $\Lambda_n \coloneqq  A_n \cup \{e\} \cup A_n^{-1}$ are $k$-approximate subgroups.
\end{lemma}
In order to establish (the remaining implication of) Theorem \ref{PGT} it is thus enough to find a family of subsets $(A_n)$ which exhaust $\Lambda$ and satisfy $|A_n^3| \leq C |A_n|$ for a uniform constant $C > 0$. We will construct $(A_n)$ as a sequence of balls in $\Lambda$ of carefully chosen radii with respect to an internal metric. More precisely, 
let $(\Lambda, \Lambda^\infty)$ be an approximate group of internal polynomial growth and let $d$ be an internal metric on $\Lambda$, which we assume to be the restriction of an internal metric on $\Lambda^3$. As before we then set, for any $r>0$,
\[
B_r \coloneqq  B(e, r) = \{x \in \Lambda \mid d(x,e) < r\}.
\]
By Corollary \ref{CoarseBoundedGeometry} we can then find a constant $C_0 \geq 1$ and $d \in \mathbb N$ such that
\begin{equation}\label{Pigeon1}
|B_r| \leq C_0 r^d.
\end{equation}
On the other hand, by Corollary \ref{TriplingInternalBalls} we can find constants $C_1, C_2, \geq 1$ and $K \geq 0$ such that
\begin{equation}\label{Pigeon2}
|B_r^3| \leq C_2 \cdot |B_{C_1r+K}|. 
\end{equation}
Now from \eqref{Pigeon1} and \eqref{Pigeon2} we can deduce:
\begin{proposition} There exist $C \geq 1$ and an unbounded increasing sequence $(r_n)_{n \in \mathbb N}$ of positive real numbers such that for all $n  \in \mathbb N$,
\[
|B_{r_n}^3| \leq C |B_{r_n}|.
\]
\end{proposition}
\begin{proof} In view of \eqref{Pigeon2} it suffices to show that there exists $C \geq 1$ such for every sequence $(r_n)_{n \in \mathbb N}$ as in the proposition and all $n \in \mathbb N$ we have 
\begin{equation}\label{Pigeon3}
 |B_{C_1 r_n + K}| \leq C |B_{r_n}|.
\end{equation}
Assume for contradiction that this is not the case. Then for every $C \geq 1$ and every sequence $(r_n)_{n \in \mathbb N}$ the Condition \eqref{Pigeon3} is violated for some $n \in \mathbb N$. However, this argument does not only apply to the sequence $(r_n)_{n \in \mathbb N}$, but also to any of its subsequences, and thus \eqref{Pigeon3} must be violated for almost all $n \in \mathbb N$. We thus find for every $C > 0$ and every sequence $(r_n)_{n \in \mathbb N}$ an element $n_0 \in \mathbb N$ such that
\[
|B_{C_1 r_n + K}| > C |B_{r_n}|
\]
holds for all $n \geq n_0$. We will specifically choose $C$ with $C > (2C_1)^d$ and $(r_n)_{n \in \mathbb N}$ to be the recursively defined sequence given by  $r_1 \coloneqq  1$ and $r_{n+1} \coloneqq  C_1 r_n + K$ for $n \geq 2$. This sequence then satisfies, for all $n \geq n_0$, the inequality $|B_{r_{n+1}}| = |B_{C_1 r_n + K}| > C|B_{r_n}|$, and hence we see inductively that
\[
|B_{r_{n_0+m}}| > C^m|B_{r_{n_0}}|.
\]
On the other hand, since $r_{m+n_0} = O_{K, n_0}(C_1^m)$ we have $r_{n_0+m} < (2C_1)^m$ for all sufficiently large $m$, and hence for all such $m$ we have
\[
C^m |B_{r_{n_0}}|< |B_{r_{n_0+m}}| \overset{\eqref{Pigeon1}}{\leq} C_0 r_{n_0 + m}^d < C_0 (2C_1)^{dm} \implies C < \sqrt[m]{\frac{C_0}{|B_{r_{n_0}}|}} \cdot (2C_1)^d,
\]
which for $m \to \infty$ yields $C \leq (2C_1)^d$, contradicting our choice of $C$. 
\end{proof}
Theorem \ref{PGT} now follows by choosing $A_n \coloneqq  B_{r_n}$ and applying Lemma \ref{TaoTaoTao} to the sequence $(A_n)_{n \in \mathbb N}$.

\chapter{Hyperbolic approximate groups}\label{chap: hyperbolic approx groups}
Among the early successes of geometric group theory is Gromov's geometric theory of hyperbolic groups; this has found many generalizations and still serves as a blueprint for a good geometric group theory today. Extending the definition of a hyperbolic group to approximate groups presents no difficulty. It turns out that hyperbolic approximate groups share many classical properties of hyperbolic groups, which are not shared by general hyperbolic spaces (see \cite{BuyaloSchroeder} for the group case). For example, they are always visual (Proposition \ref{PropVisual}) and have bounded growth at some scale (Theorem \ref{BoundedGrowth1}), and their Gromov boundaries are doubling (Corollary \ref{Doubling1}) and locally quasi-self-similar (Proposition \ref{LemmaQuasiSelfSim}).  In fact, all of these properties hold more generally for quasi-cobounded hyperbolic proper geodesic spaces. Moreover, the Gromov boundaries of hyperbolic approximate groups
have $0$, $2$ or infinitely many points (Proposition \ref{GromovBoundaryTri}), and are minimal as long as the approximate group in question is non-elementary in a suitable sense (Proposition \ref{BoundaryMinimal}). 

\begin{remark}[Where are the examples?]
One reason for the success of the geometric theory of hyperbolic groups is the fact that many naturally occurring groups are hyperbolic. While we discuss some constructions of hyperbolic approximate groups below, we are currently unable to construct a single example of a hyperbolic approximate group which we can show to not be quasi-isometric to a hyperbolic group. There are basically three reasons for this: Firstly, it is hard to show that a given approximate group is not quasi-isometric to a finitely generated group. An example is given in \cite[Ex.\ 2.21]{BH} based on \cite{ElekTardos2000}, but this is not hyperbolic. Secondly, from the way in which approximate subgroups are usually constructed and described (say, in Hrushovski's classification of commensurability classes of approximate subgroups of a group), it is very hard to decide whether they are hyperbolic or not. Thirdly, and most importantly, there are several rigidity results which ensure that all hyperbolic approximate subgroups of a certain form are actually quasi-isometric to groups. For example, we will see that every quasi-convex approximate subgroup of a hyperbolic group and every non-elementary hyperbolic approximate group of asymptotic dimension $1$ is quasi-isometric to a finitely-generated free group (see Corollary \ref{Rigidity0} and Theorem \ref{Asdim1Case} below).
\end{remark}
Why then do we devote a whole chapter to hyperbolic approximate groups at this point? We believe that there are at least four good reasons to develop a basic theory of hyperbolic approximate groups: 
\begin{enumerate}
\item The theory can sometimes be \emph{used} to show that certain hyperbolic approximate groups must be quasi-isometric to hyperbolic groups, and such results can be thought of as \emph{rigidity results}. For example, the proof of the fact that non-elementary hyperbolic approximate groups of asymptotic dimension $1$ are quasi-isometric to groups uses most of the machinery which we develop.
\item Some of the results about Gromov boundaries of hyperbolic approximate groups generalize to Morse boundaries of more general countable approximate groups and (as we will see) can serve as a blueprint for a more general theory. In this wider context, there do indeed exist plenty of examples.
\item Most of the results of this chapter, in particular all of the results from Section \ref{QCHypSpace}, apply not just to hyperbolic approximate groups but in fact to arbitrary quasi-cobounded proper geodesic Gromov hyperbolic spaces. While our motivation comes from the study of hyperbolic approximate groups, these results are of independent interest.
\item We still have the hope that hyperbolic approximate groups which are not quasi-isometric to groups do exist. To find these, one needs good necessary and sufficient conditions for an approximate group to be hyperbolic and for a space to be an apogee for a hyperbolic approximate group.
\end{enumerate}

\section{Hyperbolic approximate groups and their boundaries}
Hyperbolic spaces can be defined as metric spaces with some large-scale version of negative curvature. In fact, as we discuss in Appendix \ref{AppHyperbolic}, there are (at least) three subtly differing definitions of a hyperbolic space in the literature, which we refer to as \emph{Gromov hyperbolic}, \emph{Rips hyperbolic} and \emph{Morse hyperbolic} spaces respectively. For proper geodesic metric spaces, all three notions are defined, QI-invariant and coincide. We can thus refer unambiguously to \emph{hyperbolic proper geodesic metric spaces}.

In some cases (in particular, in the context of Morse boundaries) we will have to deal with proper metric spaces which are not geodesic. For such spaces, Rips hyperbolicity is no longer defined, and Gromov hyperbolicity is defined, but no longer a QI-invariant. Whenever we deal with non-geodesic metric spaces we will thus work with Morse hyperbolicity. Let us briefly recall the definition; see Definition \ref{def:rips hyp} for the notion of a $\delta$-slim triangle and Definition \ref{defn: Morse geodesic} for a definition of Morse quasi-geodesic.
\begin{definition} A metric space $(X,d)$ is called \emph{Morse hyperbolic}\index{Morse hyperbolic}\index{Morse hyperbolic!space} if it is large-scale geodesic and for all $K \geq 1$ and $C \geq 0$ there exists $\delta \geq 0$ so that every $(K, C)$-quasi-geodesic triangle is $\delta$-slim.
\end{definition}
\begin{remark}[Morse hyperbolic vs.\ Gromov hyperbolic]\label{MorseGromov} \mbox{}
\begin{enumerate}[(i)]
\item Since quasi-isometries preserve slim quasi-geodesic triangles, it is immediate from the definition that being Morse hyperbolic is a QI-invariant among metric spaces (cf. Theorem \ref{thm:strongly hyperbolic equiv}).
\item By Theorem \ref{thm:strongly hyperbolic equiv}, Morse hyperbolic spaces can also be characterized as those quasi-geodesic spaces in which every quasi-geodesic is Morse (with a Morse gauge which is uniform in the parameters of the quasi-geodesic), hence the name.
\item Not every Gromov hyperbolic space is Morse hyperbolic, since Gromov hyperbolic spaces need not be quasi-geodesic. Conversely, not every Morse hyperbolic space is Gromov hyperbolic: The graph of the function $f: \R \to \R$, $x \mapsto |x|$ with the induced metric as a subspace of $\R^2$ is quasi-isometric (in fact, bi-Lipschitz) to the real line, hence it is Morse hyperbolic. However, it is not Gromov hyperbolic (see \cite[Rem.\ 4.1.3]{BuyaloSchroeder}). 
\item A Morse hyperbolic space is Gromov hyperbolic if and only if it is quasi-ruled \cite{Mathieu}. In particular, every proper geodesic Gromov hyperbolic space is Morse hyperbolic.
\item One can characterize Morse hyperbolic spaces extrinsically as those metric spaces which are quasi-isometric to proper geodesic Gromov hyperbolic spaces (see Theorem \ref{thm:strongly hyperbolic equiv}). Thus the class of Morse hyperbolic spaces is the smallest QI-invariant class of metric spaces which contains all proper geodesic Gromov hyperbolic spaces.
\end{enumerate}
\end{remark}
Combining this remark with Corollary \ref{suqc} we deduce:
\begin{proposition}\label{HypAG} For a geometrically finitely-generated approximate group $(\Lambda, \Lambda^\infty)$ the following properties are equivalent.
\begin{enumerate}[(i)]
\item Some $(X,d) \in [\Lambda]_{\mathrm{int}}$ is Morse hyperbolic.
\item Every  $(X,d) \in [\Lambda]_{\mathrm{int}}$ is Morse hyperbolic.
\item Every apogee of $\Lambda$ is a semi-uniformly quasi-cobounded hyperbolic proper geodesic metric space.
\item Every generalized Cayley graph of $\Lambda$ is a  semi-uniformly quasi-cobounded hyperbolic proper geodesic metric graph.
\item $(\Lambda, \Lambda^\infty)$ admits a geometric qiqac on a hyperbolic proper geodesic metric space.
\end{enumerate}
\end{proposition}
\begin{definition}\label{DefHypAG} A geometrically finitely-generated approximate group $(\Lambda, \Lambda^\infty)$ satisfying the equivalent conditions of Proposition \ref{HypAG} is called a \emph{hyperbolic approximate group}.
\end{definition}
\begin{construction}[Hyperbolic uniform model sets] Let $X$ be a hyperbolic proper geodesic metric space and assume that $G \coloneqq \mathrm{Is}(X)$ acts cocompactly on $X$. If $\Gamma$ is a lattice in $G \times H$, where $H$ is some auxiliary locally compact group, $W$ is a compact symmetric identity neighborhood in $H$ and $\mathrm{pr}_G: G \times H \to G$ denotes the canonical projection, then by Example \ref{CuPModel} the subset
\[
\Lambda \coloneqq \mathrm{pr}_G(\Gamma \cap (G \times W))
\]
is a discrete approximate subgroup of $G$. Since $\Lambda$ is discrete, the inclusion $(\Lambda, \Lambda^\infty) \hookrightarrow G$ defines a proper isometric action on $X$. If $\Gamma$ is uniform, then $\Lambda$ is 
syndetic in $G$, hence the action is geometric and $(\Lambda, \Lambda^\infty)$ is a hyperbolic approximate group with apogee $X$. 

While this construction yields plenty of explicit examples of hyperbolic approximate groups, most of these are actually almost group. Notably, $(\Lambda, \Lambda^\infty)$ is always an almost group
 if $G$ is discrete (see Corollary \ref{DiscreteIsometry}) or $H$ is totally-disconnected \cite{BH2}.  In fact, up to commensurability of $\Lambda$ we may always assume that $H$ is a connected Lie group \cite[Theorem A.11]{BH2}.
\end{construction}
\begin{proposition} Assume that $X$ is a hyperbolic proper geodesic metric space with the following properties:
\begin{enumerate}
\item $G \coloneqq \mathrm{Is}(X)$ is non-discrete.
\item There exists a connected Lie group $H$ and an irreducible uniform lattice in $G \times H$.
\item $X$ is not quasi-isometric to a finitely-generated group. In particular, $G$ contains no uniform lattices.
\end{enumerate}
Then $X$ is an apogee for a hyperbolic approximate group which is not quasi-isometric to a hyperbolic group.
\end{proposition}
We do not know whether such a space $X$ exists. If it does exist, then it has to be quasi-cobounded and of asymptotic dimension $\geq 2$, which puts further strong restrictions on $X$. 
\begin{example}\label{ExampleHAG1}
Here are some examples of hyperbolic approximate groups, which are not almost groups, but are quasi-isometric to hyperbolic groups:
\begin{enumerate}[(i)]
\item If $X$ is a rank one symmetric space (i.e.\ real, complex or quaternionic hyperbolic space or the octonion plane), then starting from irreducible arithmetic lattices in products of semisimple Lie groups with rank one factors one can construct plenty of uniform model sets $\Lambda \subset G$ such that $(\Lambda, \Lambda^\infty)$ is not an almost group. However, by a classical theorem of Borel and Harish-Chandra, the group $G$ admits a uniform lattice $\Gamma$.
\item Similarly, starting from an irreducible lattice in $\mathrm{SL}_2(\Q_p) \times \mathrm{SL}_2(\R)$ or more generally from irreducible $S$-arithmetic lattices which contain a rank one factor over a non-Archimedean local field, one can construct interesting hyperbolic approximate groups in Bruhat--Tits trees. Again, these are quasi-isometric to uniform lattices in $G$.
\item If $X$ is a tree (not necessarily Bruhat-Tits), then a uniform lattice in a product of $G$ with a connected Lie group may exist, but the resulting hyperbolic approximate groups have asymptotic dimension $1$, hence are either elementary in the sense of Definition \ref{DefNonEl} below or quasi-isometric to finitely-generated groups by Theorem \ref{Asdim1Case}.
\end{enumerate}
\end{example}
 As far as other constructions of approximate groups are concerned, we can construct non-laminar examples using quasikernels of real-valued quasimorphisms, but for these it is hard to decide whether they are hyperbolic:
\begin{example}\label{ExampleHAG2} Let $f: \Gamma \to \bR$ be a real-valued quasimorphism on a finitely-generated group $\Gamma$. Then any quasikernel $\Lambda$ for $f$ is an approximate subgroup of $\Gamma$. One can decide whether $\Lambda$ is geometrically finitely-generated using the quasi-BNS-invariant of Heuer and Kielak \cite{HeuerKielak}, but it is hard to decide whether $\Lambda$ is hyperbolic.
\begin{enumerate}[(i)]
\item If $\Gamma$ happens to be a free group, then there are plenty of real-valued quasimorphisms on $\Gamma$, but if $\Lambda$ is geometrically finitely-generated, then $(\Lambda, \Lambda^\infty)$ is an almost group.
\item If $\Gamma$ is hyperbolic, then there are still plenty of real-valued quasimorphisms on $\Gamma$, but even if $\Lambda$ is geometrically finitely-generated it may not be hyperbolic. The most natural condition that ensures hyperbolicity of $\Lambda$ is quasi-convexity in $\Gamma$, but this forces $(\Lambda, \Lambda^\infty)$ to be an almost group by Corollary \ref{Rigidity0} below.
\item If $\Gamma$ is non-hyperbolic, then it is unclear how to show that $\Lambda$ is hyperbolic.
\end{enumerate}
\end{example}
\begin{example}[Examples from (quasi-)commensurators of hyperbolic groups]
If $\Lambda$ is an approximate subgroup of a group $G$ and $F$ is a finite subset of the commensurator of $\Lambda$ in $G$  (or even the quasi-commensurator in the sense of \cite[Def.\ A.8]{BH2}) such that $e \in F$, then by \cite[Prop.\ A.9]{BH2}, then $F\Lambda  \cup \Lambda F^{-1}$ is again an approximate subgroup of $G$. In particular, if $\Lambda$ is a hyperbolic group, then $F\Lambda  \cup \Lambda F^{-1}$ is a hyperbolic approximate subgroup, which may not be an almost group, but by construction it is still commensurable to $\Lambda$.
\end{example}

\begin{construction}[Gromov boundary of a hyperbolic approximate group]\label{Rem: boundary hyp approx group} Let $(\Lambda, \Lambda^\infty)$ be a hyperbolic approximate group and let $X$ be an apogee for $(\Lambda, \Lambda^\infty)$. Then $X$ is a hyperbolic proper geodesic space, and hence we can define a Gromov compactification and corresponding Gromov boundary $\partial X$, which is defined up to homeomorphism (cf.\ Section \ref{SecGromovBoundaries}). Since $X$ is geodesic, we can use either the sequential model $\partial_s X$ of $\partial X$, in which boundary points are equivalence classes of sequences that converge to infinity, or the ray model $\partial_r X$, in which boundary points are equivalence classes of geodesic rays up to bounded Hausdorff distance. Indeed, an explicit homeomorphism $\partial_r X \to \partial_s X$ is given by $[\gamma] \mapsto [(\gamma(n))_{n \in \mathbb N}]$. Roughly speaking, two equivalence classes of rays are close in $\partial_r X$ if they fellow travel for a long time. Since quasi-isometries between hyperbolic proper geodesic metric spaces induce homeomorphisms of the corresponding Gromov boundaries, 
the homeomorphism type of $\partial X$ does not depend on the choice of apogee $X$ and will be denoted by $\partial \Lambda$. We refer to any representative of $\partial \Lambda$ as a Gromov boundary of $(\Lambda, \Lambda^\infty)$.\index{Gromov boundary!of a hyperbolic approximate group}
\end{construction}
\begin{remark} If $(X,d)$ is a Morse hyperbolic space which is \emph{not} geodesic, then one can still define the sequential Gromov boundary $\partial_s X$, but this is not a QI-invariant and hence of limited usefulness to us. In particular, if we want to compute the Gromov boundary of a hyperbolic approximate group, then we need to first represent $[\Lambda]_{\mathrm{int}}$ by a geodesic space $(X,d)$ before passing to the boundary. 
\end{remark}
Note that if $(\Lambda, \Lambda^\infty)$ is a hyperbolic approximate group, then by construction, $\partial \Lambda$ depends only on $[\Lambda]_{\mathrm{int}}$, and hence on $[\Lambda]_c$ (by Construction \ref{CanoicalQITypeGood}); in particular it is invariant under commensurability and $2$-local isomorphisms (or Freiman $3$-isomorphisms). Moreover, if $\cI$ is a homeomorphism invariant of compact metrizable topological spaces, then we may define
\[
\cI(\partial \Lambda) \coloneqq  \cI(\partial X),
\]
where $X$ is an apogee for $(\Lambda, \Lambda^\infty)$. Following Example \ref{BoundaryInvariantMorse} we then refer to $\cI(\partial \Lambda)$ as a \emph{topological boundary invariant} of $\Lambda$. The following topological boundary invariant will play a mayor role in Chapter \ref{ChapAsdim} below:
\begin{example}[Topological dimension of the boundary]\label{BoundaryDimension}
If $(\Lambda, \Lambda^\infty)$ is a hyperbolic approximate group, then we define the \emph{topological dimension of its boundary}\index{topological dimension!of boundary of approximate group}
by
\[
\dim \partial \Lambda \coloneqq  \dim \partial X,
\]
where $X$ is an apogee for $(\Lambda, \Lambda^\infty)$ and $\dim$ denotes  topological (or covering) dimension (see Definition \ref{Def-dim} below).
\end{example}
\begin{remark}[Visual metrics]
We recall from Section \ref{sec:VisMet} that if $(X,d)$ is a hyperbolic proper geodesic metric space (for example, an apogee of a hyperbolic approximate group), then there is a canonical class of metrics on the Gromov boundary $\partial X$ of $X$, called visual metrics (see Definition \ref{DefVisual}). If $o \in X$ is a basepoint and $a > 1$ and $c \geq 1$ are constants, then a metric $d$ on $\partial X$ is called \emph{visual} with basepoint $o$ and parameters $a$ and $c$ (or \emph{$(o,a,c)$-visual} for short), if it satisfies the visual inequality \eqref{VisualInequalities} for all $\xi, \xi' \in \partial X$. A metric on $\partial X$ is called \emph{visual} if it is $(o,a,c)$-visual for some choice of $o$, $a$ and $c$.
\end{remark}

\section{Quasi-cobounded hyperbolic spaces}\label{QCHypSpace}
By Proposition \ref{HypAG}, apogees of hyperbolic approximate groups are examples of quasi-cobounded hyperbolic proper geodesic metric spaces. We now investigate some of the basic properties of such spaces. 
\begin{definition}\label{def: visual space}
Let $(X,d)$ be a hyperbolic proper geodesic metric space.
\begin{enumerate}[(i)]
\item $(X,d)$ is called  \emph{almost geodesically complete}\index{almost geodesically complete space} if there exists a constant $D\geq 0$ such that for any $o,x \in X$, there is a geodesic ray $\alpha$ emanating from $o$ such that $d(x,\alpha)\leq D$.
\item $(X,d)$ is called \emph{visual}\index{visual space} if for some (hence for any) basepoint $o\in  X$, there is a positive constant $D$ such that  every point in $X$ has distance at most $D$ from some geodesic ray emanating from $o$. 
\end{enumerate}
\end{definition}
\begin{remark} Note that in Part (ii) of Definition \ref{def: visual space} the constant $D$ is allowed to depend on the basepoint $o$, whereas in Part (i) it is assumed to be independent of $o$. Consequently, every almost geodesically complete space is visual.

There exist a number of variants of the above definitions in the literature. The notion of almost geodesic completeness is used e.g.\ in \cite{Ontaneda} or \cite{GeogheganOntaneda}, where it is attributed to Mike Mihalik. In \cite{Bonk-Schramm}, a metric space $X$ is called \emph{visual} if for some (hence for any) basepoint $o\in  X$ there is a positive constant $D$ such that each point in $X$ lies on a $D$-roughly geodesic ray emanating from $o$. In \cite{BuyaloSchroeder}, a Gromov hyperbolic space is said to be \emph{visual} if for some (hence for any) basepoint $o \in X$, there is a positive constant $D$ such that for every $x\in X$ there is a $\xi \in\partial_s X$ such that $d(o,x)\leq (x|\xi)_o+D$. These two definitions are equivalent to each other and also to the definition above for hyperbolic proper geodesic spaces. We have chosen the definition above since it makes the relation to almost geodesically complete spaces clear. Intuitively one should think of visuality as a coarse version of the geodesic extension property (the ability to extend a geodesic segment into a geodesic ray). 
\end{remark}
The fact that infinite hyperbolic groups are visual (see e.g.\ \cite[Lemma 3.1]{BestvinaMess} or \cite{GenevoisMathoverflow}) generalizes as follows:
\begin{proposition}\label{PropVisual}
Every unbounded quasi-cobounded hyperbolic proper geodesic metric space $(X,d)$ is almost geodesically complete, hence in particular visual.
\end{proposition}
\begin{proof} Since $X$ is quasi-cobounded, there exist constants $K \geq 1$ and $C, C', R$  $\geq 0$, and a collection $\cA\subseteq \widetilde{\rm QI}(X)$ of $(K, C, C')$-quasi-isometries such that for all $x, y \in X$, there is a $g\in \cA$ with $g(y)\in B(x,R)$. We choose such a collection once and for all.

Let $o \in X$ be a basepoint and $x \in X$ an arbitrary point. We first prove that there is a bi-infinite geodesic $\gamma$ at bounded distance from $x$. Since $X$ is unbounded, we can pick two sequences of points $(x_n)$ and $(y_n)$ in $X$ with the property that $d(x_n,y_n)\to \infty$. For each pair $(x_n, y_n)$ we then choose a geodesic segment $[x_n,y_n]$ connecting them. For every $n \in \bN$ we denote by $a_n$ the midpoint of $[x_n, y_n]$. We choose a quasi-isometry $g_n\in \cA$ such that $a_n' \coloneqq  g_n(a_n)\in B(x,R)$ and define $x_n' \coloneqq  g_n(x_n)$ and $y_n' \coloneqq  g_n(y_n)$. Then $d(x_n', a_n') \to \infty$ and $d(y_n', a_n') \to \infty$. Moreover, $g_n([x_n, y_n])$ is a quasi-geodesic segment from $x_n'$ to $y_n'$ which contains $a_n'$, hence is of bounded distance from $x$. Since $\cA$ is uniform, both the implied QI-constants and the distance from $x$ of $g_n([x_n, y_n])$ can be bounded independently of $n$.

Now let $[x_n', y_n']$ be a geodesic segment from $x_n'$ to $y_n'$. By Lemma \ref{Morse1} the quasi-geodesic segments $[x_n', y_n']$ and $g_n([x_n, y_n])$ are at uniformly bounded distance, hence $x$ is at uniformly bounded distance from all of the geodesic segments $[x_n', y_n']$. We can thus apply Lemma \ref{ArzelaAscoli1} to obtain a bi-infinite geodesic line $\gamma_x$ which is at uniformly bounded distance from $x$.

For every $n \in \bZ$ we now choose a geodesic segment $[o, \gamma_x(n)]$ from $o$ to $\gamma_x(n)$ and consider the geodesic triangle formed by $[o, \gamma_x(-n)]$, $\gamma_x([-n,n])$ and $[o, \gamma_x(n)]$. By Lemma \ref{RipsGromov} this triangle is $\delta'$-slim for some uniform $\delta'$. In particular, $\gamma_x(0)$ (and hence $x$) is at uniformly bounded distance from either $[o, \gamma_x(-n)]$ or $[o, \gamma_x(n)]$.

Another application of Lemma \ref{ArzelaAscoli1} shows that for some sequence $(n_k)$ the geodesic segments $[o, \gamma_x(-n_k)]$ and $[o, \gamma_x(n_k)]$ converge to geodesic rays emanating from $o$, and at least one of them is at uniformly bounded distance from $x$.
\end{proof}
Here is another property of hyperbolic groups that we can generalize:
\begin{definition}
A metric space $(X,d)$ has \emph{bounded growth at some scale}\index{growth!bounded growth} if there exist constants $R>r>0$ and $N \in \bN$ such that every open ball of radius $R$ in $X$ can be covered by $N$ open balls of radius $r$.
\end{definition}
\begin{theorem}\label{BoundedGrowth1}
Every quasi-cobounded hyperbolic proper geodesic metric space $(X,d)$ has bounded growth at some scale.
\end{theorem}
\begin{proof} Assume that $X$ is $(K_1, C_1, r, o)$-quasi-cobounded. Given $x \in X$ we choose a $(K_1, C_1, C_1)$-quasi-isometry $f:X\to X$ with $d(o, f(x)) \leq r$. We also choose a quasi-inverse $g$ of $f$ and set $h \coloneqq  g \circ f$. Then $g$ is a $(K_2, C_2, C_2)$-quasi-isometry and $h$ is a $(K_3, C_3, C_3)$-quasi-isometry which is $C_3$-close to the identity, where $K_2, C_2, K_3, C_3$ are constants depending only on $K_1$ and $C_1$. Given positive real numbers $R_3$ and $r_3$ we define real numbers $r_0, r_1, r_2, R_0, R_1, R_2$ by
\[
r_2 \coloneqq  r_1 \coloneqq  r_3-C_3, \qand r_0 \coloneqq  K_2^{-1}(r_1-C_2), 
\]
as well as
\[R_2 \coloneqq  R_3+C_3, \quad
R_1 \coloneqq  K_3 R_2 +2K_3C_3, \qand R_0 \coloneqq  K_1R_1+C_1+r.
\]
We now choose $R_3$ and $r_3$ in such a way that $R_i > r_i>0$ for all $i \in \{0, \dots, 3\}$. By properness of $X$ there exists $n \in \bN$ and $o_1, \dots, o_n \in X$ such that
\[
B(o, R_0) \subset \bigcup_{i=1}^N B(o_i,r_0).
\]
Moreover,
\[
f(B(x, R_1)) \subset B(f(x), K_1R_1+C_1) \subset B(o, K_1R_1+C_1+r) = B(o, R_0) \subset \bigcup_{i=1}^N B(o_i,r_0).
\]
If we set $x_i \coloneqq  g(o_i)$, then applying $g$ to the previous expression yields
\begin{eqnarray*}
h(B(x, R_1)) &\subset&  \bigcup_{i=1}^N g(B(o_i,r_0)) \quad \subset \quad \bigcup_{i=1}^N B(g(o_i), K_2r_0+C_2)  \ = \ \bigcup_{i=1}^N B(x_i, r_1).
\end{eqnarray*}
Now let $y \in \mathrm{Im}(h) \cap B(x, R_2)$, say $y = h(z)$. Then
\begin{eqnarray*}
d(x,z) &\leq& K_3 d(h(x), h(z)) +K_3C_3 \quad \leq \quad K_3 (d(x,y) + C_3)  + K_3C_3\\
&<& K_3 R_2 +2K_3C_3  \quad = \quad R_1,
\end{eqnarray*}
and hence $z \in B(x, R_1)$. This shows that
\[
 \mathrm{Im}(h) \cap B(x, R_2) \subset h(B(x, R_1)) \subset  \bigcup_{i=1}^N B(x_i, r_1).
\]
Since $ \mathrm{Im}(h)$ is $C_3$ relatively dense in $X$, we deduce that
\begin{eqnarray*}
B(x, R_3) &=& B(x, R_2-C_3) \subset N_{C_3}(\mathrm{Im}(h) \cap B(x, R_2) ) \\
&\subset& \bigcup_{i=1}^N B(x_i, r_1+C_3) =  \bigcup_{i=1}^N B(x_i, r_3).
\end{eqnarray*}
Since $R_3$, $r_3$ and $N$ are independent of $x$, this finishes the proof.
\end{proof}
An important application of bounded growth at some scale is the following embedding theorem (\cite[Thm.\ 1.1]{Bonk-Schramm}). Here, two metric spaces $X$ and $Y$ are called \emph{roughly similar}\index{roughly similar spaces} if there exists $f: X \to Y$ and constants $k, \lambda >0$ such that 
\[
\lambda d(x,y)-k \leq  d(f(x), f(y)) \leq \lambda d(x,y) + k
\]
and $f(X)$ is $k$-relatively dense in $Y$. This implies, in particular, that $X$ and $Y$ are quasi-isometric.
\begin{theorem}[Bonk-Schramm embedding theorem] If $X$ is a hyperbolic proper geodesic metric space with bounded growth at some scale, then there exists $n \in \bN$ such that $X$ is roughly similar (and hence quasi-isometric) to a closed convex subset of hyperbolic $n$-space $\bH^n$.
\end{theorem}
Combining this with Theorem \ref{BoundedGrowth1} we thus obtain:
\begin{corollary}[Hyperbolic embedding]\label{BSEmb1}
Every quasi-cobounded hyperbolic proper geodesic metric space is roughly similar (and hence quasi-isometric) to a closed convex subset of hyperbolic $n$-space $\bH^n$ for some $n \in \mathbb N$.
\end{corollary}
We now discuss some consequences of this embedding result. For this we need the following notions.
\begin{construction}[Limit sets in hyperbolic sets] 
In the sequel we denote by $\overline{\bH}^n$ the Gromov compactification of hyperbolic $n$-space and by $\partial \bH^n \coloneqq  \overline{\bH}^n \setminus \bH^n$ its Gromov boundary, which is an $(n-1)$-sphere. If $Y$ is an arbitrary subset of $\bH^n$ then we define its \emph{limit set}\index{limit set!in hyperbolic space} $\cL(Y)$ as the subset
\[
\cL(Y) \coloneqq  \overline{Y} \cap \partial \bH^n \subset \partial \bH^n,
\]
where the closure is taken in $\overline{\bH}^n$; this is a special case of the notion of a Gromov limit set as defined in Definition \ref{DefGromovLimitSet} below. 

If $Y$ happens to be closed and convex, then it is a CAT(-1) space and in particular uniquely geodesic. Its Gromov boundary $\partial Y$ can thus be identified with geodesic rays emanating from some basepoint $o \in Y$. Each such geodesic then defines a point in $\partial \bH^n$ and thus we obtain a continuous embedding $\iota: \partial Y \to \partial \bH^n$; it follows from Arzel\`a--Ascoli Lemma \ref{ArzelaAscoli1} that the image is given by the limit set $\cL(Y)$. More precisely, we can fix a basepoint $o \in X$ and given $\xi \in \cL(Y)$ we may choose $x_n \in Y$ with $x_n \to \xi$. Since $Y$ is closed and convex, the geodesic segments $[o, x_n]$ subsequentially converge to a geodesic ray in $Y$ with endpoint $\xi$. This shows that $\cL(Y) \subset \iota(\partial Y)$, and the converse inclusion is obvious. In particular we have
$\partial Y \cong \cL(Y)$, where, by construction, $\cL(Y)$ is a compact subset of the sphere $\partial \bH^n$. 

If $X$ is a quasi-cobounded hyperbolic proper geodesic metric space, then by Corollary \ref{BSEmb1} we can find a closed convex subset $Y \subset \bH^n$ for some $n \in \bN$ and a quasi-isometry $f: X \to Y$, which in turn induces a homeomorphism \[\partial X \cong \partial Y \cong \cL(Y) \subset \partial \bH^n.\]
\end{construction}
\begin{corollary}\label{cor:FinDimBd} 
If $(X,d)$ is a quasi-cobounded hyperbolic proper geodesic metric space, then $\partial X$ embeds homeomorphically into a finite-dimensional sphere. In particular, $\partial X$ is of finite topological dimension.
\end{corollary}
We now consider the problem of recovering a subset $Y \subset \bH^n$ from its limit set.
\begin{remark}[Hulls in real hyperbolic spaces] \label{rem:hulls in hyp spaces}
There are various different ways to associate with a compact subset $L \subset \partial \bH^n$ a set in $\bH^n$ which has $L$ as limit set:

In Definition \ref{def:weak hull} we define (in large generality) the notion of a \emph{weak hull}\index{weak hull} $\mathfrak{H}(L) \subset \bH^n$ of $L$. In the present case, the definition can be made more explicit: If we set $L^{(2)} \coloneqq  \{(\xi, \eta) \in L^2 \mid \xi \neq \eta\}$, then for all $(\xi, \eta) \in L^{(2)}$ there is a unique (up to reparametrization) geodesic $\gamma_{\xi, \eta}$ which is bi-asymptotic to $\xi$ and $\eta$. By Remark \ref{WeakHullGeodesics} we then have
\[
\mathfrak{H}(L) = \bigcup_{(\xi, \eta)\in L^{(2)}} \gamma_{\xi,\eta}(\R),
\]
i.e.\ the weak hull is the union of all geodesics connecting points in $L$. Since $L$ is compact, the weak hull is quasi-convex (Proposition \ref{WeakHullQuasiconvex}) and if $L$ is non-empty, then $L =\cL(\mathfrak{H}(L))$ (Corollary \ref{LimitSetInvertsHull}). 

The slight disadvantage of the weak hull is that it is only quasi-convex, but typically not convex. We thus define
\[
\cH(L) \coloneqq \overline{\mathrm{conv}}(\mathfrak H(L));
\]
this is the smallest closed convex subset of $\bH^n$ which contains all geodesics between points in $L$ and is thus called the \emph{convex hull}\index{convex hull} of $L$. We establish in Proposition \ref{WeakHullConvexHull} below that $\cH(L)$ is at bounded Hausdorff distance from $\mathfrak{H}(L)$. In particular, this implies that 
\begin{equation}
   \cL(\cH(L)) = \cL(\mathfrak{H}(L)) = L.
\end{equation}
\end{remark}
\begin{lemma}\label{Hulling} 
Let $Y \subset \bH^n$ be an unbounded quasi-cobounded closed convex subset. 
Then $\cH(\cL(Y)) \subset Y$ and $Y$ is at bounded Hausdorff distance from  $\cH(\cL(Y))$. In particular, $\cH(\cL(Y))$ is the unique smallest closed convex subset of $\bH^n$ with limit set $\cL(Y)$.
\end{lemma}
\begin{proof} Assume first that $x \in \mathfrak{H}(\cL(Y))$ is contained in the weak hull. Then there exist $\xi, \eta \in \cL(Y)$ with $\xi \neq \eta$ such that $x \in [\xi, \eta]$, and hence $x_n, y_n \in Y$ with $x_n \to \xi$ and $y_n \to \eta$. Since $\xi \neq \eta$ we have $(x_n|y_n)_0 \leq D$ for some $D>0$, and since $Y$ is convex it follows from Lemma \ref{ArzelaAscoli1} that suitable reparametrizations of the geodesic segments $[x_n, y_n]$ converge uniformly on compacta to a parametrization of $[\xi, \eta]$. Since $x \in [\xi, \eta]$ and $Y$ is closed we deduce that $x \in Y$. This proves that $\mathfrak{H}(\cL(Y)) \subset Y$, and since $Y$ is closed and convex we have $\cH(\cL(Y)) \subset Y$.
Conversely, let us fix a basepoint $o \in Y$. Since $Y$ is convex, it is geodesic, hence visual by Proposition \ref{PropVisual}. There thus exists $D>0$ such that for every $x \in Y$ there is a geodesic ray $\gamma$ emanating from $o$ with $\mathrm{dist}(x, \gamma) < D$. Since $\xi \coloneqq  \gamma(\infty) \in \cL(Y)$ we have $\gamma = [o, \xi) \in \cH(\cL(Y))$ and thus $Y \subset N_D(\cH(\cL(Y)))$. The lemma follows.
\end{proof}
Combining this with Corollary \ref{BSEmb1} we obtain:
\begin{corollary}\label{BSEmb2} Let $(X,d)$ be a quasi-cobounded hyperbolic proper geodesic metric space. Then there exists $n \in \bN$, a compact subset $L \subset \partial \bH^n$ and a quasi-isometry $f: X \to \cH(L)$ which induces a homeomorphism $\partial X \to \partial  \cH(L) = L$.
\end{corollary}
We now turn to boundaries of quasi-cobounded hyperbolic spaces. An important property of infinite hyperbolic groups is that their boundaries have the following doubling property with respect to visual metrics:
\begin{definition} \label{Def: doubling} Let $(Z, d)$ be a metric space. We say that $Z$ is \emph{doubling}\index{doubling property} with \emph{doubling constant}\index{doubling constant} $N \in \mathbb N$ if for all $t>0$ and all $\xi \in Z$ there exist $\xi_1, \dots, \xi_N \in Z$ such that
\[
B(\xi, 2t) \subset \bigcup_{i=1}^N B(\xi_i, t).
\]
\end{definition}
It was established by Bonk and Schramm \cite[Thm.\ 9.2]{Bonk-Schramm} that if $X$ is a hyperbolic proper geodesic metric space with bounded growth at some scale, then the Gromov boundary $\partial X$ of $X$ is doubling with respect to any choice of visual metric. From Theorem \ref{BoundedGrowth1} we may thus conclude:
\begin{corollary}\label{Doubling1} If $X$ is a quasi-cobounded hyperbolic proper geodesic metric space, then $\partial X$ is doubling with respect to any visual metric.
\end{corollary}
Another important property of hyperbolic groups that we are going to generalize is local self-similarity (cf.\ \cite[Section 2.3]{BuyaloSchroeder}).
\begin{definition}\label{quasi-homothety and generalization}\label{def: locally quasi-similar}
Let $(X, d_X)$, $(Y, d_Y)$ be metric spaces and let $\lambda\geq 1$, $K \geq 1$, $R>0$ and $R_0 >1$ be constants. 
\begin{enumerate}[(i)]
\item A map $f: X \to Y$ is called a \emph{$\lambda$--quasi-homothetic map with coefficient $R$} (or a \emph{quasi-homothety}\index{quasi-homothety} for short) if for all $x_1, x_2 \in X$ we have
\begin{equation}\label{eqn: q-homothety}
\frac{1}{\lambda}{R} \cdot d_X(x_1, x_2) \quad \leq \quad d_Y(f(x_1), f(x_2))  \quad \leq \quad \lambda {R} \cdot d_X(x_1,x_2).
\end{equation}
\item A map $f: X \to Y$ is called a \emph{generalized $(\lambda, K)$--quasi-homothetic map with coefficient $R$} (or a \emph{generalized quasi-homothety}\index{generalized quasi-homothety}\index{quasi-homothety:generalized} for short) if for all $x_1, x_2 \in X$ we have
\begin{equation}\label{locally-pqs}
\displaystyle \frac{1}{\lambda}R^K (d_X(x_1,x_2))^K\ \leq \ d_Y(f(x_1),f(x_2)) \ \leq \ \lambda \sqrt[K]{R}\sqrt[K]{d_X(x_1,x_2)}.
\end{equation}
\item $(X, d_X)$ is called \emph{locally self-similar}\index{locally self-similar}\index{self-similar!locally} with parameters $\lambda\geq 1$ and $ R_0 > 1$ if for every $R>R_0$ and every subset $A\subset X$ with diameter $\diam A \leq \frac{1}{R}$ there exists a $\lambda$--quasi-homothetic map $f: A \to X$ with coefficient $R$.
\item  $(X, d_X)$ is called \emph{locally quasi-self-similar}\index{locally quasi-self-similar}\index{quasi-self-similar!locally} with parameters $\lambda \geq 1$, $K \geq 1$ and $ R_0 > 1$ if for every $R>R_0$ and every subset $A\subset X$ with diameter $\diam A \leq \frac{1}{R}$ there exists a generalized $(\lambda,K)$--quasi-homothetic map $f: A \to X$ with coefficient $R$.
\end{enumerate}
\end{definition}
As alluded to above, the Gromov boundary of an infinite hyperbolic group is locally self-similar with respect to any visual metric (see e.g.\ \cite[Theorem 2.3.2]{BuyaloSchroeder}). This fact generalizes as follows:
\begin{proposition}\label{LemmaQuasiSelfSim} 
If $(X, d)$ is a quasi-cobounded hyperbolic proper geodesic metric space, then $\partial X$ is locally quasi-self-similar with respect to any visual metric.
\end{proposition}
\begin{proof} We fix a basepoint $o \in X$ and an $(o,a,c)$-visual metric $d_0$ on $\partial X$. By the coboundedness assumption there exist constants $K\geq1$, $C\geq 0$, $r>0$ such that for all $x, y \in X$ there exists a $(K_1, C_1,o)$-quasi-isometry (in the sense of \eqref{KCoQI}) such that $d(f(x), y) < r$.  We are going to show that $(\partial_r X,d_0)$ is locally quasi-self-similar with parameters
\[
R_0 \coloneqq  c_0+1,\quad K\coloneqq K_1 \qand \lambda \coloneqq  \max\{c_0^{1+2K_1} a_0^{4K_1\delta + C_1 + 2\delta + r}, c_0^{1-2K_1^{-1}}a_0^{C_1+r}\}.
\]
Thus let $R>R_0$ and let $B \subset \partial_r X$ with ${\rm diam}_{d_0}(B) \leq 1/R$. Since $ \log_{a_0}(R/c_0)>0$, one may argue as in the proof of \cite[Thm. 2.3.2]{BuyaloSchroeder} to construct a point $g \in X$ with $d(o,g) = \log_{a_0}(R/c_0)$ such that for all $\xi', \xi'' \in B$ one has
\begin{equation}\label{QSS1}
(\xi'\mid \xi'')_o - d(o,g) \; \leq \; (\xi'\mid \xi'')_g \; \leq \; (\xi'\mid \xi'')_o - d(o,g) + 4\delta.
\end{equation}
On the other hand, since $X$ is $(K_1, C_1, r, o)$-quasi-cobounded, we can find a $(K_1, C_1, o)$-quasi-isometry $h$ such that $h(g) \in B(o, r)$, and we define 
\[f \coloneqq  \partial_r h|_{B}: B \to \partial_r X.\]
By Lemma \ref{QIGromovInequality} we then have
\[K_1^{-1} (\xi'\mid \xi'')_g - C_1 \; \leq \;  (f(\xi')\mid  f(\xi''))_{h(g)} \;  \leq \; K_1 (\xi'\mid \xi'')_g + C_1 + 2\delta.\]
Since $d(o, h(g)) < r$ we deduce that
\begin{equation}\label{QSS2}
(f(\xi')\mid f(\xi''))_o \; \leq \; (f(\xi')\mid f(\xi''))_{h(g)} + d(o, h(g)) \; < \;  K_1 (\xi'\mid \xi'')_g + C_1 + 2\delta + r,
\end{equation}
and similarly
\begin{equation}\label{QSS3}
(f(\xi')\mid f(\xi''))_o \; \geq \; (f(\xi')\mid f(\xi''))_{h(g)} - d(o, h(g)) \; > \;  K_1^{-1} (\xi'\mid \xi'')_g - C_1 - r.
\end{equation}
Combining \eqref{QSS1} and \eqref{QSS2} we obtain
\begin{eqnarray*}
(f(\xi')\mid f(\xi''))_o &\leq&  K_1 (\xi'\mid \xi'')_g + C_1 + 2\delta + r\\
&\leq&  K_1((\xi'\mid \xi'')_o - d(o,g) + 4\delta) + C_1 + 2\delta + r\\
&=& K_1(\xi'\mid \xi'')_o - K_1 d(o,g) + 4K_1\delta + C_1 + 2\delta + r.
\end{eqnarray*}
Combining this with the visual inequalities for $d_0$ (see \eqref{VisualInequalities} in Appendix C) and the fact that $a_0^{d(o,g)} = R/c_0$ we obtain
\begin{eqnarray*} 
d_0(f(\xi'), f(\xi'')) &\geq& c_0^{-1} \cdot a_0^{-(f(\xi')\mid f(\xi''))_o}\\
&\geq& c_0^{-1} \cdot a_0^{-K_1(\xi'\mid \xi'')_o + K_1 d(o,g) - 4K_1\delta - C_1 - 2\delta - r}\\
&=& c_0^{-1} \cdot a_0^{-4K_1\delta - C_1-2\delta -r } \cdot (R/c_0)^{K_1} \cdot (a_0^{-(\xi'\mid \xi'')_o})^{K_1}\\
&=&c_0^{-1-2K_1} \cdot a_0^{-4K_1\delta - C_1-2\delta -r } \cdot R^{K_1} \cdot (c_0 a_0^{-(\xi'\mid \xi'')_o})^{K_1}\\
&\geq& {\lambda}^{-1} \cdot R^{K} \cdot d_0(\xi', \xi'') ^{K}.
\end{eqnarray*}
If instead of starting from \eqref{QSS1} and \eqref{QSS2} we start from \eqref{QSS1} and \eqref{QSS3}, then a similar computation yields
\begin{eqnarray*} 
d_0(f(\xi'), f(\xi'')) &\leq& c_0 \cdot a_0^{-K_1^{-1}(\xi'\mid \xi'')_o + K_1^{-1} d(o,g) + C_1  + r}\\
&=& c_0 \cdot a_0^{C_1+r} \cdot \sqrt[K1]{R/c_0} \cdot \sqrt[K_1]{a_0^{-(\xi'\mid \xi'')_o}}\\
&=& c_0^{1-2K_1^{-1}}  \cdot a_0^{C_1+r} \cdot\sqrt[K_1]{R} \cdot \sqrt[K_1]{c_0 a_0^{-(\xi'\mid \xi'')_o}}\\
&\leq& \lambda \cdot \sqrt[K]{R} \cdot \sqrt[K]{d_0(\xi', \xi'')}.
\end{eqnarray*}
This shows that  $(X,d_0)$ is locally quasi-self-similar with parameters $(R_0, K, \lambda)$.
\end{proof}
\section{Boundary actions of hyperbolic approximate groups}
We now apply the results of the previous section to apogees of hyperbolic approximate groups. In view of Proposition \ref{HypAG} the following results follow from Proposition \ref{PropVisual}, Theorem \ref{BoundedGrowth1}, Corollary \ref{BSEmb1}, Corollary \ref{cor:FinDimBd}, Corollary \ref{BSEmb2}, Corollary \ref{Doubling1} and Proposition \ref{LemmaQuasiSelfSim} respectively.
\begin{corollary}[Apogees of hyperbolic approximate groups]\label{HyperbolicApogees} Let $(\Lambda, \Lambda^\infty)$ be an infinite hyperbolic approximate group. Then the following is true:
\begin{enumerate}[(i)]
\item Every apogee $(X,d)$ of $(\Lambda, \Lambda^\infty)$ is visual, it has bounded growth at some scale and it has finite topological dimension.
\item There exists $n \in \mathbb N$ and a compact subset $L \subset \partial \bH^n$ which represents $\partial \Lambda$ such that $X \coloneqq \cH(L) \subset \bH^n$ is an apogee for $(\Lambda, \Lambda^\infty)$.
\item If $(X,d)$ is any apogee of $(\Lambda, \Lambda^\infty)$, then $\partial X$ is doubling and it is locally quasi-self-similar with respect to any visual metric.
\end{enumerate}
\end{corollary}
\begin{definition} A compact subset $L \subset \partial \bH^n$ as in Part (ii) of Corollary \ref{HyperbolicApogees} is called a \emph{Bonk-Schramm realization}\index{Bonk-Schramm realization} of $\partial \Lambda$. In this case,  $X \coloneqq \cH(L)$ is called the \emph{associated apogee}.
\end{definition}
As the first application of the Bonk-Schramm realization we establish the following proposition.
\begin{proposition}[Ends and boundaries of hyperbolic approximate groups]\label{GromovBoundaryTri} 
For a hyperbolic approximate group $(\Lambda, \Lambda^\infty)$ exactly one of the following three possibilities occurs:
\begin{enumerate}[(\text{Type} 1)]
\item $\Lambda$ is finite and hence $|\partial \Lambda| = |\cE_\infty(\Lambda)| = 0$.
\item $\Lambda$ is quasi-isometric to $\bZ$ and $|\partial \Lambda| = |\cE_\infty(\Lambda)| = 2$.
\item $\Lambda$ is infinite with $|\partial \Lambda| = \infty$ (and hence $|\cE_\infty(\Lambda)| \in \{1,2, \infty\}$). 
\end{enumerate}
\end{proposition}
\begin{proof} If $\Lambda$ is finite, then we are clearly in Type 1. Assume now that $\Lambda$ is infinite and fix a Bonk-Schramm realization $L \subset \partial \bH^n$ of $\partial \Lambda$ so that $k \coloneqq |L| = |\partial \Lambda| > 0$. Assume first that $L$ is finite.

Let $(\Lambda, \Lambda^\infty)$ be a hyperbolic approximate group. Assume that some representative $L$ of $\partial \Lambda$ is finite; then all representatives of $\partial \Lambda$ are finite with the same cardinality $|\partial \Lambda| \coloneqq  |L|$. If $k=1$, then $\cH(L) = \emptyset$ which is impossible, hence we always have $k \geq 2$ and $\cH(L)$ is an ideal hyperbolic polyhedron in $\bH^n$.
Since such polyhedra are thin, the hull $\cH(L)$ intersects the complement of any sufficiently large ball in $\bH^n$ in $k$ components, corresponding to the $k$ elements of $L$. This shows that the number of ends of $L$ is given by
\[
|\cE_\infty(\Lambda)| = |\cE(L)| = k.
\]
With Corollary \ref{EndsApp} we deduce that $k \leq 2$ and hence $k=2$. This however means that $[\Lambda]_{\mathrm{int}}$ is represented by the convex hull of two points in $\bH^n$, which is a geodesic line or, in other words, that $[\Lambda]_{\mathrm{int}} = [\Z]$, hence we are in Type 2. Finally, if $L$ is infinite, then by Corollary \ref{EndsApp} we are in Type 3. 
\end{proof}
In the sequel we will focus on hyperbolic approximate groups of Type 3 and in particular assume that $\Lambda$ is infinite.
\begin{construction}[Boundary action]
Let $(\Lambda, \Lambda^\infty)$ be an infinite hyperbolic approximate group, let $L \subset \partial \bH^n$ be a Bonk-Schramm realization of $\partial \Lambda$ and let $X$ denote the associated apogee. Then $(\Lambda, \Lambda^\infty)$ quasi-acts geometrically on $X$ by the left-regular quasi-action 
\[\rho: (\Lambda, \Lambda^\infty) \to \widetilde{\mathrm{QI}}(X).\] 
In particular, the quasi-orbit maps $\iota_x: \Lambda \to X$, $g \mapsto g.x$ are quasi-isometries for $x \in X$ and their images are quasi-lattices in $X$. Moreover, for every $g \in \Lambda$, the map $\rho(g)$ is a quasi-isometry of $X$, hence extends to a homeomorphism $\partial \rho(g)$ of $\mathcal L(X) = L$. It is easy to check that the map
\[
\partial \rho: (\Lambda, \Lambda^\infty) \to \mathrm{Homeo}(L)
\]
is actually a homeomorphism, hence defines an action of $\Lambda^\infty$ on $L$ by homeomorphisms, called the \emph{boundary action} on $L$; we will actually prove a more general statement in Lemma
\ref{BoundaryActionGromov} below.
\end{construction}
An important distinction in the theory of hyperbolic groups is the dichotomy between elementary and non-elementary hyperbolic groups. The following well-known theorem provides several equivalent characterizations of elementary hyperbolic groups:
\begin{theorem}[Characterizations of elementary hyperbolic groups]\label{TheoremGroupElementary} Let $\Gamma$ be an infinite hyperbolic group. Then the following conditions are equivalent:
\begin{enumerate}[(i)]
    \item There exists a subgroup $\Gamma'<\Gamma$ of index $\leq 2$ which fixes a point in $\partial \Gamma$.
    \item There exists a finite index subgroup $\Gamma' < \Gamma$ which fixes a point in $\partial \Gamma$.
    \item There exists a finite index subgroup $\Gamma'' < \Gamma$ which is infinite cyclic.
    \item $|\partial \Gamma| = 2$.
    \item $|\partial \Gamma| \leq 2$
    \item $\partial \Gamma$ is finite.
\end{enumerate}
\end{theorem}
\begin{proof} The implications (i)$\implies$(ii) and (iii)$\implies$(iv)$\implies$(v)$\implies$(vi) are obvious, and (vi)$\implies$ (v) is a special case of Proposition \ref{GromovBoundaryTri}. Also, since $\Gamma$ is assumed to be infinite, Proposition \ref{GromovBoundaryTri} implies $|\partial \Gamma| \geq 2$, hence (v)$\implies$(iv). If (iv) holds, then $\partial \Gamma = \{\xi_1, \xi_2\}$ for some $\xi_1, \xi_2 \in \partial \Gamma$, in which case (i) holds with $\Gamma' \coloneqq  \mathrm{Stab}_\Gamma(\xi_1)$. Thus the only non-trivial implication is (ii)$\implies$(iii), which follows from \cite[Theorem 8.3.30]{GhysdelaHarpe}.

More precisely, the group $\Gamma'$ from (ii) has an infinite point stabilizer in $\partial \Gamma' = \partial \Gamma$, hence by \cite[Theorem 8.3.30]{GhysdelaHarpe} contains a further finite index subgroup $\Gamma''$ which is infinite cyclic.
\end{proof}
Note that passing to an index $2$ subgroup may be necessary in condition (i), as the example of an infinite dihedral group shows.
\begin{definition}
A hyperbolic group $\Gamma$ is called \emph{elementary} if it is either finite or if it is infinite and satisfies the equivalent conditions of Theorem \ref{TheoremGroupElementary}, otherwise it is called \emph{non-elementary}.\index{hyperbolic!group!elementary}
\index{hyperbolic!group!non-elementary}
\end{definition}
Our next goal is to define a notion of a non-elementary hyperbolic approximate group. All of the properties in Theorem \ref{TheoremGroupElementary} have counterparts in the theory of hyperbolic approximate groups, but the non-trivial implication (ii)$\implies$(iii) breaks down. Not only is the iteration argument behind the proof of \cite[Theorem 8.3.30]{GhysdelaHarpe} not available in the context of approximate groups, but it turns out that the condition (iii) is actually strictly stronger than the other conditions:
\begin{example}
If $\Lambda$ is a uniform approximate lattice in $\bR$, then $\Lambda$ is quasi-isometric to $\bR$ and hence $|\partial \Gamma| = 2$. In particular, $\Lambda$ is infinite, but it need not contain any infinite cyclic subgroup (let alone of finite index). For example, the set of vertices of the Fibonacci tiling does not.
\end{example}
Other than that, the proof of Theorem \ref{TheoremGroupElementary} carries over almost verbatim to give the following corollary:

\begin{corollary}\label{AGElementary} 
Let $(\Lambda, \Lambda^\infty)$ be an infinite hyperbolic approximate group and let $L$ be a Bonk-Schramm realization of $\partial \Lambda$. We denote by $\partial \rho$ the associated boundary action on $L$ and
consider the following conditions:
\begin{enumerate}[(i)]
    \item There exists a syndetic approximate subgroup $\Lambda'' \subset \Lambda$ which is an infinite cyclic group.
    \item $|L| \leq 2$ (and hence $|L| = 2$).
    \item $L$ is finite.
    \item There exists a syndetic approximate subgroup $\Lambda' \subset \Lambda^2$ such that $\partial \rho((\Lambda')^\infty)$ fixes a point in $L$.
    \item There exists an approximate subgroup $\Lambda' \subset \Lambda^\infty$ which is commensurable to $\Lambda$ and such that $\partial \rho((\Lambda')^\infty)$ fixes a point in $L$.
\end{enumerate}
Then the implications (i)$\implies$(ii)$\iff$(iii)$\implies$(iv)$\implies$(v) hold.
\end{corollary}
\begin{proof} (i)$\implies$(ii)$\iff$(iii) are immediate from Proposition \ref{GromovBoundaryTri} and (iv)$\implies$(v) is obvious, hence it remains to prove (ii)$\implies$(iv).

If $\partial \Lambda = \emptyset$ then $\Lambda$ is finite, hence (iv) holds. Otherwise we have $\partial \Lambda = \{\xi_1, \xi_2\}$ and the stabilizer of $\xi_1$ defines a finite index subgroup $\Gamma$ of $\Lambda^\infty$. We may thus choose $\Lambda' \coloneqq  \Lambda^2 \cap \Gamma$; this stabilizes $\xi_1$ and is an approximate subgroup of $\Lambda^2$ by Lemma \ref{Intersection}.
\end{proof}
We define the strongest possible definition of non-elementary hyperbolic group by negating the weakest of the conditions in Corollary \ref{AGElementary}:
\begin{definition}\label{DefNonEl} Let $(\Lambda, \Lambda^\infty)$ be an infinite hyperbolic approximate group and let $L$ be a Bonk-Schramm realization of $\partial \Lambda$ with boundary action $\partial \rho$.
\begin{enumerate}[(i)]
\item We say that $\Lambda$ has \emph{fixpoints at infinity}\index{fixpoints at infinity} if there exists some $\xi \in L$ such that $\partial \rho(\lambda)(\xi) = \xi$ for all $\lambda \in \Lambda$ (and hence all $\lambda \in \Lambda^\infty$).
\item We say that $\Lambda$ is \emph{non-elementary} if no $\Lambda' \subset \Lambda^\infty$ which is commensurable to $\Lambda$ has fixpoints at infinity\index{hyperbolic approximate subgroup!non-elementary}. Otherwise we say $\Lambda$ is \emph{elementary}.
\end{enumerate}
\end{definition}
Note that these notions do not depend on the choice of a Bonk-Schramm realization. As in the group case, we make the convention that a finite approximate subgroup is considered to be elementary. By Corollary \ref{AGElementary} every non-elementary hyperbolic approximate group has an infinite boundary. We do not know whether the converse is true.

It is a classical fact in the theory of hyperbolic groups (see e.g. \cite[Cor. 8.2.26]{GhysdelaHarpe}) that the action of a non-elementary hyperbolic group on its boundary is minimal; here an action of a group on a topological space is called \emph{minimal}\index{action!group action!minimal} if every orbit is dense. Similarly we define:
\begin{definition} An action $\rho: (\Lambda, \Lambda^\infty) \to \mathrm{Homeo}(Z)$ of an approximate group $(\Lambda, \Lambda^\infty)$ on a topological space $Z$ is \emph{minimal} if every $\rho(\Lambda)$-orbit in $Z$ is dense.\index{action!approximate group action!minimal}
\end{definition}
Note that minimality of an action $(\Lambda, \Lambda^\infty) \curvearrowright X$ is a priori stronger than minimality of $\Lambda^\infty \curvearrowright X$. Fixpoints are an obvious obstruction to minimality (if the space has at least two points), and in the case of the boundary action of a hyperbolic approximate group they turn out to be the only obstruction:
\begin{proposition}[Boundary minimality]\label{BoundaryMinimal} For a hyperbolic approximate group $(\Lambda, \Lambda^\infty)$, the following are equivalent:
\begin{enumerate}[(i)]
\item  $(\Lambda, \Lambda^\infty)$ has no fixpoints at infinity.
\item The action of $(\Lambda, \Lambda^\infty)$ on some (hence any) representative of $\partial \Lambda$ is minimal.
\item The action of $\Lambda^\infty$ on some (hence any) representative of $\partial \Lambda$ is minimal.
\end{enumerate}
In particular, these hold if $(\Lambda, \Lambda^\infty)$ is non-elementary.
\end{proposition}
For the proof we will use a statement concerning invariance of weak hulls that we will only prove later in a more general context in Lemma \ref{lem:invariant boundary invariant wh}. We invite the reader to check that Proposition \ref{BoundaryMinimal} is not used in the proof of this lemma, so that our argument is not circular.
\begin{proof} Assume that $(\Lambda, \Lambda^\infty)$ has no fixpoints at infinity and let $\xi \in L$. Then the orbit closure $L' \coloneqq  \overline{\rho(\Lambda).\xi}$ is a compact subset of $\partial \bH^n$ which contains at least one other point than $\xi$, say $\eta$. Now the weak hull $\mathfrak{H}(L')$ contains the geodesic $[\eta, \xi]$ and thus $\mathfrak{H}(L') \neq \emptyset$.

It follows from Lemma \ref{lem:invariant boundary invariant wh} that $d_\mathrm{Haus}(\rho(\Lambda).\mathfrak{H}(L'), \mathfrak{H}(L'))<D$ for some $D>0$. On the other hand, since $\rho$ is cobounded, we know that there exists an $R\geq 0$ such that for any $z \in \mathfrak{H}(L')$ we have $d_\mathrm{Haus}(\rho(\Lambda).z, \mathfrak{H}(L))<R$. Setting $T\coloneqq R+D$, we conclude $d_\mathrm{Haus}(\mathfrak{H}(L'), \mathfrak{H}(L))<T$. Passing to the respective limit sets we deduce with Proposition \ref{prop: limit set of hull is limit set again} that
\[
L' = \cL(\mathfrak{H}(L')) = \cL(\mathfrak{H}(L)) = L.
\] 
This shows that $(\Lambda, \Lambda^\infty)$ acts minimally on $L$ (hence any representative of $\partial \Lambda$) and proves the implications (i) $\implies$ (ii). The implication (ii)$\implies$(iii)$\implies$(i) are obvious.
\end{proof}
Here is a sample application of boundary minimality to the topology of $\partial \Lambda$:
\begin{proposition}[Boundaries of non-elementary hyperbolic approx.\ groups]\label{BoundaryNonel} Assume that $(\Lambda, \Lambda^\infty)$ is an infinite hyperbolic approximate group with Bonk-Schramm realization $L \subset \partial \bH^n$ which has no fixpoints at infinity. Then precisely one of the following two options holds:
\begin{enumerate}[(i)]
\item $|L| = 2$.
\item $L$ is an infinite perfect compact subset of $\partial \bH^n$.
\end{enumerate}
In particular, the boundary of a non-elementary hyperbolic approximate group is represented by a perfect infinite compact subset of a sphere, whose homeomorphism group acts minimally.
\end{proposition}
\begin{proof} Note first that $L$ is compact as a closed subset of $\partial \bH^n$. Now assume that $L$ is not perfect. The subset $K$ of isolated points in $L$ is open and invariant, so the complement of $M=L\setminus K$, is a closed invariant subset. By minimality of the action shown in Proposition \ref{BoundaryMinimal}, we conclude that $M$ is empty. Thus $L$ consists only of isolated points and is a discrete subset of the compact space $L$, hence finite.
By Proposition \ref{GromovBoundaryTri} we thus have either $|L| = 2$ or $|L| = 0$, but since $\Lambda$ is infinite, the latter is impossible. This shows that $L$ is either perfect and compact or of cardinality $2$, and the proposition follows.
\end{proof}

\chapter{Limit sets of quasi-isometric quasi-actions}

In this chapter we consider the behavior of quasi-isometric quasi actions (qiqacs) of approximate groups at infinity. 
We have seen in the last chapter that the left-regular qiqac of a hyperbolic approximate group on its apogee extends to an action by homeomorphisms on the Gromov boundary. In Section \ref{SecMorseGromovLimit} we generalize this construction, first to Gromov boundaries (for qiqacs on hyperbolic spaces) and then to Morse boundaries (for arbitrary proper geodesic metric spaces). In general the boundary action will not be minimal anymore, and as in the group case we identify a closed invariant subset of the boundary called the \emph{limit set} (see Definition \ref{DefGromovLimitSet} and Definition \ref{MorseLimitSet}) whose elements are called \emph{limit points} of the qiqac. 

As in the group case, one can distinguish different types of limit points depending on whether they are approached by quasi-orbit points in a bounded neighborhood of a geodesic or not. Limit points of the first type are usually called \emph{conical}, and the set of conical limit points is an important refinement of the limit set. In Theorem \ref{QuasikernelsConical} we provide an explicit example of an infinite family of pairwise non-commensurable approximate subgroups of free groups (in fact, quasikernels of counting quasimorphisms) which all have full limit set in the ambient free group, but can be pairwise distinguished by their respective conical limit points. This specific example has no counterpart in the group world.

On the other hand,
in Sections \ref{SecStableQO} and \ref{Section:ConvexCocompact} we discuss examples of qiqacs in which every limit point is conical, namely convex cocompact qiqacs or, equivalently, qiqacs with stable quasi-orbits. Our results generalize work of Durham and the first-named author \cite{cordes:2016ab}, which in turn generalizes work of Swenson \cite{Swenson} in the hyperbolic case.

\section{Morse and Gromov limit sets of approximate groups}\label{SecMorseGromovLimit}

In this section we define the boundary action of a qiqac and the associated limit set. Throughout we will work in the following setting:
\begin{notation}\label{BoundarySetting}
The following notations and conventions will be used throughout this chapter:
\begin{itemize}
\item $(\Lambda, \Lambda^\infty)$ denotes an approximate group;
\item  $X$ is a proper geodesic metric space with basepoint $o\in X$;
\item $\rho\colon (\Lambda, \Lambda^\infty) \to \widetilde{{\rm QI}}(X)$ is a $(K_k, C_k, C_k')$-qiqac.
\end{itemize}
\end{notation}
Our basic constructions apply to a wide variety of boundaries. For example, if $X$ is Gromov hyperbolic, then we can define a limit set in the Gromov boundary of $X$. If $X$ is an arbitrary proper geodesic metric space, then we can define a limit set in the Morse boundary of $X$ (see Appendix \ref{AppendixMorse} for the definition of Morse boundaries). In fact, one could as well define limit sets in other type of boundaries - rather than trying to axiomatize the situation, we will discuss limit sets in Gromov and Morse boundaries and leave it to the reader to adapt the constructions presented here to their context.

For the first part of this section, up to Proposition \ref{GromovLimitSetInvariant}, we will assume that $X$ is Gromov hyperbolic. We then denote by $\partial X \coloneqq  \partial_s X$ the sequential model of the \emph{Gromov boundary}\index{Gromov boundary!sequential} of $X$ and by $\overline{X} = X \cup \partial X$ the \emph{Gromov compactification}\index{Gromov compactification}. Note that $\partial X$ is always compact, and it is non-empty as soon as $X$ is unbounded.
\begin{definition}\label{DefGromovLimitSet} Given a subset $Y \subset X$ we define its \emph{Gromov limit set}\index{Gromov limit set!of a subset}\index{limit set!Gromov} as $\cL(Y) \coloneqq  \overline{Y} \cap \partial X$, where the closure is taken inside $\overline{X}$.
\end{definition}
Note that if two subsets $Y, Z \subset X$ are at bounded Hausdorff distance, then $\cL(Y) = \cL(Z)$, which implies:
\begin{lemma} \label{lem:limit set indep1}  For all $o, o' \in X$ and $k \in \bN$ we have
\[
\cL(\rho(\Lambda).o) = \cL(\rho(\Lambda).o') = \cL(\rho(\Lambda^k).o) \subset \partial X.
\]
\end{lemma}
\begin{proof} The first equality follows from the fact that, by Remark \ref{QuasiActionRemarks}, any two $\Lambda$-quasi-orbits are at bounded Hausdorff distance. Similarly, the quasi-orbits of $\Lambda$ and $\Lambda^k$ are at bounded Hausdorff distance since $\Lambda \subset \Lambda ^k$ is of finite index by Proposition \ref{PropShuffling}.\end{proof}
\begin{definition}
The set  $\cL_\rho(\Lambda) \coloneqq  \cL(\rho(\Lambda).o)$ is called the \emph{Gromov limit set}  of the qiqac $\rho$.
\index{Gromov limit set!of qiqac}\index{qiqac!Gromov limit set of}
\end{definition}
By Lemma \ref{lem:limit set indep1} the Gromov limit set of $\rho$ is independent of the basepoint $o$.
\begin{construction}[Boundary action of a qiqac]\label{BoundaryAction1} Recall that every quasi-isometry $f: X \to X$ induces a homeomorphism $\partial f: \partial X \to \partial X$. Explicitly, if $(x_n)$ is a sequence in $X$ which represents $\xi \in  \partial X$, then the sequence $(f(x_n))$ represents $f(\xi)$. In particular, we may define a map
\[
\partial \rho\colon \Lambda^\infty \to \mathrm{Homeo}(\partial X), \quad \partial \rho(\lambda) \coloneqq  \partial(\rho(\lambda)).
\]
\end{construction}
\begin{lemma}[Boundary action]\label{BoundaryActionGromov} The map $\partial\rho$ is a group homomorphism and hence defines an action of $(\Lambda, \Lambda^{\infty})$ on $\partial X$ by homeomorphisms.
\end{lemma}
\begin{proof} Let $\xi=[(x_n)] \in \partial X$ and $\mu, \lambda \in \Lambda^\infty$. Since $\rho$ is a qiqac, there exists a constant $R$ (depending on $\lambda$ and $\mu$) such that for all $x \in X$ we have $d(\rho(\lambda)(\rho(\mu)(x)), \rho(\lambda \mu)(x)) < R$. We deduce that
\[
\partial \rho(\lambda)(\partial \rho(\mu)(\xi)) = [(\rho(\lambda)(\rho(\mu)(x_n))] = [(\rho(\lambda \mu)(x_n))] = \partial \rho(\lambda \mu)(\xi),
\]
and hence $\partial\rho(\lambda) \circ \partial \rho(\mu) = \partial \rho(\lambda \mu)$.
\end{proof}
\begin{definition} $\partial \rho: (\Lambda, \Lambda^\infty) \to \mathrm{Homeo}(\partial X)$ is called the \emph{boundary action} of the qiqac $\rho$.\index{qiqac!boundary action of}
\end{definition}
The following crucial observation follows essentially from Lemma \ref{lem:limit set indep1}:
\begin{proposition}[Invariance of the limit set]\label{GromovLimitSetInvariant} The limit set $\cL_\rho(\Lambda)$ is invariant under the boundary action $\partial \rho$, i.e.,
\[
\partial \rho(\Lambda^\infty)(\cL_\rho(\Lambda)) \subset \cL_\rho(\Lambda).
\]
\end{proposition}
\begin{proof} We fix a basepoint $o \in X$ so that $\cL_\rho(\Lambda) = \cL(\rho(\Lambda).o)$. Given $\mu \in \Lambda^\infty$ we can choose $k \in \mathbb N$ such that $\mu \in \Lambda^k$. Since $\rho$ is a qiqac, a similar argument as in the proof of Lemma \ref{BoundaryActionGromov} shows that $\rho(\mu)(\rho(\Lambda).o)$ is contained in a neighborhood of $\rho(\Lambda^{k+1}).o$. Since the quasi-orbits of $\Lambda^{k+1}$ are contained in a neighbourhood of the
quasi-orbits of $\Lambda$, the result follows. 
\end{proof}
For every qiqac of $(\Lambda, \Lambda^\infty)$ on a Gromov hyperbolic space $X$ we have thus managed to construct an action of $(\Lambda, \Lambda^\infty)$ by homeomorphism on the corresponding Gromov limit set. This construction is functorial in the following sense: Assume that $\rho'$ is another qiqac of $(\Lambda, \Lambda^\infty)$ on a proper geodesic hyperbolic space $Y$ and that $\phi: X \to Y$ is a quasi-isometry with quasi-inverse $\psi$ and that $\rho'$ is equivalent to $\phi \circ \rho \circ \psi$. In this case, $\phi(\rho(\Lambda).o)$ is at bounded Hausdorff distance from $\rho'(\Lambda).\phi(o)$, and hence $\partial \phi$ maps $\cL_{\rho}(\Lambda)$ homeomorphically onto $\cL_{\rho'}(\Lambda)$. We then obtain the following commuting diagram:
\begin{center} \begin{tikzcd}
  \partial X \arrow[rr, "{\partial \phi}"] 
    && \partial Y  \\
   \cL_\rho(\Lambda) \arrow[rr, "{\partial \phi|_{\cL_\rho(\Lambda)}}"] \arrow[u, hook]
&&\cL_{\rho'}(\Lambda) \arrow[u, hook] \end{tikzcd} \end{center}

We now want to establish similar results for qiqacs on arbitrary proper geodesic metric spaces, which are not necessarily Gromov hyperbolic.  A convenient choice of boundary for this purpose is the so-called \emph{Morse boundary}; we refer the reader to Appendix \ref{AppendixMorse} for background concerning this boundary.

From now on, $(X,d)$ is allowed to be an arbitrary proper geodesic metric space and $o \in X$ denotes a fixed choice of basepoint. We denote by $\partial_s^M X_{o}$ the sequential model of the Morse boundary of $X$ with basepoint $o$ (see Remark \ref{rem:seq morse boundary} for the definition). Since a change of basepoint induces a homeomorphism on the boundary, we will usually suppress the basepoint in the notation and write  $\partial_s^M X$ for $\partial_s^M X_{o}$. 
\begin{remark}[On the scope of our setting]
If $X$ is Gromov hyperbolic, then $\partial_s^M X = \partial_s X$ is just (the sequential model of) the Gromov boundary of $X$, and hence the results about Morse limit sets below reduce to the results concerning Gromov limit sets above.

Assume now that $X$ is not Gromov hyperbolic; then one of two cases will happen: If there are no Morse geodesics in $X$, then $\partial_s^M X_{o}$ is empty. This happens for example if every geodesic in $X$ bounds a half-flat, in particular in higher rank buildings and symmetric spaces. If there are Morse geodesics, then $\partial_s^M X_{o}$ is actually non-compact \cite[Theorem 5.7]{Liu:dynamics}. By \cite{sisto} this is the case if $X$ is quasi-isometric to an acylindrically hyperbolic group (for example a relatively hyperbolic group, a mapping class group, $\mathrm{Out}(F_n)$ or a non-directly decomposable right angled Artin group, cf.\ \cite{Osin_AcylHyp}), but there are examples outside this class, including finitely-generated groups such as Tarski monsters and infinite torsion groups \cite{OOS, fink}.
\end{remark}
In our more general setting, there is no canonical topology on $X \cup \partial_s^M X$, and hence we cannot define the limit set of a subset $Y \subset X$ as in the hyperbolic case. Instead, following  Definition \ref{defn:limit set}, we define the  \emph{Morse limit set}\index{Morse!limit set}\index{limit set!Morse} of $Y$ in $\partial_s^M X$ as the closed subset given by
\[\cL(Y)= \setcon{\xi \in \partial_s^M X}{\exists\small{\text{ Morse gauge }} N \text{ and } (y_n)\subset X^{(N)}_{o} \cap Y \text{ such that } \lim{y_n}= \xi}.\]
In the case that $X$ is hyperbolic, then $\partial_s^M X = \partial_s X$ is simply (the sequential model of) the Gromov boundary of $X$, and in this case the Morse limit set coincides with the Gromov limit set. In this sense, the Morse limit set generalizes the Gromov limit set.

As in the hyperbolic case, we observe that if two subsets $Y, Z \subset X$ are of bounded Hausdorff distance, then $\cL(Y) = \cL(Z)$. The proof of Lemma \ref{lem:limit set indep1} thus yields:
\begin{lemma} \label{lem:limit set indep}  For all $o, o' \in X$ and $k \in \bN$ we have
\[\pushQED{\qed}
\cL(\rho(\Lambda).o) = \cL(\rho(\Lambda).o') = \cL(\rho(\Lambda^k).o) \subset \partial_s^M X.\qedhere\popQED
\]
\end{lemma}
\begin{definition}\label{MorseLimitSet}
The set  $\cL_\rho(\Lambda) \coloneqq  \cL(\rho(\Lambda).o)$ is called the \emph{(Morse) limit set}\index{qiqac!limit set} of the qiqac $\rho$.
\end{definition}
Note that by Lemma \ref{lem:limit set indep},  the limit set of $\rho$ is independent of the basepoint $o$.

As in the hyperbolic case, the qiqac $\rho$ also gives rise to a boundary action of $\Lambda^\infty$ on the Morse boundary, which preserves the limit set, but the proof is slightly more involved: Firstly, by Corollary \ref{MorseBoundaryQIInvariant} every
quasi-isometry $f: X \to X$ induces a homeomorphism $\partial^M_s f: \partial^M_sX \to \partial^M_sX$, and hence we may define a map:
\[
\partial_s \rho\colon \Lambda^\infty \to \mathrm{Homeo}(\partial^M_s X), \quad \partial_s \rho(\lambda) \coloneqq  \partial_s(\rho(\lambda)).
\]
The following two lemmas show that this defines an action of $\Lambda^\infty$ which preserves the limit set; the proofs are based on the same ideas as the proofs of Lemma \ref{BoundaryActionGromov} and Lemma \ref{GromovLimitSetInvariant}; the only additional ingredient is the fact that quasi-isometries preserve Morse geodesics.
\begin{lemma}[Boundary action]\label{BoundaryActionMorse} The map $\partial_s\rho$ is a group homomorphism and hence defines an action of $(\Lambda, \Lambda^{\infty})$ on $\partial_s^M X$ by homeomorphisms.
\end{lemma}
\begin{proof} Let $\xi =[(x_n)] \in \partial_s^M X$ and $\lambda, \mu \in \Lambda^\infty$. We choose a Morse gauge $N$ such that $(x_n) \subset X^{(N)}_o$. Since $\rho$ is a qiqac we can argue as in the proof of Lemma \ref{BoundaryActionGromov} to find a constant $R$ (depending on $\lambda$ and $\mu$) such that for all $x \in X$ we have \begin{equation}\label{BoundaryActionMorseEq} 
    d(\rho(\lambda)(\rho(\mu)(x)),\rho(\lambda\mu)(x))<R.
\end{equation} To finish the proof we need to check that $[(\rho(\lambda)(\rho(\mu)(x_n)))]$ is equivalent to \\ $[(\rho(\lambda\mu)(x_n))]$ in the Morse boundary. In view of \eqref{BoundaryActionMorseEq} and the definition of the Morse boundary as a direct limit, this amounts to showing that 
there is a Morse gauge $N'$ such that $(\rho(\lambda)(\rho(\mu)(x_n)))$ and $(\rho(\lambda\mu)(x_n))$ are subsets of $X^{(N')}_o$. However, this follows from the fact that Morse geodesics are robust under quasi-isometries and change of basepoint.
\end{proof}

\begin{proposition}[Invariance of the limit set] \label{lem:limit set invariant} The limit set $\cL_\rho(\Lambda)$ is invariant under the boundary action $\partial_s \rho$, i.e.,
\[
\partial_s \rho(\Lambda^\infty)(\cL_\rho(\Lambda)) \subset \cL_\rho(\Lambda).
\]
\end{proposition} 
\begin{proof} Let $\mu \in \Lambda^\infty$ and $\xi=[(x_n)] \in  \cL_\rho(\Lambda)$. We can then choose $k \in \mathbb N$ and a Morse gauge $N$ such that $\mu \in \Lambda^k$ and $(x_n)$ is contained in $ X^{(N)}_o \cap \rho(\Lambda).o$. Since Morse geodesics are robust under quasi-isometry and change of basepoint, there exists a Morse gauge $N'$ so that the sequence $(\rho(\mu)(x_n))$ is contained in $X^{(N')}_o$ and represents $\partial_s \rho(\mu)(\xi)$.

Since $\rho$ is a qiqac, the same argument as in the proof of Lemma \ref{BoundaryActionMorse} now shows that this sequence is 
contained in a neighborhood of $\rho(\Lambda^{k+1}).o$, hence the result follows.
\end{proof}
\begin{definition} $\partial_s \rho: (\Lambda, \Lambda^\infty) \to \mathrm{Homeo}(\partial_s(X))$ is called the \emph{boundary action} of the qiqac $\rho$. \index{qiqac!boundary action of}
\end{definition}
As in the hyperbolic case, we have the following functoriality result: If $Y$ is another proper geodesic metric space and $\phi: X \to Y$ is a quasi-isometry with quasi-inverse $\psi$, then the qiqac $\rho$ induces a qiqac $\rho'$ on $Y$ by $\rho'(\lambda) \coloneqq  \phi \circ \rho(\lambda) \circ \psi$. As in the hyperbolic case we get that $\phi(\rho(\Lambda).o)$ is at bounded Hausdorff distance from $\rho'(\Lambda).\phi(o)$ and we obtain:
\begin{lemma}[Functoriality]\label{lem:equivariance of boundary maps}
In the situation above, the following diagram $\Lambda^\infty$-equivariantly commutes:
\begin{center} \begin{tikzcd}
  \partial^M_s X \arrow[rr, "{\partial_s \phi}"] 
    && \partial^M_s Y  \\
   \cL_\rho(\Lambda) \arrow[rr, "{\partial_s \phi|_{\cL_\rho(\Lambda)}}"] \arrow[u, hook]
&&\cL_{\rho'}(\Lambda) \arrow[u, hook] \end{tikzcd} \end{center}
In particular, $\partial_s\phi$ maps $\cL_\rho(\Lambda)$ homeomorphically onto $\cL_{\rho'}(\Lambda)$.
\end{lemma}
As a special case we observe that the homeomorphism type of the limit set is invariant under quasi-conjugacies. We conclude this section by discussing qiqacs $\rho$ which have \emph{full limit set} in the sense that $\cL_\rho(\Lambda) = \partial_s^M X$. An obvious class of examples is given by cobounded qiqacs:
\begin{example}[Limit sets of cobounded qiqacs] Assume that the qiqac $\rho$ is cobounded. In this case the quasi-orbit $\rho(\Lambda).o$ is at bounded Hausdorff distance from $X$ and hence
\[
\cL_\rho(\Lambda) = \cL(\rho(\Lambda).o) = \cL(X) = \partial_s^M X.
\]
In particular, every cobounded action has full limit set.
\end{example}
However, the limit set can also be full without the action being cobounded (even in the Gromov hyperbolic case).
\begin{example}[A small action with full limit set]\label{FullLimitSet} Let $\Gamma$ be a non-elementary hyperbolic group with finite generating set $S$, and the word metric $d_S$. It is known that in this case the action of $\Gamma$ on the Gromov boundary $\partial_s \Gamma$ of $(\Gamma, d_S)$ is minimal (\cite{gromov:1987aa}), i.e.\ every orbit is dense. Now let $(\Lambda, \Lambda^\infty)$ be an infinite approximate group with $\Lambda^\infty = \Gamma$ and consider the natural isometric action $\rho$ of $(\Lambda, \Lambda^\infty)$ on $(\Lambda^\infty, d_S)$ given by left-multiplication. Since $\Lambda$ is infinite, the Gromov limit set $\cL_\rho(\Lambda) \subset \partial_s \Gamma$ is non-empty. Since it is also closed and $\Lambda^\infty$-invariant, and since $\partial_s \Gamma$ is minimal, we deduce that $\cL_\rho(\Lambda)  = \partial_s \Gamma$, i.e.\ $\rho$ has full limit set. On the other hand, if $\Lambda$ does not have finite index in $\Lambda^\infty$, then $\rho$ is not cobounded.

A concrete example can be given as follows: Let $F$ be the free group on two generators $a$ and $b$ and let $f:F \to \Z$ denote the cyclic $ab$-counting quasimorphism (cf.\ Remark \ref{CyclicWords}), which by Remark \ref{FaizievGood} is the homogenization of the counting quasimorphism $\phi_{ab}$, hence in particular symmetric. To see that the quasikernel
\[
\Lambda := \{w\in F \mid |f(w)|\leq 1\} \subset F
\]
has full limit set in $\partial F$ we observe that for any $w \in F$ there exists a letter $s \in \{a,b,a^{-1}, b^{-1}\}$ such that $ws(ab)^{-f(w)}$ is reduced and $f(ws(ab)^{-f(w)})=0$, and hence $ws(ab)^{-f(w)} \in \Lambda$. Now think of elements of $\partial F$ as infinite reduced words over $\{a,b,a^{-1}, b^{-1}\}$. If $\xi = (\xi_n)_{n \in \bN}$ is such a word, then we thus find letters $s_n \in  \{a,b,a^{-1}, b^{-1}\}$ such that the words $w_n := \xi_1 \cdots \xi_n s_n (ab)^{-f(\xi_1 \cdots \xi_n)} \in \Lambda$ are reduced. We then have $w_n \to \xi$, and thus $\Lambda$ has full limit set.  
\end{example}

\section{Conical limit points}\label{SecConicalLimitPoints}
We keep the notation of the last section (see Notation \ref{BoundarySetting}). We will be concerned with a special kind of limit points, defined as follows:

\begin{definition} \label{def:conical limit point}
For $Y \subset X$, a limit point $\xi \in \cL(Y)$ is called a \emph{conical limit point of $Y$}\index{conical limit point}\index{limit point!conical} if for all geodesic rays $\gamma$ representing $\xi$ there exists $K>0$ and a strictly increasing sequence $k_n \in \mathbb{N}$ such that
\[ 
B(\gamma(k_n), K) \cap Y \neq \emptyset \quad \text{for all }n \in \bN.
\]  
The set of all conical limit points of $Y$ is denoted $\cL^{\mathrm{con}}(Y)$ and called the \emph{conical limit set}\index{limit set!conical}\index{conical limit set} of $Y$.
\end{definition}

It is immediate from the definition that if $Y, Z \subset X$ are two subspaces at bounded Hausdorff distance, then $\cL^{\mathrm{con}}(Y) = \cL^{\mathrm{con}}(Z)$. Given a qiqac $\rho$ of $(\Lambda, \Lambda^\infty)$ on $X$ we may thus say that a point $\xi \in \cL_\rho(\Lambda)$ is a \emph{conical limit point of $\rho$} if $\xi$ is a conical limit point of $\Lambda.o$ for some (hence any) $o \in X$. The \emph{conical limit set of $\rho$}\index{conical limit set! of qiqac} is 
\[
\cL^{\mathrm{con}}_\rho(\Lambda) \coloneqq  \{\xi \in \cL_{\rho}(\Lambda) \mid \xi \text{ is a conical limit point of }
\rho\}.
\]
If $\Xi \subset \Lambda^\infty$ is commensurable to $\Lambda$ we then have $\cL^{\mathrm{con}}_\rho(\Lambda) = \cL^{\mathrm{con}}_\rho(\Xi)$, in particular 
 $\cL^{\mathrm{con}}_\rho(\Lambda) = \cL^{\mathrm{con}}_\rho(\Lambda^n)$ for all $n \in \mathbb N$. From the latter property we deduce, by the same argument as in the case of $\cL_\rho$, the following invariance property of the conical limit set:
\begin{proposition}[Invariance of the conical limit set] 
The conical limit set $\cL^{\mathrm{con}}_\rho(\Lambda)$ is a $\partial \rho(\Lambda^\infty)$-invariant subset of $\cL_\rho(\Lambda)$.
\end{proposition}
One of the guiding questions for the remainder of this chapter is how different the conical limit set of a qiqac is from the actual limit set. We will later give conditions which ensure that every limit point is conical (see in particular Proposition \ref{cor:conical} below). Before we do so, however, we illustrate by explicit examples that in general the conical limit set contains much more information than the full limit set.

For the discussion of these examples we recall from Example \ref{QkerZ} that if $\Gamma$ is a countable group and $f: \Gamma \to \Z$ is a symmetric $\Z$-valued quasimorphism with defect $d(f)$, then for every $C \geq d(f)$ the subset
\[
\qker(f, C) \coloneqq\{x \in \Gamma \mid |f(x)| \leq C\}
\]
is a quasikernel for $f$. In particular, all these subsets are mutually commensurable approximate subgroups. These approximate subgroups then act on $\Gamma$ by left-multiplication, and if $\Gamma$ is finitely-generated then this action is isometric with respect to any choice of word metric. In particular, if $\Gamma$ is hyperbolic, then we can consider the conical limit set of $\qker(f)$ in the Gromov boundary of $\Gamma$.
\begin{proposition}[Conical limit sets of quasikernels]\label{ConicalQker} Let $\Gamma$ be a hyperbolic group with Gromov boundary $\partial \Gamma$ and let $f: \Gamma \to \mathbb{Z}$ be a symmetric quasimorphism and $C \geq d(f)$. Then the following hold:
\begin{enumerate}[(i)]
\item $\cL(\mathrm{qker}(f,C))$ and $\cL^{\mathrm{con}}(\mathrm{qker}(f,C))$ are independent of $C$.
\item $\cL(\mathrm{qker}(f,C)) = \partial \Gamma$.
\item $\xi \in \cL^{\mathrm{con}}(\mathrm{qker}(f,C))$ if and only if $\varliminf_{t \to \infty} |f(\gamma(t))| < \infty$ for some (hence every) geodesic ray $\gamma$ representing $\xi$.
\end{enumerate}
\end{proposition}
\begin{proof} (i) This is immediate from the fact that the approximate subgroups $\qker(f, C)$ for $C \geq d(f)$ are mutually commensurable.

(ii) By (i) we may enlarge $C$ to guarantee that $|f(s)| \leq C$ for all $s$ in some finite generating system $S$ of $\Gamma$. Then $\qker(f, C)$ generates $\Gamma$ and hence (ii) is just a special case of Example \ref{FullLimitSet}. 

(iii) We fix a finite generating set $S$ for $\Gamma$ and realize $\partial \Gamma$ as the boundary of the Cayley graph of $\Gamma$ with respect to $S$. Since any two geodesic rays representing $\xi$ are at bounded Hausdorff distance and $f$ is Lipschitz by Remark \ref{quasiremark}.(v), the condition $\varliminf_{t \to \infty} |f(\gamma(t))| < \infty$ is independent of the choice of $\gamma$.

Assume now that $\xi$ is a conical limit point which is represented by a geodesic $\gamma$ emanating from the identity. Choose a strictly increasing sequence $(k_n)$ in $\bN$ such that
\[
B(\gamma(k_n), K) \cap \qker(f, C) \neq \emptyset \quad \text{for all }n \in \bN.
\]
Since $f$ is Lipschitz and bounded along its stable quasikernel, this implies that $|f(\gamma(k_n))|$ is bounded, and hence 
\[
 \varliminf_{t \to \infty} |f(\gamma(t))| \leq \lim_{n \to \infty}|f(\gamma(k_n))| <\infty.\]
Conversely let $\gamma$ be a geodesic ray representing $\xi$ and assume that $\varliminf_{t \to \infty} |f(\gamma(t))| < \infty$. There then exists $C' > 0$ and a sequence $t_n \nearrow \infty$ such that $|f(\gamma(t_n))| \leq C'$ for all $n \in \bN$. By (i) we may enlarge $C$ to ensure that $C \geq C'$. This ensures that $\gamma(t_n) \in \qker(f, C)$ for all $n \in \bN$ and hence $\xi$ is a conical limit point of $\qker(f, C)$.
\end{proof}

 In order to illustrate Proposition \ref{ConicalQker}
it is instructive to consider the case of homogeneous counting quasimorphisms in free groups (see Example \ref{HomogeneousCounting}). 
\begin{example}[A class of homogeneous counting quasimorphisms]
We consider the situation of Notation \ref{SpecialCounting}: $F_r$ is a non-abelian free group with free generating set  $S \coloneqq  \{a_1, \dots, a_r\}$ and $u \in W_r$ is a cyclically reduced non-self-overlapping word of length $\|u\|_S \geq 2$. We then consider an arbitrary quasikernel $\Lambda_u \subset F_r$ of the cyclic counting quasimorphism $\phi^{\mathrm{cyc}}_u: F_r \to \bZ$. By Remark \ref{FaizievGood} the latter is the homogenization of the counting quasimorphism $\phi_u$, in particular conjugation-invariant and symmetric. For concretness' sake we may choose $\Lambda_u \coloneqq \qker(\phi^{\mathrm{cyc}}_u, d(\phi^{\mathrm{cyc}}_u))$, although the precise value of the constant will not matter at all.

Since $u$ is non-self-overlapping and of length $\geq 2$, its initial and final letter are distinct, and hence $\phi_u(s) = 0$ for all $s \in S$. This shows, firstly, that $\phi_u$ is not a homomorphism (since $\phi_u|_S = 0$ and $\phi_u(u) = 1$) and hence $d(\phi_u) \geq 1$ and, secondly, that $S \subset \Lambda_u$, and hence $\Lambda_u$ generates $F_r$.

The action of $F_r$ on its Cayley tree $T_{2r}$ with respect to $S$ now induces an isometric action of $(\Lambda_u, \Lambda^\infty_u)$ on $T_{2r}$. Since $T_{2r}$ is hyperbolic, we can consider the limit set $\cL(\Lambda_u) \subset \partial T_{2r}$ as well as the corresponding conical limit set $\cL^{\mathrm{con}}(\Lambda_u) \subset \partial T_{2r}$. Since $\phi^{\mathrm{cyc}}_u$ is homogeneous, for every $w \in F_r$ we have
\[
\lim_{n \to \infty} |\phi^{\mathrm{cyc}}_u(w^n)| = \left\{\begin{array}{rl} \infty, & \text{if } \phi^{\mathrm{cyc}}_u(w) \neq 0\\ 0 & \text{if } \phi^{\mathrm{cyc}}_u(w) = 0\end{array}\right.
\]
Thus if we denote by $w^\infty$ the endpoint of the unique geodesic ray from the origin through the vertices $w^n$, then by Proposition \ref{ConicalQker} we have
\begin{equation}\label{ConicalPowers}
w^\infty \in \cL^{\mathrm{con}}(\Lambda_u) \iff \phi^{\mathrm{cyc}}_u(w) = 0.    
\end{equation}
\end{example}
\begin{theorem}[Recovering quasikernels from their conical limit sets]\label{QuasikernelsConical} Let $F_r$ be a free group of rank $r\geq 3$ with free generating set $S$ and let $u,v$ be two cyclically reduced, non-self-overlapping words in $F_r$ of lengths $|u|_S, |v|_S \geq 2$. Then the following hold:
\begin{enumerate}[(i)]
\item $\cL(\Lambda_u) = \cL(\Lambda_v) = \partial T_{2r}$.
\item $\cL^{\mathrm{con}}(\Lambda_u) = \cL^{\mathrm{con}}(\Lambda_v)$ if and only if $\Lambda_u = \Lambda_v$ if and only if $u=v^{\pm 1}$.
\end{enumerate}
Thus cyclically reduced non-self-overlapping words of length $\geq 2$ in $F_r$ are uniquely determined up to inverse by the conical limit sets of the associated cyclic counting quasimorphisms, whereas the full limit set is independent of the chosen word.
\end{theorem}
\begin{remark}[Concerning the proof of Theorem \ref{QuasikernelsConical}] 
Part (i) of Theorem \ref{QuasikernelsConical} is just a special case of Example \ref{FullLimitSet}. Concerning (ii) we have the obvious implications
\[
u=v^{\pm 1} \implies \Lambda_u = \Lambda_v \implies \cL^{\mathrm{con}}(\Lambda_u) = \cL^{\mathrm{con}}(\Lambda_v).
\]
The non-trivial implication $\cL^{\mathrm{con}}(\Lambda_u) = \cL^{\mathrm{con}}(\Lambda_v) \implies u = v^{\pm 1}$ is based on three ingredients: our starting point is Proposition \ref{ConicalQker} (or rather its consequence \eqref{ConicalPowers}), which allows us to identify the conical limit sets in question. In addition, we will need a number of rather specific combinatorial arguments concerning free groups (summarized in Lemma \ref{Quasi1} and Lemma \ref{Quasi2} below) and a reconstruction theorem for quasimorphisms (Corollary \ref{Reconstruction}) due to Gabi Ben Simon and the second author.
\end{remark}
\begin{proof}[Proof of Theorem \ref{QuasikernelsConical}] Let $u,v$ be as in the theorem and assume that $\cL^{\mathrm{con}}(\Lambda_u)= \cL^{\mathrm{con}}(\Lambda_v)$. We then have to show that $v = u$ or $v=u^{-1}$. Throughout the proof we are going to abbreviate $\cL \coloneqq  \cL^{\mathrm{con}}(\Lambda_u)= \cL^{\mathrm{con}}(\Lambda_v)$. We also set $N_u \coloneqq  \{w \in F_r \mid \phi^{\mathrm{cyc}}_u(w) \neq 0\}$ and define an equivalence relation $\sigma_u$ on $N_u$ by $w \sigma_u w' :\iff \phi^{\mathrm{cyc}}_u(w)\phi^{\mathrm{cyc}}_u(w') > 0$. The set $N_v$ and equivalence relation $\sigma_v$ will be defined accordingly. We make the following two claims:\\
\textbf{Claim 1:} $N_u = N_v$.\\
\textbf{Claim 2:} $\sigma_u = \sigma_v$.\\
Assuming the two claims for the moment, we complete the proof of the theorem: By the first claim we have $N_u = N_v$ and we denote this set by $N$. Moreover, the equivalence classes of $\sigma_u$ are given by $\mathrm{Pos}(\phi^{\mathrm{cyc}}_u)=\{w \in F_r \mid \phi^{\mathrm{cyc}}_u(w)>0\}$ and 
$\mathrm{Pos}(-\phi^{\mathrm{cyc}}_u)$, and similarly the equivalence classes of $\sigma_v$ are given by $\mathrm{Pos}(\phi^{\mathrm{cyc}}_v)$ and $\mathrm{Pos}(-\phi^{\mathrm{cyc}}_v)$. We thus deduce from Claim 2 that either $\mathrm{Pos}(\phi^{\mathrm{cyc}}_v) = \mathrm{Pos}(\phi^{\mathrm{cyc}}_u)$ or $\mathrm{Pos}(\phi^{\mathrm{cyc}}_v) = \mathrm{Pos}(-\phi^{\mathrm{cyc}}_u)$. It then follows from the reconstruction theorem (Corollary \ref{Reconstruction}) that $\phi^{\mathrm{cyc}}_u$ and $\phi^{\mathrm{cyc}}_v$ are linearly dependent. However, according to \cite{HartnickTalambutsa} this is only possible if $v = u$ or $v = u^{-1}$. It thus remains to establish the two claims.

Claim 1 is immediate from \eqref{ConicalPowers}. For the proof of Claim 2 we set $N\coloneqq  N_u = N_v$.
and $C_0 \coloneqq  \max\{C_u, C_v\}$, where $C_u$ and $C_v$ are defined as in Lemma \ref{Quasi2}; we then define
\[C_1 \coloneqq  \max_{f \in \{\phi^{\mathrm{cyc}}_u, \phi^{\mathrm{cyc}}_v\}}\left(\sup_{x \in B(e, C_0)} |f(x)| + \frac 3 2  d(f) + 1\right).\]
We first observe that $x\sigma_u y$ if and only if $x' \sigma_u y'$ for some (hence all) $x', y' \in F_r$ which are conjugate to some positive power of $x$ and $y$ respectively; the same holds for $\sigma_v$ instead of $\sigma_u$. We now define 
\[
N' \coloneqq  \{w \in N\mid w \text{ cyclically reduced},\, \|w\|_S > C_0, \min\{|\phi^{\mathrm{cyc}}_u(w)|, |\phi^{\mathrm{cyc}}_v(w)|\} > C_1\}.
\]
Since $w \in N$ if and only if some conjugate of a positive power of $w$ lies in $N'$, it suffices to show that $\sigma_u$ and $\sigma_v$ agree on $N'$. 
Thus let $x,y \in N'$ and assume that $x \cancel{\sigma}_u y$; thus either $\phi^{\mathrm{cyc}}_u(x)>0>\phi^{\mathrm{cyc}}_u(y)$ or $\phi^{\mathrm{cyc}}_u(y)>0>\phi^{\mathrm{cyc}}_u(x)$. 

We first consider the case in which $\phi^{\mathrm{cyc}}_u(x)>0>\phi^{\mathrm{cyc}}_u(y)$. Since $\phi^{\mathrm{cyc}}_u$ is homogeneous and takes integer values we can then find $a,b \in \mathbb N$ such that 
\begin{equation}\label{xayb0}
\phi^{\mathrm{cyc}}_u(x^a) = -\phi^{\mathrm{cyc}}_u(y^b) > 0.
\end{equation}
This implies in particular that $\#^{\mathrm{cyc}}_u(x^a) >0$ and 
$\#^{\mathrm{cyc}}_{u^{-1}}(y^b) >0$. Moreover, $x^a$ and $y^b$ are cyclically reduced and $\|x^a\|_S > \|u\|_S$ and $\|y^b\|_S > \|u^{-1}\|_S$. It then follows from Lemma \ref{Quasi1} (applied once to the pair $(x^a, u)$ and once to the pair $(y^b, u^{-1})$ that there exist $x', y' \in F_r$ which are cyclically equivalent to $x^a$, respectively $y^b$ such that 
\begin{equation}\label{x'y'}
\phi^{\mathrm{cyc}}_u(x') = \phi_u(x') \qand \phi^{\mathrm{cyc}}_u(y') = \phi_u(y').
\end{equation}
In particular, $x'$ is conjugate to $x^a$ and $y'$ is conjugate to $y^b$ and $\min\{\|x'\|_S, \|y'\|_S\} \geq \|u\|$. By Lemma \ref{Quasi2} we now find $w, z \in F_r$ with the following properties:
\begin{itemize}
    \item $w,z \in B(e, C_u) \subset B(e,C_0)$;
    \item the product $x'wy'z$ is reduced and cyclically reduced;
    \item every copy of $u$ or $u^{-1}$ in $(x'wy'z)^n$ is contained in one of the $n$ copies of $x'$ or in one of the $n$ copies of $y'$.
\end{itemize}
Combining these properties with \eqref{xayb0} and \eqref{x'y'} we deduce that for all $n \in \mathbb N$ we have
\begin{eqnarray*}
\phi_u((x'wy'z)^n) &=& n \phi_u(x') + n \phi_u(y') = n \phi^{\mathrm{cyc}}_u(x') + n \phi^{\mathrm{cyc}}_u(y')\\
&=& n \phi^{\mathrm{cyc}}_u(x^a) + n \phi^{\mathrm{cyc}}_u(y^b) = 0.
\end{eqnarray*}
This implies that $\phi^{\mathrm{cyc}}_u(x'wy'z) = 0$ and hence $(x'wy'z)^\infty \in \cL$ by \eqref{ConicalPowers}. We have thus proved that if $x \cancel{\sigma}_u y$ and $\phi^{\mathrm{cyc}}_u(x)>0>\phi^{\mathrm{cyc}}_u(y)$, then 

$(*)$ \emph{there exist powers $x'$ and $y'$ of conjugates of $x$ and $y$ respectively and words $w,z$ of length $\leq C_0$ such that $x'wy'z$ is reduced and cyclically reduced and its powers converge to an element $\xi \in \cL$.}

If instead we have $\phi^{\mathrm{cyc}}_u(y)>0>\phi^{\mathrm{cyc}}_u(x)$, then $(*)$ holds after reversing the roles of $x$ and $y$. We now claim that, conversely, if Condition $(*)$ holds for either the pair $(x,y)$ or the pair $(y,x)$, then $x \cancel{\sigma}_u y$. Since the relation $\sigma_u$ is symmetric, we may assume that $(*)$ holds for $(x,y) \in (N')^2$. Also assume for contradiction that $x \sigma_u y$. Now $\phi^{\mathrm{cyc}}_u(x')$ and $\phi^{\mathrm{cyc}}_u(y')$ have the same sign (since $\phi^{\mathrm{cyc}}_u(x)$ and $\phi^{\mathrm{cyc}}_u(y)$ have the same sign) and $|\phi^{\mathrm{cyc}}_u(x')|$ and $|\phi^{\mathrm{cyc}}_u(y')|$ are $\geq C_1$ (since
$|\phi^{\mathrm{cyc}}_u(x)|$ and $|\phi^{\mathrm{cyc}}_u(y)|$ are). With  notation as in $(*)$ we then have
\begin{eqnarray*}
|\phi^{\mathrm{cyc}}_u(x'wy'z)| &\geq&|\phi^{\mathrm{cyc}}_u(x') + \phi^{\mathrm{cyc}}_u(y')| - |\phi^{\mathrm{cyc}}_u(w)| - |\phi^{\mathrm{cyc}}_u(z)| - 3 d(\phi^{\mathrm{cyc}}_u) \\&=&|\phi^{\mathrm{cyc}}_u(x')| + |\phi^{\mathrm{cyc}}_u(y')| - |\phi^{\mathrm{cyc}}_u(w)| - |\phi^{\mathrm{cyc}}_u(z)| - 3 d(\phi^{\mathrm{cyc}}_u) \\ &\geq& 2C_1 - 2 \sup_{x \in B(e, C_0)} |\phi^{\mathrm{cyc}}_u(x)| - 3d(\phi^{\mathrm{cyc}}_u) >0
\end{eqnarray*}
by the very definition of $C_1$. By \eqref{ConicalPowers} this implies that $\xi \not \in \cL$, which is a contradiction. The same arguments can also applied to $v$ instead of $u$. We thus obtain for all $x,y \in N''$ the equivalences 
\[
x\sigma_u y \iff \text{Condition $(*)$ holds neither for $(x,y)$ nor for $(y,x)$}\iff x \sigma_v y. \qedhere\]
\end{proof}

\section{Qiqacs with stable quasi-orbits}\label{SecStableQO}

We keep the notation of the previous sections (see Notation \ref{BoundarySetting}). We observe that if the quasi-orbit $\rho(\Lambda).o$ is quasi-convex in $X$, then every quasi-orbit of $X$ is quasi-convex in $X$; we then simply say that $\rho$ \emph{has quasi-convex quasi-orbits}. For proper qiqacs on \emph{hyperbolic} spaces having quasi-convex quasi-orbits is a rather strong condition with strong consequences for the limit set:
\begin{proposition}[Consequences of quasi-convex quasi-orbits, hyperbolic case] \label{BoundaryStableOrbitHyp} If $X$ is Gromov hyperbolic, $(\Lambda, \Lambda^\infty)$ is geometrically finitely-generated and $\rho$ is proper and has quasi-convex quasi-orbits, then the following hold:
\begin{enumerate}[(i)]
    \item $\cO \coloneqq  \rho(\Lambda).o\in [\Lambda]_{\mathrm{int}}$.
    \item $(\Lambda, \Lambda^\infty)$ is a hyperbolic approximate group.
    \item If $A$ is an apogee for $\Lambda$, then there is a quasi-isometric embedding $\iota: A \to X$ whose image is given by $\cO$.
    \item Any map $\iota$ as in (iii) induces a continuous embedding $\partial \iota: \partial A \hookrightarrow \partial X$ whose image is given by $\cL_\rho(\Lambda)$.
\end{enumerate}
In particular $\cL_\rho(\Lambda)$ is a representative of $\partial \Lambda$.
\end{proposition}
\begin{proof} (i) We first note that since $\cO \subset X$ is quasi-convex, it is quasi-geodesic by Remark \ref{QuasiconvexQuasigeodesic}, and hence large-scale geodesic by Lemma \ref{lsgquasigeodesic}. We also note that $\cO$ is quasi-invariant under $\rho(\Lambda)$ and denote by $\widetilde{\rho}$ the quasi-restriction of $\rho$ to $\cO$ (cf. Construction \ref{Quasirestriction2}).

Since $\rho$ is proper, this quasi-restriction is proper as well, and by definition it is cobounded, hence geometric. Since $\Lambda$ is geometrically finitely-generated and $\cO$ is large-scale geodesic we may thus apply the Milnor--Schwarz Lemma (Theorem \ref{ThmMS}) to conclude that $\cO \in [\Lambda]_{\mathrm{int}}$.

(ii) follows from (i) and Corollary \ref{QCMorseHyp}.

(iii) follows from (i).

(iv) follows from (i) and Corollary \ref{BoundaryStableSubsetHyp}.
\end{proof}
If we remove the assumption that $X$ be Gromov hyperbolic, then having quasi-convex orbits by itself will no longer allow one to draw any similar conclusions. However, we can still establish a version of Proposition \ref{BoundaryStableOrbitHyp} if we replace quasi-convexity by the stronger condition of \emph{stability}: Recall from Definition \ref{DefStable} that a quasi-convex subset $Y \subset X$ is called $N$-stable\index{stable subset} (with respect to some Morse gauge $N$) provided every pair of points in $Y$ can be connected by an $N$-Morse geodesic in $X$. We say that the qiqac $\rho$ \emph{has stable quasi-orbits} if some quasi-orbit is $N$-stable for some Morse gauge $N$ (and hence every quasi-orbit is stable with a Morse gauge which may depend on the given quasi-orbit).
\begin{proposition}[Consequences of stable quasi-orbits] \label{BoundaryStableOrbit} \label{lem: stable implies hyperbolic approx group} If $(\Lambda, \Lambda^\infty)$ is geometrically finitely-gene\-ra\-ted and $\rho$ is proper and has stable quasi-orbits, then the following hold:
\begin{enumerate}[(i)]
    \item $\cO \coloneqq  \rho(\Lambda).o\in [\Lambda]_{\mathrm{int}}$.
    \item $(\Lambda, \Lambda^\infty)$ is a hyperbolic approximate group.
    \item If $A$ is an apogee for $\Lambda$, then there is a quasi-isometric embedding $\iota: A \to X$ whose image is given by $\cO$.
    \item Any map $\iota$ as in (iii) induces a continuous embedding $\partial \iota: \partial A \hookrightarrow \partial X$ whose image is given by $\cL_\rho(\Lambda)$.
\end{enumerate}
In particular, $\cL_\rho(\Lambda)$ is a representative of $\partial \Lambda$, and hence compact.
\end{proposition}
\begin{proof} (i) Literally as in Proposition \ref{BoundaryStableOrbitHyp}.

(ii) follows from (i) and the fact that a stable subspace is Morse hyperbolic (see Proposition \ref{BoundaryStableSubset}).

(iii) follows from (i).

(iv) follows from (i) and Proposition \ref{BoundaryStableSubset}.
\end{proof}

\section{Convex cocompact qiqacs}\label{Section:ConvexCocompact}
We keep the notation of the previous sections (see Notation \ref{BoundarySetting}). In addition we are going to assume that $\Lambda$ is infinite and that $\rho$ is proper. These assumptions imply in particular that all quasi-orbits of $\Lambda$ are unbounded. The main goal of this section is to characterize proper qiqacs with stable quasi-orbits geometrically. It will turn out they are precisely the ``convex cocompact'' qiqacs and that all of their limit points are conical.

As the name suggests, a ``convex cocompact'' qiqac is a qiqac which quasi-restricts to a cobounded qiqac on a quasi-convex subset of $X$. To construct such a subset we need to recall from Definition \ref{def:weak hull} the notion of the weak hull $\mathfrak{H}(Z) \subset X$ of a subset $Z \subset \partial^M_s X$: By definition, $\mathfrak{H}(Z)$ is the union of all bi-infinite geodesic lines which are bi-asymptotic to a pair of distinct points in $Z$. The following observation shows that - in the generality of our present setting (cf. Notation \ref{BoundarySetting}) - the weak hull of the limit set of $\rho$ provides a candidate for a $\Lambda$-cobounded subspace of $X$, provided the limit set is compact:
\begin{proposition}[Weak hull of the limit set]\label{WeakHullQuasiInvariant} If $\mathcal{L}_\rho(\Lambda)$ is compact, then $\mathfrak{H}(\mathcal{L}_\rho(\Lambda))$ is $\Lambda$-quasi-invariant. \end{proposition}
Since the limit set is $\Lambda^\infty$-invariant by Proposition \ref{lem:limit set invariant}, the proposition is actually a special case of the following general observation:

\begin{lemma}\label{lem:invariant boundary invariant wh} Assume a set $Z \subset \partial_s^M X$ is compact and $\Lambda^\infty$-invariant. Then  $\mathfrak{H}(Z)$ is $\Lambda$-quasi-invariant. \end{lemma} 
\begin{proof} We need to show that there exists a constant $D\geq 0$ so that for any $\lambda \in \Lambda$ we have $d_{\mathrm{Haus}}(\mathfrak{H}(Z), \rho(\lambda).\mathfrak{H}(Z))<D$. Thus let $\lambda \in \Lambda$ and $x \in \mathfrak{H}(Z)$; there then exist $\xi^-, \xi^+ \in Z$ with $\xi^- \neq \xi^+$ and a geodesic line $\gamma$ bi-asymptotic to $(\xi^-, \xi^+)$ such that $x \in \gamma(\R)$.

We first observe that since $Z$ is compact we have $Z \subset \partial_s^N X$ for some fixed Morse gauge $N$. By Theorem \ref{IdealMorseTriangle}.(ii) we deduce that $\gamma$ is $N'$-Morse for some Morse-gauge $N'$ depending only on $N$. Let $\gamma' \coloneqq  \rho(\lambda).\gamma$; this is an $N''$-Morse $(K,C)$-quasi-geodesic, where $N''$, $K$ and $C$ are independent of $\gamma$ and $\lambda$ and depend only on $N'$ (hence $N$) and the QI constants of $\rho$ (see \cite[Lemma 2.5 (2)]{charney:2015aa}). We set $D\coloneqq 2N''(K,C)$.

Let $\gamma'_n$ be a geodesic segment joining $\gamma'(-n)$ and $\gamma'(n)$. By \cite[Lemma 2.5 (3)]{charney:2015aa} the geodesics $\gamma'_n$ are bounded Hausdorff distance from $\gamma'([-n,n])$, where the bound depends only on $N, K, C$ and thus by \cite[Lemma 2.5 (1)]{charney:2015aa} are $N'''$-Morse, where $N'''$ depends only on $N, K, C$. By the Arzel\`a--Ascoli Lemma \ref{ArzelaAscoli1} and \cite[Lemma 2.10]{cordes:2016ad}, this converges to an $N'''$-Morse bi-infinite geodesic $\gamma''$ that is of Hausdorff distance at most $D$ from $\gamma'$. In particular, $\rho(\lambda)(x)$ is of distance at most $D$ from $\gamma''$. Since $\gamma''$ is bi-asymptotic to $(\partial_s\rho(\lambda)(\xi^-), \partial_s\rho(\lambda)(\xi^+)) \in Z^2$, hence contained in $\mathfrak{H}(Z)$, the lemma follows.
\end{proof}

\begin{remark}[Quasi-restriction to the weak hull] \label{weak hull restricted quasi-action}
Assume that $\mathcal{L}_\rho(\Lambda)$ is compact so that the weak hull is quasi-invariant under $\rho(\Lambda)$. We can then quasi-restrict the qiqac $\rho$ to a qiqac $\widetilde{\rho}_{\mathfrak{H}(\Lambda)}$ on $\mathfrak{H}(\mathcal{L}_\rho(\Lambda))$ (see Construction \ref{Quasirestriction2}). Then, by construction,  $\widetilde{\rho}_{\mathfrak{H}(\Lambda)}$ is a qiqac on $\mathfrak{H}(\mathcal{L}_\rho(\Lambda))$ and it is at bounded distance from $\rho$ in the following sense: for any $k \in \mathbb{N}$ there exists a constant $R_k$ so that for any $\lambda \in \Lambda^k$ and $x \in X$ we have
\begin{equation}\label{QuasiRestrictionClose}
d(\widetilde{\rho}_{\mathfrak{H}(\Lambda)}(\lambda).x, \rho(\lambda).x)< R_k.
\end{equation}
\end{remark}

\begin{definition} Let $(\Lambda, \Lambda^\infty)$ be an infinite geometrically finitely-generated approximate group and let $\rho\colon (\Lambda, \Lambda^\infty) \to \widetilde{{\rm QI}}(X)$ be a qiqac on a proper geodesic space $X$. We say that $\rho$ is \emph{convex cocompact}\index{qiqac!convex cocompact} if:
\begin{enumerate}[(CC1)]
\item $\Lambda$ quasi-acts properly on $X$ via $\rho$;
\item $\cL_\rho(\Lambda) \subset \partial_s^M X$ is non-empty and compact; 
\item $\Lambda$ quasi-acts coboundedly on $\mathfrak{H}(\mathcal{L}_\rho(\Lambda))$ via 
$\widetilde{\rho} \coloneqq  \widetilde{\rho}_{\mathfrak{H}(\Lambda)}$.
\end{enumerate}
\end{definition}
It will sometimes be convenient to use the following equivalent reformulation of Condition (CC3):
\begin{itemize}
    \item[(CC3$'$)] There exists $D>0$ such that for every $z \in \mathfrak{H}(\mathcal{L}_\rho(\Lambda))$ there is some $\lambda \in \Lambda$ with $d(\rho(\lambda).o, z) < D$.
\end{itemize}
The equivalence of (CC3) and (CC3$'$) is immediate from \eqref{QuasiRestrictionClose}. We can now state the main theorem of this section; in the group case, a similarly formulated equivalence is proved in \cite{cordes:2016ab}.
\begin{theorem}[Convex cocompactness vs.\ stable quasi-orbits] \label{thm:stab equiv to cocompact}
Suppose that $X$ is a proper geodesic metric space,  $(\Lambda, \Lambda^\infty)$ is a geometrically finitely-generated approximate group, and $\rho\colon (\Lambda, \Lambda^\infty) \to \widetilde{{\rm QI}}(X)$ is a qiqac. Then the following are equivalent:
\begin{enumerate}[(i)]
    \item $\rho$ is convex cocompact.
    \item $\rho$ is proper and has stable quasi-orbits.
\end{enumerate}
In this case, $(\Lambda, \Lambda^\infty)$ is a hyperbolic approximate group and $\cL_\rho(\Lambda)$ is compact and represents $\partial \Lambda$.
 \end{theorem}
\begin{proof}
(i)$\implies$(ii): Since $\rho$ is proper by (CC1), we only need to show stability of $\cO \coloneqq  \rho(\Lambda).o$, where we may assume that $o \in \mathfrak{H}(\cL_\rho(\Lambda))$. By \eqref{QuasiRestrictionClose}, this is equivalent to showing stability of $\cO' \coloneqq  \widetilde{\rho}(\Lambda).o$. In view of (CC3) it thus suffices to show that $\mathfrak{H}(L)$ is stable.
Since $L \coloneqq  \cL_\rho(\Lambda)$ is compact by (CC2), the latter follows from Proposition \ref{WeakHullQuasiconvex}.

(ii)$\implies$(i): (CC1) holds by assumption and compactness of $\cL_\rho(\Lambda)$ follows from Proposition \ref{BoundaryStableOrbit}. Also, since $\Lambda$ is infinite and $\rho$ is proper, quasi-orbits are \emph{unbounded} stable sets and hence their limit set is non-empty. This proves (CC2), and it remains to show (CC3), or equivalently (CC3$'$).

For this we set $\cO \coloneqq  \rho(\Lambda).o$ and $L \coloneqq  \cL(\cO)$. We also fix an apogee $A$ for $(\Lambda, \Lambda^\infty)$. By Proposition \ref{BoundaryStableOrbit} there exists a $(K,C)$-quasi-isometric embedding $\iota: A \to X^{(N)}_o$ whose image is given by $\cO$ and which extends to a homeomorphism $\partial \iota: \partial A \to L$. 

Now assume that $z \in \mathfrak{H}(L)$. Then there exist $\xi^-, \xi^+ \in L$ and a geodesic line $\gamma$ bi-asymptotic to $(\xi^-, \xi^+)$ such that $z \in \gamma(\R)$. Now let $\gamma'$ denote a geodesic line in $A$ which is bi-asymptotic to $((\partial \iota)^{-1}(\xi^-), (\partial\iota)^{-1}(\xi^+))$. Then $\gamma'' \coloneqq  \iota(\gamma')$ is a $(K,C)$-quasi-geodesic line in $\cO$. Since $\cO \subset X^{(N)}_o$,  the quasi-geodesic $\gamma''$ is $N'$-Morse for some $N'$ depending only on $N$ by Theorem \ref{thm:strongly hyperbolic equiv}. By the same Arzel\`a--Ascoli argument as in the proof of Lemma \ref{lem:invariant boundary invariant wh} we thus find a geodesic line $\gamma'''$ in $X$ which is at bounded Hausdorff distance from $\gamma''$. In particular, $\gamma'''$ is bi-asymptotic to $(\xi^-, \xi^+)$, hence also at bounded Hausdorff distance from $\gamma$ by Proposition \ref{prop:limit geodesics are asymptotic}, where again the bound depends only on $N$. In particular, $d(z, \gamma'') \leq d(\gamma, \gamma'') < D$ and hence $\mathfrak{H}(L) \subset N_D(\cO)$ for some constant $D$ depending only on $N$. This finishes the proof.

We have thus established (i)$\iff$(ii), and the final sentence follows from Proposition \ref{BoundaryStableOrbit}.
\end{proof}
\begin{remark} Combining Proposition \ref{BoundaryNonel} and Theorem \ref{thm:stab equiv to cocompact} we deduce that the limit set of any convex cocompact action of a non-elementary hyperbolic approximate group is a finite-dimensional perfect compact space, which can be metrized by a locally quasi-self-similar and doubling metric.
\end{remark}
It seems worthwhile to record the following special case of Theorem \ref{thm:stab equiv to cocompact}; in the group case, this result is due to Swenson \cite{Swenson}.
\begin{corollary}[Convex cocompactness vs.\ quasi-convex quasi-orbits] \label{thm:quasi-convex equiv to cocompact}
Suppose that $X$ is a proper geodesic hyperbolic space,  $(\Lambda, \Lambda^\infty)$ is a geometrically finitely-generated approximate group, and $\rho\colon (\Lambda, \Lambda^\infty) \to \widetilde{{\rm QI}}(X)$ is a qiqac. Then the following are equivalent:
\begin{enumerate}[(i)]
    \item $\rho$ is convex cocompact.
    \item $\rho$ is proper and has quasi-convex quasi-orbits.
\end{enumerate}
In this case, $(\Lambda, \Lambda^\infty)$ is a hyperbolic approximate group and $\cL_\rho(\Lambda)$ represents $\partial \Lambda$.
 \end{corollary}

It turns out that convex cocompactness of a qiqac is enough to ensure that all limit points are conical:
\begin{proposition}[Limit points are conical] \label{cor:conical}
If $\rho\colon (\Lambda, \Lambda^\infty) \to \widetilde{{\rm QI}}(X)$ is a convex cocompact qiqac on a proper geodesic metric space $X$, then every limit point of $\rho$ is conical. 
\end{proposition}
\begin{proof} We may assume that $\Lambda$ is infinite, since otherwise the proposition holds vacuously. By Proposition \ref{GromovBoundaryTri} this implies that $|\partial \Lambda| \geq 2$. By Theorem \ref{thm:stab equiv to cocompact} we may now fix a Morse gauge $N$ such that $\cO \coloneqq  \rho(\Lambda).o$ is $N$-stable. Denote $L \coloneqq  \cL(\cO)$ and note that by Theorem \ref{thm:stab equiv to cocompact} we have $|L| \geq 2$.

Let $\xi \in L$. Since $\cO$ is $N$-stable, there exists a sequence $(x_n)$ in $\rho(\Lambda).o \cap X^{(N)}_o$ such that $\xi = [(x_n)]$. Let $\gamma_n$ be a geodesic joining $o$ to $x_n$. By Arzel\`a--Ascoli (Lemma \ref{ArzelaAscoli1}) we know that the $\gamma_n$ subsequentially converge to an $N$-Morse geodesic $\gamma$ representing $\xi$. Since $|L| \geq 2$ we can choose $\xi' \in L \setminus \{\xi\}$. By Theorem \ref{IdealMorseTriangle} there then exists a bi-infinite geodesic line $\alpha$ which is bi-asymptotic to $(\xi', \xi)$, and thus contained in $\mathfrak{H}(L)$. 

Again, by Theorem \ref{IdealMorseTriangle}, there is a constant $\delta$ depending only on $N$ and constants $T, T'$ (depending on $\gamma$ and $\alpha$) such that $d_\mathrm{Haus}(\gamma([T, \infty)), \alpha([T', \infty)))<\delta$. Since $\alpha([T', \infty]) \subset \mathfrak{H}(L)$, we know from (CC3') that $B(\gamma(n), \delta+D)\cap \cO \neq \emptyset$ for all $n \geq T$. Let $\beta$ be any other geodesic ray representing $\xi$. By the proof of Proposition 2.4 in \cite{cordes:2016ad}, we have a constant $E$ depending only on $N$  and constants $S, S'$ (depending on $\gamma$ and $\beta$) so that $d_\mathrm{Haus}(\gamma([T, \infty)), \alpha([T', \infty)))<E$. Thus it follows that $B(\gamma(n), \delta+D+E)\cap \cO \neq \emptyset$ for all $n \geq \max(S,T)$. Since $\gamma$ represents $\xi$, this shows that $\xi$ is conical.
\end{proof}
It follows from the proposition that the isometric actions studied in Theorem \ref{QuasikernelsConical} are not convex cocompact. 
\begin{remark}[Convex cococompact isometric actions]\label{CoCoIso}
There is a major difference between convex cocompact qiqacs and convex cocompact \emph{isometric} actions. Indeed, if an isometry preserves a subset $\cL \subset \partial X$, then it maps geodesics between points of $\cL$ to such geodesics and hence preserves the weak hull $\mathfrak H(\cL)$; for quasi-isometries this is not the case. Consequently, if $\rho \colon (\Lambda, \Lambda^\infty) \to \mathrm{Is}(X)$ is a convex cocompact isometric action, then the weak hull $\mathfrak{H}(\cL_\rho(X))$ is invariant under $\rho(\Lambda)$ and hence under $\rho(\Lambda^\infty)$. If $\rho$ is only a convex cocompact qiqac, then $\mathfrak{H}(\cL_\rho(X))$ is quasi-invariant under $\rho(\Lambda)$, but not quasi-invariant under $\rho(\Lambda^\infty)$. This difference has drastic consequences.
\end{remark}
\begin{corollary}\label{Rigid1} Let $\rho:(\Lambda, \Lambda^\infty) \to \mathrm{Is}(X)$ be a convex cocompact isometric action. If the stabilizer of $\cL_\rho(X)$ in $\mathrm{Is}(X)$ is discrete or $\rho(\Lambda^\infty)$ acts properly on $X$, then $(\Lambda, \Lambda^\infty)$ is an almost group.
\end{corollary}
\begin{proof} By Remark \ref{CoCoIso} the action of $(\Lambda, \Lambda^\infty)$ restricts to an action on the weak hull $\mathfrak{H}(\cL_\rho(X))$, which by Proposition \ref{WeakHullQuasiconvex} is quasi-convex in $X$ and hence large-scale geodesic. The corollary then follows from Corollary \ref{DiscreteIsometry}.
\end{proof}
Note that there exist plenty of Fuchsian groups which act isometrically and convex cocompactly on the hyperbolic plane and whose limit set has a discrete stabilizer in $\mathrm{PSL}_2(\R)$; these are mainly responsible for the rich zoo of convex cocompact Fuchsian groups. In the world of approximate groups which are not almost groups, such examples cannot exist. If we combine Corollary \ref{Rigid1} with Corollary \ref{thm:quasi-convex equiv to cocompact} then we obtain the following surprising rigidity result:
\begin{corollary}\label{Rigidity0} Let $(\Lambda, \Lambda^\infty)$ be an approximate group. 
\begin{enumerate}[(i)]
\item If $\Lambda^\infty$ is a hyperbolic group and $\Lambda$ is quasi-convex in $\Lambda^\infty$, then $(\Lambda, \Lambda^\infty)$ is an almost group.
\item If $\Lambda^\infty$ is a free group and  $(\Lambda, \Lambda^\infty)$ is geometrically finitely-generated, then $(\Lambda, \Lambda^\infty)$ is an almost group.
\end{enumerate}
\end{corollary}
\begin{proof} (i) Let $X$ be a Cayley graph for $\Lambda^\infty$; then $X$ is a proper geodesic hyperbolic metric space and
$(\Lambda, \Lambda^\infty)$ acts isometrically on $X$ via the inclusion $\rho: (\Lambda, \Lambda^\infty) \to \mathrm{Is}(X)$ and both $\Lambda$ and $\Lambda^\infty$ act properly on $X$ via $\rho$.
Since $\rho(\Lambda).e = \Lambda$ is quasi-convex, we deduce that $\rho$ has quasi-convex quasi-orbits. By Corollary \ref{thm:quasi-convex equiv to cocompact} this implies that $\rho$ is convex cocompact. Simce $\rho(\Lambda)^\infty$ acts properly, the corollary now follows from Corollary \ref{Rigid1}.

\item (ii) This follows from (i) and the fact that every coarsely-connected subset of a tree is quasi-isometric to a subtree and hence quasi-convex.
\end{proof}

\chapter{Asymptotic dimension} \label{ChapAsdim}

A well-known formula in geometric group theory, first conjectured by Gromov \cite[1.E$_1'$]{Gromov} and later proved by Buyalo and Lebedeva \cite{BuyaloLebedeva}, states that the asymptotic dimension of a hyperbolic group $\Gamma$ equals the topological dimension of its boundary plus one, i.e., $\asdim\Gamma = \dim \partial\Gamma +1$. More generally, this formula is true for cobounded hyperbolic proper geodesic metric spaces. In this chapter we extend the Buyalo--Lebedeva theorem to all \emph{quasi-cobounded} hyperbolic proper geodesic metric spaces, hence in particular to hyperbolic approximate groups (Theorem \ref{AsdimMainConvenient}).

\section{Statement of the main result}

In Definition \ref{def: asdim lambda} we defined the asymptotic dimension of countable approximate groups, which is their coarse invariant. 
 If $(\Lambda, \Lambda^\infty)$ is a \emph{hyperbolic} approximate group, then we saw, in Example \ref{BoundaryDimension}, how to define the topological dimension of the Gromov boundary of $\Lambda$ (where $\partial\Lambda$ depends on the internal QI-type of $(\Lambda, \Lambda^\infty)$). 
 The goal of this chapter is to relate $\asdim \Lambda$ and $\dim \partial\Lambda$, more precisely we will prove:
\begin{theorem}[Asymptotic dimension vs.\ boundary dimension]\label{AsdimMainConvenient} If $X$ is a quasi-cobounded hyperbolic proper geodesic metric space with Gromov boundary $\partial X$, then
\[
\asdim X = \dim \partial X + 1.
\]
In particular, if $(\Lambda, \Lambda^\infty)$ is a hyperbolic approximate group, then
\[
\asdim \Lambda  = \dim \partial \Lambda + 1. 
\]
\end{theorem}
For hyperbolic groups, Theorem \ref{AsdimMainConvenient} was conjectured by Gromov (\cite[1.E$_1'$]{Gromov}) and established by Sergei Buyalo and Nina Lebedeva \cite{BuyaloLebedeva}, based on earlier work by many authors (like \cite{BestvinaMess}, \cite{Swiatkowski-asdimhomol}, \cite{BuyaloSchroeder-hypdim}). Our proof follows the steps of the proof in the group case, as presented in \cite{BuyaloSchroeder}. We would like to point out that one of the proofs (\cite[Thm.\ 1.1]{BuyaloLebedeva} in the original, which is reproduced as \cite[Thm.\ 12.2.1]{BuyaloSchroeder}) contains a gap that we were unable to fix. Nina Lebedeva kindly provided us with a list of corrections for this proof, and our arguments will be parallel to her corrected version (see Section \ref{SecProofldim}). While the (modified) Buyalo--Lebedeva argument applies not only to hyperbolic groups, but also to cobounded hyperbolic proper geodesic spaces, the extension to the quasi-cobounded case requires additional arguments. It may be possible to further weaken the quasi-coboundedness assumption (see Remark \ref{ThreeProperties}), but as the following example shows, the assumption cannot be dropped altogether.
\begin{example}[Hyperbolic Shashlik] Let $\gamma: [0, \infty) \to \bH^n$ be a geodesic ray in hyperbolic $n$-space and let $x_1, x_2, \dots$ be points of $\gamma([0, \infty))$ such that $d(x_n, x_{n+1}) \geq 2^{n+2}$, for all $n \in \bN$. We consider the \emph{hyperbolic shashlik} \index{hyperbolic!shashlik}
\[
X \coloneqq  \gamma([0, \infty)) \cup \bigcup_{n=1}^\infty B(x_n, 2^n) \subset \bH^n,
\]
taken with the path-length metric, which is a hyperbolic proper geodesic metric space. Since $X$ contains arbitrarily large balls in $\bH^n$, we have $\asdim X = \asdim \bH^n = n$. On the other hand, $X$ contains a single geodesic ray so $\partial X$ is one point, and thus $\dim \partial X = 0$. Consequently, for general hyperbolic proper geodesic metric spaces their asymptotic dimension cannot be bounded in terms of the topological dimension of their boundary.
 \end{example}

Theorem \ref{AsdimMainConvenient} will actually be established by proving the following result, in Section \ref{AsdimMainTheoremProofSketch} below:
\begin{theorem}\label{AsdimMain} 
If $X$ is a quasi-cobounded hyperbolic proper geodesic metric space with Gromov boundary $\partial X$ and $d$ is a visual metric on $\partial X$, then
\[
{\asdim}\, X =  \dim \partial X + 1 = \ell\text{-}\dim (\partial X,d) + 1.
\]
\end{theorem}
For \emph{cobounded} (rather than just quasi-cobounded) spaces, this is \cite[Thm.\ 12.3.3]{BuyaloSchroeder}. This theorem is mentioning $\ell$-$\dim$, which is a notion of dimension for metric spaces known as \emph{linearly controlled metric dimension}\index{linearly controlled metric dimension}\index{dimension!linearly controlled metric}. We introduce the notion of $\ell$-$\dim$, along with discussing some other necessary background on dimension theory in  Section \ref{AsdimPrelim}. For the process of proving Theorem \ref{AsdimMain}, the most technical step will be Theorem \ref{l-dim-leq-dim}, whose proof will be given in Section \ref{SecProofldim}.

\section{Background on dimension theory}\label{AsdimPrelim}

For the reader's convenience, we summarize in this section some basic properties of topological dimension $\dim$ and of linearly controlled metric dimension $\ell$-$\dim$. We start with some general notation concerning covers of metric spaces.

If $X$ is a set, then a collection $\cU$ of subsets of $X$ is called a \emph{cover}\index{cover} of $X$ if $X = \bigcup \cU$. If $X$ is a topological space, then a cover $\cU$ is called an \emph{open cover}\index{cover!open} if every $U \in \cU$ is an open subset of $X$. 
While in classical dimension theory one is mostly concerned with open covers of topological or metric spaces, let us emphasize that the covers considered in the definition of asymptotic dimension need not be open. Given covers $\cU$ and $\cV$ of $X$, the cover $\cV$ \emph{refines}\index{refinement!of a cover} $\cU$ if for each $V\in \cV$ there is a $U\in\cU$ so that $V\subset U$.
The \emph{multiplicity}\index{multiplicity!of cover} 
(or \emph{order})\index{order!of cover} $\mult \cU$ of a cover $\cU$ is defined as the maximal number $n$ such that there exist $U_1, \dots, U_n \in \cU$ with $U_1 \cap \dots \cap U_n \neq \emptyset$. With this notation understood we have the following definition of topological dimension (see e.g.\ \cite{Munkres}):

\begin{definition}\label{Def-dim}
For a nonempty topological space $X$, its \emph{topological dimension}\index{topological dimension}\index{dimension!topological}\index{$\dim$}\index{covering dimension}\index{dimension!covering}
(or \emph{covering dimension}) $\dim X$ is the minimal integer $n\in\bN_0$
such that for every open cover $\cU$ of $X$, there is an open cover $\cV$ of $X$ which refines $\cU$, and which has the multiplicity $\mult\cV\leq n+1$. If no such $n\in\bN_0$ exists, we say that $\dim X=\infty$. We also set $\dim \emptyset \coloneqq  -1$.
\end{definition}

\begin{remark}\label{Remark: dim def}
Note that the definition of topological dimension $\dim$ is often stated for \emph{finite} open covers $\cU$ (see, for example, \cite[Def.\ 1.6.7]{Engelking} or \cite[Def.\ 1.1.8]{Coornaert}), but as long as a topological space $X$ is Hausdorff and paracompact, this definition is equivalent to ours (see \cite[Prop.\ 3.2.2]{Engelking}).
\end{remark}
We will often be interested in covers of metric spaces. If $\cU$ is a cover of a metric space $(X,d)$, then we can define its \emph{mesh}\index{mesh} as \[\mesh\cU\coloneqq \sup\{\diam U \mid U\in\cU\}.\]
The following definition is taken from \cite[Ch.\ 1.6]{Engelking}.

\begin{definition} \label{Def myu-dim}
Let $(X, d)$ be a nonempty metric space. 
The \emph{metric dimension}\index{metric dimension}\index{$\mu\dim$}\index{dimension!metric} $\mu\dim X$ is the smallest integer $n\in \bN_0$ such that for every $\varepsilon >0$ there is an open cover $\cU$ of $X$ with $\mult \cU\leq n+1$ and $\mesh \cU\leq \varepsilon$.
\end{definition}

\begin{remark} In \cite[p.\ 107]{BuyaloSchroeder}, the dimension $\dim X$ of a metric space $(X,d)$ is defined as $\mu \dim X$. This is somewhat unfortunate since \cite[Ex.\ 1.10.23]{Engelking} shows that $\mu\dim X$ and $\dim X$ do not coincide for general metric spaces. The reason that this will not cause any problems for us in the sequel, when applying results from \cite{BuyaloSchroeder}, is that we will only consider topological dimension of non-empty compact metric spaces, and for these we have the following result.
\end{remark}

\begin{theorem}\label{Thm-dim-compact}
If $(X,d)$ is a non-empty compact metric space, then $\mu\dim X = \dim X$.
\end{theorem}
\begin{proof}
See \cite{Engelking}, Definition 1.6.7 and Theorem 1.6.12. 
\end{proof}

Thus, when we know we are dealing with a compact metric space $X$, we may use Definition \ref{Def myu-dim} as our definition for $\dim X$. Our next goal is to introduce the notion of linearly controlled metric dimension of a metric space. For this we need the definition of the Lebesgue number of an open cover. 
\begin{definition} \label{def: Lebesgue original}
For an open cover $\cU$ of a metric space $X$,  \emph{the Lebesgue number of $\cU$}\index{Lebesgue number!of a cover} is
the number $L(\cU)\coloneqq \inf_{x\in X}L(\cU,x)$, where $L(\cU,x)\coloneqq \sup_{ U\in\cU} \dist(x, X\setminus U)$.
\end{definition}

Note that $L(\cU)$ may be strictly larger than $\mesh \cU$ (and even infinite, as will be the case for $L(\widehat\cU_{-1})$ in the proof of Theorem \ref{l-dim-leq-dim} below). The following definition of linearly controlled metric dimension is given in Subsection 9.1.4 of \cite{BuyaloSchroeder}. 
\begin{definition}\label{def:ldim}
If $(X,d)$ is a non-empty metric space, then its \emph{linearly controlled metric dimension}\index{linearly controlled metric dimension}\index{dimension!linearly controlled metric}\index{$\ell$-$\dim$} $\ell$-$\dim X$ is the minimal integer $n\in\bN_0$ with the following property: there exists a $\delta \in (0,1)$ such that for every sufficiently small $r>0$ there is an open cover $\cU$ of $X$ with $\mult \cU\leq n+1$, $\mesh \cU\leq r$ and $L(\cU)\geq \delta r$.  If no such $n\in\bN_0$ exists, we say that $\ell$-$\dim X=\infty$.
Since $\ell$-$\dim$ of a metric space depends on the choice of metric, we will often emphasize this choice by writing $\ell$-$\dim (X,d)$ instead of just $\ell$-$\dim X$. 
\end{definition}
In \cite{BuyaloLebedeva}, linearly controlled metric dimension is referred to as \emph{capacity dimension}; \index{dimension!capacity} the term \emph{microscopic Assouad-Nagata dimension}\index{dimension!microscopic Assouad-Nagata} is also sometimes used.

\begin{remark} \label{rem: l-dim and dim}
As a consequence of \cite[Proposition 2.2]{Lang-Sch}, note that for general metric spaces $(X,d)$ we always have the inequality $\ell$-$\dim X \geq  \dim X$, but the converse is not true in general. See \cite[Chapter 11]{BuyaloSchroeder} for some examples of calculation of $\ell$-$\dim$ and $\dim$. 
\end{remark}

\begin{remark}[Coloring definitions of $\dim$ and $\ell$-$\dim$] \label{Rem: coloring def} \index{asymptotic dimension!coloring definition of}\index{topological dimension!coloring definition of}
Our definition of asymptotic dimension (Definition \ref{def: asdim} above) was in terms of \emph{colored}\index{cover!colored} coverings. There are similar \emph{coloring versions} of definitions of $\dim$ and $\ell$-$\dim$, which are often easier to work with than the original definitions, and which we briefly recall here. As explained on page 109 of \cite{BuyaloSchroeder}, a family of subsets $\cU$ of a space $X$ is disjoint if it is pairwise disjoint, i.e., if its multiplicity $\mult\cU=1$. 

For $m\in\bN$, a family $\cU$ is called \emph{$m$-colored} if it is the union of $m$ families $\bigcup_{i=1}^m \cU^i$, and each family $\cU^i$ (of the same color $i$) is disjoint. Clearly, the multiplicity of an $m$-colored family is at most $m$, i.e., each point contained in the union of this family can be contained in at most $m$ different elements of the family. As it turns out (see in \cite[Section 9.1.6]{BuyaloSchroeder}), for a metric space $X$, in the definitions for $\mu \dim$ and $\ell$-$\dim$ above one can replace the requirement that multiplicity of the relevant family be $\leq n +1$ by the requirement that this family be $(n+1)$-colored. Thus, for a nonempty compact metric space $X$, and $n\in\bN_0$, we have $\dim X\leq n$ if and only if for every $\varepsilon >0$, there is an $(n+1)$-colored open cover $\cU=\bigcup_{i=0}^n\cU^i$ of $X$ with  $\mesh \cU\leq \varepsilon$.
Similarly, for a nonempty metric space $X$, and $n\in\bN_0$, we have $\ell$-$\dim X\leq n$ if and only if there exists a $\delta \in (0,1)$ such that for every sufficiently small $r>0$ there is an  $(n+1)$-colored open cover $\cU=\bigcup_{i=0}^n\cU^i$ of $X$ with $\mesh \cU\leq r$ and $L(\cU)\geq \delta r$.
\end{remark}


\section{Asymptotic dimension vs.\ dimension of the boundary}\label{AsdimMainTheoremProofSketch}
The remainder of this chapter is devoted to the proof of Theorem \ref{AsdimMain}.  Note that Theorem \ref{AsdimMain} holds trivially if $X$ is bounded: in this case, $\asdim X = 0$ and $\partial X = \emptyset$, hence $\dim \partial X = -1 = \asdim X-1$. Thus, from now on, \emph{$X$ is an unbounded quasi-cobounded hyperbolic proper geodesic metric space with Gromov boundary $\partial X$ and $d$ denotes a fixed choice of visual metric on the Gromov boundary $\partial X$.} The assumptions on $X$ imply in particular that $\partial X$ is compact, so it is safe to use (the coloring version of) Definition \ref{Def myu-dim}
 of $\dim\partial X$
 (see Remark \ref{Rem: coloring def}).

One of the two inequalities in Theorem \ref{AsdimMain} actually holds in the complete generality of hyperbolic proper geodesic spaces without any (quasi-)coboundedness assumptions, see \cite[Thm. 10.1.2]{BuyaloSchroeder}: 
\begin{theorem}\label{Thm: BS asdim-lower bound}
For every hyperbolic proper geodesic metric space $X$, we have 
\[
\asdim X \geq \dim \partial X+1.
\]
\end{theorem}
In order to establish Theorem \ref{AsdimMain} it thus remains to show only that
\begin{equation}\label{AsdimMainToShow}
\asdim X \leq \ell\text{-}\dim (\partial X,d)+1 \leq \dim (\partial X)+1.
\end{equation}
The main step in the proof of \eqref{AsdimMainToShow} (and hence Theorem \ref{AsdimMain}) is the following theorem, which generalizes \cite[Thm.\ 12.2.1]{BuyaloSchroeder}.
\begin{theorem} \label{l-dim-leq-dim}
If a metric space $X$ is locally quasi-similar to a compact metric space $Y$ and its linearly controlled metric dimension $\ell$-$\dim X$ is finite, then $\ell$-$\dim X \leq \dim Y$.
\end{theorem}

We will prove Theorem \ref{l-dim-leq-dim} in Section \ref{SecProofldim} below. We now explain how this implies inequalities \eqref{AsdimMainToShow} and hence Theorem \ref{AsdimMain}. 
\begin{remark}\label{ThreeProperties} In the proof of Theorem \ref{AsdimMain}, we will use three properties of $X$:
\begin{enumerate}[(i)]
\item The space $X$ is visual (see Proposition \ref{PropVisual}).
\item $(\partial X, d)$ is locally quasi-self-similar, i.e., locally quasi-similar to itself (see Proposition \ref{LemmaQuasiSelfSim}). 
\item $(\partial X, d)$ is doubling (see Corollary \ref{Doubling1}).
\end{enumerate}
\end{remark}

Equipped with these three properties we can now complete the proof of Theorem \ref{AsdimMain} using the same embedding argument as in \cite[Chapter 12]{BuyaloSchroeder}, which is based on the following theorem:
\begin{theorem}[Buyalo--Schroeder embedding theorem]\label{BuSchEmbedding}
Let $Z$ be a visual Gromov hyperbolic space such that $n \coloneqq  \ell\text{-}\dim \partial Z < \infty$. Then there exists a quasi-isometric embedding $Z \to T_1\times\ldots\times T_{n+1}$, where $T_1, \dots, T_{n+1}$ are (not necessarily locally finite) simplicial trees.
\end{theorem}

\begin{remark} See Definition \ref{Def1Tree} and Remark \ref{MetricTree} for the notions of metric trees and simplicial trees.
The conclusion of \cite[Theorem 12.1.1]{BuyaloSchroeder} is that a certain metric space admits a quasi-isometric embedding into a product of metric trees. However, as remarked on \cite[p.147]{BuyaloSchroeder}, the proof actually constructs an embedding into a product of (typically non-locally finite) \emph{simplicial} trees. 
\end{remark}
We now complete the proof of Theorem \ref{AsdimMain}.
\begin{proof}[Proof of Theorem \ref{AsdimMain} modulo Theorem \ref{l-dim-leq-dim}]
The fact that $(\partial X,d)$ is doubling implies by \cite[Lemma 12.2.2]{BuyaloSchroeder} that
 $n\coloneqq \ell\text{-}\dim (\partial X, d)< \infty$. Since $X$ is visual we can apply Theorem \ref{BuSchEmbedding} to $X$. We deduce that there exists a quasi-isometric embedding
$F:X\to T_1\times\ldots\times T_{n+1}$ of our space $X$ into a product of $(n+1)$ simplicial trees. By \cite[Prop.\ 10.2.1]{BuyaloSchroeder} we have $\asdim T\leq 1$ for every metric tree $T$, and hence by the product theorem for asymptotic dimension (see \cite[Thm.\ 32]{BellDran1}) we have
\[
\asdim (T_1 \times \dots \times T_{n+1}) \leq n+1.
\]
Now recall from Lemma \ref{AsdimInvariant}.(iii) that asdim is monotone under coarse embeddings. We thus obtain (compare \cite[Cor.\ 12.1.11]{BuyaloSchroeder}) that
\[
\asdim X \leq \asdim (T_1\times\ldots\times T_{n+1}) \leq n+1 =  \ell\text{-}\dim (\partial X,d) + 1.
\]
Finally, since $(\partial X,d)$ is locally quasi-self-similar and has finite $\ell$-$\dim$, we apply Theorem \ref{l-dim-leq-dim} to get $\ell\text{-}\dim(\partial X,d) \leq \dim \partial X$, so
\[
\asdim X \leq  \ell\text{-}\dim (\partial X,d) + 1 \leq \dim \partial X +1,
\]
which is the desired inequality \eqref{AsdimMainToShow}.
\end{proof}

\begin{remark}
By Remark \ref{rem: l-dim and dim} the inequality
$\ell$-$\dim Z\geq \dim Z$ holds for any metric space $Z$. For $X$ as in Theorem \ref{AsdimMain} and any choice of visual metric $d$ on $\partial X$ we thus have both $\ell\text{-}\dim(\partial X,d) \geq \dim \partial X$ and $\ell\text{-}\dim(\partial X ,d)\leq \dim \partial X$ by Theorem \ref{l-dim-leq-dim}. We thus get the equality $\ell\text{-}\dim(\partial X,d) = \dim \partial X$ without referring to inequalities involving $\asdim X$.
\end{remark}

\section{The proof of Theorem 7.14}\label{SecProofldim}

In the previous section we have reduced the proof of our main theorem (Theorem \ref{AsdimMain}) to the proof of Theorem \ref{l-dim-leq-dim}, which we give in this section. Before we start with the actual proof we collect various preliminary notions and results which will be used in the proof.

\subsection{Restrictions of coverings} In the proof of Theorem \ref{l-dim-leq-dim} we will have to restrict various open coverings to subspaces. We thus introduce some terminology to deal with such restrictions.

From now on let $(X,d)$ be a metric space and let $A \subset X$ be a subset. If $\cU$ is a family of subsets of $X$, then we define the \emph{restriction of $\cU$ to $A$} by \[
\cU|_A=\{U\cap A\ | \ U\in\cU\}.
\]
We say that a family $\cU$ of (open) subsets of $X$ is an \emph{(open) covering family of $A$ in $X$}\index{covering family} if $\bigcup\cU\supseteq A$. This is equivalent to $\cU|_A$ being an (open) cover of $A$. We introduce the following notion of the Lebesgue number for open covering families:
\begin{definition} If $A$ is a subset of a metric space $X$ and $\cU$ is an open covering family of $A$ in $X$, then \emph{the Lebesgue number of $(\cU, A, X)$}\index{Lebesgue number!the Lebesgue number of a covering family} is defined as 
\begin{align}
L_X(\cU, A)\coloneqq \inf_{x\in A}L_X(\cU,x), \ \text{ where } L_X(\cU,x)\coloneqq \sup_{ U\in\cU} \dist(x, X\setminus U).
\end{align}
\end{definition}
Note that for $A=X$ we have
\[
 L_X(\cU, X) = \inf_{x\in X}L_X(\cU,x) = L(\cU),
\]
hence in the case $A=X$ we recover the usual definition of the Lebesgue number. The following inequalities are immediate from the definitions:
\begin{lemma}\label{lemma: Lebesgue-restriction-of-family}\label{Lebesgue-subspace}
If $\cU$ is an open cover of a metric space $X$ and $A\subseteq B\subseteq X$, then
\[
L(\cU) = L_X(\cU, X)\leq L_X (\cU,A) \leq L_B(\cU|_B, A). 
\]
\end{lemma}

We emphasize that, in the situation of Lemma \ref{Lebesgue-subspace}, the Lebesgue number of the covering family $\cU|_B$ of $A$ in $B$ may be larger than the Lebesgue number of the original family.  In the case $A=B$ the last inequality of Lemma \ref{lemma: Lebesgue-restriction-of-family} specializes to \[L_X(\cU, A) \leq  L_A(\cU|_A, A) = L(\cU|_A),\] and it is important to note that even in this special case the inequality can be strict.  Similarly, if $\cU$ is an open cover of a metric space $X$ and $A\subset X$ is a subset, then
\begin{equation}\label{MeshRestriction}
\mesh \cU|_A \leq \mesh \cU,
\end{equation}
but equality does not hold in general. Thus, when restricting an open cover of $X$ to a subspace, the Lebesgue number may increase and the mesh may decrease. The gap in the original proof of \cite[Thm.\ 1.1]{BuyaloLebedeva} (\cite[Thm.\ 12.2.1]{BuyaloSchroeder}) is caused by the fact that the authors do not take these effects into account carefully enough. This illustrates the importance of keeping the ambient space in the Lebesgue number notation and mesh notation, which we may omit only when the ambient space is the obvious one.

\subsection{Saturation of covering families} Let $(X,d)$ be a metric space. Given a subset $W\subset X$ and $r>0$, we will denote by $N_r(W)$ the open $r$-neighborhood of the set $W$, and we set $N_{-r}(W)\coloneqq X\setminus \overline{N}_r(X\setminus W)$.
\begin{definition} Let $U \subset X$ and let $\cU, \cV$ be families of subsets of $X$.
\begin{enumerate}[(i)]
\item The \emph{saturation}\index{saturation!of a set by a family} $U\ast\cV$ of the set $U$ by the family $\cV$ is the union of $U$ with all members $V\in\cV$ with $U\cap V\neq \emptyset$.
\item The \emph{saturation}\index{saturation!of a family by a family} of the family $\cU$ by the family $\cV$ is the family 
\[
\cU\ast\cV=\{U\ast\cV\ | \ U\in\cU\}.
\]
\end{enumerate} 
\end{definition}
The following result is \cite[Lemma 3.1]{BuyaloLebedeva} (see also \cite[Proposition 9.6.1]{BuyaloSchroeder}):
\begin{proposition}\label{saturated-family}
Suppose that $X$ is a metric space and let $A, B\subset X$. Let $\cU=\bigcup_{c\in C} \cU^c$ and  $\cV=\bigcup_{c\in C} \cV^c$ be covering families of $A$ and $B$, respectively,  which are both open in $X$ and $m$-colored with $m=|C|\geq 1$. If $\mesh \cV\leq \frac{L_X(\cU, A)}{2}$, then the family $\cW=\bigcup_{c\in C} \cW^c$ given by
\begin{eqnarray*}
\cW^c\coloneqq  \left( N_{\frac{-L_X(\cU,A)}{2}}(\cU^c)\ast \cV^c\right) \cup \left\{V\in \cV^c \ | \  N_{\frac{-L_X(\cU,A)}{2}}(U)\cap V=\emptyset\  \!\text{ for all } U\in \cU^c\right\}
\end{eqnarray*}
is an open $m$-colored covering family of $A\cup B$ in $X$, with $L_X(\cW,A\cup B)$ $\geq$  $\min\{\frac{L_X(\cU,A)}{2}, L_X(\cV,B)\}$ and $\mesh\cW\leq \max \{\mesh\cU, \mesh\cV\}$.
\end{proposition}
In the sequel we will denote the family $\cW$ constructed in Proposition \ref{saturated-family} by $\cW=\cU\circledast\cV$.

\subsection{The mesh and the Lebesgue number under generalized quasi-homotheties} Recall the Definition \ref{quasi-homothety and generalization}, which introduces quasi-homotheties and their generalized version.
An important role in the original theorem of Buyalo and Lebedeva is played by \cite[Lemma 3.4]{BuyaloLebedeva} (see \cite[Lemma 12.2.3]{BuyaloSchroeder}) which describes the behavior of the mesh and the Lebesgue number under quasi-homotheties. This lemma also works for generalized quasi-homotheties, in the following form:

\begin{lemma}\label{adjusted-lemma12.2.3}
Let $h:A\to B$ be a generalized quasi-homothety between metric spaces, i.e., a map such that, for some $\lambda\geq 1$, $K\geq 1$ and for some $R>0$, we have, for each $a_1,a_2 \in A$,
\begin{equation}\label{lambda-quasi-homothetic}
\displaystyle \frac{1}{\lambda} R^K \cdot (d_A(a_1,a_2))^K \ \leq \ d_B(h(a_1),h(a_2)) \ \leq \ \lambda \sqrt[K]{R}\cdot \sqrt[K]{d_A(a_1,a_2)}.
\end{equation}
Let $C\subseteq A$, let $\widetilde\cU$ be an open covering family of $h(C)$ in $h(A)$, and let $\cU=h^{-1}(\widetilde\cU)=\{h^{-1}(\widetilde U)\ | \ \widetilde U\in \widetilde\cU\}$.
Then $\cU$ is an open covering family of $C$ in $A$, $h:A\to h(A)$ is bijective, and the following is true:
\begin{enumerate}
\item[(i)] $\frac{1}{\lambda}R^K\cdot (\mesh\cU)^K \ \leq \  \mesh\widetilde\cU \ \leq \ \lambda \sqrt[K]{R}\cdot \sqrt[K]{\mesh\cU}$, \ and
\item[(ii)] $\frac{1}{\lambda}R^K\cdot (L_A(\cU,C))^K\ \leq\  L_{h(A)}(\widetilde\cU, h(C)) \ \leq\  \lambda
\sqrt[K]{R}\cdot \sqrt[K]{L_A(\cU, C)}$.
\end{enumerate}
\end{lemma}

\begin{proof}
First note that $h$ is continuous, so $\cU$ is a family of open sets in $A$, and there is a bijective correspondence between elements of $\widetilde \cU$ and $\cU$: each $\widetilde U$ from $\widetilde\cU$ corresponds to $U=h^{-1}(\widetilde U)$ from $\cU$, and $h(U)=\widetilde U$, that is, $h(\cU)=\widetilde\cU$. \\
For (i): 
\begin{eqnarray*}\displaystyle \mesh \widetilde\cU & = & \sup_{\widetilde U \in \widetilde\cU} \diam \widetilde U=\sup_{\widetilde U \in \widetilde\cU}\  \sup_{b_1, b_2 \in \widetilde U}d_B(b_1,b_2)
 = \sup_{h( U) \in h(\cU)}\  \sup_{b_1, b_2 \in h( U)}d_B(b_1,b_2)\\
& = & \sup_{U \in \cU}\  \sup_{a_1, a_2 \in  U}d_B(h(a_1),h(a_2)) \overset{\eqref{lambda-quasi-homothetic}}\leq \sup_{U \in \cU}\  \sup_{a_1, a_2 \in  U}\lambda \sqrt[K]{R}\cdot \sqrt[K]{d_A(a_1,a_2)} \\
& = & \lambda \sqrt[K]{R}\cdot \sqrt[K]{\sup_{U\in \cU}\diam U} = \lambda \sqrt[K]{R}\cdot \sqrt[K]{\mesh\cU}.
\end{eqnarray*}
The other inequality for $\mesh$ is proven analogously.\\
For (ii):  
\begin{eqnarray*}\displaystyle
L_{h(A)}(\widetilde  \cU, h(C)) & = & \inf_{b\in h(C)} L_{h(A)}(\widetilde\cU, b)=\inf_{b\in h(C)}\  \sup_{\widetilde U\in \widetilde\cU}\ \dist_B(b,h(A)\setminus\widetilde U)\\
& = & \inf_{h(a)\in h(C)}\  \sup_{h( U)\in h(\cU)}\ \dist_B(h(a),h(A)\setminus h(U))\\
& = & \inf_{a\in C}\  \sup_{U\in \cU}\ \dist_B(h(a),h(A\setminus U))\\
&=&\inf_{a\in C}\  \sup_{U\in \cU} \ \inf_{x\in A\setminus U} d_B(h(a), h(x))\\
& \overset{\eqref{lambda-quasi-homothetic}}\leq & \inf_{a\in C}\  \sup_{U\in \cU} \ \inf_{x\in A\setminus U} \lambda \sqrt[K]{R}\cdot \sqrt[K]{d_A(a,x)}\\
&=& \lambda \sqrt[K]{R} \cdot \inf_{a\in C}\  \sup_{U\in \cU} \sqrt[K]{\dist_A (a, A\setminus U)}\\
& = & \lambda \sqrt[K]{R} \cdot \sqrt[K]{\inf_{a\in C} L_A(\cU, a)}= \lambda \sqrt[K]{R}\cdot \sqrt[K]{L_A(\cU, C)}.
\end{eqnarray*}
The other inequality for the Lebesgue number is proven analogously.
\end{proof}

\subsection[Proof of Theorem 7.14]{Proof of Theorem \ref{l-dim-leq-dim}
}

We now turn to the proof of Theorem \ref{l-dim-leq-dim}. This will also conclude the proof of Theorem \ref{AsdimMain} and thereby conclude this chapter. Our proof is based on the corrected proof of \cite[Thm.\ 1.1]{BuyaloLebedeva} (which is \cite[Thm.\ 12.2.1]{BuyaloSchroeder}), using a list of corrections provided to us by Nina Lebedeva in private correspondence. The authors wish to wholeheartedly thank Nina Lebedeva for sharing these corrections with them. 

\begin{proof}[Proof of Theorem \ref{l-dim-leq-dim}]
Let $N \coloneqq  \ell$-$\dim X < \infty$. If $\dim Y=\infty$ we are done, so let us assume that  $n \coloneqq  \dim Y < \infty$. We have to show that $N \leq n$.

From $\ell$-$\dim X=N$, we know that there is a constant $\delta\in (0,1)$ such that for every sufficiently small $\tau>0$ there exists an $(N+1)$-colored open cover $\cV=\bigcup_{j=0}^N\cV^j$ of $X$ with $\mesh \cV\leq \tau$ and $L(\cV)\geq \delta\tau$. 

Since $X$ is locally quasi-similar to $Y$, there exist constants $\lambda\geq 1$ and $K\geq 1$ such that for every sufficiently large $R>1$ and every $V\subset X$ with $\diam V\leq \frac{1}{R}$, there is a generalized quasi-homothety from $V$ to $Y$, i.e., a map $h_V:V\to Y$ with 
\begin{equation}\label{lambda-quasi-power-hv}
\displaystyle \frac{1}{\lambda}R^K (d_X(v_1,v_2))^K\ \leq \ d_Y(h_V(v_1),h_V(v_2)) \ \leq \ \lambda \sqrt[K]{R}\sqrt[K]{d_X(v_1,v_2)}, \quad \forall v_1,v_2 \in V.
\end{equation}
We would like to build such a generalized quasi-homothety $h_V$ for every $V\in \cV$. Since $\mesh\cV\leq \tau$, let us require that $\tau=\frac{1}{R}$ (thus for $\tau$ sufficiently small, $R$ is sufficiently large).
Now for every $V \in \cV$, we fix a choice of such a map $h_V:V\to Y$ once and for all.
 
Using that $Y$ is compact and $\dim Y=n$, we can choose, for each $j\in\{0,\ldots , N\}$, a finite $(n+1)$-colored open cover $\widetilde\cU_j$ of $Y$, say $\displaystyle\widetilde\cU_j =\bigcup_{c=0}^n \widetilde\cU^c_j$, such that the following hold:
\begin{enumerate}
\item[(i)] $\displaystyle \mesh \widetilde\cU_0\leq \frac{\delta^K}{2^K\cdot\lambda}$,\ \  and
\item[(ii)] $\displaystyle \mesh \widetilde\cU_{j+1}\leq \frac{1}{2^K\cdot\lambda^{1+K^2}}\cdot \min \left\{\left(L(\widetilde\cU_j)\right)^{K^2}, \left(\mesh \widetilde\cU_j\right)^{K^2}\right\}$,  \ for every $j=0,\ldots , N-1$.
\end{enumerate}
We claim that (i) and (ii) imply that
\begin{equation}\label{eq: mesh tilda}
\displaystyle \mesh\widetilde\cU_j\leq \frac{\delta^K}{2^K\cdot\lambda}, \qquad \forall j=0,\ldots , N.
\end{equation}
Indeed, for $j=0$ this is exactly the statement of (i). Now let $j\geq 1$. Since $\delta \in (0,1), \lambda \geq 1$ and $K\geq 1$, it follows from (i) that $\mesh\widetilde\cU_0 \in (0, \frac{1}{2})$, so $(\mesh\widetilde\cU_0)^{K^2}\leq \mesh\widetilde\cU_0$. Therefore, by (ii), $\mesh \widetilde\cU_{1}\leq \frac{1}{2^K\cdot\lambda^{1+K^2}}\cdot (\mesh\widetilde\cU_0)^{K^2} \leq \frac{1}{2}\mesh\widetilde\cU_0 < \mesh\widetilde\cU_0 \leq \frac{\delta^K}{2^K\cdot\lambda}$. This proves the claim, and the rest can be proven by induction:

Using the notation $l_j\coloneqq \min\{L(\widetilde\cU_j), \mesh\widetilde\cU_j\}$, notice that \eqref{eq: mesh tilda}, (i) and (ii) imply
\begin{equation}\label{eq: elj}
\displaystyle l_j\leq \frac{\delta^K}{2^K\cdot\lambda}, \qquad \forall j=0,\ldots , N, \text{ and}
\end{equation}
\begin{equation}\label{eq: elj+1}
\displaystyle l_{j+1}\leq \frac{1}{2^K\cdot\lambda^{1+K^2}}\cdot (l_j)^{K^2}, \qquad \forall j=0,\ldots , N-1,
\end{equation}
from which, considering that $l_j\in \left( 0,\frac{1}{2}\right]$, it follows that
\begin{equation}\label{eq: elj-root}
\displaystyle l_{j+1}\leq \frac{1}{\sqrt[K]{2}}\cdot l_j, \qquad \forall j=0,\ldots , N-1.
\end{equation}

Also define the constant $\displaystyle l\coloneqq \min\left\{\frac{1}{2^{N-j}}\left( L(\widetilde\cU_j)\right)^{K} \mid j=0, \ldots , N\right\}>0$.
 
 \medskip

We proceed by shrinking $\cV$ a little: for each $V\in\cV$, consider the smaller subset $V'\coloneqq N_{\frac{-\delta\tau}{2}}(V)$. Since $L(\cV)\geq \delta\tau$, we still have that $\{V'\ | \ V\in\cV\}$ is an open cover of $X$. 
Put  $\widetilde V\coloneqq h_V(V')$. Now for every $j\in\{0,\ldots , N\}$, and every $V \in \cV^j$, we define the family 
\begin{eqnarray*}
\widetilde\cU_{j,V}\coloneqq \{\widetilde U\in \widetilde\cU_j\ | \ \widetilde V \cap \widetilde U \neq \emptyset\}.
\end{eqnarray*}
 This $\widetilde\cU_{j,V}$ is a finite, open, at most $(n+1)$-colored covering family of $\widetilde V$ in $Y$. Note that $\widetilde\cU_{j,V}$ need not be contained in $h_V(V)$ in general, so it is important to distinguish between $\widetilde\cU_{j,V}$ and $\widetilde\cU_{j,V}|_{h_V(V)}$.
 
 Continuing with the proof, one considers
\begin{eqnarray*}\cU_{j,V} &\coloneqq & h_V^{-1}(\widetilde\cU_{j,V}) =\{h_V^{-1}(\widetilde U)\mid \widetilde U \in \widetilde\cU_{j,V}\} \\
&=&\{h_V^{-1}(\widetilde U \cap h_V(V))\mid \widetilde U \in \widetilde\cU_{j,V}\}
= h_V^{-1}(\widetilde\cU_{j,V}|_{h_V(V)}).
\end{eqnarray*}
Each $\cU_{j,V}$ defines an at most $(n+1)$-colored open covering family of $V'$ in $V$.

One can now apply Lemma  \ref{adjusted-lemma12.2.3}  to the map  $h_V:V\to Y$, with $C=V'$,
 $\widetilde\cU=\widetilde\cU_{j,V}|_{h_V(V)}$ and $\cU=\cU_{j,V}$, to obtain
\begin{align}\label{Lemma12.2.3-ineq-adj-(i)}
\frac{1}{\lambda}R^K\cdot (\mesh \cU_{j,V})^K\leq \mesh \widetilde\cU _{j,V}|_{h_V(V)} \leq \lambda \sqrt[K]{R}\cdot \sqrt[K]{\mesh \cU_{j,V}}, \quad \text{ and}
\end{align}
\begin{align}\label{Lemma12.2.3-ineq-adj-(ii)}
\frac{1}{\lambda}R^K\cdot (L_V (\cU_{j,V}, V'))^K\leq L_{h_V(V)}( \widetilde\cU _{j,V}|_{h_V(V)}, \widetilde V) \leq \lambda \sqrt[K]{R}\cdot\sqrt[K]{ L_V (\cU_{j,V},V')}.
\end{align}
 From the left side of \eqref{Lemma12.2.3-ineq-adj-(i)} and from \eqref{MeshRestriction} it follows that 
 \begin{eqnarray}\label{Lemma12.2.3-ineq-adj-(i)-half}
\frac{1}{\lambda}R^K\cdot (\mesh \cU_{j,V})^K\leq \mesh \widetilde\cU _{j,V},
\end{eqnarray}
while from the right side of \eqref{Lemma12.2.3-ineq-adj-(ii)} and from Lemma \ref{lemma: Lebesgue-restriction-of-family}, it follows that 
\begin{eqnarray}\label{Lemma12.2.3-ineq-adj-(ii)-half}
 L_Y( \widetilde\cU _{j,V}, \widetilde V) \leq \lambda \sqrt[K]{R}\cdot\sqrt[K]{ L_V (\cU_{j,V},V')}.
\end{eqnarray}

\medskip

Given a $j\in\{0, \ldots , N\}$, define $X_j\coloneqq \bigcup_{V\in\cV^j} V'$. Then $\{ X_j \mid j=0,\ldots , N\}$ form an open cover of $X$.
The family of sets 
\begin{eqnarray*}
\cU_j\coloneqq \displaystyle\bigcup_{V\in\cV^j} \cU_{j,V}
\end{eqnarray*}
 is an open covering family of the set $X_j$ in $X$.
The main step of the proof is to establish the following properties of this family:
\begin{enumerate}
\item[(1)] for each $j\in\{0,\ldots , N\}$,  the family $\cU_j$ is at most $(n+1)$-colored,
\medskip

\item[(2)] $\displaystyle\mesh \cU_j\leq \tau \sqrt[K]{\lambda}\sqrt[K]{\mesh\widetilde \cU_j}, \quad \forall j \in\{0, \ldots , N\}$,
\medskip

\item[(3)] \label{eq: mesh short}
$\displaystyle \mesh\cU_{j+1}\leq \frac{1}{2}\cdot\frac{\tau}{\lambda^K}\cdot l_j^K, \quad\forall j\in\{0,\ldots , N-1\}, $
\medskip

\item[(4)] $\displaystyle \mesh\cU_{j}\leq \frac{\delta\tau}{2}, \quad\forall j\in\{0,\ldots , N\},  $
\medskip

\item[(5)]   $\displaystyle L_X(\cU_j, X_j)\geq \min\left\{\frac{\tau}{\lambda^K} \cdot(L(\widetilde\cU_j))^K, \frac{\delta\tau}{2}\right\}, \ \forall j \in\{0, \ldots , N\}$,
\medskip 

\item[(6)] $ \displaystyle L_X(\cU_j, X_j)\geq \frac{\tau}{\lambda^K}\cdot l_j^K, \quad\forall j\in\{0,\ldots , N\}$, and
\medskip

\item[(7)] $\displaystyle \mesh \cU_{j+1}\leq \frac{1}{2} \cdot L_X(\cU_j, X_j), \ \forall j\in\{0,\ldots , N-1\}$. 
\end{enumerate}

\medskip
The only thing to check for (1)  is that within each of the colors $c\in\{0,\ldots, n\}$, i.e., in $\cU_j^c$, the sets are disjoint.
Indeed, if $U_1$, $U_2$ are elements of $\cU_j$ of the same color $c$, i.e., $U_1,\ U_2 \in \cU_j^{c}$, this means $U_1\in \cU_{j,V_1}^c$, $U_2\in \cU_{j,V_2}^c$. If $V_1=V_2$, then $U_1=h_{V_1}^{-1}(\widetilde U_1)$, $U_2=h_{V_1}^{-1}(\widetilde U_2)$, for some $\widetilde\U_1, \widetilde U_2\in \widetilde\cU_{j,V_1}^c$, but these are disjoint since $\widetilde \cU_j^c$, as well as  $\widetilde \cU_{j,V_1}^c$, consists of disjoint elements, $\forall c=0,\ldots , n$. If $V_1\neq V_2$, then these two are disjoint as different elements of $\cV^j$, so $U_1$ and $U_2$ are disjoint, too.

The inequality in (2) follows from\eqref{Lemma12.2.3-ineq-adj-(i)-half} and the fact that $\mesh \widetilde\cU_{j,V}\leq \mesh\widetilde\cU_j$, for all indexes $j=0, \ldots, N$:
\begin{eqnarray*}
 (\mesh \cU_{j,V})^K & \overset{\eqref{Lemma12.2.3-ineq-adj-(i)-half}}{\leq} &\lambda\cdot \frac{1}{R^K}\cdot\mesh\widetilde\cU_{j,V}\leq \lambda\cdot \frac{1}{R^K}\cdot\mesh\widetilde\cU_j\\
\Rightarrow\quad \mesh \cU_{j,V }& \leq & \frac{1}{R}\cdot\sqrt[K]{\lambda}\sqrt[K]{\mesh \widetilde\cU_j}, \quad \forall V\in\cV^j\\
\Rightarrow \qquad \mesh\cU_j &= & \sup_{V\in \cV^j}\mesh \cU_{j,V}\ \ \leq \ \ \tau\cdot\sqrt[K]{\lambda}\sqrt[K]{\mesh\widetilde\cU_j}.
\end{eqnarray*}
\medskip

To prove (3),
we use (ii), (2) and definition of $l_j$, for indexes $j=0, \ldots, N-1$:
\begin{eqnarray*}\label{eq: long mesh}
\mesh \cU_{j+1} &\overset{(2)}{\leq} & \tau \sqrt[K]{\lambda}\sqrt[K]{\mesh\widetilde \cU_{j+1}} \nonumber \\
&\overset{(ii)}{\leq} &
\tau \sqrt[K]{\lambda}\cdot\frac{1}{2}\cdot\frac{1}{\sqrt[K]{\lambda}}\cdot \frac{1}{\lambda^K}\cdot \min\left\{(L(\widetilde\cU_j))^K, (\mesh\widetilde\cU_j)^K\right\} \nonumber\\
& =& \frac{1}{2}\cdot\frac{\tau}{\lambda^K}\cdot\min\left\{(L(\widetilde\cU_j))^K, (\mesh\widetilde\cU_j)^K\right\}=\frac{1}{2}\cdot\frac{\tau}{\lambda^K}\cdot l_j^K.
\end{eqnarray*}

Also, since $\delta\in (0,1)$, $\lambda \geq 1$ and $K\geq 1$, for $j=0, \ldots , N$ we have
\begin{align}\label{two*}
\frac{\tau}{\lambda^K}(\mesh\widetilde\cU_j)^K
\overset{\eqref{eq: mesh tilda}}{\leq} \frac{\tau}{\lambda^K}\left( \frac{\delta^K}{2^K\cdot\lambda}\right)^K = \frac{\tau \delta^{K^2}}{\lambda^{2K}2^{K^2}}\leq\frac{\delta\tau}{2}.
\end{align}

Now let us prove (4). First, using (3) and \eqref{two*} we get:
\begin{align*}
\mesh \cU_{j+1} \overset{(3)}{\leq} \frac{1}{2}\min\left\{\frac{\tau}{\lambda^K}(L(\widetilde\cU_j))^K, \frac{\tau}{\lambda^K}(\mesh\widetilde\cU_j)^K\right\} \overset{\eqref{two*}}{\leq} \frac{1}{2}\cdot\frac{\delta\tau}{2} < \frac{\delta\tau}{2},
\end{align*}
which gives us (4) for indexes $j=1, \ldots , N$. Finally, for index $j=0$, (4) follows from (i) and (2):
\begin{align*}
    \mesh \cU_0 \overset{(2)}{\leq} \tau \sqrt[K]{\lambda}\sqrt[K]{\mesh\widetilde \cU_0} \overset{(i)}{\leq} \tau \sqrt[K]{\lambda}\sqrt[K]{\frac{\delta^K}{2^K\cdot\lambda}}=\frac{\delta\tau}{2}.
\end{align*}

\bigskip
For the statement in (5), note that, since elements $V'$ of $X_j$ are disjoint, we get
\begin{eqnarray*}
L_X(\cU_j, X_j) &=& \inf_{x\in X_j} \sup_{U\in \cU_j} \dist (x, X\setminus U) = \inf_{V\in \cV^j} \inf_{x\in V'}\sup_{U\in \cU_{j, V}} \dist (x, X\setminus U)\\
&=& \inf_{V\in \cV^j} \inf_{x\in V'}\sup_{U\in \cU_{j, V}} \min\left\{ \dist (x, X\setminus V),\dist (x, V\setminus U) \right\}.
\end{eqnarray*}
Since $x\in V'=N_{-\frac{\delta\tau}{2}}(V)$ implies $\dist (x, X\setminus V)\geq \frac{\delta\tau}{2}$, we have
\begin{eqnarray}\label{eq: L long}
L_X(\cU_j, X_j) &\geq& \nonumber \inf_{V\in \cV^j} \inf_{x\in V'}\sup_{U\in \cU_{j, V}} \min\left\{ \frac{\delta\tau}{2},\dist (x, V\setminus U) \right\}\\
& = & \nonumber \min\left\{ \frac{\delta\tau}{2},  \inf_{V\in \cV^j} \inf_{x\in V'}\sup_{U\in \cU_{j, V}}\dist (x, V\setminus U) \right\}\\
& = & \min\left\{ \frac{\delta\tau}{2},  \inf_{V\in \cV^j} L_V(\cU_{j,V}, V')\right\}.
\end{eqnarray}
On the other hand,
\begin{align}\label{eq: L tilda Uj}
 L(\widetilde\cU_j)=L_Y(\widetilde\cU_j,Y)\overset{Lemma \ \ref{lemma: Lebesgue-restriction-of-family}}\leq L_Y(\widetilde\cU_j,\widetilde V)= L_Y(\widetilde\cU_{j,V}, \widetilde V), \quad \forall V\in\cV^j,
\end{align}
where the last equality is true because, for $y \in \widetilde V$,
the only sets $\widetilde U$ from $\widetilde \cU_j$ that enter in the definition of $L_Y(\widetilde\cU_{j}, y)$ are those that are in $\widetilde\cU_{j,V}$.

Using \eqref{Lemma12.2.3-ineq-adj-(ii)-half} and \eqref{eq: L tilda Uj}, we may deduce that 
\begin{eqnarray*}
\nonumber L(\widetilde\cU_j) &\overset{\eqref{eq: L tilda Uj}}{\leq}& L_Y(\widetilde\cU_{j,V}, \widetilde V)
\overset{\eqref{Lemma12.2.3-ineq-adj-(ii)-half}}{\leq}\lambda \sqrt[K]{R}\cdot\sqrt[K]{ L_V (\cU_{j,V},V')}, \quad \forall V\in\cV^j\\
\nonumber \Rightarrow \quad L(\widetilde\cU_j) &\leq & \lambda \sqrt[K]{R}\cdot\inf_{V\in\cV^j} \sqrt[K]{ L_V (\cU_{j,V},V')}=
\lambda \sqrt[K]{R}\cdot \sqrt[K]{\inf_{V\in\cV^j} L_V (\cU_{j,V},V')}
\end{eqnarray*}
\begin{eqnarray}\label{eq: L to the K}
\Rightarrow \quad \frac{1}{\lambda^K}\cdot\frac{1}{R}\cdot (L(\widetilde\cU_j))^K \leq  \inf_{V\in\cV^j} L_V (\cU_{j,V},V').
\end{eqnarray}
Therefore, \eqref{eq: L long} and \eqref{eq: L to the K} imply
\begin{eqnarray*}
L_X(\cU_j,X_j) &\overset{\eqref{eq: L long}}{\geq}& \min\left\{ \frac{\delta\tau}{2},  \inf_{V\in \cV^j} L_V(\cU_{j,V}, V')\right\} \\
& \overset{\eqref{eq: L to the K}}{\geq} & \min\left\{ \frac{\delta\tau}{2}, \frac{\tau}{\lambda^K}\cdot(L(\widetilde \cU_j))^K\right\}, \ \  \forall j=0,\ldots , N.
\end{eqnarray*}
This finishes the proof of (5).

\medskip
To prove (6), we combine (5) and \eqref{two*}:
\begin{eqnarray*}
 L(\cU_j, X_j) &\overset{(5)}{\geq} &  \min\left\{  \frac{\tau}{\lambda^K}\cdot(L(\widetilde \cU_j))^K,  \frac{\delta\tau}{2}    \right\}\\  &\overset{\eqref{two*}}{\geq} &
 \min\left\{\frac{\tau}{\lambda^K}\cdot(L(\widetilde\cU_j))^K, \frac{\tau}{\lambda^K}\cdot(\mesh\widetilde\cU_j)^K\right\}\\
& = & \frac{\tau}{\lambda^K}\cdot l_j^K, \quad \forall j=0, \ldots , N.
\end{eqnarray*}

Finally, (7) follows from (3), \eqref{two*} and (5):
\begin{eqnarray*}
\mesh \cU_{j+1} &\overset{(3)}{\leq}& \frac{1}{2}\min\left\{\frac{\tau}{\lambda^K}(L(\widetilde\cU_j))^K, \frac{\tau}{\lambda^K}(\mesh\widetilde\cU_j)^K\right\} \\
 & \overset{\eqref{two*}}{\leq} &
\frac{1}{2}\min\left\{\frac{\tau}{\lambda^K}(L(\widetilde\cU_j))^K, \frac{\delta\tau}{2}\right\} \\
&\overset{(5)}{\leq} &\frac{1}{2}L_X(\cU_j, X_j), \quad \forall j=0,\ldots , N-1.
\end{eqnarray*}

\medskip

Now that the properties (1) -- (7) are established, the final goal of the construction is to use them to build an open $(n+1)$-colored cover $\cU$ of $X$ that will have properties which guarantee that $\ell$-$\dim X\leq n$. The way to build it is by induction in finitely many steps.

First, put $\widehat \cU_{-1}\coloneqq \{X\}$, noting that $\displaystyle L(\widehat\cU_{-1})=\inf_{x\in X}\sup_{U\in\widehat\cU_{-1}}\dist (x,X\setminus U)=\inf_{x\in X} \dist(x, \emptyset)=\infty$. Put 
 $\widehat\cU_0\coloneqq \cU_0$,
and assume that, for some $j\in\{0,\ldots, N-1\}$, we have already constructed families $\widehat\cU_0, \ldots , \widehat\cU_j$ so that the following properties are true, for indexes $m\in\{0,\ldots, j\}$ (except (II)$_0$, which does not make sense):
\begin{enumerate}
\item[(I)$_m$] $\widehat\cU_m$ is an $(n+1)$-colored open covering family of $X_0\cup\ldots \cup X_m$ in $X$,
\item[(II)$_m$] $\mesh \cU_m\leq\frac{1}{4}\cdot L_X(\widehat\cU_{m-2},\bigcup_{i=0}^{m-2} X_i)$,
\item[(III)$_m$] $\mesh \cU_m\leq\frac{1}{2}\cdot L_X(\widehat\cU_{m-1},\bigcup_{i=0}^{m-1} X_i)$,
\item[(IV)$_m$] $\mesh\widehat\cU_m\leq\frac{\delta\tau}{2}$,
\item[(V)$_m$] $L_X(\widehat \cU_m,\bigcup_{i=0}^{m} X_i)\geq \min\left\{\frac{1}{2}L_X(\widehat\cU_{m-1},\bigcup_{i=0}^{m-1} X_i) , L_X(\cU_m,X_m) \right\}$, and
\item[(VI)$_m$] $L_X(\widehat \cU_m,\bigcup_{i=0}^{m} X_i)\geq \frac{\tau}{\lambda^K}\cdot l_m^K$.
\end{enumerate}

\medskip

Before proceeding, we should check steps $m=0$ and $m=1$, for the basis of induction. 
We leave to the reader to check (I)$_0$ and (III)$_0$ through (VI)$_0$ (skipping (II)$_0$, which is void).

For $m=1$: (II)$_1$ is true because $\mesh \cU_1 \leq \frac{1}{4}L_X(\widehat \cU_{-1})=\infty$. Property (III)$_1$ follows from (7): $\mesh \cU_{1}\leq \frac{1}{2} \cdot L(\cU_0, X_0)=\frac{1}{2} \cdot L(\widehat\cU_0, X_0)$.
So now we  apply Proposition \ref{saturated-family} to produce $\widehat\cU_1\coloneqq \widehat\cU_0\circledast \cU_1$. Clearly (I)$_1$ is true, while properties (IV)$_1$ and (V)$_1$ follow directly from conclusions of Proposition \ref{saturated-family}: $L_X(\widehat\cU_1,X_0\cup X_1)\geq \min \left\{\frac{1}{2}L_X(\widehat\cU_0, X_0) , L_X(\cU_1,X_1) \right\}$ is exactly the statement of (V)$_1$, while from $\mesh \widehat\cU_1\leq\max\{\mesh\widehat\cU_0, \mesh \cU_1\}$, (IV)$_0$ and (4) we get the statement of (IV)$_1$.
It remains to show (VI)$_1$, which follows from (V)$_1$, (VI)$_0$, \eqref{eq: elj-root} and (6). First note:
\[\frac{1}{2}L_X(\widehat\cU_0, X_0)\overset{\text{(VI)}_0}{\geq}\frac{1}{2}\cdot\frac{\tau}{\lambda^K}\cdot l_0^K= \frac{\tau}{\lambda^K}\cdot \left(\frac{1}{\sqrt[K]{2}}\cdot l_0\right)^K \overset{\eqref{eq: elj-root}}{\geq}
\frac{\tau}{\lambda^K}\cdot l_1^K,\]
and $L_X(\cU_1,X_1) \overset{(6)}{\geq}\frac{\tau}{\lambda^K}\cdot l_1^K$, so these two facts combined with (V)$_1$ yield (VI)$_1$.

\medskip

For the step of induction, assuming, as above, that properties (I)$_m$ -- (VI)$_m$ are true for all $m\leq j$, let us prove these properties for index $m=j+1$.

Start by checking (II)$_{j+1}$:
\begin{eqnarray*} 
\mesh\cU_{j+1} \overset{(3)}{\leq}\frac{1}{2}\cdot\frac{\tau}{\lambda^K}\cdot l_j^K &\overset{\eqref{eq: elj-root}}{\leq}& \frac{1}{2}\cdot\frac{\tau}{\lambda^K}\cdot \left(\frac{1}{\sqrt[K]{2}}\cdot l_{j-1}\right)^K = \frac{1}{4}\cdot\frac{\tau}{\lambda^K}\cdot l_{j-1}^K\\
&\overset{\text{(VI)}_{j-1}}{\leq}& \frac{1}{4}\cdot  L_X(\widehat \cU_{j-1},\bigcup_{i=0}^{j-1} X_i).
\end{eqnarray*}
Now we can show (III)$_{j+1}$, because from (II)$_{j+1}$ and (7) we get:
\[ 
\mesh \cU_{j+1}\leq \min\left\{ \frac{1}{4}\cdot  L_X(\widehat \cU_{j-1},\bigcup_{i=0}^{j-1} X_i), \frac{1}{2} L_X(\cU_j,X_j)\right\}\overset{\text{(V)}_j}{\leq} \frac{1}{2}\cdot  L_X(\widehat \cU_{j},\bigcup_{i=0}^{j} X_i).
\]
Now, according to Proposition \ref{saturated-family}, we may produce $\widehat\cU_{j+1}\coloneqq \widehat\cU_j\circledast \cU_{j+1}$. Clearly (I)$_{j+1}$ is true, property (V)$_{j+1}$  is delivered directly by the conclusion about the Lebesgue number of Proposition \ref{saturated-family}, while we get the statement of (IV)$_{j+1}$ from $\mesh \widehat\cU_{j+1}\leq\max\{\mesh\widehat\cU_j, \mesh \cU_{j+1}\}$, and properties (IV)$_j$ and (4) for $j+1$.
It remains to show (VI)$_{j+1}$, which can be done by repeating the same argument used to obtain (VI)$_{1}$:
$\frac{1}{2}L_X(\widehat\cU_j, \bigcup_{i=0}^{j} X_i)\geq\frac{\tau}{\lambda^K}\cdot l_{j+1}^K$
and $L_X(\cU_{j+1},X_{j+1}) \geq\frac{\tau}{\lambda^K}\cdot l_{j+1}^K$ 
combined with (V)$_{j+1}$ yield (VI)$_{j+1}$.

\medskip

After finitely many steps of induction, we obtain an open $(n+1)$-colored cover $\cU\coloneqq \widehat\cU_N$ of $X=\bigcup_{i=0}^N X_i$, which satisfies (IV)$_N$, that is, $\mesh\cU=\mesh\widehat\cU_N\leq\frac{\delta\tau}{2}$. Also, after applying property (V) $N$ times, as well as (5) and the fact that 
$\ \displaystyle\min_{i\in I} \min\{a_i,b_i\}=\min \{\min_{i\in I} a_i, \min_{i\in I} b_i \}$,  we get
\begin{eqnarray*}
L(\cU)=L(\widehat\cU_N) 
&\overset{(V)}\geq& \min\left\{ \frac{1}{2^{N-j}}L_X(\cU_j, X_j)\ | \ j\in\{0,\ldots , N\} \right\}\\
&\overset{(5)}\geq & \min\left\{
\min\left\{ \frac{\tau}{\lambda^K}\cdot\frac{1}{2^{N-j}}(L(\widetilde\cU_j))^K, \frac{1}{2^{N-j}}\cdot\frac{\delta\tau}{2}\right\} \mid  j\in\{0,\ldots , N\} \right\}\\
& =& \min\left\{\frac{\tau}{\lambda^K}\cdot
\min\left\{ \frac{1}{2^{N-j}}(L(\widetilde\cU_j))^K\mid  j\in\{0,\ldots , N\} \right\} , \frac{1}{2^{N}}\cdot\frac{\delta\tau}{2}\right\}\\
& =& \min\left\{ \frac{\tau}{\lambda^K}\cdot l, \frac{1}{2^N}\cdot\frac{\delta\tau}{2}\right\}=\frac{\delta\tau}{2}\cdot \min\left\{ \frac{2l}{\delta\lambda^K}, \frac{1}{2^N}\right\}.
\end{eqnarray*}
Since we can choose $\tau>0$ to be arbitrarily small, and the constants $\delta$, $\lambda$, $l$, $K$ and $N$ are independent of $\tau$, this shows that $N=\ell$-$\dim X\leq n$: just put 
$\delta '\coloneqq \min\left\{\frac{2l}{\delta\lambda^K}, \frac{1}{2^N}\right\}$ and $\tau '\coloneqq \frac{\delta}{2}\tau$, and now we can say that there exists  $\delta'$ in $(0,1)$ such that for all $\tau'>0$ small enough, we can find an open, $(n+1)$-colored cover $\cU$ of $X$ with $\mesh\cU\leq\tau'$ and $L(\cU)\geq\delta'\tau'$.
\end{proof}

\chapter{Hyperbolic approximate groups revisited}\label{ChapterHyp2}

In this final chapter we discuss how the results of the previous chapters restrict the possible QI types of hyperbolic approximate groups. 

Our first main result (Theorem \ref{Asdim1Case}) provides the QI classification of non-elementary hyperbolic approximate groups of asymptotic dimension $1$. Recall that every non-elementary hyperbolic group of asymptotic dimension $1$ is quasi-isometric to some (hence any) non-abelian finitely-generated free group, hence there is only one QI class of non-elementary hyperbolic groups of $\asdim =1$. We are going to extend this result by showing that there is also only one QI class of non-elementary hyperbolic approximate groups of asymptotic dimension $1$, in particular, every non-elementary hyperbolic approximate group of asymptotic dimension $1$ is quasi-isometric to a group. After establishing this result, we will also discuss a number of further related rigidity results.

Our second main result establishes that non-elementary hyperbolic approximate groups have exponential internal growth (Theorem \ref{ExponentialGrowth}). The corresponding result in the group case is easily established by a ping-pong argument, but this argument is not available in the approximate case, since non-elementary hyperbolic approximate groups do not have to contain non-abelian free subsemigroups. In our proof we distinguish two cases. If the asymptotic dimension of the approximate group in question equals $1$, then we can refer to our first main result. For non-elementary hyperbolic approximate groups of asymptotic dimension $\geq 2$ we can establish exponential growth via classical arguments comparing asymptotic dimension and critical exponent in quasi-cobounded hyperbolic spaces. One advantage of this approach is that it provides an explicit lower bound on the exponential growth rate in terms of asymptotic dimension.

\section{Hyperbolic approximate groups of asymptotic dimension 1}

The goal of this section is to establish the following theorem; here, following \cite{MSW}, a locally finite simplicial tree is called \emph{bushy}, if every vertex is at uniformly bounded distance from a vertex which has at least $3$ unbounded complementary components.\index{simplicial tree!bushy}\index{bushy simplicial tree}
\begin{theorem}[Non-elementary hyperbolic approximate groups of asymptotic dimension $1$]\label{Asdim1Case}
Gi\-ven a non-elementary hyperbolic approximate group $(\Lambda, \Lambda^\infty)$, the following are equivalent:
\begin{enumerate}[(i)]
\item $\asdim \Lambda = 1$.
\item $\ell$-$\dim \partial \Lambda = 0$.
\item $\dim \partial \Lambda = 0$.
\item Some (hence any) representative of $\partial \Lambda$ is totally disconnected.
\item Some (hence any) representative of $\partial \Lambda$ is homeomorphic to a Cantor space.
\item Some (hence any) representative of $[\Lambda]_{\mathrm{int}}$ admits a quasi-isometric embedding into a simplicial tree.
\item $[\Lambda]_{\mathrm{int}}$ can be represented by a simplicial tree.
\item $[\Lambda]_{\mathrm{int}}$ can be represented by a locally finite bushy simplicial tree with uniformly bounded valences.
\item $\Lambda$ is quasi-isometric to a finitely-generated non-abelian free group.
\end{enumerate}
\end{theorem}
For the proof of Theorem \ref{Asdim1Case}, we note that the equivalences (i)$\iff$(ii)$\iff$(iii) are immediate from Theorems \ref{AsdimMainConvenient} and \ref{AsdimMain}.
In view of Proposition \ref{BoundaryNonel} the equivalence (iii)$\iff$(iv) is a special case of the following classical result in dimension theory (see e.g.\ \cite[Thm.\ 2.7.1]{Coornaert}):
\begin{theorem}[Zero-dimensional compact metrizable spaces] If $Z$ is a compact metrizable space, then $\dim Z = 0$ if and only if $Z$ is totally-disconnected.
\end{theorem}
The implication (iv)$\implies$(v) is immediate from Proposition \ref{BoundaryNonel} and the following classical theorem of Brouwer \cite{Brouwer}:
\begin{theorem}[Brouwer] Every non-empty totally-disconnected perfect compact metrizable space is homeomorphic to a Cantor space.
\end{theorem}
Since any Cantor space has topological dimension $0$ (see e.g. \cite[Thm.\ 2.7.1]{Coornaert}) we have thus established the equivalence of the Conditions (i)--(v).

As for the remaining implications, the implication (ii)$\implies$(vi) is a special case of Theorem \ref{BuSchEmbedding} above. The implication (vi)$\implies$(vii)$\implies$(viii) follows from the following construction:
\begin{construction}[Growing bushy trees] Let $(\Lambda, \Lambda^\infty)$ be a geometrically finitely-generated infinite approximate group, let $d$ be an internal metric on $\Lambda$ and let $(\cT, d_{\cT})$ be a simplicial tree. We assume that we are given a quasi-isometric embedding $\varphi_0: (\Lambda, d) \to (\cT, d_{\cT})$.

Since all edges in $\cT$ have length $1$, the isometric embedding $V \hookrightarrow \cT$ of the vertex set of $\cT$ has relatively dense image, so it is a quasi-isometry. Composing $\varphi_0$ with a quasi-inverse of this quasi-isometry we obtain a quasi-isometric embedding $\varphi: (\Lambda, d) \to (V, d_{\cT}|_{V \times V})$. We denote by $V_0$ the image of $\varphi$; since $(\Lambda, \Lambda^\infty)$ is geometrically finitely-generated, $(\Lambda, d)$ is coarsely connected, and hence $V_0$ is $C$-coarsely connected for some $C \geq 0$. 

Now let $\cT_0$ denote the convex hull of $V_0$. Since $\cT_0$ can be obtained from $V_0$ by connecting each vertex in $V_0$ to its nearest neighbors in $V$ by a geodesic segment of length $\leq C$, it follows that $\cT_0$ is at Hausdorff distance at most $C$ from $V_0$. Consequently, the map $\varphi$ induces a quasi-isometry (denoted by the same letter)
$\varphi: (\Lambda, d) \to (\cT_0, d_{\cT_0})$ and $(\cT_0, d_{\cT_0})$ represents $[\Lambda]_{\mathrm{int}}$. It remains to show that the tree $\cT_0$ is locally finite with uniformly bounded valences and bushy.

By Corollary \ref{CoarseBoundedGeometry}, the space $(\Lambda, d)$ has coarse bounded geometry, and hence $(\cT_0, d_{\cT_0})$ has coarse bounded geometry by Proposition \ref{QIGrowth}. In particular, $(\cT_0, d_{\cT_0})$ is uniformly locally finite. We have thus established that $(\Lambda, \Lambda^\infty)$ is quasi-isometric to a locally finite simplicial tree with uniformly bounded valences.

Moreover, since $(\Lambda, \Lambda^\infty)$ is non-elementary, $\partial \cT_0$ is infinite. This implies that there exists at least one vertex with at least $3$ unbounded complementary components. Moreover, $\partial \cT_0$ is quasi-cobounded, and hence the set of all such vertices is relatively dense. This shows that $\cT_0$ is bushy and finishes the proof.
\end{construction}
As pointed out in \cite[Sec.\ 2.1]{MSW}, any two locally finite bushy simplicial trees with uniformly bounded valences are quasi-isometric to each other, hence quasi-isometric to (the Cayley tree of) a finitely-generated non-abelian free group. This shows that (viii)$\implies$(ix). Since the boundary of a finitely-generated non-abelian free group is homeomorphic to a Cantor space, we also deduce that (ix)$\implies$(v), and hence we have established Theorem \ref{Asdim1Case}. 

\begin{remark} If $\Gamma$ is a non-elementary hyperbolic group of asymptotic dimension $1$, then it follows from Theorem \ref{Asdim1Case} that $\Gamma$ quasi-acts geometrically on a  locally finite bushy simplicial tree with uniformly bounded valences. By the celebrated QI rigidity theorem of Mosher--Sageev--Whyte \cite{MSW}, this qiqac is in fact quasi-conjugate to an isometric action on a (possibly different) locally finite bushy simplicial tree with uniformly bounded valences. This shows that every non-elementary hyperbolic group has asymptotic dimension $1$ if and only if it admits a geometric isometric action on a locally finite bushy simplicial tree with uniformly bounded valences.

In the approximate group case, we only obtain a geometric qiqac, and it is unclear to us whether it is quasi-conjugate to an actual isometric action on a tree.
\end{remark}
We record three immediate consequences of Theorem \ref{Asdim1Case}:
\begin{corollary}  Let $(\Lambda, \Lambda^\infty)$ be a non-elementary hyperbolic approximate group of asymptotic dimension $1$ without fixpoints at infinity. Then any apogee for $\Lambda$ has infinitely many ends.
\end{corollary}
\begin{corollary}\label{Asdim1Rigidity} Every non-elementary hyperbolic approximate group of asymptotic dimension $1$ is quasi-isometric to a finitely-generated group.
\end{corollary}
\begin{corollary}\label{ExponentialGrowth0} Let $(\Lambda, \Lambda^\infty)$ be a non-elementary hyperbolic approximate group of asymptotic dimension $1$ without fixpoints at infinity. Then $\Lambda$ has exponential internal growth.
\end{corollary}
Corollary \ref{ExponentialGrowth0} will be generalized to hyperbolic approximate groups of higher asymptotic dimension in the next section.

\section{Exponential growth}\label{SecExpGrowth}
In this section we are finally going to establish the following result:
\begin{theorem}[Exponential growth]\label{ExponentialGrowth} Every non-elementary hyperbolic approximate group $(\Lambda, \Lambda^\infty)$ has exponential internal growth.
\end{theorem}
For hyperbolic approximate groups of asymptotic dimension 1 this was established in Corollary \ref{ExponentialGrowth0}. Thus for the rest of this section we are going to assume that $(\Lambda, \Lambda^\infty)$  is a non-elementary hyperbolic approximate group with $\asdim \Lambda \geq 2$. We will work in the following setting: 
\begin{notation}[General Setting]\label{BSAsdim2} Throughout this section we fix a Bonk--Schramm realization $L \subset \partial \bH^n$ of $\partial \Lambda$ so that in particular,
\begin{equation}\label{dimBonkSchramm}
    \dim L = \dim \partial \Lambda = \asdim \Lambda -1 \geq 1.
\end{equation}
Then $X \coloneqq  \cH(L) \subset \bH^n$ is a closed convex apogee for $\Lambda$ with $\partial X = \cL(X) = L$. We fix a left-regular quasi-action $\rho: (\Lambda, \Lambda^\infty) \to \widetilde{\mathrm{QI}}(X)$  with $\rho(e) = \mathrm{Id}$. 

We also denote by ${\mathrm{pr}}_X: \bH^n \to X$ the nearest point projection and fix a basepoint $o \in X$. We then denote by $\iota: \Lambda \to X$, $\lambda \mapsto \rho(\lambda)(o)$ the associated orbit map. Since $\rho$ is geometric, the map $\iota$ is a quasi-isometric embedding with respect to some (hence any) internal metric on $\Lambda$.

We denote by $\cO \coloneqq  \rho(\Lambda).o$ the image of $\iota$ (i.e. the quasi-orbit of $o$), and given $\lambda \in \Lambda$ we set
\[
\|\lambda\| \coloneqq  d(o, \iota(\lambda)).
\]
Finally we abbreviate $\overline{B}_R \coloneqq  \{\lambda \in \Lambda \mid \|\lambda\| \leq R\}$ and $A_R \coloneqq  \overline{B}_{R+1} \setminus \overline{B}_R$ so that $\Lambda = \bigsqcup A_R$.
\end{notation}

\begin{remark}[Exponential growth rate] 
Since $d$ is an internal metric on $\Lambda$, internal exponential growth of $\Lambda$ amounts to showing that $0<\delta(\Lambda) < \infty$, where 
\[
  \delta(\Lambda) \coloneqq  \liminf \frac{1}{R} \cdot \log|B_R| < \infty
\]
is the \emph{exponential growth rate}\index{exponential growth rate} of $\Lambda$. Since $(\Lambda, d)$ is large-scale geodesic, it is quasi-isometric to a locally finite graph, and since it has bounded local geometry any such graph has bounded valences. This implies that $\delta(\Lambda) < \infty$, and thus to establish exponential growth we only need to show that $\delta(\Lambda) > 0$.
\end{remark}
\begin{definition} 
The \emph{growth series}\index{growth series!of approximate group} of $\Lambda$ is defined as
\[
S_\alpha \coloneqq   \sum_{\lambda \in \Lambda} e^{-\alpha \|\lambda\|},
\]
and the \emph{critical exponent}\index{critical exponent!of approximate group} $\delta_\Lambda$ is defined as
\[
\delta_\Lambda \coloneqq  \inf\{\alpha \mid S_\alpha \text{ converges}\}.
\]
\end{definition}
With this terminology at hand, we are going to establish the following more precise version of Theorem \ref{ExponentialGrowth} (in the case of higher asymptotic dimension):
\begin{theorem}[Asymptotic dimension and exponential growth rate]\label{ExpGrowthHigher}
Under our standing assumption that $\asdim \Lambda \geq 2$ we have
\[\delta(\Lambda) \quad \geq \quad \delta_\Lambda \quad \geq \quad \asdim \Lambda -1 \quad \geq \quad 1.\]
In particular, $\delta(\Lambda) > 0$, i.e.\ $\Lambda$ has internal exponential growth.
\end{theorem}
The proof of the first inequality is standard:
\begin{lemma}[Exponential growth rate vs.\ critical exponent]
We have $\delta(\Lambda) \geq \delta_\Lambda$.
\end{lemma}
\begin{proof} Let $\epsilon > 0$ and $\alpha \coloneqq  \delta(\Lambda) + \epsilon$. Since
\[
\alpha > \liminf \frac{1}{R+1} \cdot \log|B_{R+1}| + \epsilon =  \liminf \frac{1}{R} \cdot \log|B_{R+1}| + \epsilon,
\]
we can find $R_0$ such that for all $R>R_0$ we have $\alpha R \geq \log  |B_{R+1}| + \epsilon R/2$. For all $R>R_0$ and all $\lambda \in A_R = B_{R+1} \setminus B_R$ we then have
\[
e^{-\alpha \|\lambda\|} \leq e^{-\alpha R} \leq |B_{R+1}|^{-1} \cdot e^{-\epsilon R/2} \leq |A_{R}|^{-1} \cdot e^{-\epsilon R/2}.
\]
This implies
\begin{eqnarray*}
S_\alpha &=& \sum_{R=0}^\infty \sum_{\lambda \in A_R} e^{-\alpha\|\lambda\|} \quad \leq \quad  \sum_{R=0}^{R_1} \sum_{\lambda \in A_R} e^{-\alpha\|\lambda\|} +  \sum_{R=R_1+1}^\infty \sum_{\lambda \in A_R} |A_{R}|^{-1} \cdot e^{-\epsilon R/2}\\
&\leq& |B_{R+1}| + \sum_{R = R+1}^\infty e^{-\epsilon R/2} \quad < \quad \infty.
\end{eqnarray*}
This shows that $S_\alpha$ converges for every $\alpha > \delta(\Lambda)$, and thus $\delta(\Lambda) \geq \delta_\Lambda$.
\end{proof}

In order to prove the final inequality 
\begin{equation}
    \delta_\Lambda \geq \asdim\Lambda-1,
\end{equation}
we will interpret the right-hand side as a lower bound on the Hausdorff dimension of the conical limit set of $\Lambda$.
\begin{remark}[Conical limit set]\label{ConicalLimtSetExpGrowth} Since $\rho$ is geometric, it is convex cocompact and thus, by Proposition \ref{cor:conical} every limit point of $\Lambda$ (or $\cO$) in $\partial X$ is conical. On the other hand, since $\cO$ is relatively dense in $X$ and the latter is a closed convex subset of $\bH^n$, we have $\cL(\Lambda) = \cL(X) = L$. This shows that if we consider $\Lambda$ as a subset of $\bH^n$, then
\begin{equation}\label{ConicalLimitBS}
\cL^{\mathrm{con}}(\Lambda) = \cL(\Lambda) = L.
\end{equation}
\end{remark}
\begin{remark}[Hausdorff dimension]\label{Hausdorff}
Let $E$ be a subset of a metric space $(Z,d)$. Given $\alpha > 0$ and $\delta \in [0, \infty]$ we define
\[
H_\alpha^\delta(E,d) \coloneqq  \inf\left\{\sum_{j=1}^\infty r_j^\alpha \mid \exists\, x_j \in E, r_j \in (0, \delta):\;  E \subset \bigcup_{j=1}^\infty B(x_j, r_j)\right\}.
\]
Then $H_\alpha(E,d) \coloneqq  \lim_{\delta \to 0} H_\alpha^\delta(E,d)$ is called the \emph{$\alpha$-dimensional Hausdorff measure}\index{Hausdorff measure!$\alpha$-dimensional} of $E$ and $H_\alpha^\infty(E,d)$ is called the \emph{$\alpha$-dimensional Hausdorff content}\index{Hausdorff content!$\alpha$-dimensional} of $E$. With this notation understood, the \emph{Hausdorff dimension}\index{Hausdorff dimension}\index{dimension!Hausdorff} of $E$ (with respect to $d$) can be characterized as
\begin{eqnarray*}
\dim_{\mathrm{Haus}}(E,d) &\coloneqq & \inf\{\alpha\mid H_\alpha(E,d) = 0\} = \inf\{\alpha\mid H^\infty_\alpha(E,d) = 0\}\\
&=&\inf\{\alpha\mid H_\alpha(E,d) < \infty\}. 
\end{eqnarray*}
Each Hausdorff measure is monotone and $\sigma$-additive on Borel sets \cite[Theorem 27]{Rogers}, and hence if a Borel set $E$ can be written as a countable union of Borel sets $E_M$, then we have
\[
\dim_{\mathrm{Haus}}(E,d) = \sup_{M} \dim_{\mathrm{Haus}}(E_M,d).
\]
\end{remark}
\begin{proposition}[Hausdorff dimension of the conical limit set]
If $d$ is any metric on $\partial \bH^n$ which induces the usual topology, then
\[
\dim_{\mathrm{Haus}}(\cL^{\mathrm{con}}(\Lambda), d) \geq \asdim\Lambda - 1.
\]
\end{proposition}
\begin{proof} By Theorem \ref{AsdimMainConvenient} we have 
\[
\dim L =\dim \partial \Lambda = \asdim\Lambda - 1.
\]
Now $L$ is a separable metrizable space, and thus for $L$ the topological dimension $\dim L$ coincides with the small inductive dimension $\ind L$ by \cite[Thm.\ 1.6.4 and Thm.\ 4.1.3]{Engelking}. Finally, as explained in \cite[Section VII.4]{HurewiczWallman}, it follows from \cite[Thm.\ VII.2]{HurewiczWallman} that $\ind L$ (which unfortunately is denoted $\dim L$ in \emph{loc.\ cit.}) is bounded above by the Hausdorff dimension $\dim_{\mathrm{Haus}}(L, d)$. In view of Remark \ref{ConicalLimtSetExpGrowth} this establishes the proposition.
\end{proof}
We have thus reduced the proof of Theorem \ref{ExpGrowthHigher} to the following statement:
\begin{theorem}[Critical exponent and Hausdorff dimension]\label{CritExp}
We have $\delta_\Lambda \geq \dim_{\mathrm{Haus}}(\cL^{\mathrm{con}}(\Lambda))$.
\end{theorem}
The remainder of this section is devoted to the proof of Theorem \ref{CritExp}, which is essentially the same as the proof of \cite[Thm.\ 2.1]{BishopJones}. Firstly, we need a lemma which allows us to provide upper bounds on Hausdorff dimension:
\begin{lemma}[Upper dimension bound]\label{HausdorffUpper} Let $(X,d)$ be a metric space and let
$\alpha>0$. Assume that there exists a sequence $(x_n)$ in $X$ and a sequence $(r_n)$ of positive real numbers such that each $x \in X$ is contained in infinitely many of the balls $B(x_n, r_n)$ and such that $\sum r_n^\alpha < \infty$. Then $\dim_{\mathrm{Haus}}(X) \leq \alpha$.
\end{lemma}
\begin{proof} Since every $x \in X$ is contained in infinitely many balls, the collection of balls $\cU_N \coloneqq  (B(x_n, r_n))_{n \geq N}$ covers $X$ for every $N \in \mathbb N$. By definition of the $\alpha$-dimensional Hausdorff content we thus have
\[
H_\alpha^\infty(X) \leq C_N \coloneqq  \sum_{n=N}^\infty r_n^\alpha.
\]
Since $C_1$ is finite we have $C_n \to 0$ and thus $H_\alpha^\infty(X) = 0$, which implies the lemma.
\end{proof}
As a second ingredient we need some estimates from hyperbolic trigonometry. 
\begin{notation}\label{HypBallModel}
We are going to work in the Poincar\'e ball model of $\bH^n$; thus from now on $\bH^n$ denotes the unit ball in $\R^n$ with the hyperbolic metric and we assume that the basepoint is given by $o = 0$. Given $x \in \bH^n$ we denote by $\|x\|$ the hyperbolic distance of $x$ from $0$ and by $\|x\|_E \in [0,1)$ the Euclidean distance of $x$ from $0$ so that
\[
\|x\| = 2 \mathrm{arctanh} (\|x\|_E) = \log\left(\frac{1+\|x\|_E}{1-\|x\|_E}\right).
\]
We then identify $\partial \bH^n$ with the unit sphere in $\R^n$ and denote by $d_{\angle}$ the angular metric on $\partial \bH^n$. By definition, if $\eta, \xi \in \partial \bH^n$, then $d_{\angle}(\xi,\eta)$ is the arc length of the shortest segment of a great circle which contains $\xi$ and $\eta$.
\end{notation}
\begin{lemma}\label{HypTrig} For every $C>0$ there exists $M_C>0$ such that the following holds: If $x \in \bH^n$ with $\|x\|_E > \frac 1 2$ and $\gamma \colon [0, \infty) \to \bH^n$ is a geodesic ray in $\mathbb{H}^n$ emanating from $0$ with $d(\gamma, x)\leq C$, then the endpoint $\xi$ of $\gamma$ in $\partial \bH^n$ satisfies\[d_{\angle}\left(\xi, \frac{x}{\|x\|_E}\right)\leq M_C(1-\|x\|_E) \quad \text{or} \quad d_{\angle}\left(\xi, \frac{x}{\|x\|_E}\right) \geq \frac{\pi}{2}.\]
\end{lemma}
\begin{proof} Since the statement is invariant under isometries, we may assume without loss of generality that $n=2$. Now let 
\[
\theta \coloneqq  d_\angle\left(\xi, \frac{x}{\|x\|_E}\right).
\]
We will assume that $\theta < \frac \pi 2$. Note that, by definition, $\theta$ is just the hyperbolic angle between the rays through $\xi$ and $x$. If $C'$ is the hyperbolic distance between $x$ and its nearest point projection onto $\gamma$, then $C' \leq C$, and by basic hyperbolic trigonometry we have
\[\sin(\theta) = \frac{\sinh(C')}{\sinh(\|x\|)}\leq \frac{\sinh(C)}{\sinh(2\mathrm{arctanh}(\|x\|_E))}=\frac{\sinh(C)(1-\|x\|_E^2)}{2\|x\|_E}.\]
Since $\theta < \frac \pi 2$ we have $\sin(\theta) \geq \frac{1}{2}\theta$ and since $\|x\|_E \geq \frac 1 2$ we have $\frac{1-\|x\|_E^2}{\|x\|_E} \leq 4(1-\|x\|_E)$. We deduce that
\[
\theta \leq 2 \sin(\theta) < 4 \sinh(C) \cdot (1-\|x\|_E).
\]
We may thus choose $M_C \coloneqq  4 \sinh(C)$.
\end{proof}

\begin{proof}[Proof of Theorem \ref{CritExp}] It will be convenient for us to work in the ball model of $\bH^n$; we will use the notation of Notation \ref{HypBallModel}.

Let $\epsilon>0$ and $\alpha \coloneqq  \delta_\Lambda + \epsilon$; by definition of $\delta_\Lambda$ this means that $S_\alpha$ converges. Since $d(0,x) = \|x\| =  \log\left(\frac{1+\|x\|_E}{1-\|x\|_E}\right)$, convergence of $S_\alpha$ is equivalent to the statement that
\begin{equation}\label{NichollsFormula}
    \sum_{\lambda \in \Lambda} (1-\|\rho(\lambda)(0)\|_E)^\alpha < \infty,
\end{equation}
see \cite[Sec.\ 1.6]{Nicholls}. For every $x \in \bH^n \setminus\{0\}$ we denote by $\gamma_x$ the unique geodesic ray from $0$ through $x$. We then define
\[
\pi_\infty: \bH^n \setminus\{0\} \to \partial\bH^n, \quad x \mapsto \gamma_x(\infty) = \frac{x}{\|x\|_E}.
\]
We extend this map to all of $\bH^n$ by setting $\pi_\infty(0) \coloneqq  (1, 0\dots, 0)^\top$ and enumerate $\Lambda = \{\lambda_1, \lambda_2, \dots\}$. For every $M \in \mathbb N$ we then define
\[
\xi_n \coloneqq  \pi_\infty(\rho(\lambda_n)(0)) \qand r^{(M)}_n \coloneqq  M(1-\|\rho(\lambda_n)(0)\|_E).
\]
We then set $\cU_M \coloneqq  \{B(\xi_n, r^{(M)}_n) \mid n \in \mathbb N\}$ and denote by $E_M$ the set of all $\xi \in \partial \bH^n$ which are contained in infinitely many elements of $\cU_M$. We also define $E \coloneqq  \bigcup E_M$.
By \eqref{NichollsFormula} and Lemma \ref{HausdorffUpper} we then have $\dim_{\mathrm{Haus}}(E_M) \leq \alpha$ for all $M \in \mathbb N$ and thus $\dim_{\mathrm{Haus}}(E) \leq \alpha$ by Remark \ref{Hausdorff}; letting $\epsilon \to 0$ we therefore obtain $\dim_{\mathrm{Haus}}(E) \leq \delta_\Lambda$.

To conclude the proof, it suffices to show that every conical limit point is contained in the set $E$. Thus let $\xi$ be a conical limit point and let $\gamma$ be a geodesic ray emanating from $0$ which represents $\xi$. Then there exist $C>0$ and infinitely many elements $x_1, x_2, \dots,$ in the quasi-orbit $\rho(\Lambda).o$ which are at distance $< C$ from $\gamma$. Since the quasi-orbit is locally finite, we have \[d_\angle\left(\xi, \frac{x_n}{\|x_n\|_E}\right)< \frac \pi 2 \qand \|x_n\|_E \geq \frac 1 2 \quad \text{for almost all } n.\]
If we now choose $M\coloneqq M_C$ as in Lemma \ref{HypTrig}, then the lemma shows that $\xi \in E_M$. This finishes the proof.
\end{proof}

\begin{appendix}

\chapter{Background from large-scale geometry}\label{AppendixLSG}

\section{Notations concerning (pseudo-)metric spaces}
In this section we fix our notation concerning metric spaces. In fact, it will be convenient for us to allow for a slight generalization of metric spaces called pseudo-metric spaces:
\begin{definition}\label{DefPseudoMetric}  
Let $X$ be a set. A function $d: X \times X \to [0, \infty)$ is called a \emph{pseudo-metric}\index{pseudo-metric} if for all $x,y,z \in X$ the following hold:
\begin{enumerate}[(i)]
\item $d(x,x) = 0$,
\item $d(x,y) = d(y,x)$,
\item $d(x,z) \leq d(x,y) + d(y,z)$.
\end{enumerate}
In this case the pair $(X,d)$ is called a \emph{pseudo-metric space}\index{pseudo-metric!space}.
\end{definition}
If $(X,d)$ is a pseudo-metric space, then we obtain an equivalence relation $\sim$ on $X$ by setting $x \sim y :\Leftrightarrow d(x,y) = 0$. The pseudo-metric $d$ then descends to a metric $\bar d$ on the quotient space $\bar X \coloneqq  X/_\sim$, and we refer to $(\bar X, \bar d)$ as the \emph{metric quotient}\index{metric quotient} of $(X,d)$. Our reason to consider pseudo-metric spaces is that they arise naturally in the context of isometric group actions:
\begin{example} Assume that a group $\Gamma$ acts by isometries on a metric space $(X,d_X)$ and let $o \in X$ be a basepoint. Then we obtain a left-invariant pseudo-metric $d$ on $\Gamma$ by setting $d(g,h) \coloneqq  d_X(g.o, h.o)$. In general, this pseudo-metric is not a metric.
\end{example}

We will use the following notations concerning pseudo-metric spaces: If $(X, d)$ is a pseudo-metric space, a point $x \in X$ and $R > 0$, we denote by 
\[B(x, R) \coloneqq  \{x' \in X\mid d(x,x')< R\} \quad \text{and} \quad \bar B(x, R) \coloneqq   \{x' \in X\mid d(x,x')\leq R\} \]
the open, respectively closed ball of radius $R$ around $x$.  Given $A \subset X$ we denote by $N_R(A)$ the (open) $R$-neighborhood of $A$ in $X$, i.e.
\[
N_R(A) \coloneqq  \bigcup_{x \in A} B(x, R).
\]
The open balls $B(x, R)$ form the basis for a topology $\tau_d$ on $X$, which is Hausdorff if and only if $d$ is a metric. By abuse of language, if $(X, \tau_d)$ satisfies a topological property, we will also say that $(X, d)$ has this property. For example, we say that $(X,d)$ is connected if $(X, \tau_d)$ is. We say that $(X,d)$ is \emph{proper}\index{proper!space} if for some (hence any)
$o \in X$ the map $x \mapsto d(o, x)$ is proper. Equivalently, the balls $B(o, R)$ are \emph{relatively compact}\index{relatively compact} (i.e.\ they have compact closure) for all $R>0$ and some (hence any) $o \in X$.
\begin{definition}\label{DefDelone1} Let $(X, d)$ be a pseudo-metric space, let $A \subset X$ and let $R > r > 0$.
\begin{enumerate}[(i)]
\item $A$ is called \emph{$r$-uniformly discrete}\index{uniformly discrete}\index{uniformly discrete!$r$-uniformly discrete} in $X$ if $d(x,y) > r$ for all $x,y \in A$ with $x \neq y$.
\item $A$ is called \emph{relatively dense}\index{relatively dense!$R$-relatively dense} in $X$ if $N_R(A) = X$. 
\item $A$ is called \emph{$(r, R)$-Delone}\index{Delone!set}\index{Delone!$(r,R)$-Delone} if it is $r$-uniformly discrete and $R$-relatively dense. In this case $(r, R)$ are called \emph{Delone parameters}\index{Delone!parameters} for $A$.
\end{enumerate}
\end{definition}
It is well-known that every pseudo-metric space contains a Delone set (cf. \cite[Prop. 3.C.3]{CdlH}). There are other terms for the properties from the previous definition, especially for (ii): such a set $A$ is often called quasi-dense, or coarsely dense. 

\section{Coarse maps and coarse equivalences} 

The following terminology follows \cite[Chapter 3]{CdlH} (except for coarse equivalence, which is called \emph{metric coarse equivalence} in \cite{CdlH}, even among pseudo-metric spaces).

\begin{definition}\label{def: coarsely - maps} Let $(X, d_X)$, $(Y, d_Y)$ be pseudo-metric spaces and $f, f': X \to Y$ be maps.
\begin{enumerate}[(i)]
\item A non-decreasing function $\Phi_+: [0, \infty) \to [0, \infty)$ is called an \emph{upper control function}\index{upper control function}. A non-decreasing function $\Phi_-:  [0, \infty) \to [0, \infty]$ with $\lim_{t \to \infty} \Phi_-(t) = \infty$ is called a \emph{lower control function}\index{lower control function}.
\item An upper control function $\Phi_+$ (respectively a lower control function $\Phi_-$) is an \emph{upper control for $f$} (respectively a \emph{lower control for $f$}) if $d_Y(f(x), f(x')) \leq \Phi_+(d_X(x,x'))$ (respectively $d_Y(f(x), f(x'))$ $\geq \Phi_-(d_X(x,x'))$) for all $x, x' \in X$.
\item $f$ is called \emph{coarsely Lipschitz}\index{coarsely Lipschitz} (respectively \emph{coarsely expansive}\index{coarsely expansive}, a \emph{coarse embedding}\index{coarse embedding}) if it admits an upper control (respectively a lower control, both an upper and a lower control).
\item $f$ is called \emph{coarsely proper}\index{coarsely proper} if pre-images of bounded sets are bounded.
\item  $f$ is called \emph{essentially surjective}\index{essentially surjective} or \emph{coarsely surjective}\index{coarsely surjective} if its image is relatively dense in $(Y, d_Y)$. It is called a \emph{coarse equivalence}\index{coarse equivalence} if it is a coarse embedding and essentially surjective.
\item $f$ and $f'$ are \emph{close}\index{close functions}, denoted $f \sim f'$, if they are at uniformly bounded distance in the sense that
\[
\sup_{x \in X}d_Y(f(x), f'(x)) < \infty.
\]
We denote by $[f]$ the equivalence class of $f$ under closeness.
\item A map $g: Y \to X$ is called a \emph{coarse inverse}\index{coarse inverse} to $f$ if $[g \circ f] = [{\rm Id}_X]$ and $[f \circ g] = [{\rm Id}_Y]$.
\end{enumerate}
\end{definition}

We refer the reader to \cite[Chapter 3]{CdlH} for basic properties of coarsely expansive and coarsely Lipschitz functions. The terminology varies a lot throughout the literature (cf.\ \cite[Remark 3.A.4]{CdlH}). In \cite{Roe}, coarsely Lipschitz maps between metric spaces are referred to as (uniformly) bornologous, while in \cite{BDLM} they are called large scale uniform; in \cite{BellDran1}, coarsely expansive maps are called uniformly expansive maps. Coarsely Lipschitz maps and coarsely expansive maps can also be characterized without explicit mentioning of control functions (see \cite[Prop.\ 3.A.5]{CdlH}):
\begin{lemma} \label{characterisation-CL,CE}
A map $f:X\to Y$ between pseudo-metric spaces is coarsely Lipschitz if and only if for every $r\geq 0$, there exists $s\geq 0$ such that, if $x, x' \in X$ satisfy $d_X(x, x')\leq r$, then $d_Y (f(x),f(x'))\leq s$. 

A map $f:X\to Y$ between pseudo-metric spaces is coarsely expansive if and only if for each $s\geq 0$, there exists $r\geq 0$ such that, if $x, x' \in X$ satisfy $d_X(x, x')\geq r$, then $d_Y (f(x),f(x'))\geq s$. \qed
\end{lemma}

\begin{remark}[Metric coarse category]\label{Coa0} Note that coarsely expansive (respectively coarsely Lipschitz) functions are closed under composition, and that composition descends to closeness classes. We may thus define a category $\mathbf{Coa}_0$ as follows: Objects are non-empty pseudo-metric spaces, and morphisms from $(X,d)$ to $(Y,d)$ are closeness classes of coarsely Lipschitz maps. This category is called the \emph{metric coarse category}\index{metric coarse category} in \cite[Def.\ 3.A.7]{CdlH}, and in this category the following hold (\cite[Prop. 3.A.16]{CdlH}).
\end{remark}

\begin{lemma}[Morphisms in the metric coarse category]\label{CoarseMorphisms} Let $f: X \to Y$ be a coarsely Lipschitz map between non-empty pseudo-metric spaces and $[f]$ the corresponding morphism in $\mathbf{Coa}_0$. 
\begin{enumerate}[(i)]
\item $[f]$ is a monomorphism in $\mathbf{Coa}_0$ if and only if $f$ is coarsely expansive.
\item $[f]$ is an epimorphism in $\mathbf{Coa}_0$ if and only if $f$ is coarsely surjective.
\item $[f]$ is an isomorphism in $\mathbf{Coa}_0$ if and only if $f$ is a coarse equivalence if and only if $[f]$ is a monomorphism and an epimorphism.
\end{enumerate}
\end{lemma}
Explicitly, (iii) states that a coarsely Lipschitz map $f: X \to Y$ (with $X \neq \emptyset$) is a coarse equivalence if and only if there exists a coarsely Lipschitz map $g: Y \to X$ which is a coarse inverse for $f$.

\begin{definition} Two pseudo-metric spaces $(X, d_X)$ and $(Y, d_Y)$ are called \emph{coarsely equivalent} if there exists a coarse equivalence $f: X \to Y$.
\end{definition}

By Lemma \ref{CoarseMorphisms}, for non-empty pseudo-metric spaces, coarse equivalence is the same as isomorphism in the category $\mathbf{Coa}_0$. In particular, coarse equivalence defines an equivalence relation on pseudo-metric spaces. An invariant of metric spaces will be called a \emph{coarse invariant}\index{coarse invariant} if it is invariant under this equivalence relation. We denote by $[(X, d)]_c$ (or simply $[X]_c$) the coarse equivalence class of the pseudo-metric space $(X,d)$.

\begin{remark}\label{coarse structure def}
 The above notions can be defined in the wider context of coarse spaces in the sense of Roe \cite{Roe} (cf.\ \cite[3.E]{CdlH}). Recall that if $X$ is a set, then a non-empty collection $\mathcal{E}$ of subsets of $X\times X$ is called a \emph{coarse structure}\index{coarse structure} (or a \emph{uniform bornology}\index{uniform bornology}) on $X$ and the pair $(X,\mathcal{E})$ is called a \emph{coarse space}\index{coarse space} if the following properties hold. 
\begin{enumerate}[({E}1)]
\item$\Delta(X)=\{(x,x) \mid x\in X\} \in \mathcal E$,
\item If $E\in \mathcal{E}$, and $F\subset E$, then $F\in \mathcal{E}$,
\item If $E_1,E_2\in \mathcal{E}$, then $E_1\cup E_2\in\mathcal{E}$,
\item If $E\in \mathcal{E}$, then $E^{-1}\in \mathcal{E}$,
\item If $E, F \in \mathcal{E}$, then $E \circ F \in \mathcal E$,
\end{enumerate}
where  $E^{-1}=\{(x',x) \mid (x,x')\in E\}$, $E \circ F= \{(x', x'') \mid \exists x\in X \text{ with } (x', x)\in E, \ (x,x'' )\in F\}$.\\
The elements of $\mathcal E$ are then called \emph{controlled sets}\index{controlled sets} or \emph{entourages}\index{entourages}. A subset $B \subset X$ is called \emph{$\mathcal E$-bounded}\index{bounded!$\mathcal E$-bounded} if $B \times B$ is controlled.

If $(X, \mathcal E_X)$ and $(Y, \mathcal E_Y)$ are coarse spaces, then a map $f: X \to Y$ is called \emph{proper}\index{proper!map} if pre-images of bounded sets under $f$ are bounded, and
 \emph{bornologous}\index{bornologous}\index{bornologous!map} if $(f\times f)(\mathcal E_X) \subset \mathcal E_Y$. A proper bornologous map is called a \emph{coarse map}\index{coarse map}. Two coarse maps $f, f': X \to Y$ are considered \emph{equivalent}\index{coarse map!equivalence}, denoted $f \sim f'$, if $\{(f(x), f'(x)) \mid x \in X\}$ is controlled. Coarse spaces and equivalence classes of coarse maps form a category $\mathbf{Coa}$. (Let us note here that what we call a bornologous map, in \cite[ Def.\ 3.E.1]{CdlH} it is called a coarse map, i.e., their coarse maps do not have to satisfy  the requirement of properness.)

If $(X,d)$ is a pseudo-metric space, then 
\[
\mathcal E_d \coloneqq  \{E \subset X \times X \mid  \sup_{(x,y) \in E} d(x,y) < \infty\}
\]
defines a coarse structure on $X$, called the \emph{coarse structure generated by $d$}\index{coarse structure!generated by $d$}. The $\mathcal E_d$-bounded subsets of $X$ are just the $d$-bounded subsets. If $(X,d_X)$ and $(Y, d_Y)$ are pseudo-metric spaces, then a map $f: (X, \mathcal E_{d_X}) \to (Y, \mathcal E_{d_Y})$ is bornologous if and only if it is coarsely Lipschitz. On the other hand, while $f: (X, \mathcal E_{d_X}) \to (Y, \mathcal E_{d_Y})$ being coarsely expansive implies that $f$ is proper, a proper map $f$ need not be coarsely expansive (see Remark 3.A.13 (2), \cite{CdlH}).
Still, it is true that if two pseudo-metric spaces are coarsely equvialent, then their associated coarse spaces are isomorphic in the category $\mathbf{Coa}$. 
\end{remark}

\section{Quasi-isometries, bi-Lipschitz maps and rough isometries}
We will be particularly interested in coarse equivalences with (at most) affine  control functions.
\begin{definition} Let $(X, d_X)$, $(Y, d_Y)$ be pseudo-metric spaces and $f: X \to Y$ be a map. Given $C, C' \geq 0$ and $K\geq 1$ we say that $f$ is
\begin{enumerate}[(i)]
\item a \emph{$(K,C)$-quasi-isometric embedding}\index{quasi-isometric embedding}\index{embedding!quasi-isometric} if for all $x_1, x_2 \in X$,
\[
\frac{1}{K} \cdot d_X(x_1,x_2) - C \leq d_Y(f(x_1), f(x_2)) \leq K \cdot d_X(x_1,x_2)+C;
\]
\item a \emph{$(K, C, C')$-quasi-isometry}\index{quasi-isometry} if it is a $(K,C)$-quasi-isometric embedding whose  image is, moreover, $C'$-relatively dense in $Y$;
\item a \emph{$K$-bi-Lipschitz embedding}\index{bi-Lipschitz embedding}\index{embedding!bi-Lipschitz} if it is a $(K, 0)$-quasi-isometric embedding;
\item a \emph{$K$-bi-Lipschitz equivalence}\index{bi-Lipschitz equivalence} if it is a surjective bi-Lipschitz embedding;
\item a \emph{$C$-coarse isometric embedding}\index{coarse isometric embedding}\index{embedding!coarse isometric} if it is a $(1,C)$-quasi-isometric embedding;
\item a \emph{$(C, C')$-coarse isometry}\index{coarse isometry} if it is a $(1,C, C')$ quasi-isometry;
\item an \emph{isometric embedding}\index{isometric embedding}\index{embedding!isometric} if it is a $(1,0)$-quasi-isometric embedding;
\item an \emph{isometry}\index{isometry} if it is a surjective isometric embedding.
\end{enumerate}
\end{definition}
\begin{remark}\label{rem:qi facts}
 If the precise constants do not matter, we simply refer to $f$ as a quasi-isometric embedding, quasi-isometry etc.\ An equivalence class of quasi-isometries will be called a $(K, C, C')$-quasi-isometry class if it admits a representative which is a $(K, C, C')$-quasi-isometry, and we will use similar terminology for the other types of maps. A collection of quasi-isometric embeddings, quasi-isometries etc.\ (or equivalence classes thereof) is called \emph{uniform}\index{uniform collection!of maps} if the implied constants are uniformly bounded.\index{uniform collection!of quasi-isometries}

If $f: X \to Y$ is a $(K_1, C_1)$-quasi-isometric embedding and $g: Y \to Z$ is a $(K_2, C_2)$-quasi-isometric embedding, then $g \circ f: X \to Z$ is a $(K_1K_2, K_2C_1 + C_2)$ quasi-isometric embedding. If $f$ is a $(K_1, C_1, C_1')$-quasi-isometry and $g$ is a $(K_2, C_2, C_2')$-quasi-isometry, then $f \circ g$ is a $(K_1K_2, K_2C_1+C_2, K_2C_1'+C_2 + C_2')$-quasi-isometry. In particular, if $A$ is a uniform collection of quasi-isometries from $X \to Y$ and $B$ is a uniform collection of quasi-isometries from $Y$ to $Z$, then the set $B \circ A$ of all compositions is a uniform collection of quasi-isometries.

We say that $(X, d_X)$ and $(Y, d_Y)$ are \emph{quasi-isometric} (respectively \emph{bi-Lipschitz, coarsely-isometric, isometric}) if there exists a quasi-isometry (bi-Lipschitz equivalence, coarse isometry, isometry) $f: X \to Y$. If $f$ is a $(K, C, C')$-quasi-isometry, then $f$ admits a coarse inverse $\bar f$ which is a $(K, KC+2KC', KC+KC')$-quasi-isometry. It follows that quasi-isometry (and similarly bi-Lipschitzness, coarse isometry, isometry) is an equivalence relation on pseudo-metric spaces. Note that a coarse inverse for a quasi-isometry is usually referred to as quasi-inverse.

In the sequel we will denote the equivalence class of $(X, d_X)$ under quasi-isometry by $[X, d_X]$ (or simply $[X]$) and refer to it as the \emph{QI type}\index{QI-type} of $X$. Note that $[X, d_X] = [\bar X, \bar d_X]$, hence every QI type can be represented by a metric (as opposed to pseudo-metric) space.
\end{remark}

\begin{remark}\label{ProperVsLocallyFinite} As mentioned earlier, every pseudo-metric space $(X,d)$ contains a Delone set $A \subset X$. Since $A$ is relatively dense in $X$, the inclusion $A \hookrightarrow X$ is a quasi-isometry. It follows that every pseudo-metric space is quasi-isometric to a uniformly discrete metric space. In particular, every proper pseudo-metric space is quasi-isometric to a locally finite metric space.
\end{remark}

\begin{remark} Given a pseudo-metric space $(X, d_X)$, we denote by $\widetilde{\rm QI}_{K,C, C'}(X)$ the set of all $(K, C, C')$-quasi-isometries from $X$ to itself. We also set
\[
\widetilde{\rm QI}(X) \coloneqq  \underset{C, C'\geq 0}{\bigcup_{K \geq 1}} \widetilde{\rm QI}_{K,C, C'}(X).
\]
By the previous remark, $\widetilde{\rm QI}(X)$ is a semigroup under composition, which contains the \emph{isometry group} ${\rm Is}(X)  = \widetilde{\rm QI}_{1,0, 0}(X)$ as a subgroup. 
Moreover, the quotient ${\rm QI}(X) \coloneqq  \widetilde{\rm QI}(X)/_\sim$ is a group with respect to composition of representatives, and we refer to this group as the \emph{quasi-isometry group}\index{quasi-isometry group} of $X$. Composing the inclusion ${\rm Is}(X) \to  \widetilde{\rm QI}(X)$ with the quotient map $ \widetilde{\rm QI}(X) \to {\rm QI}(X)$ we obtain a \emph{canonical homomorphism}
${\rm Is}(X) \to {\rm QI}(X)$. In general, the canonical homomorphism is neither injective nor surjective.
\end{remark}
\section{Geodesics and quasi-geodesics in pseudo-metric spaces}
\begin{definition} Let $(X,d)$ be a pseudo-metric space, $K \geq 1$, $C \geq 0$ and $I \subset \R$ be connected. A $(K,C)$-quasi-isometric embedding $\gamma \colon I \to X$ is called a: 
\begin{enumerate}[(i)]
    \item $(K,C)$-\emph{quasi-geodesic segment} if $I$ is bounded; \index{quasi-geodesic!segment}
    \item $(K,C)$-\emph{quasi-geodesic ray} if $I$ is bounded below; \index{quasi-geodesic!ray}
    \item $(K,C)$-\emph{quasi-geodesic line} if $I=\R$. 
    \index{quasi-geodesic!line}
\end{enumerate}
We say that $\gamma$ is a \emph{quasi-geodesic} segment, ray or line (or simply a \emph{quasi-geodesic}) if it is a $(K,C)$-quasi-geodesic segment, ray or line for some $K$ and $C$ as above.

If $K=1$ we refer to $\gamma$ as a \emph{$C$-rough geodesic segment}\index{rough geodesic segment}, a \emph{$C$-rough geodesic ray}\index{rough geodesic ray} or a \emph{$C$-rough geodesic}\index{rough geodesic line}, instead, and If $K=1$ and $C=0$, i.e., $\gamma$ is an isometric embedding, then $\gamma$ is referred to as a \emph{geodesic segment}\index{geodesic!segment}, a \emph{geodesic ray}\index{geodesic!ray}, or a \emph{geodesic line}\index{geodesic!line}.
\end{definition}

Unless otherwise stated, rays will have domain $[0, \infty)$. Note that by this definition, all geodesics are parametrized by arc length. We will sometimes need to reparametrize geodesics:

\begin{definition}
Let $I \subset \R$ be connected. A map $\gamma \colon I \to X$ is said to be a \emph{linearly reparametrized geodesic}\index{geodesic!linearly reparametrized}, if there exists a constant $\lambda$ such that $d(\gamma(t), \gamma(t'))=\lambda|t-t'|$ for all $t,t' \in I$.
\end{definition}

We now state two corollaries of the Arzel\`a--Ascoli Theorem that will be useful in proofs about the Gromov and Morse boundaries.

\begin{lemma}[Arzel\`a--Ascoli]\label{ArzelaAscoli1}
Let $X$ be a proper metric space and $o \in X$. Let $a_n \leq 0 \leq b_n$ and let $(\beta_n: [a_n, b_n] \to X)_{n\in\mathbb{N}}$ be a sequence of geodesics in $X$. Assume that there exists $C>0$ such that $d(\beta_n(0), o) < C$ for all $n \in \bN$.
\begin{enumerate}[(i)]
\item If $a_n \to -\infty$ and $b_n \to \infty$, then  $(\beta_n)_n$ has a subsequence which converges uniformly on compact sets to a bi-infinite geodesic line $\beta: \R \to X$.
\item If  $a_n=0$ and $b_n \to \infty$, then  $(\beta_n)_n$ has a subsequence which converges uniformly on compact sets to a geodesic ray $\beta: [0, \infty) \to X$.
\end{enumerate}
\end{lemma} 
It was observed by Morse \cite{morse} that geodesics in negatively curved spaces have additional properties; one of these properties is nowadays called the \emph{Morse property}\index{Morse!property} and can be defined as follows:
\begin{definition} \label{defn: Morse geodesic}
Let $X$ be a proper geodesic metric space, and let $N: [1, \infty) \times [0, \infty) \rightarrow [0, \infty)$ be an arbitrary function. A quasi-geodesic $\gamma: I \to X$ is called $N$-\emph{Morse}\index{Morse!quasi-geodesic} if any $(K,C)$-quasi-geodesic segment $q$ with endpoints on $\gamma$ is contained in the $N(K,C)$-neighborhood $\mathcal N_{N(K,C)}(\gamma)$ of $\gamma$. If $\gamma$ is a quasi-geodesic ray we obtain the notion of an \emph{$N$-Morse quasi-geodesic ray}\index{Morse!quasi-geodesic ray}, and if $\gamma$ is a quasi-geodesic line then we call $\gamma$ an \emph{$N$-Morse quasi-geodesic line}\index{Morse!quasi-geodesic line}. If $\gamma$ is an $N$-Morse quasi-geodesic, then we call $N$ a \emph{Morse gauge}\index{Morse!gauge} for $\gamma$.
\end{definition}

Two important classes of metric spaces for us will be quasi-geodesic and geodesic metric spaces. 

\begin{definition} \label{QuasiGeodesic}
A pseudo-metric space $(X,d)$ is called a \emph{$(K,C)$-quasi-geo\-de\-sic space}\index{quasi-geodesic space} if for all $x,y\in X$ there exists a $(K,C)$-quasi-geodesic segment $\gamma:[a,b] \to X$ such that $\gamma(a) = x$ and $\gamma(b) = y$. If $K=1$ it is called a \emph{$C$-roughly geodesic space}\index{roughly geodesic space}, and if moreover $C=0$ then it is called a \emph{geodesic space}\index{geodesic space}. We say that $(X,d)$ is quasi-geodesic if it is $(K, C)$-quasi-geodesic, for some $K$ and $C$ as above.
\end{definition}
\begin{remark}[QI-invariance]\label{QuasiGeodesicQI} Since the image of a geodesic under a quasi-isometry is a quasi-geodesic, the image of a geodesic metric space under a quasi-isometric embedding is a quasi-geodesic metric space. Moreover, if $X$ is a relatively dense subset of a metric space $Y$, then $X$ is quasi-geodesic if and only if $Y$ is quasi-geodesic. Combining these two facts we see that being quasi-geodesic is QI-invariant among pseudo-metric spaces.
\end{remark}
In general, a subspace of a geodesic metric space need not be geodesic, or even quasi-geodesic. 
\begin{definition} Let $(X,d)$ be a geodesic metric space. A subspace $Y$ is called a \emph{quasi-convex subspace}\index{quasi-convex subspace} of $X$ if there exists $C>0$ such that every geodesic in $X$ with endpoints in $Y$ is contained in $N_C(Y)$.
\end{definition}
\begin{remark}\label{QuasiconvexQuasigeodesic}
We note that if $Y \subset X$ is a quasi-convex subset of a geodesic metric space, then $Y$ is quasi-geodesic with respect to the induced metric.

Indeed, assume that $Y$ is $R$-quasi-convex for some constant $R$; we may then join any two points in $Y$ by a geodesic in $X$ in the $R$-neighborhood of $Y$. It follows that one can construct a $(1, 2R+1)$-quasi-geodesic in $Y$ with bounded Hausdorff distance $2R$ from the geodesic. Thus $Y$ is quasi-geodesic where the quasi-constants depend only on $R$. 
\end{remark}

\begin{example}\label{CanonicalMetric} If $\Gamma$ is a connected graph with edge set $\mathcal E$ and $w: \mathcal E \to [0, \infty)$ is an arbitrary function, then there exists a unique path metric $d_w$ on
(the geometric realization of) $\Gamma$ such that (the geometric realization of) any given edge $e \in \mathcal E$ is isometric to an interval of length $w(e)$. If $w$ is the constant function $1$, then we refer to $d$ as the \emph{canonical metric}\index{canonical metric (on a graph)} on $\Gamma$. Each of the metric spaces $(|\Gamma|, d_w)$ is geodesic; in particular, every graph is a geodesic metric space with respect to its canonical metric.

The canonical metric of a graph is proper if and only if the graph is locally finite; by contrast, every countable graph admits a proper metric of the form $d_w$ for some (weight) function $w$.
\end{example}
We will see in the next section that, up to quasi-isometry, locally finite graphs (with their canonical metrics) are the only examples of proper geodesic metric spaces.

\section{Coarsely connected and large-scale geodesic spaces}
\begin{definition} Let $(X,d)$ be a pseudo-metric space.
\begin{enumerate}[(i)]
\item Let $n \in \mathbb N$ and $c > 0$. A sequence $(x_0, \dots, x_n)$ of points in $X$ is called a \emph{$c$-path of length $n$}\index{path!$c$-path} from $x_0$ to $x_n$ if $d(x_i, x_{i+1}) < c$, for all $i = 0, \dots, n-1$.
\item $(X,d)$ is \emph{coarsely connected}\index{coarsely connected} if there exists $c>0$ such that for all $x, x' \in X$ there exists a $c$-path from $x$ to $x'$.
\end{enumerate}
\end{definition}

The terminology is justified by the following observation, see \cite[Prop. 3.B.7]{CdlH}:
\begin{lemma} A pseudo-metric space is coarsely connected if and only if it is coarsely equivalent to a connected metric space. In particular, coarse connectedness is a coarse invariant, hence a QI-invariant.\qed
\end{lemma}

We compare the notion of being quasi-geodesic (Definition \ref{QuasiGeodesic}) to the following terminology from \cite{CdlH}:
\begin{definition}\label{DefLargeScaleGeodesic} Let $(X,d)$ be a pseudo-metric space. Moreover, let $\Phi_+$ be an upper control and let $a,c >0$, $b \geq 0$.
\begin{enumerate}[(i)]
\item $(X,d)$ is \emph{$(\Phi_+, c)$-coarsely geodesic}\index{coarsely geodesic space} if for all $x, x' \in X$ there exists a $c$-path from $x$ to $x'$ of length $n \leq 
\Phi_+(d(x,x'))$.
\item $(X,d)$ is \emph{$(a,b,c)$-large-scale geodesic}\index{large-scale geodesic space} if it is $(\Phi_+, c)$-coarsely geodesic with respect to the affine control $\Phi_+(x) = ax+b$.
\end{enumerate}
We say that $(X,d)$ is coarsely geodesic (respectively  large-scale geodesic) if it is $(\Phi_+, c)$-coarsely geodesic ($(a,b,c)$-large-scale geodesic), for some $\Phi_+$  and $c$ as above ($a, b,c$ as above).
\end{definition}
For the following lemma see again \cite[Prop. 3.B.7 and Lemma 3.B.6(6)]{CdlH}:
\begin{lemma}[Coarse connectedness properties]\label{CharLSG}\ 
\begin{enumerate}[(i)]
\item Every large-scale geodesic space is coarsely geodesic, and every coarsely geodesic space is coarsely connected.
\item A pseudo-metric space is coarsely geodesic if and only if it is coarsely equivalent to a geodesic metric space.
\item A pseudo-metric space is large-scale geodesic if and only if it is quasi-isometric to a geodesic metric space.
\item A proper pseudo-metric space is large-scale geodesic if and only if it is quasi-isometric to a locally finite graph (with its canonical metric).
\end{enumerate}
In particular, coarse geodesicity is a coarse invariant and large-scale geodesicity is a QI-invariant. 
\end{lemma}

\begin{remark} The proof of (iv) is actually constructive: Indeed, let $(X, d)$ be an $(a,b,c)$-large-scale geodesic pseudo-metric space. We construct a metric graph $\Gamma_0$ as follows. The vertex set of $\Gamma_0$ is simply $X$, and two vertices $v_1, v_2 \in X$ are connected by an edge if and only if they are distinct and $d(v_1, v_2) \leq c$. Finally let $w$ be the constant function on edges of $\Gamma_0$ taking value $c$. Then by \cite[Lemma 3.B.6(6)]{CdlH} the metric graph $(|\Gamma_0|, d_w)$ is quasi-isometric to $(X,d)$, albeit in general not locally finite. 

Now assume that $(X,d)$ is moreover proper. By Remark \ref{ProperVsLocallyFinite} we can replace $(X,d)$ by a quasi-isometric space $(X', d')$ which is locally finite. Applying the above construction then yields a \emph{locally finite} graph $(|\Gamma_0|, d_w)$ which is quasi-isometric to $(X', d')$ (and hence to $(X,d)$). In general, the metric $d_w$ will not be the canonical metric on $|\Gamma_0|$, but will differ from the latter by a factor $c$. However, this can be corrected by subdividing each edge of $\Gamma_0$ into $\lfloor c \rfloor$ edges; then $(X,d)$ is quasi-isometric to the resulting graph $\Gamma$ with its canonical metric. Note that $\Gamma$ is locally finite, since $\Gamma_0$ is.
\end{remark}

\begin{lemma}\label{lsgquasigeodesic}
A metric space $(X,d)$ is large-scale geodesic if an only if it is quasi-geodesic.
\end{lemma}

\begin{proof}
The forward implication follows from Remark \ref{QuasiGeodesicQI} and Lemma \ref{CharLSG}.

For the reverse implication, assume that $X$ is $(K,C)$-quasi-geodesic. Given $x,x' \in X$ we pick a $(K,C)$-quasi-geodesic segment $\gamma \colon [a,b] \to X$ joining them. The points $\{\gamma(a), \gamma(a+1),\ldots,  \gamma(a+\lfloor b-a \rfloor-1), \gamma(b)\}$ form a $(K+C)$-path of length $n \leq \lfloor b-a \rfloor+1\leq b-a +1$. To show that $X$ is large-scale geodesic, it suffices to find constants $a_0>0$ and $b_0\geq 0$, which are independent of the choice of $x,x'$, such that  $n \leq a_0 d(x,x')+ b_0$. Since $\gamma$ is a $(K,C)$-quasi-geodesic, we know that $b-a \leq K d(x,x') + KC$. Thus we get $n \leq Kd(x,x')+KC+1$.
\end{proof}

The usefulness of large-scale geodesic metric spaces in coarse geometry is based on the fact that in large-scale geodesic spaces, coarse equivalences are quasi-iso\-me\-tries. This was first observed by Gromov \cite[0.2.D]{Gromov}, who refers to this fact as a ``trivial lemma''.

\begin{lemma} \label{GromovTrivial}
Let $(X, d_X)$ be a large-scale geodesic metric space, $(Y, d_Y)$ an arbitrary metric space and $f: X \to Y$ a map. Assume that there exists an upper control $\Phi_+: [0, \infty) \to [0, \infty)$ of $f$.
Then there exist constants $a\geq 1$, $b \geq 0$ such that for all $x_1, x_2 \in X$ we have
\begin{equation*}
d_Y(f(x_1), f(x_2))  \leq a \cdot d_X(x_1,x_2) + b. 
\end{equation*}
\end{lemma}
\begin{proof} Assume that $(X, d_X)$ is $(a_0, b_0, c_0)$-large-scale geodesic, and, given $x, x' \in X$, let $x = x_0, \dots, x_n = x'$ in $X$ be such that $d_X(x_{i-1}, x_i) \leq c_0$, for $i=1,\ldots , n $, and let $n \leq a_0 d_X(x,x') +b_0$. Then, since  
$d_Y(f(x_{i-1}), f(x_i))  \leq \Phi_+(d_X(x_{i-1},x_i))$,
\begin{eqnarray*}
d_Y(f(x), f(x')) &\leq& \sum_{i=1}^n d_Y(f(x_{i-1}), f(x_i)) \leq n \cdot \Phi_+(c_0) \\
&\leq& \Phi_+(c_0)\cdot a_0 d_X(x, x') + \Phi_+(c_0)\cdot b_0.
\end{eqnarray*}
We may thus choose $a \coloneqq  \Phi_+(c_0)a_0$ and $b \coloneqq  \Phi_+(c_0)b_0$. This finishes the proof.
\end{proof}
\begin{corollary}\label{GromovTrivialSymmetric} Every coarse equivalence between large-scale geodesic spaces is a quasi-isometry, hence two large-scale geodesic metric spaces are coarsely equivalent if and only if they are quasi-isometric. 
\end{corollary}

While this symmetric version of Gromov's trivial lemma is often sufficient in applications, we will also need the asymmetric version of Lemma \ref{GromovTrivial} in the proof of our generalized Milnor--Schwarz lemma. Moreover, it will be important to us to have also a uniform version of Gromov's lemma at our disposal. Given metric spaces $(X, d_X)$, $(Y, d_Y)$ and a collection $A$ of coarse equivalences from $X$ to $Y$, we say that the collection $A$ is \emph{uniform}\index{uniform collection!of coarse equivalences} if there exist an upper control function $\Phi_+$ and a lower control function $\Phi_-$, which are an upper, respectively lower control for every $f \in A$. Recall that, in Remark \ref{rem:qi facts}, we similarly introduced the notion of uniform collection of quasi-isometries, as a collection with uniformly bounded QI-constants.
 With this terminology understood we have:

\begin{proposition}\label{GromovUniform} Let $(X, d_X)$, $(Y, d_Y)$ be large-scale geodesic metric spaces. Then every uniform collection of coarse equivalences from $X$ to $Y$ is a uniform collection of quasi-isometries from $X$ to $Y$.
\end{proposition}
\begin{proof} For the upper bound we argue as in the proof of Lemma \ref{GromovTrivial} and observe that the QI constants $a \coloneqq  \Phi_+(c_0)a_0$ and $b \coloneqq  \Phi_+(c_0)b_0$ depend only on the upper control function and the constants $a_0, b_0, c_0$, which depend only on $X$. For the lower bounds we apply the same argument to a quasi-inverse of the coarse equivalence in question.
\end{proof}

\section{Spaces of coarse bounded geometry}\label{Growth}
The following terminology was introduced by Block and Weinberger \cite{BlockWeinberger1992}:
\begin{definition}\label{Weinberger} Let $(X,d)$ be a pseudo-metric space, let $R \geq 0$ and let $\Phi_+$ be an upper control function. A subset $\Lambda \subset X$ is \emph{$\Phi_+$-uniformly locally finite}\index{uniformly locally finite subset} if $|\Lambda \cap \overline{B}(x,r)| \leq \Phi_+(r)$ holds for all $x \in X$ and $r >0$. This $\Lambda$ is called a \emph{$(R,\Phi_+)$-quasi-lattice}\index{quasi-lattice} if it is moreover $R$-relatively dense. The space $(X,d)$ is called of \emph{coarse bounded geometry}\index{coarse bounded geometry!space of} if it contains a quasi-lattice.
\end{definition}
It turns out that having coarse bounded geometry is a coarse invariant:
\begin{lemma}\label{Quasilattice} Let $f: X \to Y$ be a coarse equivalence with control functions $(\Psi_-, \Psi_+)$ and $C$-relatively dense image and let $\Lambda \subset X$ be a $(R, \Phi_+)$-quasi-lattice. Then $f(\Lambda)$ is a $(\Psi_+(R)+C, \Phi_{++})$-quasi-lattice with upper control function $\Phi_{++}$ given by
\[
\Phi_{++}(r) \coloneqq  \Phi_+(\widehat{\Psi}_-(r+C)), 
\quad \text{where} \quad \widehat{\Psi}_-(s) \coloneqq  \sup\{t \in [0, \infty) \mid \Psi_-(t) \leq s\}.
\]
\end{lemma}
\begin{proof} Since $\Lambda$ is $R$-relatively dense in $X$, the set $f(\Lambda)$ is $\Psi_+(R)$-relatively dense in $f(X)$ and thus $\Psi_+(R)+C$ relatively dense in $Y$. 

Now let $y \in f(\Lambda) \cap \overline B(y_0,r)$, for some $y_0 \in Y$ and $r >0$. There then exists $x \in X$ such that $d(f(x), y_0) < C$ and thus $y \in \overline{B}(f(x), r+C)$. Choose $\lambda \in \Lambda$ such that $y = f(\lambda)$. Then
\[
\Psi_-(d(\lambda,x)) \leq d(f(\lambda), f(x)) \leq r+C \implies d(\lambda,x) \leq  \widehat{\Psi}_-(r+C).
\]
We deduce that
\[
f(\Lambda) \cap \overline B(y_0,r) = f(\Lambda \cap \overline B(x, \widehat{\Psi}_-(r+C))) \implies |f(\Lambda) \cap \overline B(y_0,r)| \leq  \Phi_+(\widehat{\Psi}_-(r+C)),
\]
and the lemma follows.
\end{proof}

\begin{remark}[Growth equivalence of functions]\label{GrowthEq} Let $f,g: [0, \infty) \to [0, \infty)$ be two functions. We say that $g$ \emph{dominates}\index{dominating function} $f$ and write $f \preceq g$ if there are $a,b \geq 1$ and $c,d \geq 0$ such that for all $r \geq c$ we have $f(r) \leq a g(br+d)$. In fact, it is always possible to choose $d=0$ (at the cost of increasing $b$ and $c$). We write $f \approx g$ if $f \preceq g$ and $g \preceq f$. This defines an equivalence relation, and the corresponding equivalence class $[f]$ of $f$ is called the \emph{growth equivalence class}\index{growth!equivalence class} of $f$.
\end{remark}

\begin{lemma}[Growth of quasi-lattices]\label{GrowthQuasiLattices} Let $(X,d)$ be a metric space of coarse bounded geometry and let $\Lambda_1, \Lambda_2 \subset X$ be two quasi-lattices. Then for any $x_1, x_2 \in X$, the functions
\[
\gamma_1(r) \coloneqq  |B(x_1,r) \cap \Lambda_1| \qand \gamma_2(r) \coloneqq  |B(x_2,r) \cap \Lambda_2|
\]
are growth-equivalent.
\end{lemma}
\begin{proof} Choose $R$ and $\Phi_+$ such that $\Lambda_1$ and $\Lambda_2$ are $(R, \Phi_+)$-quasi-lattices. Fix a choice of $x_1, x_2 \in X$, then choose $\delta > d(x_1, x_2)$ and let $\lambda_1 \in B(x_1,r) \cap \Lambda_1$, for some $r>0$. Choose $\lambda_2 \in \Lambda_2$ such that $d(\lambda_1, \lambda_2) \leq R$. Then
\[
d(\lambda_2, x_2) \leq d(\lambda_2, \lambda_1) + d(\lambda_1, x_1) + d(x_1, x_2) < R+ r+\delta.
\]
Since $\lambda_1 \in \overline{B}(\lambda_2, R)$, we obtain
\[
B(x_1,r) \cap \Lambda_1 \quad \subset \bigcup_{\lambda_2 \in B(x_2, r+R+\delta) \cap \Lambda_2} \overline{B}(\lambda_2, R),
\]
and hence $\gamma_1(r) \leq \gamma_2( r+R+\delta)\cdot \Phi_+(R)$. This shows that $\gamma_1 \preceq \gamma_2$, and similarly $\gamma_2 \preceq \gamma_1$.
\end{proof}

\section{Coarse ends of metric spaces}
The purpose of this section is to discuss a coarse version (due to \'Alvarez L\'opez and Candel) of the notion of ends of a topological space. We start by recalling the classical notion; our notation is set up in such a way that the correspondence to the coarse version becomes particularly apparent.

\begin{construction}[Ends]\label{Ends} Let $X$ be a topological space. We denote by $\cK(X)$ the poset of compact subsets of $X$, ordered by inclusion. Given $K \in \cK(X)$ we denote by $\cU_{\cK}$ the space of unbounded connected components of $X \setminus K$, equipped with the discrete topology.

If $K \subset K'$ are compact subsets of $X$, then we define $\eta_{K, K'}: \cU_{K'} \to \cU_{K}$ by mapping every $U' \in \cU_{K'}$ to the unique $U \in \cU_{K}$ with $U' \subset U$. With respect to the maps $\eta_{K, K'}$ the family $(\cU_{K})_{K \in \cK(X)}$ is an inverse system and the \emph{space of ends}\index{space of ends}\index{ends!space of} of $X$ can be defined as the inverse limit space
\[
\cE(X) \coloneqq  \lim_\leftarrow \cU_{K}.
\]
\end{construction}
We now introduce a coarse version of the space of ends for metric spaces, following \'Alvarez L\'opez and Candel \cite[Chapter 6]{AlvarezCandel}. We need the following notion:
\begin{definition}[$c$-connected components]
If $(X,d)$ is a metric space and $c>0$, then we can define an equivalence relation $\sim_c$ on $X$ by setting $x \sim_c y$ if there is a $c$-path from $x$ to $y$. The equivalence classes of this equivalence relation are precisely the maximal $c$-connected subsets of $X$, and they are called the \emph{$c$-connected components}\index{connected components!$c$-connected components} of $X$. 
\end{definition}

Note that a connected metric space is $c$-connected for all $c>0$ (see, for example, \cite{Joshi}), so for all $c>0$ every connected component of a metric space is contained in a unique $c$-connected component.

\begin{construction}[Coarse ends]\label{CoarseEnds} Let $(X,d)$ be a metric space. We denote by $\cB(X)$ the poset of bounded subsets of $X$, ordered by inclusion. Given $c>0$ and $B \in \cB(X)$, we denote by $\cU_{c,B}$ the set of unbounded $c$-connected components of $X \setminus B$, considered as a discrete topological space. 

Now let $c'>c>0$ and $B \subset B'$ be bounded subsets of $X$. We define $\eta_{c,B, B'}: \cU_{c,B'} \to \cU_{c,B}$ by mapping every $U' \in \cU_{c,B'}$ to the unique $U \in \cU_{c, B}$ with $U' \subset U$. Similarly, we define $\theta_{c, c', B}: \cU_{c,B} \to \cU_{c', B}$ by mapping every $U \in \cU_{c,B}$ to the unique $U' \in \cU_{c', B}$ such that $U \subset U'$. Then we obtain a commuting square
\begin{equation}\label{ENDetatheta}
\begin{xy}\xymatrix{
\cU_{c,B}  \ar[d]_{\theta_{c,c', B}} &&   \ar[ll]_{\eta_{c, B, B'}}  \cU_{c, B'} \ar[d]^{\theta_{c,c', B'}}\\
\cU_{c', B}  &&  \ar[ll]_{\eta_{c', B, B'}} \cU_{c',B'}.
}\end{xy}\end{equation}
With respect to the maps $\eta_{c,B,B'}$, the family $(\cU_{c,B})_{B \in \cB(X)}$ is an inverse system and we define a topological space
\[
\cE_c(X) \coloneqq  \lim_\leftarrow \cU_{c,B}.
\]
The elements of $\cE_c(X)$ are called the \emph{$c$-ends}\index{ends!$c$-ends} of $X$.
By \eqref{ENDetatheta} the maps $\theta_{c,c', B'}$ induce continuous maps $\theta_{c,c'}: \cE_c(X) \to \cE_{c'}(X)$, and with respect to these maps the family $(\cE_c(X))_{c>0}$ is a direct system. We may thus define a topological space
\[
\cE_\infty(X) \coloneqq  \lim_{\rightarrow} \cE_c(X);
\]
the elements of $\cE_\infty(X)$ are called the \emph{coarse ends}\index{ends!coarse ends}\index{coarse ends} of $X$.
\end{construction}
The following is established in \cite[Prop.\ 6.35]{AlvarezCandel}:
\begin{theorem}[Coarse invariance of coarse ends]\label{CoarseEndsInvariant}
Let $X$ and $Y$ be metric spaces and let $f: X \to Y$ be a coarsely proper coarsely Lipschitz map. Then $f$ induces a continuous map $f_\infty: \cE_\infty(X) \to \cE_\infty(Y)$, which depends only on the closeness class of $f$. In particular, the homeomorphims type of $\cE_\infty(X)$ is a coarse invariant of metric spaces.
\end{theorem}
While the space of coarse ends can be defined for an arbitrary metric space $X$ and is a coarse invariant in this context, some assumptions on $X$ are needed to make this space well-behaved. The following is \cite[Prop.\ 6.29]{AlvarezCandel}:
\begin{proposition}[Compactness condition] If $X$ is a coarsely connected metric space of coarse bounded geometry, then $\cE_\infty(X)$ is compact.
\end{proposition}
Our next goal is to compare the space $\cE_\infty(X)$ of coarse ends to the space $\cE(X)$ for proper metric spaces $X$. We first observe that there is always a natural continuous map $\cE(X) \to \cE_\infty(X)$:
\begin{construction}[Comparison maps] Assume that $(X,d)$ is a proper metric space and let $o \in X$ be a basepoint. We abbreviate $\overline{B}_r \coloneqq  \overline{B}(o,r)$; then the sequence $\overline{B}_n$ is cofinal both in $\cK(X)$ and $\cB(X)$, i.e.\ every element of $\cK(X)$ or $\cB(X)$ is contained in some $\overline{B}_n$.  
Moreover, since for every $c>0$ every connected component is contained in a unique $c$-connected component, we have maps $c_{n,c}: \cU_{\overline{B}_n} \to \cU_{c,\overline{B}_n}$. Passing to the inverse limit over $n$ we obtain a comparison map $c_{\infty, c}: \cE(X) \to \cE_c(X)$, and then taking a direct limit over all $c>0$ we obtain a comparison map $c_{\infty, \infty}: \cE(X) \to \cE_\infty(X)$.
\end{construction}

By \cite[Sec.\ 6.3]{AlvarezCandel} the comparison map is a homeomorphism for connected graphs and complete Riemannian manifolds. We extend this list as follows:
\begin{proposition}[Ends vs.\ coarse ends]\label{EndsCoarse} If $(X,d)$ is any proper geodesic metric space, then the comparison map yields a homeomorphism $\cE(X) \cong \cE_\infty(X)$.
\end{proposition}

The proof is based on the following observation, which is similar to \cite[Prop. 6.31]{AlvarezCandel}:
\begin{lemma}\label{ComparisonIso} Let $(X,d)$ be a proper metric space. Assume that there exists $c_0 \geq 0$ such that for all $c>c_0$, there exists $n_0(c) \in \bN$ such that for all $n \geq n_0$ there exists some $m = m(c,n) \in \bN$ such that every coarse $c$-connected component of $X \setminus \overline{B}_n$ is contained in a unique connected component of $X \setminus \overline{B}_m$ and such that $\lim_{n \to \infty} m(c, n) = \infty$ for every $n \in \bN$. Then $c_{\infty, \infty}: \cE(X) \to \cE_\infty(X)$ is a homeomorphism.
\end{lemma}
\begin{proof} The condition ensures (for sufficiently large $n$ and $c$) the existence of maps $\mu_{n,c}: \cU_{c,\overline{B}_n} \to \cU_{\overline{B}_{m(c,n)}}$, which induce maps $\mu_{\infty, c}: \cE_c(X) \to \cE(X)$ and hence a map $\mu_{\infty, \infty}: \cE_\infty(X) \to\cE(X)$, which can be shown to be inverse to the comparison map, as in the proof of \cite[Prop. 6.31]{AlvarezCandel}.
\end{proof}

\begin{proof}[Proof of Proposition \ref{EndsCoarse}] We show that Lemma \ref{ComparisonIso} applies with $c_0 \coloneqq  0$, $n_0(c) \coloneqq  5c$ and $m(c,n) \coloneqq  n - 2c$. Indeed, let $c$ and $n$ be chosen accordingly and let $U$ be a  coarse $c$-connected component of $X \setminus \overline{B}_n$. Let $x,y \in U$ be two points of distance $\leq c$ and let $\gamma$ be a geodesic segment in $X$ which connects $x$ and $y$. Since the endpoints of $\gamma$ lie outside $\overline{B}_n$ and $\gamma$ has length $c$ we have $\gamma \subset X \setminus \overline{B}_m$, where $m \coloneqq  n-2c$. We deduce that $x$ and $y$ lie in the same connected component of $X \setminus \overline{B}_m$. Consequently, $U$ is contained in a single connected component of 
$X \setminus \overline{B}_m$, and the proposition follows.
\end{proof}

It is well-known that if $X$ is a Cayley graph of a finitely-generated group, then the space of ends of $X$ with the topology defined in Construction \ref{Ends} (which coincides with the space of coarse ends by Proposition \ref{EndsCoarse}) is homeomorphic to a Cantor space as soon as it contains at least three points. It is an interesting question in which generality this result holds. The following general criterion was established by 
\'Alvarez L\'opez and Candel \cite[Thm.\ 6.39]{AlvarezCandel}; see  Definition \ref{def: quasi cobounded} for the notion of a quasi-cobounded space.
\begin{theorem}[\'Alvarez L\'opez--Candel criterion]\label{ACCriterion} If $X$ is a quasi-cobounded large-scale geodesic metric space of coarse bounded geometry, then either $|\cE_\infty(X)| \leq 2$ or $\cE_\infty(X)$ is homeomorphic to a Cantor set.
\end{theorem}
Combining Proposition \ref{EndsCoarse} and Theorem \ref{ACCriterion} we obtain:
\begin{corollary}\label{ACConvenient} If $X$ is a quasi-cobounded proper geodesic metric space of coarse bounded geometry, then either $|\cE(X)| \leq 2$ or $\cE(X)$ is homeomorphic to a Cantor set.
\end{corollary}

\chapter{Some notions from geometric group theory}\label{AppendixGGT}

\section{Left-invariant (pseudo-)metrics on groups}\label{SecGroupMetrics}

A pseudo-metric $d$ on a group $G$ is called \emph{left-invariant}\index{left-invariant!pseudo-metric}\index{pseudo-metric!left-invariant} if $d(gh, gk) = d(h, k)$ for all $g,h, k \in G$. Left-invariant pseudo-metrics on groups correspond to pseudo-norms in the sense of the following definition (see e.g. \cite{DranSmith}).
\begin{definition} A \emph{pseudo-norm}\index{pseudo-norm} on a group $G$ is a map $\|\cdot \|: G \to [0, \infty)$ with the following properties
\begin{enumerate}[(N1)]
\item $\|e\| = 0$,
\item $\|g\|  = \|g^{-1}\|$ for all $g\in G$,
\item $\|gh\| \leq \|g\| + \|h\|$ for all $g,h \in G$.
\end{enumerate}
It is called a \emph{norm}\index{norm} if instead of (N1) we have
\begin{enumerate}[(N1')]
\item $\|g\| = 0$ if and only if $g = e$.
\end{enumerate}
\end{definition}
Given a pseudo-norm $\|\cdot\|$ we define $d_{\|\cdot\|}: G \times G \to [0,\infty)$ by $d_{\|\cdot\|}(g,h) \coloneqq  \|g^{-1}h\|$, and given a pseudo-metric $d: G \times G \to [0,\infty)$ we define $\|\cdot\|_d: G \to [0, \infty)$ by $\|g\|_d \coloneqq  d(e,g)$. With these notations understood, the following lemma follows straight from the definitions:
\begin{lemma} The maps $d \mapsto \|\cdot\|_d$ and $\|\cdot\| \mapsto d_{\|\cdot\|}$ define mutually inverse bijections between the set of left-invariant pseudo-metrics on $G$ and the set of pseudo-norms on $G$. Under these bijections, left-invariant metrics correspond to norms.\qed
\end{lemma}
If $G$ carries a topology, we are interested in left-invariant metrics on $G$ which to some extent are compatible with the given topology. 

\begin{definition}\label{DefLeftAdmissible} Let $G$ be a topological group. A pseudo-metric $d$ on $G$ is called
\begin{enumerate}[(i)]
\item \emph{locally bounded}\index{pseudo-metric!locally bounded}\index{locally bounded!pseudo-metric} if every point in $G$ has a $d$-bounded open neighborhood;
\item \emph{left-adapted}\index{left-adapted!pseudo-metric}\index{pseudo-metric!left-adapted} if it is left-invariant, proper (i.e.\ closed balls are compact) and locally bounded;
\item \emph{left-admissible}\index{left-admissible!pseudo-metric}\index{pseudo-metric!left-admissible} if it is left-adapted and induces the given topology on $G$.
\end{enumerate}
\end{definition}
We say that a metric is \emph{continuous} if the map $d \colon G \times G \to [0,\infty)$ is continuous. Note that every continuous pseudo-metric on a locally compact group is automatically locally bounded. We emphasize that, following \cite{CdlH}, we do not require left-adapted pseudo-metrics to be continuous. The following is contained in \cite[Prop. 2.A.9, Thm. 2.B.4, Prop 4.A.2]{CdlH}:
\begin{lemma}\label{AutomaticAdmissible} For a locally compact group $G$ the following are equivalent:
\begin{enumerate}[(i)]
    \item There exists a left-adapted continuous pseudo-metric on $G$.
    \item There exists a left-adapted pseudo-metric on $G$.
    \item There exists a left-adapted metric on $G$.
    \item There exists a left-admissible metric on $G$.
    \item $G$ is second-countable.
\end{enumerate}
\end{lemma}
Throughout this book we refer to a locally compact second-countable group $G$ as an \emph{lcsc}\index{group!lcsc}\index{lcsc group} group. Such groups arise naturally as isometry groups of pseudo-metric spaces:
\begin{example} Let $(X, d)$ be a proper pseudo-metric space; then on the isometry ${\rm Is}(X,d)$ the topologies of pointwise convergence, of uniform convergence on compacta and the compact-open topology all coincide and define a lcsc topology on ${\rm Is}(X,d)$ \cite[5.B.5]{CdlH}. In this book we will always consider ${\rm Is}(X,d)$ as a lcsc group with respect to this topology. 

By a theorem of Malicki and Solecki \cite[Thm 5.B.14]{CdlH}, every lcsc group can be realized as ${\rm Is}(X,d)$ for a proper metric space $(X,d)$, hence we can think of lcsc groups simply as isometry groups of proper metric spaces.

It can be shown that the action of $\mathrm{Is}(X,d)$ on $X$ is jointly continuous (i.e. the map $\mathrm{Is}(X,d) \times X \to X$ is continuous with respect to the product topology on $\mathrm{Is}(X,d)\times X$) and proper \cite[Lemma 5.B.4]{CdlH}.
\end{example}

A particularly important class of examples of lcsc groups for our purposes are countable discrete groups. In this context we have, by Lemma \ref{AutomaticAdmissible} and \cite[Section 1]{Smith}:
\begin{corollary}\label{AutomaticAdmissibleDiscrete} A discrete group $\Gamma$ admits a left-admissible metric if and only if $\Gamma$ is countable. Moreover, for a metric $d$ on a countable discrete group $\Gamma$ the following are equivalent:
\begin{enumerate}[(i)]
\item $d$ is left-invariant and proper;
\item $d$ is left-adapted;
\item $d$ is left-admissible.\qed
\end{enumerate}
\end{corollary}
While left-admissible metrics on lcsc groups are far from unique, different left-admissible metrics share a number of important properties. For example, we will see in the next section that they all lie in the same coarse equivalence class. To given another example, recall from Definition \ref{DefDelone1} the definition of a Delone set in a metric space. Let $G$ be a lcsc group and $\Lambda \subset G$. Recall from Definition \ref{def:syndetic, commensurable} that $\Lambda$ is \emph{left-syndetic} in $(G,d)$ if there exists a compact $F \subset G$ so that $G \subset \Lambda F$. The following proposition guarantees that Delone sets in lcsc groups with respect to a left-admissible metric admit a purely group theoretic characterization:
\begin{proposition}[Characterization of Delone sets in lcsc groups, \cite{BH}]\label{TopCharDelone} Let $G$ be a lcsc group, $\Lambda \subset G$ a subset and $d$ a left-admissible metric.
\begin{enumerate}[(i)]
\item $\Lambda$ is uniformly discrete in $(G,d)$ if and only if the identity $e \in G$ is not an accumulation point of $\Lambda^{-1}\Lambda$.
\item $\Lambda$ is relatively dense in $(G,d)$ if and only if it is left-syndetic.
\end{enumerate}
In particular, the property of being a Delone set in $(G, d)$ is independent of the choice of $d$.\qed
\end{proposition}
\begin{definition}\label{DefDelone}
A subset $\Lambda$ of a lcsc group $G$ is called a \emph{Delone set in $G$}\index{Delone!set in a lcsc group} if it is a Delone set with respect to some (hence any) left-admissible metric on $G$.
\end{definition}

\section{The coarse class of a countable group}\label{SecCoarseClass}
With every countable discrete group and, more generally with every lcsc group we can associate a canonical coarse equivalence class of metric spaces. This construction is based on the following coarse uniqueness property of left-adapted metrics (see \cite[Cor. 4.A.6]{CdlH}):
\begin{proposition}\label{CoarseGroupClass} Let $d$ be a left-adapted pseudo-metric on a lcsc group $G$ and let $d'$ be a left-adapted pseudo-metric on a closed subgroup $H< G$. Then the inclusion map $(H, d') \hookrightarrow (G, d)$ is a coarse embedding. In particular, if $d$ and $d'$ are left-adapted pseudo metrics on $G$, then $(G, d)$ and $(G,d')$ are coarsely equivalent.
\end{proposition}
Let us spell out the special case of countable discrete groups, since it lies at the heart of our approach:
\begin{corollary}\label{InclusionCoarseEmbedding}\label{CoarseGroupClassCountable} 
Let $\Gamma_1$ be a countable group and $\Gamma_2 < \Gamma_1$ be a subgroup. If $d_1$ and $d_2$ are left-adapted metrics on $\Gamma_1$ and $\Gamma_2$ respectively, then the inclusion $(\Gamma_2, d_2) \hookrightarrow (\Gamma_1, d_1)$ is a coarse embedding. In particular, if $d, d'$ are two left-adapted pseudo-metrics on a countable group $\Gamma$, then the identity $(\Gamma, d) \to (\Gamma, d')$ is a coarse equivalence.\qed
\end{corollary}
\begin{definition}\label{DefCoarseClass}\index{coarse class!of a lcsc group} If $G$ is a lcsc group, then the \emph{coarse class} $[G]_c$ of $G$ is the coarse equivalence class of the metric space $(G,d)$, where $d$ is some (hence any) left-adapted pseudo-metric on $G$.
\end{definition}
The definition applies in particular to the case of a countable discrete group $\Gamma$. In order to give an explicit representative of $[\Gamma]_c$ we need to construct an explicit left-adapted metric on $\Gamma$. This can be done by means of the following construction:
\begin{definition} Let $\Gamma$ be a countable group and $S\subset \Gamma$ be a symmetric subset. A function $w: S\cup\{e\}\to [0, \infty)$ is called a \emph{weight function on $S$}\index{weight function} if it is proper and satisfies $w^{-1}(0) = \{e\}$ and $w(s) = w(s^{-1})$ for all $s \in S$.
\end{definition}
\begin{lemma}\label{ExistLeftAdmissible} Let $S$ be a symmetric generating set of a countable group $\Gamma$ and let $w: S \cup\{e\}\to [0, \infty)$ be a weight function. Then 
\[
\|g\|_{S, w} : = \inf\left\{\sum_{i=1}^n w(s_i)\mid g = s_1 \cdots s_n, \; s_i \in S\right\}
\]
defines a norm on $\Gamma$, and the associated metric $d_{S,w}$ is left-adapted. In particular, $(\Gamma, d_{S,w})$ represents $[\Gamma]_c$.
\end{lemma}
\begin{proof} (N1) follows from $w^{-1}(0) = \{e\}$, (N2) follows from symmetry of $S$ and $w$ and (N3) holds by construction. Properness of $w$ implies that the associated metric is proper.
\end{proof}
If $S$ is finite, then we may choose $w \coloneqq  1$, that is, $w(s)=1, \forall s\in S$. In this case, $d_S \coloneqq  d_{S,1}$ is called a \emph{word metric}\index{word metric} on $\Gamma$.

\section{The canonical QI class of a finitely-generated group}
With any generating set $S$ of a lcsc group $G$ we can associate a word metric $d_S = d_{S,1}$ by the same formula as in Lemma \ref{ExistLeftAdmissible}. This metric is proper if and only if $S$ is compact. In this case it is even left-adapted and large-scale geodesic \cite[Prop.\ 4.B.4]{CdlH}.
\begin{lemma}\label{CompactWordMetric} If $G$ is a compactly-generated lcsc group and $S$ is a compact symmetric generating set of $G$, then the associated word metric $d_S$ is left-adapted and large-scale geodesic.
\end{lemma}
Recall that a lcsc group $G$ is called \emph{compactly generated}\index{compactly generated group}\index{group!compactly generated} if it admits a compact generating set $S$. By Lemma \ref{CompactWordMetric}, every compactly-generated lcsc group admits a left-adapted large-scale geodesic metric. While the word metric $d_S$ from Lemma \ref{CompactWordMetric} is in general not continuous, it is always possible to find a \emph{continuous}\index{metric!continuous} left-adapted large-scale geodesic metric (\cite[Prop. 4.B.4]{CdlH}). For example, if $G$ is a connected Lie group, then any left-invariant Riemannian metric on $G$ is a large-scale geodesic left-admissible metric. This shows:
\begin{proposition}\label{CompactlyGeneratedConsequences} Let $G$ be a compactly-generated lcsc group. Then there exists a large-scale geodesic left-admissible metric $d$ on $G$.\qed
\end{proposition}
In fact, Proposition \ref{CompactlyGeneratedConsequences} can be turned into a characterization of compactly-generated lcsc groups:
\begin{proposition}[Geometric characterizations of compactly-generated groups]\label{CompGenGroup} Let $G$ be a lcsc group. Then the following are equivalent:
\begin{enumerate}[(i)]
\item $G$ is compactly-generated.
\item Every left-admissible metric on $G$ is coarsely connected.
\item $G$ admits a coarsely connected left-admissible metric.
\item $G$ admits a large-scale geodesic left-admissible metric.
\item Every representative of $[G]_c$ is coarsely connected.
\item Every representative of $[G]_c$ is coarsely geodesic.
\item $[G]_c$ admits a large-scale geodesic representative.
\end{enumerate}
If any one of these equivalent conditions holds, then any two large-scale geodesic representatives of $[G]_c$ are quasi-isometric.
\end{proposition}
\begin{proof} The implication (i) $\Rightarrow$ (iv) follows from Proposition \ref{CompactlyGeneratedConsequences}. Using Lemma \ref{CharLSG} we also have implications (iv)$\Rightarrow$ (iii) $\Rightarrow$ (v) $\Rightarrow$ (ii) and implications (iv)$\Rightarrow$(vii)$\Rightarrow$(vi)$\Rightarrow$(v). To see that (ii) $\Rightarrow$ (i), observe that if $d$ is a left-admissible metric on $G$ such that $(G,d)$ is $C$-coarsely connected, then $G$ is generated, then the compact set $\overline{B}(e, C)$ generates $G$, hence $G$ is compactly-generated. This process yields the equivalence of conditions (i)-(vii), and the final statement is immediate from Gromov's trivial lemma (Corollary \ref{GromovTrivialSymmetric}).
\end{proof}
\begin{definition} The \emph{canonical QI-type}\index{canonical QI-type}\index{QI-type!canonical} of a compactly-generated lcsc group is
\[
[G] \coloneqq  \{X \in [G]_c \mid X \text{is large-scale geodesic}\}.
\]
\end{definition}
Note that the canonical QI type of a compactly-generated lcsc group is indeed a QI type by Proposition \ref{CompGenGroup} and Lemma \ref{CharLSG}. It follows from Lemma \ref{CompactWordMetric} that if $S$ is a compact generating set of a lcsc group $G$, then $(G, d_S)$ is a representative of $[G]$. By Proposition \ref{CompactlyGeneratedConsequences} the canonical QI type can always be represented by a continuous metric on $G$. We emphasize once more that the canonical QI type of a lcsc group is only defined if the group is compactly generated.

Another important aspect of geometric group theory are nice group actions:
\begin{definition}\label{ProperAct} According to \cite[Def. 4.C.1]{CdlH}, an isometric action of a topological group $G$ on a pseudo-metric space $(X,d)$ is called 
\begin{enumerate}[(i)]
\item \emph{cobounded}\index{action!group action!cobounded}, if some (hence any) $G$-orbit is relatively dense in $X$;
\item \emph{locally bounded}\index{action!group action!locally bounded}, if for every $g \in G$ and $B \subset X$ bounded there exists an open neighborhood $U$ of $g$ such that $UB \subset X$ is bounded;
\item \emph{metrically proper}\index{action!group action!metrically proper}, if $\{g \in G\mid d(x, g.x) \leq R \}$ is relatively compact for all $x \in X$ and $R \geq 0$;
\item \emph{geometric}\index{action!group action!geometric}, if it is cobounded, locally bounded and metrically proper.
\end{enumerate}
\end{definition}
\begin{remark}[On the notion of geometric actions of groups] \label{rem:geom_act}
However, in many situations of interest, geometric actions admit a much simpler characterization. Assume for example that $(X,d)$ is a non-empty proper metric space, and that $G$ is a locally compact group acting continuously on $X$ by isometries. (This includes the case where $G$ is discrete.) Then the $G$ action on $X$ is automatically locally bounded, and it is cobounded if and only if it is \emph{cocompact}\index{action!group action!cocompact} (i.e.\ there exists a compact subset $K \subset X$ such that $G.K = X$), and metrically proper if and only if it is proper (i.e.\ the map $G \times X \to X \times X$ given by $(g, x) \mapsto (x, g.x)$ is proper). In particular, such an action is geometric if and only if it is proper and cocompact. In the case of actions of discrete groups on proper metric spaces, this is the definition given in most books on geometric group theory.
\end{remark}

\begin{example}[Examples of compactly-generated lcsc groups]
The following are examples of compactly-generated lcsc groups.
\begin{enumerate}[(i)]
\item Every connected lcsc group $G$ is compactly-generated - any compact identity neighborhood is a generating set. 
\item More generally, every lcsc group $G$ with finitely many connected components is compactly-generated - any compact identity neighborhood which intersects all of the components is a generating set.
\item Isometry groups of  \lsg proper metric spaces are compactly-generated if they act coboundedly; this is a consequence of the following proposition.
\end{enumerate}
\end{example}
\begin{proposition}\label{MSTop} Let $(X, d)$ be a \lsg proper metric space and assume that $G \coloneqq  {\rm Is}(X,d)$ acts coboundedly on $X$. Then the following hold:
\begin{enumerate}[(i)]
\item $G$ is compactly generated (hence the canonical QI type $[G]$ is well-defined).
\item If $S$ is any compact symmetric generating set of $G$, then every orbit map defines a quasi-isometry $(G, d_S) \to (X, d)$. In particular, $(X,d)$ represents the canonical QI type $[G]$ of $G$.
\end{enumerate}
\end{proposition}
\begin{proof} (i) This is contained in \cite[Prop. 5.B.10]{CdlH}.

(ii) By \cite[Prop. 5.B.10]{CdlH}, the action of $G$ on $X$ is geometric in the sense of \cite[Def. 4.C.1]{CdlH}. It thus follows from  \cite[Theorem 4.C.5]{CdlH} that if $o\in X$ is an arbitrary basepoint, then the map $d_o: G \times G \to [0, \infty)$ given by $d_o(g, h) \coloneqq  d(g.o, h.o)$ defines a left-adapted pseudo-metric on $G$ such that the orbit map $(G, d_o) \to (X,d)$ with respect to $o$ is a quasi-isometry. Since $d_o$ is left-adapted we have $(G, d_o) \in [G]_c$, and hence $(X,d) \in [G]_c$. Since $(X,d)$ is large-scale geodesic, so is $(G, d_o)$ and hence we even have $(G, d_o) \in [G]$ and $(X,d) \in [G]$. Finally, if $S$ is a compact generating set of $G$, then the identical map $(G, d_o)\to(G, d_S)$ is a coarse equivalence by Proposition \ref{CoarseGroupClass}, and since both spaces are \lsg (see Lemma \ref{CompactWordMetric}) it is even a quasi-isometry. The proposition follows.
\end{proof}
Again, we are particularly interested in the case of countable discrete groups. For such groups, being finitely-generated is also equivalent to the classical finiteness property (F1) (see e.g.\ \cite[Prop.~7.2.1]{Geoghegan}). We thus obtain the following version of Proposition \ref{CompGenGroup}:
\begin{proposition}[Geometric characterizations of finitely-generated groups]\label{FinGenGroup} Let $\Gamma$ be a countable group. Then the following are equivalent.
\begin{enumerate}[(i)]
\item $\Gamma$ is finitely-generated.
\item $\Gamma$ is of type $F_1$, i.e.\ admits a classifying space with finite $1$-skeleton.
\item Every left-admissible metric on $\Gamma$ is coarsely connected.
\item $\Gamma$ admits a coarsely connected left-admissible metric.
\item $\Gamma$ admits a large-scale geodesic left-admissible metric.
\item Every representative of $[\Gamma]_c$ is coarsely connected.
\item Every representative of $[\Gamma]_c$ is coarsely geodesic.
\item $[\Gamma]_c$ admits a large-scale geodesic representative.
\end{enumerate}
If any one of these equivalent conditions holds, then any two large-scale geodesic representatives of $[\Gamma]_c$ are quasi-isometric.
\end{proposition}
Recall from Lemma \ref{CompactWordMetric} that if $S$ is any finite symmetric generating set of a finitely-generated group $\Gamma$ with associated word metric $d_S$, then $(\Gamma, d_S)$ represents the canonical QI class $[\Gamma]$. Since this QI class is \lsg, we know from Proposition \ref{CharLSG} that it can be represented by a locally finite graph. In fact, it is easy to construct such a graph explicitly:
\begin{example}[Cayley graphs of finitely-generated groups] If $\Gamma$ is finitely-generated with finite symmetric generating set $S$, then the \emph{Cayley graph}\index{Cayley graph} of $\Gamma$ with respect to $S$ is the locally finite graph ${\rm Cay}(\Gamma, S)$ with vertex set $\Gamma$ and edges given by pairs $(g, gs)$ with $g \in \Gamma$ and $s\in S$. If we restrict the canonical metric of ${\rm Cay}(\Gamma, S)$ to the vertices then we recover the word metric $d_S$. Since the vertex set of any locally finite graph is relatively dense with respect to the canonical metric, the isometric inclusion $(\Gamma, d_S) \hookrightarrow {\rm Cay}(\Gamma, S)$ is a quasi-isometry. This shows that $ {\rm Cay}(\Gamma, S)$ is a locally finite graph which represents $[\Gamma]$.
\end{example}
We conclude this section by recalling the Milnor--Schwarz lemma for countable discrete groups; recall the definition of geometric action (Definition \ref{ProperAct} and Remark \ref{rem:geom_act}).

\begin{lemma}[Milnor--Schwarz lemma]\label{MSClassic} Let $(X,d)$ be a large-scale geodesic proper metric space and let $\Gamma$ be a countable discrete group acting geometrically on $X$. Then $\Gamma$ is finitely-generated and $(X,d)$ is a representative of the canonical QI class $[\Gamma]$.
\end{lemma}

\section{Free groups and combinatorics of words}
In this section we collect some combinatorial results about words in free groups which will be needed in the proof of Theorem \ref{QuasikernelsConical}. Throughout this section let $r \geq 2$ be a natural number and $S = \{a_1, \dots, a_r\}$ be a set of cardinality $r$. We denote by $F_r$ the free group over $S$. Given $g \in F_r$ we denote by $\|g\| \coloneqq  \|g\|_S$ the word length of $g$ with respect to $S$.
\begin{notation}[Reduced products] 
If $u,u_1, \dots, u_l \in F_r \setminus\{e\}$ then we write $u \overset{r}=u_1 \cdots u_l$ provided $u = u_1 \cdots u_l$ and $\|u\| = \|u_1\| + \dots + \|u_l\|$; we then say that $u$ is the \emph{reduced product}\index{reduced product} of $u_1, \dots, u_l$.

Set $S^* = S \cup S^{-1}$; then every $g \in F_r\setminus \{e\}$ can be written uniquely as $g \overset{r}{=} g_1 \cdots g_l$ with $g_1, \dots, g_l \in S^*$ and $\|g\| = l$; we then call $(g_1, \dots, g_l)$ the \emph{reduced word} representing $g$. Similarly, we say that the empty word is the reduced word representing $e$.
\end{notation} 
\begin{definition}\label{Occurrence} Let $u \overset{r}{=} u_1 \cdots u_l$ and $v \overset{r}= v_1 \cdots v_n$ with $l,n \in \mathbb N$, $u_1, \dots, u_l$, $ v_1, \dots, v_n \in S^*$. If for some $k \in \{1, \dots, n-l+1\}$ we have
\[
v_k = u_1, \dots, v_{k+l} = u_l,
\]
then we say that $u$ \emph{occurs} in $v$ at position $k$. We then denote by $\#_u(v)$\index{$\#_u(v)$} the number of such occurrences of $u$ in $v$.
\end{definition}
It will sometimes be convenient to also define $\#_u(e) \coloneqq 0$ and $\#_e(v) = \|v\|$ for $u,v \in F_r$. Note that in the definition of $\#_u(v)$ occurrences of $u$ in $v$ may overlap, e.g.\ $\#_{a_1^2}(a_1^5) = 4$.
\begin{definition}
Let $u, x,y\in F_r \setminus\{e\}$ with $\|u\| \geq 2$.
We say that $w$ is a \emph{$u$-blocker}\index{$u$-blocker} for $(x,y)$ if $xwy$ is a reduced product and 
\[
\#_{u}(xwy) = \#_{u}(x) + \#_{u}(y) \qand \#_{u^{-1}}(xwy) = \#_{u^{-1}}(x) + \#_{u^{-1}}(y).
\]
\end{definition}
\begin{example}[An unblockable word]
Let $u=a_1a_2 \in F_2$. We claim that there is no $u$-blocker for the pair $(a_1, a_1^{-1})$ or in fact for any pair $(x,y)$ where $x$ ends in $a_1$ and $y$ starts with $a_1^{-1}$. Assume for contradiction that $w$ was such a blocker. Since $xw$ is reduced and does not contain any copies of $a_1a_2$ or $a_2^{-1}a_1^{-1}$ overlapping $w$, we see that $w$ has to start from either $a_1$ or $a_2^{-1}$. Moreover, every $a_1$ in $w$ can only be followed by $a_1$ or $a_2^{-1}$, and every $a_2^{-1}$ can only be followed by $a_1$ or $a_2^{-1}$. But this means that $w$ has to end with either $a_1$ or $a_2^{-1}$. In the former case, $wy$ is not reduced, and in the latter case we generate a copy of $u^{-1}$. This is a contradiction.
\end{example}
We are going to show that this issue does not exist in free groups of higher rank:
\begin{theorem}[Existence of blockers]\label{BlockersExist} Assume that $r \geq 3$ and let $u \in F_r$ with $\|u\| \geq 2$. Then there exists a constant $C_u>0$ with the following property:
For all $x,y\in F_r \setminus\{e\}$ there exists a $u$-blocker $w$ for the pair $(x,y)$ which satisfies
\[
\|u\| \leq \|w\| < C_u.
\]
\end{theorem}
The following proof of Theorem \ref{BlockersExist} arose from discussions with Alexey Talambutsa. We are grateful to him for allowing us to include the proof in this book. Until the end of the proof of Lemma \ref{Alexey1} let $r \geq 3$ and fix $u \in F_r \setminus\{e\}$ with $l\coloneqq  \|u\| \geq 2$; we write $u \overset{r} = u_1 \cdots u_l$ with $u_1, \dots, u_l \in S^*$ and $l \geq 2$. 
\begin{lemma}[Reduction to the equal length case] Theorem \ref{BlockersExist} holds if for every $(x,y) \in F_r^2$ with $\|x\| = \|y\| = l-1$ there exists a $u$-blocker of length $\geq \|u\|$.
\end{lemma}
\begin{proof} On the one hand, the property of $w$ being a $u$-blocker for $(x,y)$ depends only on the last $(l-1)$ letters of $x$ and the first $(l-1)$ letters of $y$, hence we may assume that $\|x\|, \|y\| \leq l-1$. Since there are only finitely many words of length at most $(l-1)$ this implies in particular that the sizes of blockers are uniformly bounded (by some constant $C_u$ depending on $u$), provided such blockers exist.

On the other hand, if $x$ or $y$ have length less than $(l-1)$, then we can prolong them to words $x'$ and $y'$ of length $(l-1)$ by adding letters to the left of $x$ and the right of $y$; every $u$-blocker for $(x', y')$ is then automatically a blocker for $(x,y)$. 
\end{proof}
From now on we will only consider elements $x,y \in F_r$ with $\|x\| = \|y\|=l-1 < \|u\|$. Note that this assumption implies that \[\#_u(x) = \#_{u^{-1}}(x) = \#_u(y) = \#_{u^{-1}}(y) = 0,\]
and thus $w$ is a $u$-blocker for $(x,y)$ if and only if
\[
\#_u(xwy) = \#_{u^{-1}}(xwy) = 0. 
\]
To establish the existence of such a blocker we consider the following class of graphs which is a variant of the classical De Bruijn graphs \cite{DeBruijn46} and was studied by Reiner Martin in his PhD thesis \cite{MartinPhD}:

\begin{definition} Let $l \geq 2$. The \emph{(oriented, labelled) $l$-th De Bruijn--Martin graph}\index{De Bruijn--Martin graph} $\Gamma_{l,r}$ of $F_r$ is the graph with vertex set $V \coloneqq  \{v \in F_r \mid \|v\|=l-1\}$ and set of oriented edges $\vec{E} \coloneqq  \{u \in F_r\mid\|u\| = l\}$. If $u \overset{r} = u_1 \cdots u_l$ with $u_1, \dots, u_l \in S^*$, then the edge $u$ is considered as an edge from the vertex $u_1 \cdots u_{l-1}$ to the vertex $u_2 \cdots u_l$ and labelled by $\ell(u) \coloneqq  u_l$. Moreover, the pair of edges $(u, u^{-1}) \in \vec{E}^2$ is called an \emph{inverse edge pair}\index{inverse edge pair}.
\end{definition}
To relate these graphs to the problem at hand, we use the following terminology:
\begin{definition} Let $y$ be a vertex of $\Gamma_{l,r}$. We say that a vertex $v$ is a \emph{$u$-acceptance state}\index{$u$-acceptance state} for $y$ if the product $vy$ is reduced with 
\[
\#_u(vy) = \#_{u^{-1}}(vy) = 0.
\]
\end{definition}
We then have the following reformulation of Theorem \ref{BlockersExist}
\begin{proposition}[Graph theoretic formulation]\label{GraphFormulation} Theorem \ref{BlockersExist} holds provided for all vertices $x,y$ of $\Gamma_{l,r}$ there exists an oriented edge path of length $\geq l$ in $\Gamma_{l,r}$ from $x$ to a $u$-acceptance state $v$ for $y$ which does not use either of the edges $u$ or $u^{-1}$.
\end{proposition}
\begin{proof} Assume that there is a path as in the proposition with labels given by $w_1, \dots, w_N$ with $N \geq l$. We define $w \coloneqq  w_1 \cdots w_N$. By definition of $\Gamma_{l,r}$ we then have $\|w\|=N$ and $xw$ is reduced and ends with $v$. We claim that $w$ is a $u$-blocker for $(x,y)$. For this we first note that since $w$ ends in $v$ and $v$ is an acceptance state for $y$, the word $xwy$ is reduced. It thus remains to show that neither $u$ nor $u^{-1}$ appears in $xwy$.

Since $\|w\| \geq l$, every occurrence of $u$ or $u^{-1}$ in $xwy$ occurs either in $xw$ or in $vy$; the latter is excluded  since $v$ is an acceptance state for $y$. On the other hand, an occurrence of $u$ or $u^{-1}$ in $xw$ means precisely that at some point the considered path from $x$ to $v$ crosses one of the two forbidden edges. This finishes the proof.
\end{proof}
Proposition \ref{GraphFormulation} reduces the proof of Theorem \ref{BlockersExist} to showing that acceptance states always exist and that they can be reached from any vertex, avoiding any given forbidden edge pair. This is established by the following two lemmas.
\begin{lemma}[Acceptance states exist] For every vertex $y$ of $\Gamma_{l,r}$ there exists a $u$-acceptance state.
\end{lemma}
\begin{proof} The total number of vertices of $\Gamma_{l,r}$ is given by $2r(2r-1)^{l-2}$. Now let $v$ be a vertex; we write $v \overset{r} = v_1 \cdots v_{l-1}$ and $y \overset{r}= y_1 \cdots y_{l-1}$ with $v_1, \dots, v_{l-1}, y_1, \dots, y_{l-1} \in S^*$. Then $v$ is \emph{not} a $u$-acceptance state if one of the following cases occurs:
\begin{itemize}
    \item $v_{l-1} = y_{1}^{-1}$; this is the case for $(2r-1)^{l-2}$ vertices;
    \item $v_1 \cdots v_{l-1} y_1$ equals $u$ or $u^{-1}$; this is the case for at most $2$ vertices, namely $v=u_1\cdots u_{l-1}$ and $v=u_{l}^{-1} \cdots u_2^{-1}$;
    \item $v_2 \cdots v_l y_1y_2$ equals $u$ or $u^{-1}$; in this case there are at most two choices for $v_2 \cdots v_l$ and thus at most $2(2r-1)$ choices for $v$;
    \item \dots
    \item $v_{l-1}y_1 \cdots y_{l-1}$ equals $u$ or $u^{-1}$; here there are $2(2r-1)^{l-2}$ choices for $v$.
\end{itemize}
The number $\alpha$ of $u$-acceptance states thus satisfies
\begin{eqnarray*}
\alpha &\geq& 2r(2r-1)^{l-2} - (2r-1)^{l-2} -2 -2(2r-1) -\dots - 2(2r-1)^{l-2}\\
&=& 2r(2r-1)^{l-2} - (2r-1)^{l-2} -2 \frac{(2r-1)^{l-1}-1}{(2r-1)-1}\\
&=& (2r-1)^{l-2} \cdot \left(2r -1 - \frac{(2r-1)^{l-1}-1}{(2r-1)^{l-2}(r-1)} \right)\\
& > & (2r-1)^{l-2} \cdot \left(2r -1 -\frac{2r-1}{r-1} \right) \ = \  (2r-1)^{l-2}\cdot\left(2r-3-\frac{1}{r-1}\right) \ \geq \ 0.
\end{eqnarray*}
Thus $\alpha>0$ and the lemma follows.
\end{proof}
Note that the previous lemma even applies in the case $r=2$. This is also the case for Part (i), but not for Part (ii) of the following lemma.
\begin{lemma}[Strong connectedness of de Bruijn-Martin graphs]\label{Alexey1} Let $x,y$ be vertices of $\Gamma_{l,r}$.
\begin{enumerate}[(i)]
    \item There exists an oriented edge path of length $\geq l$ from $x$ to $y$ in $\Gamma_{l,r}$.
    \item There exists an oriented edge path of length $\geq l$ from $x$ to $y$ in $\Gamma_{l,r}$ which does not use either of the edges $u$ or $u^{-1}$.
\end{enumerate}
\end{lemma}
\begin{proof} (i) Connectedness of the unoriented De Bruijn--Martin graph was proved in \cite[Lem\-ma 2.10]{HartnickTalambutsa}, and the argument actually shows that if $x \neq y$ there always exists an oriented edge path from $x$ to $y$ (and hence from $y$ to $x$) of length $\geq 1$. Concatenating such paths we can construct arbitrarily long edge paths from $x$ to $y$ for any two vertices $x$ and $y$.

(ii) Since $u \overset{r}=u_1 \cdots u_l$ with $u_1, \dots, u_l \in S^*$, the forbidden edges and their labels are given by
\[
u_1 \cdots u_{l-1} \xlongrightarrow{u_l} u_2 \cdots u_l \qand u_l^{-1} \cdots u_2^{-1} \xlongrightarrow{u_1^{-1}} u_{l-1}^{-1} \cdots u_1^{-1}.
\]
In view of (i) it suffices to show that there exist  oriented edge paths from $u_1 \cdots u_{l-1}$ to $u_2 \cdots u_l$ and from $u_l^{-1} \cdots u_2^{-1}$ to $u_{l-1}^{-1} \cdots u_1^{-1}$ which bypass the forbidden edges $u$ and $u^{-1}$. 

Since $r \geq 3$ (and thus $|S^*| \geq 6$), we can find $t, t' \in S^*$ such that
\[
t \not \in \{u_1, u_1^{-1}, u_{l-1}^{-1}, u_l, u_l^{-1}\} \qand t' \not \in \{u_1, u_2^{-1}, u_l^{-1}, t^{-1}\}.
\]
Since $t \neq u_{l-1}^{-1}$, there is an oriented edge path in $\Gamma_{l,r}$ such that
\[
u_1 \cdots u_{l-1} \xlongrightarrow{t} u_2 \cdots u_{l-1}t \xlongrightarrow{t}  \dots \xlongrightarrow{t} t^{l-1}.
\]
All of the edges in this path are labelled by $t$, which is different from $u_l$ and $u_1^{-1}$, and hence avoid the two forbidden edges. Since $t' \neq t^{-1}$, we may also define an oriented edge path in $\Gamma_{l,r}$ by
\[
t^{l-1} \xlongrightarrow{t'} t^{l-2}t'  \xlongrightarrow{t'}  \dots \xlongrightarrow{t'} (t')^{l-1}.
\]
All of these edges start from a vertex with initial letter $t$, and since $t \not \in \{u_1, u_l^{-1}\}$, none of them can be forbidden. Finally, since $t' \neq u_2^{-1}$ we may define an oriented edge path in $\Gamma_{l,r}$ by
\[
(t')^{l-1} \xlongrightarrow{u_2} (t')^{l-2}u_2 \xlongrightarrow{u_3} \dots \xlongrightarrow{u_l} u_2 \cdots u_l.
\]
All of these edges start from a vertex with initial letter $t'$, and since $t
'\not \in \{u_1, u_l^{-1}\}$, none of them can be forbidden. Concatenating the three paths above we obtain an edge path in $\Gamma_{l,r}$ from $u_1 \cdots u_{l-1}$ to $u_2 \cdots u_l$ which avoids the two forbidden edges. The argument for a bypass from  $u_l^{-1} \cdots u_2^{-1}$ to $u_{l-1}^{-1} \cdots u_1^{-1}$ is similar.
\end{proof}
At this point we have established Theorem \ref{BlockersExist}.

\begin{remark}[Cyclic words]\label{CyclicWords} Let $u \overset{r}{=} u_1 \cdots u_l$ and $v \overset{r} = v_1 \cdots v_k$ be elements of $F_r$ with $l,n\in \mathbb N$ and $u_1, \dots, u_l, v_1, \dots, v_k \in S^*$. 

\item We say that $u$ is \emph{cyclically reduced}\index{cyclically reduced word} if $u_l^{-1} \neq u_1$. This implies that $\|u^n\| = n\|u\|$, i.e.\ the powers of $u$ lie on a (unique) geodesic of the Cayley tree of $F_r$ with respect to $S$. If $u$ and $w$ are cyclically reduced, then we say that they are \emph{cyclically equivalent}\index{cyclically equivalent words} if there exists a cyclic permutation $\sigma$ of $\{1, \dots, l\}$ such that $w = u_{\sigma(1)} \cdots u_{\sigma(l)}$. We refer to the equivalence class $[u]$ of $u$ as a \emph{cyclic word}\index{cyclic word}. 

\item If $u$ and $v$ are cyclically reduced then we say that $u$ \emph{cyclically occurs}\index{cyclical occurrence!of a word in another word} in $v$ at position $s$ if $v_s = u_1$, $v_{[s+1]}= u_2$, \dots, $v_{[s+l-1]} = u_l$, where  $[m] \in \{1, \dots, k\}$ denotes the residue of $m$ modulo $k$. The number of such cyclic occurrences depends only on $[v]$, and hence will be denote by $\#^{\mathrm{cyc}}_u([v])$.

If $w \in F_r \setminus\{e\}$ is a word which is not necessarily cyclically reduced then we can write $w$ uniquely as $w\overset{r}= tw_ot^{-1}$ with $t, w_o \in F_r$ and $v_o$ cyclically reduced. We then refer to $w_o$ as the \emph{cyclic reduction}\index{cyclic reduction!of a word} of $w$ and to $[w] \coloneqq  [w_o]$ as the \emph{associated cyclic word}\index{cyclic word!associated}. In particular, this allows us to define $\#^{\mathrm{cyc}}_u([w])$ for a cyclically reduced word $u$ and an arbitrary $w \in F_r \setminus\{e\}$. Finally, we also set $\#^{\mathrm{cyc}}_u([e]) \coloneqq  0$ for every cyclically reduced word $u$.
\end{remark}

\begin{lemma}[Homogeneity]\label{CyclicCountHom}
If $u,w \in F_r \setminus\{e\}$ are cyclically reduced words, then we have $\#^{\mathrm{cyc}}_{u}([w]^n) = n \cdot \#^{\mathrm{cyc}}_{u}([w])$ for all $n \in \mathbb N$.
\end{lemma}
\begin{proof} Since $w^n \overset{r} = w \cdots w$ is a \emph{reduced} product, we deduce that $u$ cyclically occurs in $w$ at position $s$ if and only if it occurs cyclically in $w^n$ at positions $s, s+\|w\|, \dots, s+(n-1)\|w\|$. Conversely, every cyclic occurrence of $u$ in $w^n$ is of this form, and the lemma follows.
\end{proof}

\begin{remark}[Normal forms]\label{uNormalForm} If $u \overset{r}= u_1 \cdots u_l$ with $u_1, \dots, u_l \in S^*$, then we denote by
\[
P(u) \coloneqq  \{u_1 \cdots u_k \mid k \in \{1, \dots, l-1\}\} \qand S(u) \coloneqq  \{u_k \cdots u_l\} \mid k \in \{2, \dots, l\}
\]
the collections of proper prefixes, respectively proper suffixes of $u$. If $v \in F_r \setminus\{e\}$ is another non-trivial word, then we say that $u$ and $v$ \emph{do not overlap}\index{non-overlapping words} and write $u \pitchfork v$ if $P(u) \cap S(v) = \emptyset = P(v) \cap S(u)$ and none of the two words is a proper subword of the other. We say that $u$ is \emph{non-self-overlapping}\index{non-self-overlapping word} if $u \pitchfork u$. We also say that $(u,v)$ is an \emph{independent pair}\index{independent pair of words} if $u \neq v$, $u \pitchfork u$, $v \pitchfork v$ and $u\pitchfork v$. It was established in \cite[Lemma 2.5 and Cor.\ 2.6]{HartnickSchweitzer} that if $u$ is cyclically reduced and non-self-overlapping, then $(u, u^{-1})$ is an independent pair. Assume from now on that this is the case.

Given $w \in F_r \setminus\{e\}$ and $g_1, \dots, g_m \in F_r \setminus\{e\}$, we call $(g_1, \dots, g_m)$ a decomposition of $w$ if $w \overset{r}{=} g_1 \cdots g_m$. If $g_j \overset{r} = a_jb_j$ for some $j \in \{1, \dots, m\}$, then $(g_1, \dots, g_{j-1}, a_j, b_j, \dots, g_m)$ is another decomposition of $w$, called a 
\emph{simple refinement}\index{refinement (word decomposition)!simple refinement} of $(g_1, \dots, g_m)$. We call a decomposition of $w$ a \emph{refinement}\index{refinement (word decomposition)} of $(g_1, \dots, g_m)$ if it can be obtained by a sequence of simple refinements. We say that a decomposition $(g_1, \cdots, g_m)$ is \emph{$u$-maximal} if the number of $g_i$ with $g_i \in \{u, u^{-1}\}$ is maximal. 
By \cite[Lemma 5.13]{HartnickSchweitzer} every $w \in F_r \setminus\{e\}$ admits a unique $u$-maximal decomposition $(g_1, \dots, g_m)$ which is not a refinement of a $u$-maximal decomposition. We call $(g_1, \dots, g_m)$ the \emph{$u$-decomposition} of $w$. Note that in a $u$-maximal decomposition, every $g_j$ is either equal to $u$ or equal to $u^{-1}$ or does not contain any copy of $u$ or $u^{-1}$; moreover, every other element of the decomposition is either a copy of $u$ or $u^{-1}$. For example, if $w = a_1a_2a_1^{-1}a_2^{-1}a_1^{-1}a_2a_1^{-1}$, then its $a_1a_2$-decomposition is
$(a_1a_2, a_1^{-1}, a_2^{-1}a_1^{-1}, a_2a_1^{-1})$.
\end{remark}
The following observation will be needed in the proof of Theorem \ref{QuasikernelsConical}:
\begin{lemma}\label{Quasi1}
Let $u \in F_r$ be cyclically reduced and non-self-overlapping, and let $x \in F_r$ be cyclically reduced with $\#^{\mathrm{cyc}}_u(x) >0$ and $\|x\| \geq \|u\| \geq 2$. Then there exists $x' \in F_r$ which is cyclically equivalent to $x$ such that 
\begin{equation}\label{Quasi11}
\#^{\mathrm{cyc}}_u(x') = \#_u(x') \qand
\#^{\mathrm{cyc}}_{u^{-1}}(x') = \#_{u^{-1}}(x').
\end{equation}
\end{lemma}
\begin{proof} Since $\#^{\mathrm{cyc}}_u(x) >0$ and $\|x\| \geq \|u\|$ we can find a word which is cyclically equivalent to $x$ and contains $u$. We may thus assume without loss of generality that $\#_u(x)>0$, which implies that the word $x$ admits a $u$-decomposition $(g_1, \dots, g_n)$ such that $g_{j} = u$ for some $j \in \{1, \dots, n\}$. Then $x' \coloneqq  g_j g_{j+1}\cdots g_{j-1}$ is cyclically equivalent to $x$ and has $u$ as a prefix. To see that \eqref{Quasi11} holds for $x'$, assume that $u$ (or $u^{-1}$) occurs cyclically in $x'$ at position $s$ in $x'$. Then $s\leq \|x'\|-\|u\|+1$, since otherwise the occurrence of $u$ (or $u^{-1}$) would overlap with $g_j$, contradicting the fact that $(u,u^{-1})$ is an independent pair. But then $u$ (or $u^{-1}$) actually occurs in $x'$ at position $s$, and hence \eqref{Quasi11} holds. 
\end{proof}
We will also need the following consequence of Theorem \ref{BlockersExist}:
\begin{lemma}\label{Quasi2} 
Let $r \geq 3$ and let $u \in F_r$ with $l \coloneqq  \|u\| \geq 2$. Then there exists a constant $C_u$ with the following properties: if $x', y' \in F_r$ with $\min\{\|x'\|, \|y'\|\} \geq \|u\|$, then there exist $w,z \in F_r$ with
\[
\|u\|\leq \|w\| \leq C_u \qand \|u\|\leq \|z\| \leq C_u
\]
such that $x'wy'z$ is reduced and cyclically reduced, and such that every copy of $u$ or $u^{-1}$ in $(x'wy'z)^n$ is contained in one of the $n$ copies of $x'$ or in one of the $n$ copies of $y'$.
\end{lemma}
\begin{proof} Let $C_u$ be the constant from Theorem \ref{BlockersExist}. By the theorem we can find a blocker $w$ for the pair $(x',y')$ and a blocker $z$ for the pair $(y',x')$ of the desired lengths. Then $x'wy'$ and $y'zx'$ are reduced, and hence $x'wy'zx'$ is reduced, i.e.\  $x'wy'z$ is reduced and cyclically reduced. Moreover, since all four factors of $x'wy'z$ have length $\geq \|u\|$, no copy of $u$ or $u^{-1}$ can overlap more than two of the factors, and then the blocker properties of $w$ and $z$ imply that every copy of $u$ or $u^{-1}$ in $(x'wy'z)^n$ is contained in one of the $n$ copies of $x'$ or in one of the $n$ copies of $y'$.
\end{proof}

\section{Real-valued quasimorphisms}\label{SecQuasimorphisms}

Let $\Gamma$ be a group and $H$ be a (possibly discrete) locally compact group. Recall from Definition \ref{DefQM} and Remark \ref{TopQM} that a map $f: \Gamma \to H$ is called a \emph{topological quasimorphism}, if its left-defect set $D(f) = \{f(y)^{-1}f(x)^{-1}f(xy) \mid x,y \in \Gamma\}$ is relatively compact.  According to Example \ref{RQM} a real-valued function is a topological quasimorphism in this sense if and only if its \emph{defect}
\[
d(f) \coloneqq \sup_{g,h \in \Gamma} |f(gh)-f(g)-f(h)|
\]
is finite. We then say that $f$ is a \emph{real-valued quasimorphism}\index{real-valued quasimorphism}. Real-valued quasimorphisms are studied extensively in the geometric group theory literature (see e.g.\ \cite{Calegari_scl}). A real-valued quasimorphism $f: \Gamma \to \R$ is called \emph{homogeneous} if $f(g^n) = n \cdot f(g)$ for all $n \in \bZ$.\index{real-valued quasimorphism!homogeneous}
\begin{remark}[Real-valued quasimorphisms]\label{quasiremark} Let $\Gamma$ be a group. 
\begin{enumerate}[(i)]
\item If $f: \Gamma \to \R$ is a quasimorphism, then $|f(e)| = |f(ee) -f(e) - f(e)| \leq d(f)$, hence $|f(e)| \leq d(f)$. We deduce that for all $g \in G$ we have
\begin{eqnarray*}
    d(f) &\geq& |f(gg^{-1}) + f(g) + f(g^{-1})| \quad \geq \quad  |f(g) + f(g^{-1})| - |f(e)| \\ &\geq& |f(g) + f(g^{-1})| -d(f),
\end{eqnarray*}
and hence $f$ is almost symmetric in the sense that
\begin{equation}
-f(g) - 2d(f) \leq f(g^{-1}) \leq -f(g) + 2d(f).
\end{equation}
\item Every quasimorphism $f: \Gamma \to \R$ can be written as $f = f_{\mathrm{sym}} + b$, where  $f_{\mathrm{sym}}: \Gamma \to \R$, $g \mapsto \frac 1 2 (f(g)-f(g^{-1}))$ is a symmetric quasimorphism and $b$ is bounded. In this sense, every real-valued quasimorphism is at bounded distance from a symmetric one. 
\item In fact, every real-valued quasimorphism $f: \Gamma \to \R$ is at bounded distance from a unique \emph{homogeneous} quasimorphism $\widehat{f}: \Gamma \to \R$. Explicitly, we have
\[
\widehat{f}(g) = \lim_{n \to \infty} \frac{f(g^n)}{n},
\]
where the limit exists by Fekete's lemma. Since $|f(g^n) - nf(g)|\leq (n-1) \cdot d(f)$ for all $n \in \bN$, we have $\|f-\widehat{f}\|\leq d(f)$. Every homogeneous quasimorphism is automatically conjugation-invariant (see e.g. \cite[Sec. 2.2]{Calegari_scl}). 
\item If $f: \Gamma \to \R$ is a quasimorphism with defect set $D(f)$ and defect $d(f)$, then $D(f) \subset [-d(f), d(f)]$, and $d(f)$ is the supremum over the absolute values of elements in $D(f)$.
\item If $\Gamma$ is a finitely-generated group, then every real-valued quasimorphism is a Lipschitz function with respect to some (hence any) word metric on $\Gamma$. Indeed, if $S$ is a finite generating set and $s_1, \dots, s_n \in S \cup S^{-1}$, then
\[
f(s_1 \cdots s_n) \leq n \cdot (d(f) + \sup\{f(s) \mid s \in S\}).
\]
\item If $f: \Gamma \to \R$ is a quasimorphism, then we define $\bar f(g)$ as the nearest integer to $f(g)$. This yields an integer-valued quasimorphism $\bar f: \Gamma \to \Z$, which is at distance at most $\frac 1 2$ from $f$ and hence of defect $d(\bar f) \leq d(f) + \frac 3 2$. Moreover, $\bar f$ is symmetric provided $f$ is symmetric.
\end{enumerate}
\end{remark}
\begin{example}[Perturbed homomorphisms]
If $\Gamma$ is an infinite group, then there are plenty of real-valued quasimorphisms on $\Gamma$. Indeed, if $f_o: \Gamma \to \R$ is a homomorphism (possibly trivial), then $f_o + b$ is a quasimorphism for every bounded function $b: \Gamma \to \R$. Quasimorphisms of this form are called \emph{perturbed homomorphisms}\index{perturbed homomorphism}.
\end{example}
If $\Gamma$ is amenable (as a discrete group), then every real-valued quasimorphism on $\Gamma$ is a perturbed homomorphism, but for many non-amenable groups there are also plenty of exotic examples of real-valued quasimorphisms which are not perturbed homomorphisms. The most basic examples are counting quasimorphisms on free groups: 
\begin{example}[Counting quasimorphisms]\label{CountingQuasimorphism}\label{HomogeneousCounting}
Let $F_r$ be a free group of rank $r$ with free generating set $S \coloneqq  \{a_1, \dots, a_r\}$; we identify elements of $F_r$ with reduced words over $S \cup S^{-1}$. If $u \in F_r \setminus\{e\}$ is a non-trivial reduced word, then the function \[
\phi_u: F_r \to \bZ, \quad v \mapsto \#_u(v) - \#_{u^{-1}}(v)
\]
(with notation as in Definition \ref{Occurrence}) defines a symmetric quasimorphism on $F_r$ (see e.g.\ \cite[p. 22]{Calegari_scl}). It is called the \emph{(big) $u$-counting quasimorphism}.\index{quasimorphism!$u$-counting} We will need two variants of this construction: firstly, the homogenization $f_u: F_r \to \mathbb{R}$ of $\phi_u$ is called the \emph{homogeneous $u$-counting quasimorphism}.\index{quasimorphism!homogeneous $u$-counting} Secondly, if $u$ is cyclically reduced, then the function \[\phi_u^{\mathrm{cyc}}: F_r \to \bZ, \quad v \mapsto \#^\mathrm{cyc}_u([v]) - \#^\mathrm{cyc}_{u^{-1}}([v]) \]
is called the \emph{cyclic $u$-counting quasimorphism}\index{quasimorphism!cyclic $u$-counting}. Here, $[v]$ denotes the cyclic word associated with $v$ and $\#^\mathrm{cyc}_u$ denotes the cyclic counting function as in Remark \ref{CyclicWords}.
\end{example}
Counting quasimorphisms were introduced by Brooks \cite{Brooks}, whereas cyclic counting quasimorphisms were introduced by Fa\u{\i}ziev \cite{Faiziev} and popularized by the paper \cite{Grigorchuk}.
\begin{lemma}\label{FaizievLemma} Assume that $u$ is cyclically reduced. Then 
\begin{equation}\label{GrigorchukFormula1}
f_u = \phi_u^{\mathrm{cyc}}.
\end{equation}
In particular, $f_u$ is a $\bZ$-valued homogeneous quasimorphism and
$f_u(w)$ depends only on $u$ and the cyclic word $[w]$.
\end{lemma}
\begin{proof} We have to show that for every $w\in F_r$ we have
\begin{equation}\label{GrigorchukFormula}
f_u(w) = \#^{\mathrm{cyc}}_u([w]) - \#^{\mathrm{cyc}}_{u^{-1}}([w]).
\end{equation}
For $w=e$ this holds by definition, thus assume that $w \neq e$. We first observe that both sides of \eqref{GrigorchukFormula} remain unchanged if we replace $w$ by its cyclic reduction, and that both sides multiply by $n$ if we replace $w$ by $w^n$ in view of Lemma \ref{CyclicCountHom}. We may thus assume that $w$ is cyclically reduced and that $\|w\|_S\geq \|u\|_S$.

In this situation, every occurrence of $u$ in $w^n$ yields a cyclic occurrence of $u$ in $[w^n]$, and every such cyclic occurrence of $u$ yields an occurrence of $u$ in $w^{n+1}$. Applying the same argument to $u^{-1}$ yields
\begin{small}
\[
\frac{\#_u(w^n)- \#_{u^{-1}}(w^n)}{n} \leq \frac{\#^{\mathrm{cyc}}_u([w^n])-\#^{\mathrm{cyc}}_{u^{-1}}([w^n])}{n} \leq \frac{n+1}{n} \cdot \frac{\#_u(w^{n+1})-\#_{u^{-1}}(w^{n+1})}{n+1}.
\]
\end{small}
Now the expression in the middle equals $\#^{\mathrm{cyc}}_u([w]) - \#^{\mathrm{cyc}}_{u^{-1}}([w])$, and the two outer expressions both converge to $f_u(w)$.
\end{proof}
We will use the following class of counting quasimorphisms to illustrate properties of conical limit sets in Theorem \ref{ConicalQker}
\begin{notation}[Special counting quasimorphisms]\label{SpecialCounting}
Let $S \coloneqq  \{a_1, \dots, a_r\}$ be a set of cardinality $r\geq 2$ and let $F_r$ denote the free group over $S$. We denote by $W_r$ the set of all cyclically reduced and non-self-overlapping words of length $\geq 2$ in $F_r$ and set
\[
\cQ_o(F_r) \coloneqq  \{\phi_u: F_r \to \bZ \mid u \in W_r\} \qand \cQ_o^{\mathrm{cyc}}(F_r) \coloneqq  \{\phi_u^{\mathrm{cyc}}:F_r \to \bZ \mid u \in W_r\}.
\]
\end{notation}
\begin{remark}[Properties of special counting quasimorphisms]\label{FaizievGood}
Let $u \in W_r$ and $w \in F_r \setminus\{e\}$ with cyclic reduction $w_o$.
\begin{enumerate}[(i)]
    \item $\phi_u^{\mathrm{cyc}}$ is the homogenization of $\phi_u$ (by Lemma \ref{FaizievLemma}), in particular it is symmetric and conjugation-invariant.
    \item If $w$ has $u$-decomposition $(g_1, \dots, g_n)$ in the sense of Remark \ref{uNormalForm}, then
    \[
    \phi_u(w) = \{j \in \{1, \dots, n\} \mid g_j = u\} - \{j \in \{1, \dots, n\} \mid g_j = u^{-1}\}.
    \]
   \item It follows from (ii) that $\phi_u(u^n) = n$ for every $n \in \mathbb{Z}$; consequently also $\phi_u^{\mathrm{cyc}}(u^n) = n$ for all $n \in \mathbb Z$. In particular, for $u\in W_r$ the quasimorphisms $\phi_u: F_r \to \mathbb Z$ and $\phi_u^{\mathrm{cyc}}: F_r \to \mathbb Z$ are surjective.
\end{enumerate}
\end{remark}
\begin{remark}[Groups with many non-trivial real-valued quasimorphisms]
It turns out that for every non-abelian free group, the space of finite real linear combinations of counting quasimorphisms modulo perturbed homomorphisms is infinite-dimensional; in fact, the only linear relations between equivalence classes of counting quasimorphisms are ``the obvious ones'' (see \cite{HartnickTalambutsa}). 

In general, if $H$ is a subgroup of a group $G$, then not every quasimorphism $f: H \to \R$ can be extended to a quasimorphism on $G$; in fact, not every group containing a free subgroup admits a non-trivial real-valued quasimorphism. However, if $H$ is ``hyperbolically embedded'' in $G$, then every homogeneous quasimorphism on $H$ extends to $G$ \cite{HullOsin}. It follows that if $G$ is a group which contains a hyperbolically embedded non-abelian free group, then it has plenty of non-trivial quasimorphism. Since groups which admit hyperbolically embedded non-abelian free subgroups are precisely  acylindrically hyperbolic groups \cite{Osin_AcylHyp}, this recovers the fact that acylindrically hyperbolic groups admit an infinite-dimensional space of homogeneous quasimorphisms \cite{BF02}. Conversely, every homogeneous quasimorphism on an acylindrically hyperbolic group arises as an extension of a homogeneous quasimorphism on a hyperbolically embedded free subgroup of rank $2$ \cite{HartnickSisto}, which in turn can be written as a pointwise limit of linear combinations of cyclic counting functions \cite{Grigorchuk}. This illustrates the importance of cyclic counting functions for the general theory.
\end{remark}

\section{Reconstruction theorems for homogeneous quasimorphisms}

In our study of conical limit points we will use a reconstruction theorem for homogeneous quasimorphisms, which is based on work by Ben Simon and the second named author \cite{BSH}; this requires the following terminology.
\begin{notation}[Sandwiched order] Let $f: \Gamma \to \R$ be a non-zero homogeneous quasimorphism on a group $\Gamma$. We say that a partial order $\leq$ on $\Gamma$ is \emph{bi-invariant}\index{bi-invariant partial order}\index{order!bi-invariant partial} if $g\leq h$ implies $gx\leq hx$ and $xg \leq xh$ for all $g,h,x \in \Gamma$; such an order $\leq$ is then determined by the associated \emph{order semigroup}\index{order semigroup} $S_{\leq} \coloneqq  \{g \in \Gamma \mid g \geq e\}$. We say that $\leq$ is \emph{sandwiched}\index{sandwiched order}\index{order!sandwiched} by $f$ if there exist constants $C_2 > C_1>0$ such that
\[
\{g \in \Gamma \mid f(g) > C_2\} \subset S_{\leq} \subset \{g \in \Gamma \mid f(g) > C_1\};
\]
in fact, the lower bound implies the upper bound with $C_1\coloneqq 0$ \cite[Lemma 3.2]{BSH}. We denote by $\mathrm{IPO}(f)$ the set of all invariant partial orders on $\Gamma$ which are sandwiched by $f$.

If $\leq$ is sandwiched by $f$, then there exists $g \in \Gamma$ such that for all $h \in \Gamma$ there exists $n \in \mathbb N$ with $f(g^n) \geq h$. Any such element $g$ is called \emph{dominant element}\index{dominant element}, and following \cite{EliashbergPolterovich} we define the \emph{growth function}\index{growth function} associated with the pair $(g, \leq)$ by \[\gamma_{\leq, g}(h) \coloneqq  \lim_{n \to \infty}\frac{1}{n}\cdot \min \{p \geq 0 \mid g^p \geq h^n\}.\]
\end{notation}
The following is \cite[Prop.\ 1.3]{BSH}
\begin{lemma}[Sandwich lemma] Let $f: \Gamma \to \R$ be a non-zero homogeneous quasimorphism and let $\leq \in \mathrm{IPO}(f)$. Then for every dominant $g$ we have $f(g)>0$ and
\[
\pushQED{\qed}
\gamma_{\leq, g}(h) = \frac{f(h)}{f(g)}.\qedhere \popQED
\]
\end{lemma}
The sandwich lemma has the following immediate consequence:
\begin{corollary}[First reconstruction theorem] Let $f_1,f_2: \Gamma \to \bR$ be non-zero homogeneous quasimorphisms. If $\mathrm{IPO}(f_1) \cap \mathrm{IPO}(f_2) \neq \emptyset$, then $f_1 = C \cdot f_2$ for some $C>0$.\qed
\end{corollary}
We can give a more convenient reformulation using the following notion:
\begin{definition} If $f: \Gamma \to \bR$ is a homogeneous quasimorphism, then the \emph{positivity set}\index{positivity set $\mathrm{Pos}(f)$} $\mathrm{Pos}(f)$ is given by $\mathrm{Pos}(f) \coloneqq  \{g \in \Gamma \mid f(g) > 0\}$.
\end{definition}
\begin{corollary}[Second reconstruction theorem]\label{Reconstruction} Let $f_1,f_2: \Gamma \to \bR$ be non-zero homogeneous quasimorphisms. If $\mathrm{Pos}(f_1) \subset \mathrm{Pos}(f_2)$, then $f_1 = C \cdot f_2$ for some $C>0$.\qed
\end{corollary}
\begin{proof} Let $g_0 \in\mathrm{Pos}(f_1)$ (and hence $g_0\in \mathrm{Pos}(f_2)$) and set $P\coloneqq \mathrm{Pos}(f_2)$. Denote by $d(f_j)$ the supremum of the defect set of $f_j$ and choose $n\in \mathbb N$ with $n > 2\frac{d(f_2)}{f_2(g_0)}$. Then for every $g \in Pg_0^nP$ we have
\[
f_2(g) \geq n \cdot f_2(g_0) - 2 d(f_2) > 2\frac{d(f_2)}{f_2(g_0)} f_2(g_0)- 2d(f_2) \geq 0,
\]
and hence $Pg_0^nP \subset P$. This shows that $g_0^nP$ is a semigroup, and hence 
\[
S = \{e\}\cup\bigcap_{g \in \Gamma} g(g_0^nP)g^{-1}.
\]
is a conjugation-invariant monoid with $S \cap S^{-1} = \{e\}$ (since $f_2(s) > 0$ for all $s\in S \setminus\{e\}$), which implies that $S = S_{\leq}$ for some
conjugation-invariant partial order (see \cite{BSH}).

We claim that $\leq$ is sandwiched by both $f_1$ and $f_2$; this will imply the corollary by the first reconstruction theorem. For the proof of the claim let $j \in \{1,2\}$ and set $C_j \coloneqq  n \cdot f_j(g_0) + d(f_j)$. Now if $x_j \in \Gamma$ with $f(x_j) > C_j$, then for all $g \in \Gamma$ we have
\[
f_j(g_0^{-n}g^{-1}x_jg) \geq f_j(g_0^{-n}) + f(g^{-1}x_j g) - d(f) =-n \cdot f_j(g_0) + f(x_j) - d(f) > 0,
\]
and hence
\[
g_0^{-n}g^{-1}x_jg \in \mathrm{Pos}(f_j) \subset P \implies x_j \in \bigcap_{g \in \Gamma} g(g_0^nP)g^{-1} \subset S_{\leq},
\]
which shows that $f_j$ sandwiches $\leq$ and finishes the proof.
\end{proof}
\section{Trees and their almost automorphism groups}\label{trees}
In this section we fix our terminology concerning trees and their (almost) automorphism. We also provide some background for Example \ref{AAutMain} following \cite{Lederle, AdrienLeBoudec}.
\begin{definition}\label{Def1Tree}
A connected non-oriented graph $\cT$ is called a \emph{tree}\index{tree} if it has no loops, no multiple edges and no cycles. A subgraph of $\cT$ is called a \emph{subforest}\index{subforest} and it is called a \emph{subtree}\index{subtree} if it is connected. We denote by $\mathrm{Aut}(\cT)$ the \emph{automorphism group}\index{tree automorphism} of $\cT$, i.e.\ the group of graph automorphisms of $\cT$. If $\cT$ is a tree and $d_\cT$ is the canonical metric on the geometric realization $|\cT|$
(see Example \ref{CanonicalMetric}), then the metric space $(|\cT|, d_\cT)$ is called a \emph{simplicial tree}\index{tree!simplicial}. 

Given a tree $\cT$ with vertex set $V$ and edge set $E$ we use the following terminology:
Two vertices are called \emph{adjacent} if there is an edge between them. If $v \in V$ then the \emph{valence} $|v|$ is the cardinality of the set of adjacent vertices of $v$, and $v$ is called a \emph{leaf} if $|v| = 1$. We denote by $\cL\cT$ the set of leaves of $\cT$. If $|v|< \infty$ for all $v \in V$, then $\cT$ is called \emph{locally finite}\index{tree!locally finite} and $(|\cT|, d_\cT)$ is called a \emph{locally finite simplicial tree}. 

If $o \in V$, then the pair $(\cT, o)$ is called a \emph{rooted tree}\index{rooted tree} with \emph{root} $o$. In this case, the \emph{level}\index{rooted tree!level} of $v \in V$ is $\ell(v) \coloneqq d_{\cT}(o,v)$, and given $n \in \bN_0$ we set $\cT^{(n)} \coloneqq \{v \in V \mid \ell(v)= n\}$. 
If $\ell(v)=n$ and $w \in V$ is adjacent to $v$ with $\ell(w) = n+1$, then we write $v \to w$ and say that $v$ is the \emph{parent} of $w$ and $w$ is the \emph{child} of $v$. If $v \to v_1 \dots \to v_n \to w$, then we call $w$ a \emph{descendant} of $v$. Given $v \in V$ we denote by $C(v)$ the set of children of $v$ and by $\cT_v$ the subtree of $\cT$ whose vertex set is given by the descendants of $v$ (including $v$).

If $(\cT, o)$ is a rooted tree, then a subtree $\cT'$ of $\cT$ is called a \emph{rooted subtree} if it contains the root $o$; in this case the \emph{rooted diameter} of $\cT'$ is $\mathrm{diam}_o(\cT') \coloneqq \sup\{\ell(v) \mid v \in V(\cT')\}$, where $V(\cT')$ is the vertex set of $\cT'$. The \emph{rooted automorphism group}\index{tree automorphism!rooted} of $(\cT, o)$ is defined as
\[
\mathrm{Aut}(\cT, o) \coloneqq \{\phi \in \Aut(\cT)\mid \phi(o) = o\}.
\]
\end{definition}
\begin{remark}\label{MetricTree} In metric geometry, a  geodesic metric space $(X,d)$ is usually called a \emph{metric tree}\index{tree!metric} if every triangle in a (possibly degenerate) tripod. Every simplicial tree is a metric tree in this sense, and these are the only metric trees which we will encounter in this book.
\end{remark}
For the remainder of this section, $(\cT, o)$ will denote a locally finite rooted tree. Given a subtree $T \subset \cT$ we denote by $V(T)$ the set of vertices of $T$.
\begin{definition}
    A \emph{complete finite subtree}\index{subtree!complete} $T$ of $\mathcal T$ is a finite rooted subtree such that every $v \in V(T)$ we have $C(v) \cap V(T) \in \{\emptyset, C(v)\}$. We denote by $\mathrm{CF}(\cT, o)$ the set of complete finite subtrees of $(\cT, o)$ and for every $m \in \bN_0$ we define
   \[\mathrm{CF}_m(\cT,o) \coloneqq \{T \in \mathrm{CF}(\cT,o) \mid \mathrm{diam}_o(T) \leq m \}.\] 
Given $T \in \mathrm{CF}(\cT, o)$ we set
     \[
    \cT \setminus T \coloneqq \bigsqcup_{v \in \mathcal{L}T} \cT_v.
    \] 
  \end{definition}
Since $\cT$ is locally finite, $\mathrm{CF}_m(\cT, o)$ is finite for every $m \in \bN_0$, and $\mathrm{CF}(\cT,o)$ is closed under finite unions. If $T \in \mathrm{CF}(\cT,o)$, then $\cT \setminus T$ is a subforest of $\cT$ with $|\mathcal{L}T|$ connected components.
\begin{definition}
Let $T_1, T_2, T_1', T_2'\in  \mathrm{CF}(\cT,o)$.
\begin{enumerate}[(i)]
\item A graph isomorphism $\varphi \colon \mathcal{T}\setminus T_1 \to  \mathcal{T}\setminus T_2$ is called an \emph{honest almost automorphism}\index{almost automorphism!honest} of $(\mathcal{T}, o)$ of \emph{depth}\index{almost automorphism!depth of} $\mathrm{depth}(\varphi) \coloneqq \max\{\diam_o(T_1), \diam_o(T_2)\}$.
\item Two honest almost automorphisms $\varphi \colon \mathcal{T}\setminus T_1 \to  \mathcal{T}\setminus T_2$ and $\psi \colon \mathcal{T}\setminus T_1' \to  \mathcal{T}\setminus T_2'$ are called \emph{equivalent} if there exists a finite complete subtree $T \supset T_1 \cup T_1'$ such that $\varphi\vert_{\mathcal{T}\setminus T} = \psi\vert_{\mathcal{T}\setminus T}$.
\item If $\varphi$ is an honest almost automorphism of $\cT$, then its equivalence class $[\varphi]$ is called an \emph{almost automorphism}\index{almost automorphism} of $(\mathcal{T}, o)$, and its \emph{depth} is defined as 
\[
\mathrm{depth}[\varphi] \coloneqq \min\{\mathrm{depth}(\psi) \mid \psi \in [\varphi]\}.
\]
\end{enumerate}
\end{definition}
\begin{construction}\label{CompositionAAut} Let $T_1, \ldots, T_4 \in \mathrm{CF}(\cT, o)$ and let $\varphi \colon \mathcal{T}\setminus T_1 \to  \mathcal{T}\setminus T_2$ and $\psi \colon \mathcal{T}\setminus T_3 \to  \mathcal{T}\setminus T_4$ be honest almost automorphisms. We can then define complete finite subtrees
\[
T \coloneqq T_1 \cup T_4, \quad T_3' \coloneqq T_3 \cup \psi^{-1}(T \setminus T_4) \qand T_2' \coloneqq \varphi(T\setminus T_1) \cup T_2.
\]
and obtain honest almost automorphisms
\[
\varphi': \cT \setminus T \to \cT \setminus T_2' \qand \psi': \cT \setminus T_3' \to \cT \setminus T.
\]
such that $[\varphi'] = [\varphi]$ and $[\psi'] = [\psi]$. Moreover, the composition $\varphi' \circ \psi': \cT \setminus T_3' \to \cT \setminus T_2'$ is a well-defined honest almost automorphisms. If $\varphi''$ and $\psi''$ are any other representatives of  $[\phi]$ and $[\psi]$ respectively, then $\varphi'' \circ \psi''$ is equivalent to $\varphi' \circ \psi'$ and hence we may define
\[
[\varphi] \circ [\psi] \coloneqq [\varphi' \circ \psi'].
\]
One can show that with this multiplication, the set $\mathrm{AAut}(\cT, o)$ of almost automorphisms of $(\cT, o)$ becomes a group (cf. \cite{Lederle}).
\end{construction}
\begin{definition} Given $m \in \bN$ we define a subset
  \[\mathrm{AAut}^m(\mathcal{T}, o) \coloneqq \{ [\varphi] \in \mathrm{AAut}(\mathcal{T}, o) \mid \mathrm{depth}[\varphi] \leq m\}.\]
\end{definition}
Note that every class in $\mathrm{AAut}^0(\mathcal{T}, o)$ can be represented by a unique rooted automorphism of $(\cT, o)$, and hence $\mathrm{AAut}^0(\cT, o)$ is a subgroup isomorphic to $\Aut(\cT, o)$. For $m \geq 1$ the subsets $\mathrm{AAut}^m(\cT, o)$ need not be subgroups; however, we have:
\begin{lemma}\label{AAutFiltration} For all $m, n \in \bN_0$ we have
\[\mathrm{AAut}^m(\mathcal{T}, o) \mathrm{AAut}^n(\mathcal{T}, o) \subset \mathrm{AAut}^{m+n}(\mathcal{T}, o).\]
\end{lemma}
\begin{proof} Let $\alpha \in \mathrm{AAut}^m(\mathcal{T}, o)$ and $\beta \in \mathrm{AAut}^n(\mathcal{T}, o)$. We choose representatives $\varphi \colon \mathcal{T}\setminus T_1 \to  \mathcal{T}\setminus T_2$ and $\psi \colon \mathcal{T}\setminus T_3 \to  \mathcal{T}\setminus T_4$ of $\alpha$ and $\beta$ respectively such that $\mathrm{depth}(\varphi) = m$ and $\mathrm{depth}(\psi) = n$. Accordingly we then have
\[
\mathrm{diam}_o(T_1) \leq m, \; \mathrm{diam}_o(T_2) \leq m, \; \mathrm{diam}_o(T_3) \leq n \; \text{and}\; \mathrm{diam}_o(T_4) \leq n.
\] 
We now define $T$, $T_2'$, $T_3'$, $\varphi'$ and $\psi'$ as in Construction \ref{CompositionAAut}. Since $\varphi' \circ\psi'$ represents $\alpha \circ \beta$, we then have
\[
\mathrm{depth}(\alpha \circ \beta) \leq \max\{\mathrm{diam}_o(T_2'), \mathrm{diam}_o(T_3')\}.
\]
Since $\psi$ is distance preserving and maps the leaves of $T_3$ to the leaves of $T_4$, the rooted diameter of $T_3' = T_3 \cup \psi^{-1}(T_1 \setminus T_4)$ is at most $\mathrm{diam}_o(T_3) + \diam_o(T_1) \leq n+m$. Similarly, since $\phi$ is distance preserving and maps  the leaves of $T_1$ to the leaves of $T_2$, the rooted diameter of $T_2' = \varphi(T_4 \setminus T_1) \cup T_2$ is at most $\mathrm{diam}_o(T_2) + \diam_o(T_4) \leq m+n$.
\end{proof}
Given $d \in \bN$, there is a unique rooted tree $(\cT_d, o)$ such that every vertex has precisely $d$ children. Equivalently, $|o| = d$ and $|v| = d+1$ for every vertex $v \neq o$. We warn the reader that $\cT_d$ is not to be confused with the $d$-regular tree.
\begin{definition}\label{regularrootedtree} $(\cT_d, o)$ is called the \emph{$d$-regular rooted tree}.
\end{definition}
\begin{notation}
We will need two subgroups of $\mathrm{AAut}(\cT_d, o)$. Firstly, we denote by $W$ the pointwise stabilizer of the first level $\cT_d^{(1)}$ of $\cT_d$ in $\Aut(\cT_d, o) = \mathrm{AAut}^0(\cT, o)$. Since for every vertex $v$ of $\cT_d^{(1)}$ we have $(\cT_d)_v \cong \cT_d$, we have $W \cong \prod_{i=1}^d \mathrm{Aut}(\cT_{d},o)$. Another important subgroup of $\mathrm{AAut}(\cT_d, o)$ is the \emph{Higman--Thompson group} $V_d$, which consists of all ``locally order-preserving'' elements of $\mathrm{AAut}(\cT_{d},o)$. We refer to \cite[Sec.\ 2 and 3]{AdrienLeBoudec} for the precise definition\footnote{Note that the tree $\cT_d$ is denoted $\cT_{d,d}$ in \cite{AdrienLeBoudec}.}.
\end{notation}
\begin{proposition}\cite[Prop.\ 4.7]{AdrienLeBoudec}\label{prop:decomp}
    For any $g \in \mathrm{AAut}(\mathcal{T}_{d}, o)$ there exists $u \in W$ and $v \in V_d$ such that $g = uv$.
\end{proposition}
\begin{remark}\label{LeBoudecConvenient} The proof of \cite[Prop.\ 4.7]{AdrienLeBoudec} actually provides a more precise statement: If $\mathrm{depth}(g) = n$, then we can choose $u \in W$ and $v \in V_d$ with the following properties:
\begin{itemize}
\item $g = uv$.
\item There exist a unique representative $\phi: \cT \setminus T \to \cT \setminus T'$ of $v$ such that $T, T' \in \mathrm{CF}_n(\cT,o)$. In particular, $\mathrm{depth}(v)= n$.
\item $u$ is a rooted automorphism of $\cT$ which fixes every vertex of $T'$ and hence maps $\cT \setminus T'$ to itself. 
\end{itemize}
In particular, every class in $V_d^n \coloneqq V_d \cap \mathrm{AAut}^n(\cT, o)$ can be represented by a unique honest almost automorphism $\phi: \cT \setminus T \to \cT \setminus T'$ with $\mathrm{diam}_o(T) = \mathrm{diam}_o(T')=n$. Moreover, since $\phi$ is locally order-preserving it is uniquely determined by the induced bijection $\cL T\to \cL T'$. Since there are only finitely many such bijections, we deduce that $V_d^n$ is in fact finite.
\end{remark}

\chapter{Morse hyperbolic spaces and their Gromov boundaries} \label{AppHyperbolic}

\section{Gromov, Rips and Morse hyperbolic spaces}
In this section we collect some basic notions concerning hyperbolic spaces and their boundaries. We use three definitions of hyperbolicity. The first definition, called the 4-point or Gromov definition of hyperbolicity, is defined for any metric space. 

\begin{definition}[Gromov hyperbolic spaces] \label{def:gromov hyp} Let $(X, d)$ be a metric space. Given three points $x,y,w \in X$, the \emph{Gromov product}\index{Gromov product} of $x$ and $y$ with respect to the basepoint $w$ is defined as 
\[
\gp{x}{y}{w}=\frac{1}{2}\left(d(w,x) + d(w, y)- d(x,y)\right).
\]
Given a $\delta\geq 0$, a metric space $X$ is said to be $\delta$-\emph{Gromov hyperbolic}\index{Gromov hyperbolic}\index{Gromov hyperbolic!metric space}\index{hyperbolic!Gromov} if for all $x,y,z,w \in X$:
\[
\gp{x}{y}{w} \geq \min\set{\gp{x}{z}{w}, \gp{y}{z}{w}}-\delta.
\]
A metric space $X$ is called Gromov hyperbolic if it is $\delta$-Gromov hyperbolic for some $\delta\geq 0$.
\end{definition}

The second definition, credited to Rips, requires that the metric space be geodesic because it characterizes hyperbolicity by asking that all geodesic triangles be ``slim''. 

\begin{definition}[Rips hyperbolic spaces] \label{def:rips hyp}
Let $(X, d)$ be a geodesic metric space. Given a $\delta \geq 0$, a geodesic triangle in $X$ is said to be \emph{$\delta$-slim}\index{slim triangle}\index{triangle!slim} if each of its sides is contained in the $\delta$-neighborhood of the union of the other two sides. 

The space $(X,d)$ is called  \emph{$\delta$-Rips hyperbolic}\index{Rips-hyperbolic}\index{hyperbolic!Rips}  if every geodesic triangle in $X$ is $\delta$-slim. It is called Rips hyperbolic if it is $\delta$-Rips hyperbolic for some $\delta\geq 0$.
\end{definition}

Rips' definition is often easier to handle; in the case of geodesic metric spaces, it coincides with Gromov's definition:

\begin{proposition}[Rips vs.\ Gromov hyperbolicity, {\cite[Prop. III.H.1.22]{bridson-haefliger}}]\label{RipsGromov}
For every $\delta\geq 0$ there exists $\delta'\geq 0$ with the following property: If a geodesic metric space $(X,d)$ is $\delta$-Gromov hyperbolic (respectively $\delta$-Rips hyperbolic), then it is $\delta'$-Rips hyperbolic (respectively $\delta'$-Gromov hyperbolic). In particular, a geodesic metric space is Gromov hyperbolic if and only if it is Rips hyperbolic.\end{proposition}

Among geodesic metric spaces, Gromov hyperbolicity (or, equivalently, Rips hyperbolicity) is actually a QI-invariant, see e.g.\ {\cite[Thm. III.H.1.9]{bridson-haefliger}}:
\begin{theorem}[QI-invariance of hyperbolicity among geodesic metric spaces]\label{HyperbolicQI}
If $X$, $X'$ are quasi-isometric geodesic metric spaces, then $X$ is hyperbolic if and only if $X'$ is hyperbolic.\qed
\end{theorem}

The proof of this theorem uses a particularly nice property of hyperbolic spaces, that says, roughly, that (the image of) a quasi-geodesic segment is close to any geodesic segment joining its endpoints. The following lemma details this property. 
\begin{lemma}[Morse lemma]\label{Morse1} There exists a function $\Theta(K, L, \delta)$ so that the following holds: If $X$ is a $\delta$-Gromov hyperbolic geodesic space, then for every $(K,L)$-quasi-geodesic segment $f \colon [a,b] \to X$, the Hausdorff distance between the image of $f$ and any geodesic segment between $f(a)$ and $f(b)$ is at most $\Theta(K, L, \delta)$.\qed
\end{lemma}

Occasionally, in particular in the context of Morse boundaries, we will have to deal with non-geodesic hyperbolic spaces. In this wider context, Gromov hyperbolicity is no longer a quasi-isometry invariant.
\begin{example} Let $X \subset \R^2$ be the graph of the function $x \mapsto |x|$ with the metric induced from the Euclidean metric of $\R^2$. Then $X$ is not hyperbolic, but $\pi_1: X \to \R$, $(x,y) \mapsto x$ is a quasi-isometry (even bi-Lipschitz), and $\R$ is $0$-Gromov hyperbolic. See \cite[Remark 4.1.3]{BuyaloSchroeder} for a more refined example.
\end{example}
The next theorem gives four characterizations of a quasi-isometrically invariant notion of hyperbolicity. Two of these characterization will use the notion of a Morse quasi-geodesic (Definition \ref{defn: Morse geodesic}); compare the definition with the Morse lemma (Lemma \ref{Morse1}). We also use the notation of a $\delta$-slim quasi-geodesic triangle; $\delta$-slimness of quasi-geodesic triangles is defined exactly as for geodesic triangles (see Definition \ref{def:rips hyp}).
\begin{theorem}[Conditions for Morse hyperbolicity] \label{thm:strongly hyperbolic equiv}
Let $(X,d)$ be a proper metric space. Then the following are equivalent:
\begin{enumerate}
\item $X$ is quasi-isometric to a proper geodesic hyperbolic metric space. \label{equiv: SH}
\item There exists a $K\geq 1, C \geq 0$ so that $X$ is $(K,C)$-quasi-geodesic and for all $K'\geq K, C' \geq C$ there exists a Morse gauge $N_{(K',C')}$, depending on $K',C'$, so that every $(K', C')$-quasi-geodesic is $N_{(K',C')}$-Morse. \label{equiv: Morse}
\item\label{equiv: thin triangles} There exists a $K\geq 1, C \geq 0$ so that $X$ is $(K,C)$-quasi-geodesic and for all $K'\geq K, C' \geq C$ there exists $\delta_{(K',C')} \geq 0$, depending on $K',C'$, so that all $(K', C')$-quasi-geodesic triangles are $\delta_{(K',C')}$-slim.
\item $X$ is large scale geodesic and for all $K \geq 1$ and $C \geq 0$ there exists a $\delta \geq 0$ so that every $(K,C)$-quasi-geodesic triangle is $\delta$-slim. \label{equiv: thin triangles 2}
\end{enumerate}
\end{theorem}
\begin{definition}
A proper metric space $(X,d)$ is \emph{Morse hyperbolic}\index{Morse hyperbolic}\index{Morse hyperbolic!space}\index{hyperbolic!Morse} if it satisfies one (hence any) of the four equivalent properties in Theorem \ref{thm:strongly hyperbolic equiv}.
\end{definition}
We record the following consequence for later reference:
\begin{corollary}[Quasi-convex subsets]\label{QCMorseHyp}
Every quasi-convex subset of a geodesic Morse hyperbolic space is Morse hyperbolic.
\end{corollary}
\begin{proof} It is convenient to work with Definition (3) of a Morse hyperbolic space. Thus let $X$ be geodesic and Morse hyperbolic and let $Y \subset X$ be $C$-quasi-convex. It then follows from Remark \ref{QuasiconvexQuasigeodesic} that $Y$ is quasi-geodesic. In fact, by a similar argument as in the remark one can show that every quasi-geodesic in $Y$ is uniformly close to a quasi-geodesic in $X$, and hence that every quasi-geodesic triangle in $Y$ is at uniformly bounded Hausdorff distance from a quasi-geodesic triangle in $X$. The corollary follows.
\end{proof}

The remainder of this section is devoted to the proof of Theorem \ref{thm:strongly hyperbolic equiv}.

Firstly, recall from Lemma \ref{lsgquasigeodesic} that being a quasi-geodesic metric space is equivalent to being a large scale geodesic metric space. Since for $K' \geq K$ and $C' \geq C$ every $(K,C)$-quasi-geodesic is also a $(K', C')$-quasi-geodesic, we deduce that (3)$\iff$(4). We prefer to work with Condition (3) in the sequel, since it is closer in spirit to Conditions (1) and (2). We are going to show that (\ref{equiv: SH}) $\implies$ (\ref{equiv: Morse}) $\implies$  (\ref{equiv: thin triangles}) $\implies$  (\ref{equiv: SH}). The implication (1)$\implies$(2) follows from the following lemma:
\begin{lemma}[Stability of quasi-geodesics] \label{lem:strong hyp have stable quasi-geodesics}
Assume that $X$ is quasi-isometric to a proper geodesic hyperbolic metric space. Then there exist $K\geq 1, C \geq 0$ such that $X$ is $(K,C)$-quasi-geodesic and for all $K'\geq K, C' \geq C$ there exists a Morse gauge $N_{(K',C')}$, depending on $K',C'$, so that every $(K', C')$-quasi-geodesic is $N_{(K',C')}$-Morse.
\end{lemma}

\begin{proof} Since $X$ is quasi-isometric to a geodesic metric space, it is large-scale geodesic by Lemma \ref{CharLSG}, hence by Lemma \ref{lsgquasigeodesic} there exist $K\geq 1, C \geq 0$ such that $X$ is $(K,C)$-quasi-geodesic.

Now we fix a $(K_1,C_1)$-quasi-isometry $f\colon X \to Y$ into a proper geodesic $\delta$-hyperbolic metric space. Let $x,y \in X$ and let $\gamma$ be a $(K',C')$-quasi-geodesic joining $x$ and $y$, where $K \leq K', C \leq C'$. We know that $f(\gamma)$ is a $(K'', C'')$-quasi-geodesic with endpoints $f(x)$ and $f(y)$, where $K'', C''$ depend on $K_1,K',C_1,C'$. Since $Y$ is geodesic and $\delta$-hyperbolic, by Lemma \ref{Morse1} we know that there exists a geodesic segment $[f(x), f(y)]$, which is moreover $N$-Morse, for some Morse gauge $N$ which depends only on $\delta$. So we know that $f(\gamma)$ is in the $N'(K'',C'')$-neighborhood of $[f(x), f(y)]$. Indeed, since $Y$ is geodesic, $f(\gamma)$ and $[f(x), f(y)]$ have Hausdorff distance bounded by $M=2N'(K'', 2(K''+C''))+K''+C''$ \cite[Lemma 2.1]{cordes:2016ad}.

Let $\beta$ be a $(k,c)$-quasi-geodesic with endpoints on $\gamma$. We know that $f(\beta)$ is a $(k', c')$-quasi-geodesic with endpoints on $f(\gamma)$, where $k', c'$ depend on $K_1,k, C_1, c$. We also note that $f(\beta)$ is a $(k', c'+2N(K'', C''))$-quasi-geodesic with endpoints on $[f(x), f(y)]$, so it is in the $M'=N'(k', c'+2N(K'', C''))$-neighborhood of $[f(x), f(y)]$. Thus we can conclude that $f(\beta)$ is in the $M+M'$ neighborhood of $f(\gamma)$. It follows that $\beta$ lies in the $K_1(M+M'+C_1)$-neighborhood of $\gamma$. Thus $\gamma$ is $N_{(K',C')}$-Morse, where $N_{(K',C')}=K_1(M+M'+C_1)$.
\end{proof}

The implication (2)$\implies$(3) is provided by the following lemma:
\begin{lemma}[Slim quasi-triangles] \label{lem: SH implies thin triangles}
Assume $X$ is $(K,C)$-quasi-geodesic and there exists a Morse gauge $N_{(K,C)}$, depending on $K,C$, so that every $(K, C)$-quasi-geodesic is $N_{(K,C)}$-Morse. Then there exists $\delta_{(K,C)} \geq 0$, depending on $K,C$, so that all $(K, C)$-quasi-geodesic triangles are $\delta_{(K,C)}$-slim as in Definition \ref{def:rips hyp}.
\end{lemma}

\begin{proof}
We follow closely the proof of Lemma 6.2 in \cite{Masur-MinskyI}. By abuse of notation for any two points $x,y \in X$ we will use $[x,y]$ to denote a $(K,C)$-quasi-geodesic joining $x$ and $y$. Let $x,y,z \in X$. We will show that $[x,y]$ lies in a $\delta$-neighborhood of $[x,z]\cup[y,z]$, for $\delta$ depending on $K,C, N$. 

Let $z'$ be a point on $[x,y]$ closest to $z$. We will show that the concatenation $[x,z'] \cup [z',z]$ is a $(3K, 2C)$-quasi-geodesic. Let $u \in [x,z']$ and $v \in [z',z]$.  We note that $d(z',v)\leq d(u,v)$, because $z'$ is a closest point to $v$. Let the difference in parameters along $[x,z']$ between $u$ and $z'$ be $|b-a|$ and the difference in parameters along $[z', z]$ between $z'$ and $v$ be $|c-d|$.

Putting everything together, we get the following two inequalities:

\begin{equation*}
\begin{split}
d(u,v) & \leq d(u,z')+d(z',v)\\
& \leq K|b-a|+C+K|d-c|+C\\
& = K(|b-a|+|d-c|) +2C
\end{split}
\end{equation*}
and
\begin{equation*}
\begin{split}
|b-a|+|d-c| & \leq Kd(u,z')+KC+Kd(z',v)+KC\\
& \leq K(d(u,v)+d(z',v))+KC+Kd(z',v)+KC\\
& \leq K(2d(u,v))+KC+Kd(u,v)+KC\\
& \leq 3Kd(u,v)+2KC.
\end{split}
\end{equation*}

Now since all $(K,C)$-quasi-geodesics are $N_{(K,C)}$-Morse, we have that $[x,z']\cup[z',z]$ is in the $\delta_{(K,C)}=N_{(K,C)}(3K,2C)$-neighborhood of $[x,z]$ and thus, in particular, $[x,z']$ is. We can repeat this argument substituting $y$ for $x$ and we see that all of $[x,y]$ is in the $\delta_{(K,C)}$-neighborhood of $[x,z] \cup [y,z]$.
\end{proof}

Finally, the following lemma establishes the implication (3)$\implies$(1) and finishes the proof:
\begin{lemma} \label{lem: thin triangles implies SH}
Let $X$ be a proper $(K,C)$-quasi-geodesic metric space with the property that for all $K'\geq K, C' \geq C$ there exists $\delta_{(K',C')} \geq 0$, depending on $K',C'$, so that all $(K', C')$-quasi-geodesic triangles are $\delta_{(K',C')}$-slim as in Definition \ref{def:rips hyp}. Then $X$ is quasi-isometric to a hyperbolic proper geodesic metric space.
\end{lemma}

\begin{proof}
Since $X$ is a $(K,C)$-quasi-geodesic metric space, it is large scale geodesic (Lemma \ref{lsgquasigeodesic}) and thus there exists a proper geodesic metric space $Y$ and a $(K',C')$-quasi-isometry $f \colon Y \to X$ (Lemma \ref{CharLSG}). 

Let $x,y,z \in Y$ and let $T$ be a geodesic triangle in $Y$ with vertices $x,y,z$. We note that $T'=f(T)$, is a $(K',C')$-quasi-geodesic triangle. By assumption on $X$ there exists a $\delta_{(K',C')}>0$ so that $T'$ is $\delta_{(K',C')}$-slim. It follows from the quasi-isometric inequality that $T$ is $(K\delta_{(K',C')}+C)$-slim, hence $Y$ is Rips (equivalently, Gromov) hyperbolic.
\end{proof}

\section{Boundaries of Gromov and Morse hyperbolic spaces}\label{SecGromovBoundaries}

From now on let $(X,d)$ be a Gromov hyperbolic space (not necessarily geodesic) and let $o \in X$ be a basepoint. A sequence of points $\set{x_i}\subset X$ \emph{converges to infinity}\index{convergence to $\infty$} if 
\[
\lim_{i,j \to \infty} \gp{x_i}{x_j}{o} = \infty.
\]
We say that two sequences $(x_i)$ and $(y_i)$ that converge to infinity are equivalent if 
\[
\lim_{i \to \infty} \gp{x_i}{y_i}{o} = \infty.
\] 
One can check that neither of these notions depend on the choice of basepoint.

\begin{definition} If $(X,d)$ is a Gromov hyperbolic metric space, then its \emph{sequential Gromov boundary}\index{Gromov boundary!sequential}\index{sequential Gromov boundary} $\partial_s X$ is defined as the set of equivalence classes of sequences converging to infinity.
\end{definition}

To topologize the set $\overline{X} \coloneqq  X \cup \partial_s X$, one observes first that the Gromov product can be extended to $\overline{X}$ as follows: If $\xi, \eta \in \partial_s X$, then we set
\[ 
\gp{\xi}{\eta}{o} \coloneqq  \inf \ \liminf_{i,j\to\infty}\set{\gp{x_i}{y_j}{o}},
\]
where the infimum is taken over all sequences $(x_i)$ and $(y_j)$ in $X$ representing $\xi$ and $\eta$ respectively. If $\xi \in \partial_s X$ and $y\in X$, then 
\[
\gp{y}{\xi}{o} = \gp{\xi}{y}{o} = \inf\ \liminf_{i \to \infty} \gp{x_i}{y}{o},
\]
where the infimum is again taken over all sequences $(x_i)$ representing $\xi$. Then the \emph{Gromov topology}\index{Gromov topology} on $\overline{X}$ is the unique topology which restricts to the metric topology on $X$ and such that a neighborhood basis of $\xi \in \partial_s X$ is given by the sets
\[
\mathcal{U}_R(\xi) = \setcon{x\in \overline{X}}{\gp{\xi}{x}{o} > R},
\]
as $R$ ranges through $(0, \infty)$. It turns out that this topology is independent of the choice of basepoint $o$. From now on we will always consider $\overline{X}$ as a topological space with respect to the Gromov topology.
\begin{lemma}[Gromov compactification of a proper hyperbolic space, {\cite[Chapter 7]{GhysdelaHarpe}}] The subset $X \subset \overline{X}$ is open and hence $\partial_s X$ is closed in $\overline{X}$. If $(X,d)$ is proper, then $\overline{X}$ and $\partial_s X$ are compact.
\end{lemma}
Thus if $X$ is a proper hyperbolic space (not necessarily geodesic), then $\overline{X}$ is a compactification of $X$, i.e.\ a compact space which contains $X$ as a dense open subset.

Now let $X$ be a \emph{proper geodesic} hyperbolic space. In this case, the Gromov compactification $\overline{X}$ admits the following more explicit description (see \cite[Section III.H.3]{bridson-haefliger}: Recall that a (quasi-)isometric embedding $c: [0, \infty) \to X$ is called a \emph{(quasi-)geodesic ray} in $X$ and that two quasi-geodesic rays $c, c'$ in $X$ are \emph{asymptotic}\index{asymptotic quasi-geodesic rays} provided
\[
\sup\{ d(c(t), c'(t))\mid t \geq 0\} < \infty.
\]
This property is actually equivalent to finite Hausdorff distance of their images, hence independent of the chosen parametrization. Denote by $\partial_r X$ the collection of all asymptoticity classes of quasi-geodesic rays in $X$. Given a quasi-geodesic ray $c$, we denote by $c(\infty)$ the associated class in $\partial_r X$. 

By {\cite[Lemma III.H.3.1]{bridson-haefliger}}, every asymptoticity class admits a representative $c_o$ which is a geodesic ray emanating from $o$. Thus,
\[
\partial_r X = \{c(\infty) \mid c: [0, \infty) \to X \text{ geodesic ray}, c(0) = o\}.
\] 
This is sometimes taken as the definition of $\partial_rX$, but we prefer our original definition, which has better functoriality properties. 
\begin{lemma}[{\cite[Lemma III.H.3.13]{bridson-haefliger}}]\label{GromovBoundarieModels} If $c$ is a quasi-geodesic ray in a proper geodesic hyperbolic space $X$, then the sequence $(c(n))_{n \in \mathbb N}$ converges to infinity and there is a well-defined bijection $\partial_r X \to \partial_s X$ given by $c(\infty) \mapsto [(c(n))_{n \in \mathbb N}]$.\qed
\end{lemma}
In view of this lemma we refer to $\partial_r X$ as the \emph{ray model}\index{Gromov boundary!ray model of} of the Gromov boundary of $X$. We will often identify $\partial_r X$ and $\partial_s X$ via the bijection above, and denote either of these spaces simply by $\partial X$. In particular, we will topologize $\partial_r X$ via this identification. Similarly, we will identify $\overline{X}^{(r)}\coloneqq  X \cup \partial_r X$ and $\overline{X}$, and use this identification to topologize $\overline{X}^{(r)}$. The resulting topology admits the following more explicit description:

\begin{lemma}[{\cite[Lemma III.H.3.6]{bridson-haefliger}}]\label{GromovBoundaryNeighborhoods} Assume that $X$ is proper, geodesic and $\delta$-Rips hyperbolic and let $k>2\delta$. Let $c_0:[0, \infty) \to X$ be a geodesic ray with $c_0(0) = o$ and for every $n \in \mathbb N$ let
\begin{equation}\label{GromovBoundaryNeighborhoods1}
V_n(c_0) \coloneqq  \{c(\infty) \mid  c: [0, \infty) \to X \text{ geodesic ray}, c(0) = o, d(c(n), c_0(n))<k\}.
\end{equation}
Then $\{V_n(c_0) \mid n \in \mathbb N\}$ is a fundamental system of (not necessarily open) neighborhoods of $c(\infty)$ in $\overline{X}^{(r)}$.
\end{lemma}
Now let $X$, $X'$ be proper geodesic hyperbolic metric spaces and let $f: X \to X'$ be a quasi-isometry. If $c$ is a quasi-geodesic ray in $X$, then $f \circ c$ is a quasi-geodesic ray in $X'$, and asymptoticity is preserved. Thus $f$ induces a map
\[
\partial_r f: \partial_r X \to \partial_r X', \quad c(\infty) \mapsto (f\circ c)(\infty), 
\]
and hence a map $\overline{f}^{(r)}: \overline{X}^{(r)} \to {\overline{X'}}^{(r)}$. It follows from Lemma \ref{GromovBoundaryNeighborhoods} that $\partial_rf$ is continuous (see {\cite[Thm.\ III.H.3.9]{bridson-haefliger}}), even if $f$ itself is not continuous (and hence $\overline{f}^{(r)}$ is not continuous). Using the identification from Lemma \ref{GromovBoundarieModels} we also obtain maps $\overline{f}: \overline{X} \to \overline{X'}$ and $\partial_s f: \partial_s X \to \partial_s X'$, the latter of which is continuous.
\begin{corollary}[Boundary extensions of quasi-isometries] \label{cor: boundary qi-inv}
If $X$, $X'$ are proper geodesic hyperbolic metric spaces and $f: X \to X'$ is a quasi-isometry, then $\partial_r f: \partial_r X \to \partial_r X'$ is a homeomorphism. In particular, if $X$ and $X'$ are quasi-isometric, then $\partial_r X$ and $\partial_r X'$ (and hence $\partial_s X$ and $\partial_s X'$) are homeomorphic.\qed
\end{corollary}
\begin{remark} While the extension $\overline{f}$ of a quasi-isometric embedding $f: X \to X'$ does not need to be continuous on $\overline{X}$, it is still continuous at infinity in the following sense: If $x_n$ converges to infinity in $X$, then $(f(x_n))$ converges to infinity in $X'$, and $\lim f(x_n) = \partial f(\lim x_n)$. In particular, if $X$ is discrete, then $\overline{f}$ is continuous.
\end{remark}
We now use quasi-isometry invariance of Gromov boundaries of proper geodesic Gromov hyperbolic spaces to define Gromov boundaries for Morse hyperbolic spaces. Thus for the remainder of this section, $X$ denotes a Morse hyperbolic space. We denote by $[X]$ the quasi-isometry class of $X$ and by $[X]_{\mathrm{geod}} \subset [X]$ the subclass of \emph{proper geodesic} representatives of $[X]$. By Theorem \ref{thm:strongly hyperbolic equiv} this subclass is non-empty and by Remark \ref{MorseGromov} every $X_1 \in [X]_{\mathrm{geod}}$ is a proper geodesic Gromov- (hence Rips-)hyperbolic metric space. 

\begin{definition} \label{def: Morse hyperbolic boundary}
If $X_1 \in [X]_{\mathrm{geod}}$ then the Gromov boundary $\partial X_1$ of $X_1$ is called a \emph{Gromov boundary} of the Morse hyperbolic space $X$.\index{Gromov boundary!of Morse hyperbolic space}
\end{definition}

\begin{remark}[Uniqueness properties of Gromov boundaries] By definition, a Morse hyperbolic space has many different Gromov boundaries. However, if $X_1, X_2 \in [X]_{\mathrm{geod}}$, then there exists a quasi-isometry $f: X_1 \to X_2$ and by Corollary \ref{cor: boundary qi-inv} this map extends to a boundary map $\partial f: \partial X_1 \to \partial X_2$. Any property preserved by such boundary maps is thus an invariant of $X$.
\end{remark}
\begin{example}[Topological boundary invariants]\label{BoundaryInvariantMorse}
By Corollary \ref{cor: boundary qi-inv}, all Gromov boundaries of $X$ have the same (compact, metrizable) homeomorphism type, which we denote by $\partial [X]$. If $\mathcal I$ is any homeomorphism-invariant of compact metrizable spaces, then we may define
\[
\mathcal I(\partial[X]) \coloneqq  \mathcal I(\partial X_1), \quad \text{ where }X_1 \in [X].
\]
We then refer to $\mathcal I(X)$ as a \emph{topological boundary invariant}\index{topological boundary invariant} of $X$. A typical example of such an invariant is topological dimension (cf.\ Definition \ref{Def-dim}): we may thus define the \emph{topological dimension of the boundary}, $\dim(\partial[X])$, for every Morse hyperbolic space $X$. We will sometimes abuse notation and simply denote this invariant by $\dim \partial X$.
\end{example}

\section{Visual metrics}\label{sec:VisMet}
In this section, $(X,d)$ denotes a proper geodesic hyperbolic space. Given a basepoint $o \in X$ and a constant  $a>1$, we define a map
\[
\rho_{o,a}: \partial_s X \times \partial_s X \to [0, \infty), \quad \rho_{o,a}(\xi, \xi') \coloneqq   a^{-\gp{\xi}{\xi'}{o}}.
\]
This map is then a \emph{$K$-quasi-metric}\index{quasi-metric} for some $K \geq 1$ in the sense that it is symmetric, with $\rho_{o,a}(\xi, \xi') = 0$ iff $\xi = \xi'$, and for all $\xi, \xi', \xi'' \in \partial_s X$ we have
\[
\rho_{o,a}(\xi, \xi'') \leq K \max\{\rho_{o,a}(\xi, \xi'), \; \rho_{o,a}(\xi', \xi'')\}.
\] 
The dependence on the parameters is as follows \cite[Rem.\ 2.2.4]{BuyaloSchroeder}: For all $a, a'>1$ and $o, o' \in X$ we have
\begin{equation}\label{ParametersVisualMetrics1}
a^{-d(o,o')} \leq \frac{\rho_{o,a}(\xi, \xi')}{\rho_{o',a}(\xi, \xi')} \leq a^{d(o,o')} \qand  \rho_{o,a'}(\xi, \xi') = \rho_{o,a}(\xi, \xi')^{\frac{\log a'}{\log a}}.
\end{equation}

\begin{definition}\label{DefVisual} A metric $d$ on $\partial_s X$ is called \emph{visual}\index{metric!visual}\index{visual metric} with basepoint $o \in X$ and parameters $a>1$ and $c\geq 1$ if it is bi-Lipschitz equivalent to $\rho_{o,a}$, i.e., for all $\xi, \xi'\in \partial_s X$ we have
\begin{equation}\label{VisualInequalities}
 c^{-1}  a^{-\gp{\xi}{\xi'}{o}} = c^{-1}\rho_{o,a}(\xi, \xi') \leq d(\xi, \xi') \leq c\rho_{o,a}(\xi, \xi')=c a^{-\gp{\xi}{\xi'}{o}}.
\end{equation}
These inequalities are then called the \emph{visual inequalities}\index{visual inequalities} for $d$.
\end{definition}
In the geodesic case, visual metrics on $\partial_r X$ are defined analogously, using the identification $\partial_rX \cong \partial_sX$. By \cite[Thm. 2.2.7]{BuyaloSchroeder}, every hyperbolic space (geodesic or not) admits a visual metric on its Gromov boundary. By \eqref{ParametersVisualMetrics1} any two visual metrics with the same parameter $a$ are bi-Lipschitz. More generally, if $d$ and $d'$ are visual metrics with respective parameters $a$ and $a'$, then $d'$ is bi-Lipschitz to $d^\alpha$, where $\alpha \coloneqq  \frac{\log a'}{\log a}$, i.e. there exists $c\geq 1$ such that
\begin{equation}\label{VisualQuasiconformal}
c^{-1} d(\xi, \xi')^\alpha \leq d'(\xi, \xi') \leq c \, d(\xi, \xi')^\alpha.
\end{equation}
One can show that any visual metric on $\partial_s X$ induces the given topology {\cite[Proposition III.H.3.21]{bridson-haefliger}}. We call $a > 1$ a \emph{visuality parameter}\index{visuality parameter} for $X$ if there exists a visual metric on $\partial_s X$ with parameters $(a, c)$ for some $c\geq 1$. 

\begin{lemma}\label{VisualityConnected} The set of visuality parameters is a non-empty connected set, hence either $(1, \infty)$ or of the form $(1, a_{\max})$ or of the form $(1, a_{\max}]$, for some $a_{\rm max}>1$.
\end{lemma}
\begin{proof} Non-emptiness follows from the existence of a visual metric. If $d$ is a visual metric with parameters $(a,c)$ and $\epsilon \in (0,1)$, then $d^\epsilon$ is a visual metric with parameters $(a^\epsilon, c^\epsilon)$.
\end{proof}

We recall from \cite[Cor. 4.4.2 and Prop. 4.3.2]{BuyaloSchroeder} that if $X$ is a geodesic hyperbolic metric space with basepoint $o\in X$ and $f: X \to X$ is a $(K_0, C_0)$-quasi-isometry, then there exist constants $K \geq 1$ and $C \geq 0$ depending only on $K_0$, $C_0$, $X$ and $o$ such that $f$ is a $(K, C)$-quasi-isometry and moreover
\begin{equation}\label{KCoQI}
K^{-1}(x \mid x')_o - C \quad \leq  \quad  (f(x)\mid f(x'))_{f(o)} \quad \leq \quad K(x \mid x')_{o} + C, 
\end{equation}
for all $x, x' \in X$. We then say that $f$ is a \emph{$(K,C, o)$-quasi-isometry}\index{quasi-isometry}.  By \cite[Lemma 2.2.2]{BuyaloSchroeder} this carries over to the boundary in the following form:
\begin{lemma}\label{QIGromovInequality}
If $X$ is a geodesic $\delta$-Gromov hyperbolic metric space, then for every $(K,C, o)$-quasi-isometry  $f: X \to X$ the boundary extension $\partial_r f: \partial_r X \to \partial_r X$ satisfies
\[
K^{-1} (\xi'\mid \xi'')_o - C \quad \leq \quad  (\partial_r f(\xi')\mid \partial_r f(\xi''))_{f(o)} \quad  \leq \quad K (\xi'\mid \xi'')_o + C + 2\delta,
\] 
for all $o \in X$ and $\xi', \xi'' \in \partial_r X$.\qed
\end{lemma}
Note that by taking exponentials, from Lemma \ref{QIGromovInequality} we can deduce the quantitative estimates concerning the behavior of $\partial_r f$ with respect to any visual metric, which involve the hyperbolicity constant of $X$, quasi-isometry constants of $f$ and the parameter of the visual metric in question.

\chapter{Morse boundaries of proper geodesic metric spaces}\label{AppendixMorse}

In this appendix we discuss a vast generalization of the notion of Gromov boundary known as Morse boundary, which can be defined for arbitrary proper geodesic spaces and is a quasi-isometry invariant for such spaces.

\section{Definition and basic properties}

Since the definition of the Morse boundary is based on the notion of a Morse geodesic, we briefly recall the relevant definitions from Definition \ref{defn: Morse geodesic}:

\begin{remark}[Morse geodesic rays and asymptoticity] \label{rem: C_N close}\label{RayAsympMorse} 
Let $X$ be a proper geodesic metric space and let $N: [1, \infty) \times [0, \infty) \rightarrow [0, \infty)$ be an arbitrary function. Recall that a quasi-geodesic ray $\gamma: I \to X$ is called $N$-\emph{Morse} if any $(K,C)$-quasi-geodesic segment $q$ with endpoints on $\gamma$ is contained in the $N(K,C)$-neighborhood $\mathcal N_{N(K,C)}(\gamma)$ of $\gamma$; we then also call $N$ a \emph{Morse gauge} for $\gamma$. 
We say that two Morse geodesic rays $c$ and $c'$ in $X$ are \emph{asymptotic}\index{asymptotic Morse geodesic rays} if 
\[
\sup\{ d(c(t), c'(t))\mid t \geq 0\} < \infty.
\]
This property is actually equivalent to finite Hausdorff distance of their images, hence independent of the chosen parametrization. This defines an equivalence relation on the space of Morse geodesic rays, and the equivalence class of such a ray $c$ is denoted by $c(\infty)$. The following observation (see \cite[Prop. 2.4]{cordes:2016ad}) will be crucial for us: There exists a constant $C_N$ (depending only on $N$) such that two $N$-Morse geodesic rays $c$ and $c'$ are equivalent if and only if $d(c(t), c'(t))\leq C_N$ for all $t$.
\end{remark}

\begin{definition} Let $X$ be a proper geodesic metric space with basepoint $o \in X$. The \emph{Morse boundary}\index{Morse!boundary} of $X$ based at $o$ is the set
\[
\partial_r^M X_o \coloneqq  \{c(\infty) \mid c: [0, \infty) \to X \text{ Morse geodesic ray}, c(0) = o\}
\] up to asymptotic equivalency.
\end{definition}
\begin{remark}

In general, it may happen that $\partial_r^M X_o = \emptyset$. For example, this happens if every geodesic in $X$ bounds a half-flat, as is the case in a symmetric space (or Bruhat-Tits building) of higher rank. If $X$ is hyperbolic, then $\partial_r^M X_o = \partial_rX$. In general, a point in $\partial_r^MX_o$ may have many different representatives $c$.
\end{remark}
We will describe here two different, but equivalent, approaches to define a quasi-isometry invariant topology on the Morse boundary; for more details see \cite{cordes:2016ad}.

\begin{remark}[Direct limit topology on the Morse boundary]
Let $X$ be a (not necessarily hyperbolic) proper geodesic metric space and let $o \in X$ be a basepoint. Consider all of the geodesic rays emanating from $o$ up to asymptoticity. For each Morse gauge $N$ we choose this subset
\begin{eqnarray*}
 \partial_r^N X_{o} &=& \{c(\infty) \in \partial_r X \mid c: [0, \infty) \to X \small{\text{ is an $N$-Morse geodesic ray with }}c(0)=o\}\\
 &\subset& \partial^M_r X_o
\end{eqnarray*}
and topologize $\partial_r^N X_o$ as in Lemma \ref{GromovBoundaryNeighborhoods}, by replacing $2\delta$ by the $C_N$ from Remark \ref{rem: C_N close}.
 We then define a partial order on the set $\mathcal M$ of all Morse gauges as follows: Given $N, N' \in \mathcal M$ we set $N \leq N'$ if and only if $N(\lambda,\epsilon) \leq N'(\lambda,\epsilon)$ for all $\lambda,\epsilon \in \bN$.  This partial order makes the set of Morse gauges a directed set; for any two Morse gauges $N$ and $N'$, the gauge $N''=\max_{\lambda, \epsilon}\{N(\lambda, \epsilon), N'(\lambda, \epsilon)\}$ has the property that $N \leq N''$ and $N' \leq N''$. One can show that if $N \leq N'$, then the inclusion $ \partial_r^N X_{o} \hookrightarrow  \partial_r^{N'} X_{o}$ is continuous. We can thus topologize  the Morse boundary $\partial_r^M X_o$ as the direct limit
 \begin{equation*} \partial_r^M X_{o}=\varinjlim_{N \in (\mathcal{M}, \leq)} \partial^N_r X_{o}. \end{equation*} 
\end{remark}

From now on we always consider the Morse boundary $\partial^M_r X_{o}$ as a topological space with respect to the above direct limit topology.  If $X$ is hyperbolic, then we simply recover the (ray model of the) Gromov boundary $\partial_r X$ by Lemma \ref{Morse1}. In general, however, the direct limit topology on the Morse boundary may be non-compact. We emphasize that $ \partial_r^N X_{o}$ is defined using $N$-Morse geodesic rays rather than $N$-Morse quasi-geodesic rays. It is possible to define the direct limit topology using Morse quasi-geodesic rays, but for this one has to carefully keep track of the quasi-isometry constants involved.

We now provide an alternative model for the Morse boundary, which is more in line with the sequential model of the Gromov boundary:
\begin{remark}[Sequential description of the Morse boundary] \label{rem:seq morse boundary} If $X$ is a proper geodesic metric space with basepoint $o\in X$, then an alternative description of the Morse boundary $ \partial_r^M X_{o}$ and its direct limit topology can be given as a direct limit of Gromov boundaries of certain hyperbolic subspaces of $X$. More precisely, given a Morse gauge $N$ we denote by $X^{(N)}_{o} \subset X$ the set of all $y\in X$ such that there exists a $N$--Morse geodesic segment from $o$ to $y$ in $X$, which we denote $[o,y]$. Note that there is no reason why this should be geodesic or even connected. However, by \cite[Proposition 3.2]{Cordes:2016aa} it is always Gromov hyperbolic with respect to the restricted metric, and hence we can consider its sequential Gromov boundary $\partial_s X_{o}^{(N)}$ which is a (possibly empty) compact topological space with respect to the topology induced by some (hence any) visual metric. For $N \leq N'$ the isometric inclusions $X_o^{(N)} \hookrightarrow X_o^{(N')}$ induce continuous inclusions $\partial_s X_{o}^{(N)} \hookrightarrow \partial_s X_{o}^{(N')}$, and hence we can define the \emph{sequential Morse boundary}\index{Morse!boundary!sequential} as
\[
\partial_s^M X_{o}\coloneqq  \varinjlim_{N \in (\mathcal{M}, \leq)}\partial_s X_{o}^{(N)}.
\]
We note that every (relatively) \emph{compact} subset of $\partial_s^M X_{o}$ is contained in one of the strata $\partial_s X_{o}^{(N)}$ (see \cite[Lemma 4.1]{cordes:2016ab}). One can show that there are natural homeomorphisms $\partial_s X^{(N)}_o \to\partial_r^N X_o$ which are compatible with inclusions, see \cite{Cordes:2016aa}. This implies in particular, that
\[
\partial_s^M X_{o} \cong \partial_r^M X_{o},
\]
hence we refer to $\partial_s^M X_{o}$ as the \emph{sequential model}\index{Morse!boundary!sequential model} of the Morse boundary. While this model is very much in line with the sequential approach to the Gromov boundary, we still need to assume that $X$ be geodesic in order to obtain a sensible definition of the spaces $X_{o}^{(N)}$.

Note that every element $\xi \in \partial_s^M X_o$ can be represented by a sequence $(x_n)$ of points which are contained in some fixed hyperbolic approximant $X_{o}^{(N)}$. We then write $\xi = \lim x_n$ and say that $\xi$ is an $N$-boundary point.
\end{remark}
\begin{notation}[Representation by geodesic rays and lines] In the sequel we will mostly work with the sequential model $\partial_s^M X_{o}$ of the Morse boundary. Nevertheless, we sometimes want to represent such points by geodesic rays: If $\gamma$ is a geodesic ray in $X$ (not necessarily emanating from $o$), then we say that $\gamma$ is \emph{asymptotic} to $\xi \in \partial_s^M X_o$ if it is asymptotic (in the sense of Remark \ref{RayAsympMorse}) to a geodesic ray $\gamma'$ emerging from $o$ such that $\gamma(\infty)$ corresponds to $\xi$ via the homeomorphism $\partial_s^M X_{o} \cong \partial_r^M X_{o}$. Note that any geodesic which is asymptotic to a point $\xi$ in the Morse boundary is automatically Morse. In fact, if $\xi$ is an $N$-boundary point, then $\gamma$ is $N'$-Morse for a gauge $N'$ depending only on $N$ and $d(\gamma(0), o)$.

If $\gamma: \R \to X$ is a bi-infinite geodesic line in $X$, then we can define two geodesic rays $\gamma^{\pm}: [0, \infty) \to X$ by $\gamma^+(t) \coloneqq  \gamma(t)$ and $\gamma^-(t) \coloneqq  \gamma(-t)$. We then say that $\gamma$ is \emph{bi-asymptotic}\index{bi-asymptotic} to $(\xi^-, \xi^+) \in (\partial_s^M X_o)^2$ if $\gamma^-$ is asymptotic to $\xi^-$ and $\gamma^+$ is asymptotic to $\xi^+$. As in the ray case one observes that if $\xi^\pm$ are both $N$-Morse and $\gamma$ is a geodesic line which is bi-asymptotic to $(\xi^-, \xi^+)$, then $\gamma$ is $N'$-Morse for a Morse gauge $N'$ depending only on $N$.
\end{notation}
It turns out that all geodesic lines which are bi-asymptotic to the same pair of points are at uniformly bounded Hausdorff distance (\cite{cordes:2016ab}):
\begin{proposition}[Bi-asymptoticity classes of geodesic lines]\label{prop:limit geodesics are asymptotic}
For any Morse gauge $N$, there exists a constant $K'>0$ such that if $\gamma, \gamma'$ are geodesic lines bi-asymptotic to $(\xi^-, \xi^+) \in \partial_s X^{(N)}_{o} \times \partial_s X^{(N)}_{o}$, then
\[\pushQED{\qed}d_{\mathrm{Haus}}(\gamma, \gamma')<K'.\qedhere\popQED
\]
\end{proposition}

Note that if $X$ and $X'$ are two proper geodesic metric spaces and $f: X \to X'$ is a quasi-isometric embedding which maps a given basepoint $o \in X$ to $o' \in X'$, then it induces a map
\[
\partial_r f: \partial_r X _o\to \partial_r X'_{o'}, \quad c(\infty) \mapsto (f \circ c)(\infty).
\]
The following example from \cite{cordes:2016ad} shows that this map does not in general map $\partial_r^M X _o$ to $\partial_r^M X' _{o'}$:
\begin{example}[Failure of functoriality of the Morse boundary] 
Let $X$ be the hyperbolic plane and let $\widetilde{c}: \R \to X$ be a bi-infinite geodesic. Then the restriction $c: [0, \infty) \to X$ of $\widetilde{c}$ is a Morse quasi-geodesic ray. Now let $X'$ be the space obtained from $X$ by gluing a Euclidean half-plane along $\widetilde{c}$. Then the inclusion $f: X \to X'$ is a (quasi-)isometric embedding, but $f \circ c$ is no longer Morse. 
\end{example}
Even if $f:X \to X'$ is a quasi-isometric embedding such that $\partial_r f(\partial_r^M X _o) \subset \partial_r^M X' _{o'}$, then the map $\partial_r f: \partial_r^M X _o \to  \partial_r^M X' _{o'}$ need not be continuous with respect to the direct limit topology. The following condition is immediate from the definitions (cf.\ \cite[Proposition 4.2]{cordes:2016ad}):
\begin{proposition}\label{prop:morse preserving}
Let $X,X'$ be proper geodesic metric spaces with basepoints $o, o'$ and let $f \colon X \rightarrow X'$ be a quasi-isometric embedding with $f(o) = o'$. If for every $N \in \mathcal M$ there exists $N' \in \mathcal M$ such that
\[
\partial_r X^{(N)}_o \subset \partial_r {X'}^{(N')}_{o'},
\] 
then $\partial_r f$ restricts to a map
\[
\partial_r^M f \colon \partial_r^M X _o \to  \partial_r^M X' _{o'},
\]
which is continuous and in fact a homeomorphism onto its image.\qed
\end{proposition}
A quasi-isometric embedding satisfying the condition of the proposition will be called \emph{Morse-preserving}\index{Morse!preserving}. Clearly every quasi-isometry is Morse-preserving, and hence one concludes that the Morse boundary with its direct limit topology is a quasi-isometry invariant among proper geodesic metric spaces:
\begin{corollary}[Quasi-isometry invariance]\label{MorseBoundaryQIInvariant}
Let $X,X'$ be proper geodesic metric spaces with basepoints $o,o'$ and $f \colon X \rightarrow X'$ be a quasi-isometry with $f(o) = o'$. Then $\partial_r^M f \colon \partial_r^M X_o \to \partial_r^M X'_{o'}$ is a homeomorphism.\qed
\end{corollary}
Also note that if $o$ and $o'$ are two basepoints in $X$, then there exists a self-quasi-isometry of $X$ (at bounded distance from the identity) which maps $o$ to $o'$. We may thus record:
\begin{corollary}[Basepoint independence]
Up to homeomorphism the Morse boundary $\partial_r^M X_o$ of a proper geodesic metric space $X$ with respect to a basepoint $o$ is independent of the basepoint $o$.\qed
\end{corollary}
In view of the corollary we will sometimes be sloppy about basepoints and simply write $\partial_r^M X$ instead of $\partial_r^M X_o$ if we only care about its homeomorphism type.
\begin{remark}[Morse boundary for large-scale geodesic spaces] If $X$ is a proper large-scale geodesic metric space, then $X$ is quasi-isometric to a proper geodesic metric space $X'$. If $X''$ is any other proper-geodesic metric space in the same QI type, then $X'$ and $X''$ have homeomorphic Morse boundaries by Corollary \ref{MorseBoundaryQIInvariant}. We may thus define the Morse boundary $\partial^MX$ as their common homeomorphism type.
\end{remark}
\begin{remark}[Boundary maps in the sequential model] A Morse-preserving quasi-isometric embedding $f: (X, o) \to (X', o')$ also induces a map $\partial_s f \colon \partial_s^M X _o \to  \partial_s^M X' _{o'}$ via the identifications of the ray model and the sequential model. Explicitly, this map can be described as follows: Every point $\xi \in \partial_s^M X_o$ can be represented by a sequence $(x_n)$ of points, which are contained in a fixed hyperbolic approximant $X^{(N)}_o$. One can show that if $f$ is Morse-preserving, then $(f(x_n))$ will be contained in some fixed hyperbolic approximant $(X')^{N'}_{o'}$ and $\partial_s f (\xi)$ is represented by this sequence.
\end{remark}

\section{Ideal Morse triangles}
Throughout this section, $X$ denotes a proper geodesic metric space and $o \in X$ is a basepoint.
\begin{definition} Let $\xi^-, \xi^+ \in \partial_s X^{(N)}_{o}$ with $\xi^- \neq \xi^+$. If $
\gamma^\pm$ are geodesic rays emanating from $o$ which are asymptotic to $\xi^\pm$ respectively and $\gamma$ is a geodesic line which is bi-asymptotic to $(\xi^-, \xi^+)$, then $(\gamma^-, \gamma, \gamma^+)$ is called an \emph{ideal $N$-Morse triangle}\index{Morse!ideal $N$-Morse triangle} with vertices $(o, \xi^-, \xi^+)$.
\end{definition}
In this section we are going to show that, even if $X$ is not hyperbolic, ideal $N$-Morse triangles always behave like ideal triangles in hyperbolic spaces:
\begin{theorem}[Ideal Morse triangles are slim]\label{IdealMorseTriangle} If $\xi^-, \xi^+ \in \partial_s X^{(N)}_{o}$ with $\xi^- \neq \xi^+$, then there exists an ideal $N$-Morse triangle $(\gamma^-, \gamma, \gamma^+)$ with vertices $(o, \xi^-, \xi^+)$. Moreover, the following hold:
\begin{enumerate}[(i)]
\item The Hausdorff distance between any two ideal $N$-Morse triangles with vertices $(o, \xi^-, \xi^+)$ is bounded by a constant depending only on $N$.
\item If $(\gamma^-, \gamma, \gamma^+)$ is such a triangle, then the geodesics $\gamma^-, \gamma, \gamma^+$ are $N'$-Morse for some $N'$ depending only on $N$.
\item There exists a constant $\delta$ depending only on $N$ such that every $N$-Morse triangle with vertices $(o, \xi^-, \xi^+)$ is $\delta$-slim.
\end{enumerate}
\end{theorem}
Note that (i) is immediate from Remark \ref{RayAsympMorse} and Proposition \ref{prop:limit geodesics are asymptotic}. It is thus enough to construct a single ideal $N$-Morse triangle with vertices $(o, \xi^-, \xi^+)$ and to show that this specific triangle satisfies (ii) and (iii). For this we will use the following construction:
\begin{construction}[Limit geodesics and triangles]\label{LimitGeodesics}
Let $(x_n), (y_n) \in X^{(N)}_o$ be sequences asymptotic to $\xi^-, \xi^+ \in \partial_s X^{(N)}_{o}$ respectively, with $\xi^- \neq \xi^+$.  Let $\gamma_{x,n} = [o, x_n]$, $\gamma_{y,n} = [o,y_n]$ be $N$-Morse geodesic segments.  Since $X$ is proper, the Arzel\`a--Ascoli Lemma \ref{ArzelaAscoli1} implies that there exist geodesic rays $\gamma_x, \gamma_y$ such that $\gamma_{x,n}, \gamma_{y,n}$ subsequentially converge uniformly on compact sets to $\gamma_x, \gamma_y$. Let $\gamma_n$ be a geodesic joining $\gamma_x(n)$ and $\gamma_y(n)$ for each $n \in \mathbb{N}$. By Proposition 3.11 in \cite{cordes:2016ad}, $(\gamma_n)$ has a Morse subsequential limit $\gamma$. It is then obvious that $(\gamma_x, \gamma, \gamma_y)$ is an ideal $N$-Morse triangle with vertices $(o, \xi^-, \xi^+)$ and satisfies (ii).
\end{construction}
\begin{definition} In the situation of Construction \ref{LimitGeodesics}, $\gamma_x$ and $\gamma_y$ are \emph{limit legs}\index{limit legs} based at $o$ for $\xi^+$ and $\xi^-$ respectively, and $\gamma$ is called a \emph{limit geodesic line}\index{limit geodesic line} from $\xi^-$ to $\xi^+$. The ideal $N$-Morse triangle  $(\gamma_x, \gamma, \gamma_y)$ is called a \emph{limit triangle}\index{limit triangle} based at $o$.
\end{definition}
Working with concrete limit triangles is often more convenient for practical computation than working with abstract ideal Morse triangles; for example, the following is established in \cite{cordes:2016ab}:
\begin{proposition}[Limit triangles are slim]\label{prop:limit triangles are thin}
For any Morse gauge $N$, if $(x_n), (y_n) \subset X^{(N)}_{o}$ are asymptotic to $\xi^- \neq \xi^+ \in \partial_s X^{(N)}_{o}$, then any limit triangle is $4N(3,0)$-slim.\qed
\end{proposition}
This shows in particular that limit triangles satisfy Property (iii) of Theorem \ref{IdealMorseTriangle} and thereby finishes the proof of the theorem. For later reference we record the following consequence from Property (i) of Theorem \ref{IdealMorseTriangle}:
\begin{corollary} \label{cor: limit geodesic bdd hausdorff}
If $\gamma$ is a geodesic ray which is asymptotic to some $\xi \in \partial_s^M X_o$, then it is of bounded Hausdorff distance to a limit leg for $\xi$. Similarly, if $\gamma$ is a geodesic line which is bi-asymptotic to some $(\xi^-, \xi^+) \in \partial_s^M X_o \times \partial_s^M X_o$, then it is at bounded Hausdorff distance from a limit geodesic. Moreover, the bounds depend only on the respective Morse gauges.\qed
\end{corollary}

\section{Limit sets and stable subspaces}
Throughout this section, $X$ denotes a proper geodesic metric space and $o \in X$ denotes a basepoint. We can associate with every subset $Y$ of $X$ a corresponding subset of the Morse boundary $\partial_s^M X_o$:
\begin{definition} \label{defn:limit set}
Let $X$ be a proper geodesic metric space and let $Y \subset X$ be a subset. Then the \emph{Morse limit set}\index{Morse!limit set}\index{limit set!Morse} of $Y$ in $\partial_s^M X$ is
\[\cL(Y)= \setcon{\xi \in \partial_s^M X}{\exists N \text{ and } (y_n)\subset X^{(N)}_{o} \cap Y \text{ such that } \lim{y_n}= \xi}.\]
\end{definition}
Note that if $X$ happens to be Gromov hyperbolic, then $\cL(Y) \subset \partial X$ is just the Gromov limit set as defined in Definition \ref{DefGromovLimitSet}. A number of basic properties of Gromov limit sets carry over to the current, more general setting. In particular, if two subspaces $Y, Y' \subset X$ are at bounded Hausdorff distance, then $\cL(Y) = \cL(Y')$.
\begin{example}[Some limit sets]
If $Y$ is bounded, then the limit set $\cL(Y)$ is empty; the converse is not true, since for example the Morse limit set of every flat of dimension $>2$ is empty. The limit set of a geodesic ray is either empty (if it is not Morse) or a singleton (if it is Morse); the limit set of a geodesic line $\gamma$ has at most $2$ points, and if $\gamma$ is bi-asymptotic to $(\xi^-,\xi^+)$ then its limit set is $\{\xi^-,\xi^+\}$.
\end{example}
To obtain more interesting examples, we are going to need the following definition, which is a variant of a definition of Durham and Taylor \cite{Durham}.
\begin{definition}\label{DefStable}
Let $X$ be a proper geodesic metric space, let $Y\subset X$ be a subset and let $N$ be a Morse gauge.
The subset $Y$ is called  $N$-\emph{stable}\index{stable subset} if it is quasi-convex and if every pair of points in $Y$ can be connected by a geodesic which is $N$-Morse in $X$.  
\end{definition}
\begin{remark}[Properties of stable subsets]\label{StableHyperbolic} Let $X$ be a proper geodesic metric space and $Y\subset X$.
\begin{enumerate}
\item A quasi-convex subset $Y \subset X$ is stable if and only if every pair of points in $Y$ can be connected by a \emph{quasi-}geodesic which is $N$-Morse for some Morse gauge $N$, since every such quasi-geodesic is then uniformly close to an actual geodesic, which is $N'$-Morse for a uniform $N'$.
\item By definition, if $Y \subset X$ is stable and $x,y \in Y$, then there exists a Morse geodesic connecting $x$ and $y$ inside $X$. Since $Y$ is quasi-convex, this geodesic is contained in $N_C(Y)$, and hence we can find a quasi-geodesic in $Y$ connecting $x$ and $y$ of distance at most $C$. Thus any two points in $Y$ can be connected inside $Y$ by a Morse quasi-geodesic with uniform QI constants and uniform Morse gauge.
\item From (2) one deduces that every stable subspace is large-scale geodesic and quasi-geodesic. 
\item Note that a quasi-geodesic proper metric space in which any two points can be joined by a uniform Morse quasi-geodesic is Morse hyperbolic by Theorem \ref{thm:strongly hyperbolic equiv}.  It thus follows from (2) and (3) that every stable subset of a proper geodesic metric space is necessarily Morse hyperbolic.
\item If $X$ is a proper geodesic metric space and $Y$ is a proper metric space, then a quasi-isometric embedding $f: Y \to X$ is called a \emph{stable embedding}\index{stable embedding} if its image is a stable subset in the sense of Definition \ref{DefStable}. This definition is equivalent to the original definition of Durham and Taylor (\cite{Durham}). Note that if $Y$ is a stable subset of $X$, then its inclusion is a quasi-isometric embedding (since $Y$ is quasi-convex), hence a stable embedding.
\item If $X$ is hyperbolic, then every quasi-convex subset of $X$ is stable by definition.
\end{enumerate}
\end{remark}
It is immediate from the definitions that if $X$, $Y$ are proper geodesic metric spaces, then every stable embedding $f: Y \to X$ is Morse-preserving. Let $o \in Y$ be a basepoint and set $o' \coloneqq  f(o)$.  Since $Y$ is hyperbolic by Remark \ref{StableHyperbolic}, we have $\partial_s^M Y = \partial_s Y$, and thus by Proposition \ref{prop:morse preserving}
we obtain an injection\[
\partial_s f \colon \partial_s Y \hookrightarrow \partial_s^M X, \quad [(x_n)] \to [(f(x_n))].
\]
which is a homeomorphism onto its image. Note that the image of $\partial_s f$ is compact, since $\partial_s Y$ is. This argument does not apply directly to the inclusion of an arbitrary stable subspace $\iota: Y \to X$, since $Y$ might not be geodesic.
Nevertheless, we note that $\iota(Y)$ is Morse hyperbolic and thus there exists a hyperbolic geodesic metric space $Y' \in [Y]_{\mathrm{geod}}$ and, thus a quasi-isometry $g \colon Y' \to Y$. So the map $g \circ \iota$ is a stable embedding and we get that $\partial_sY'=\partial Y$ (Definition \ref{def: Morse hyperbolic boundary}) topologically embeds in $\partial_s^M X$ and its image is $\cL(Y)$ \cite[Theorem 3.16, Proposition 3.18]{Cordes:2016aa}. 
This proves the following proposition:
\begin{proposition}\label{BoundaryStableSubset} Let $X$ be a proper geodesic metric space and let $Y \subset X$ be a stable subspace. Then there exists a proper geodesic hyperbolic metric space $Y'\in [Y]_{\mathrm{geod}}$ and a quasi-isometry $g \colon Y' \to Y$, and the following hold:
\begin{enumerate}[(i)]
\item The inclusion $\iota: Y \to X$ induces a map $\partial_s (g\circ \iota): \partial_s Y' \to \partial_s^M X$. 
\item The image of $\partial_s (g\circ \iota)$ is precisely the limit set $\mathcal L(Y)$.
\end{enumerate}
In particular, $Y$ is Morse-hyperbolic and the limit set $\mathcal L(Y)$ is a compact subset of the Morse boundary $\partial_s^M X$ and homeomorphic to the Gromov boundary $\partial Y$ of $Y$ in the sense of Definition \ref{def: Morse hyperbolic boundary}.\qed
\end{proposition}
For ease of reference we spell out the hyperbolic case:
\begin{corollary}\label{BoundaryStableSubsetHyp} If $Y$ is a quasi-convex subset of a proper geodesic Gromov hyperbolic space $X$, then $Y$ is Morse hyperbolic and its Gromov limit set $\mathcal L(Y) \subset \partial X$ is homeomorphic to the Gromov boundary $\partial Y$ of $Y$ in the sense of Definition \ref{def: Morse hyperbolic boundary}.\qed
\end{corollary}

\section{Weak hulls}
Throughout this section, $X$ denotes a proper geodesic metric space and $o \in X$ denotes a basepoint. We have seen in the previous section that limit sets of stable subsets are compact. In this section we are going to establish the converse:
\begin{proposition} \label{prop:compact hyp hull}
Let $X$ be a proper geodesic metric space and let $Z \subset \partial_s^M X_o$ be a subset. Then the following are equivalent:
\begin{enumerate}[(i)]
\item $Z$ is compact.
\item There exists a stable subset $Y \subset X$ such that $\mathcal L(Y) = Z$.
\end{enumerate} 
\end{proposition}
As mentioned above, the implication (ii)$\implies$(i) was already established in Proposition \ref{BoundaryStableSubset}. For the converse implication we need a way to construct a subset of $X$ from a given subset of $\partial_s^M X_o$. In the case of a CAT(-1) space we could simply take the convex hull; in the general case, some more care has to be taken.
\begin{definition}\label{def:weak hull}\index{weak hull}
Let $Z \subset \partial_s^M X_o$ and set $Z^{(2)} \coloneqq  \{(\xi, \eta) \in Z^2 \mid \xi \neq \eta\}$. The \emph{weak hull} $\mathfrak{H}(Z)$ of $Z$ is then defined as
\[
\mathfrak{H}(Z) \coloneqq  \bigcup_{(\xi, \eta)\in Z^{(2)}}\{\gamma(\R) \mid \gamma: \R \to X \text{ is bi-asymptotic to }(\xi, \eta)\} \subset X.
\]
\end{definition}
By definition, if $|Z|<2$ then the weak hull is empty. 
\begin{remark}[Variants of the weak hull]\label{WeakHullGeodesics}
If $X$ is a CAT(-1) space and $Z \subset \partial X$, then for all $(\xi, \eta) \in Z^{(2)}$ there is a unique bi-infinite geodesic line $\gamma_{\xi, \eta}$ which is bi-asymptotic to $(\xi, \eta)$. In the general case there may be more than one geodesic line which is bi-asymptotic to a given $(\xi, \eta) \in Z^{(2)}$. However, by Theorem \ref{IdealMorseTriangle} for each pair $(\xi, \eta) \in Z^{(2)}$ there exists always at least one such geodesic, and we can pick a preferred geodesic $\gamma_{\xi, \eta}$ (for example, a limit geodesic) for each pair. We can then define
\begin{equation}\label{H0Z}
\mathfrak{H}_0(Z) \coloneqq  \bigcup_{(\xi, \eta)\in Z^{(2)}} \gamma_{\xi, \eta}(\R).
\end{equation}
If $X$ is CAT(-1), then we simply have $\mathfrak{H}(Z) = \mathfrak{H}_0(Z)$, but in general the inclusion $\mathfrak{H}_0(Z) \subset \mathfrak{H}(Z)$ is strict. However, if $Z \subset \partial_s X^{(N)}_{o}$ for a fixed Morse gauge $N$, then it follows from Theorem \ref{IdealMorseTriangle} that $\mathfrak{H}_0(Z)$ and $\mathfrak{H}(Z)$ are at bounded Hausdorff distance. In particular this holds if $X$ is Gromov hyperbolic or  $Z$ is relatively compact (see Remark \ref{rem:seq morse boundary}).
\end{remark}
\begin{example}[Ideal polygons] If $Z$ is a finite subset in the (Morse) boundary of the hyperbolic plane $\bH^2$, then $\mathfrak{H}(Z) = \mathfrak{H}_0(Z) \subset \bH^2$ consists of the boundary of the ideal polygon with vertex set $Z$ together with the ``diagonal'' geodesic lines which connect the vertices. If $|Z| > 2$, then the weak hull is not convex; however, since ideal polygons in hyperbolic space are slim, it is of bounded Hausdorff distance from its convex hull, and thus at least quasi-convex.
\end{example}
In general we cannot expect the weak hull of an arbitrary subset of the Morse boundary to be quasi-convex. However, the situation is better for compact subsets (\cite[Prop.\ 4.2]{cordes:2016ab}):
\begin{proposition}\label{WeakHullQuasiconvex} Let $Z\subset \partial_s^M X_o$ be compact. Then $\mathfrak{H}(Z)$ is a stable subspace of $X$ and in particular Morse hyperbolic and quasi-convex.\qed
\end{proposition}
In the case of real hyperbolic space we can do even better. We can find a closed convex set at bounded Hausdorff distance from the weak hull (see Remark \ref{rem:hulls in hyp spaces}):
\begin{proposition}[Weak hull vs.\ convex hull in hyperbolic spaces]\label{WeakHullConvexHull}
Let $L \subset \partial \bH^n$ and set
\[
\cH(L) \coloneqq  \overline{\mathrm{conv}}(\mathfrak{H}(L)),
\] where $\overline{\mathrm{conv}}(\mathfrak{H}(L))$ is the smallest closed convex subset of $\mathbb{H}^n$ which contains all geodesics between points in $L$.
Then $\mathfrak{H}(L)$ is at bounded Hausdorff distance from $\cH(L)$.
\end{proposition}
\begin{proof} It suffices to construct a closed convex set $\cH$ containing $X \coloneqq  \mathfrak{H}(L)$ which is at bounded Hausdorff distance from $X$. For this, using quasi-convexity of $X$, we fix a constant $K$ such that every geodesic segment containing two points of $X$ is contained in a $K$-neighborhood of $X$.

Let $o\in \mathfrak{H}(L)$. Given a geodesic ray $\gamma\colon [0,\infty)\to \bH^n$ emanating from $o$ we define
\[
t_\gamma \coloneqq  \sup \{t \geq 0 \mid \gamma(t) \in X\} \in [0, \infty].
\]
We then denote by $\mathcal G_l$ the set of all such geodesic rays $\gamma\colon [0,\infty)\to \bH^n$ which leave $X$ in the sense that $t_\gamma < \infty$. If $\gamma \in \cG_l$ then $p_\gamma \coloneqq  \gamma(t_\gamma + 1) \in N_2(X) \setminus\{o\}$. There thus exists a unique hyperplane $H(\gamma, p_\gamma)\subset \bH^n$ through $p_\gamma$ which is perpendicular to $\gamma$ at $p_\gamma$.
This hyperplane separates $\bH^n$ into two half-spaces, and we denote by $H^+(\gamma, p_\gamma)$ the closure of the half-space containing $o$. We now claim that $\mathfrak H(L)$ is at bounded Hausdorff distance from the closed convex set
\[
\cH \coloneqq \bigcap_{\gamma \in \mathcal G_l} H^+(\gamma, p_\gamma).
\]
It is immediate that $X \subset \cH$, for if $\gamma$ is a geodesic ray from $o$ through some $x \in X$ then the segment of $\gamma$ from $o$ to $x$ is contained in $\cH$. Conversely let $x \in \subset \cH$ and let $\gamma$ be a geodesic ray emanating from $o$ through $x$. If $\gamma$ does not leave $X$ or $x = \gamma(t)$ with $t< t_\gamma$, then $x$ lies on a geodesic segment between $o$ and some other point in $X$, and hence $x \in N_K(X)$. On the other hand, if $\gamma \in \cG_l$ and $x = \gamma(t)$ for some $t \in [t_\gamma, t_\gamma+1]$, then $d(x, \gamma(t_\gamma)) \leq 1$ and thus $x \in N_2(X)$. We thus have $\cH \subset N_L(X)$, where $L \coloneqq  \max\{2,K\}$.
\end{proof}
As another application of Proposition \ref{WeakHullQuasiconvex} we can now finish the proof of Proposition \ref{prop:compact hyp hull}. It only remains to show the following:
\begin{proposition}\label{prop: limit set of hull is limit set again}
Let $Z \subset \partial_s^M X_o$ be compact. Then $\cL(\mathfrak{H}(Z))=Z$.
\end{proposition}
\begin{proof}
We first show that $Z \subset \cL(\mathfrak{H}(Z))$. Let $\eta, \xi \in Z$. By definition of the Morse limit set we have $\eta, \xi \in X^{(N)}_o$ for some Morse gauge $N$. Let $\gamma_\eta, \gamma_\xi$ be $N$-Morse geodesic rays based at $o$ representing $\eta$ and $\xi$ respectively. Let $\alpha$ be a bi-infinite geodesic line joining $\eta$ and $\xi$. Since ideal Morse triangles are $\delta$-slim (Theorem \ref{IdealMorseTriangle}), we know that, for some constant $\delta$ depending on $N$, eventually $\gamma_\eta$, respectively $\gamma_\xi$ will be $\delta$-close to $\alpha$. It follows that $\xi, \eta \in \cL(\mathfrak{H}(Z))$.

We now show that $\cL(\mathfrak{H}(Z))\subset Z$. Let $[(x_n)] \in \cL(\mathfrak{H}(Z))$. By definition we know that $(x_n) \subset \mathfrak{H}(Z)$, so  $x_n \in \alpha_n$, where $\alpha_n$ is a bi-infinite geodesic line. By Proposition \ref{WeakHullQuasiconvex}, since $Z$ is compact, we know that the $\alpha_n$ are all $N$-Morse for some Morse gauge $N$. Let $\gamma_n$ and $\gamma_n'$ be two geodesic rays based at $o$ representing the endpoints of $\alpha_n$. Since triangles are $\delta$-thin (Theorem \ref{IdealMorseTriangle}), we know that $x_n$ is either $\delta$ away from some point on $\gamma_n$ or some point on $\gamma_n'$. Up to renaming  $\gamma_n$ and $\gamma_n'$ we may assume that $x_n$ is always $\delta$-close to some point on $\gamma_n$. Note that the $\gamma_n$ represent points in $Z$. Up to replacing by a subsequence, we know by 
Arzel\`a--Ascoli Lemma \ref{ArzelaAscoli1} that $\{\gamma_n\}$ converges uniformly on compact sets to a geodesic ray $\gamma$ based at $o$.  Since $Z$ is closed, $\gamma$ represents a point in $Z$. And since the sequence of $\gamma_n$ converges uniformly, we know that $\{x_n\}$ is $\delta$-close to $\gamma$ and the result follows.  
\end{proof}
As a special case we record:
\begin{corollary}\label{LimitSetInvertsHull}
Let $X$ be a proper, geodesic $\delta$-Rips hyperbolic metric space. Let $Z$ be a non-empty closed subset of $\partial X$. Then $\cL(\mathfrak{H}(Z))=Z$.
\end{corollary} 

\chapter[Consequences of Breuillard--Green--Tao theory (by S.\ Machado)]{Consequences of Breuillard--Green--Tao theory (by S.\ Machado)} \label{AppendixBGT}

\section{Statement of the result}
The following result was established in \cite[Thm.\ 7.1]{Hrushovski} (see also \cite[Thm.\ 11.1]{BGT}):
\begin{theorem}[Hrushovski]\label{ThmHru}
Let $\Gamma$ be a finitely-generated group. Assume that there exist $k \in \bN$ and finite $k$-approximate subgroups $\Lambda_1 \subset \Lambda_2 \subset \dots \subset \Gamma$ such that $\Gamma = \bigcup \Lambda_n$. Then there exists a finite index subgroup $\Gamma' \subset \Gamma$ which is nilpotent.
\end{theorem}
We emphasize that finite generation of $\Gamma$ needs to be assumed; indeed, every locally finite group admits an exhaustion by finite ($1$-approximate) subgroups. In our proof of the polynomial growth theorem we have used the following approximate version of Theorem \ref{ThmHru} (cf. Theorem \ref{BGTConvenient}); as in the group case, the assumption of (geometric) finite generation is crucial.
\begin{theorem}\label{AppBGTMain}
Let $(\Lambda, \Lambda^\infty)$ be a geometrically finitely-generated approximate group. Assume that there exist $k \in \bN$ and finite $k$-approximate subgroups $\Lambda_1 \subset \Lambda_2 \subset \dots \subset \Lambda$ such that $\Lambda = \bigcup \Lambda_n$. Then there exists a finite index approximate subgroup $\Lambda' \subset \Lambda^{16}$ which generates a nilpotent group.
\end{theorem}
We will actually construct a finite index subgroup $\Lambda' \subset \Lambda^8$ such that $\Lambda'$ generates a \emph{virtually} nilpotent group. This will be sufficient to deduce the theorem, since if $\Gamma < (\Lambda')^\infty$ is a finite-index nilpotent subgroup $\Gamma$, then $(\Lambda')^2 \cap \Gamma$ is an approximate subgroup of $\Lambda^{16}$ which generated a nilpotent group and is of finite index in $\Lambda'$ (hence in $\Lambda$). We will explain in this appendix how to derive Theorem \ref{AppBGTMain} from the following result, which is contained in \cite[Thm.\ 1.6]{BGT}:
\begin{theorem}[Breuillard--Green--Tao]\label{BGTOriginal} For every $k \in \mathbb N$ there exists $C(k) \in \mathbb N$ such that if $(\Lambda_n, \Lambda_n^\infty)_{n \in \bN}$ is a sequence of finite $k$-approximate groups, then there exist a constant $C(k)$ and approximate subgroups $\Lambda_n' \subset \Lambda_n^4$ such that each $\Lambda_n$ is covered by $C(k)$ left-translates of $\Lambda_n'$ and each $\Lambda_n'$ generates a virtually nilpotent subgroup of $\Lambda_n^\infty$.
\end{theorem}
Note that in Theorem \ref{BGTOriginal} the groups $\Lambda^\infty_n$ are entirely unrelated, i.e.\ we are allowed to consider a family of approximate subgroups in different groups.

One can check that if $C(k)$ is as in Theorem \ref{BGTOriginal}, then our approximate subgroup $\Lambda' \subset \Lambda^8$ (the one generating a virtually nilpotent group) will be of index at most $k^5C(k)$.
Using a more precise version of Theorem \ref{BGTOriginal} one could also deduce that the rank and step of the nilpotent group in question are bounded by a constant depending only on $k$.

\section{The main construction}
We now turn to the setting of Theorem \ref{AppBGTMain}; thus  $(\Lambda, \Lambda^\infty)$ denotes a geometrically finitely-generated approximate group and $\Lambda_1 \subset \Lambda_2 \subset \dots \subset \Lambda$ is an exhaustion of $\Lambda$ by finite $k$-approximate subgroups. We now apply Theorem \ref{BGTOriginal} to the sequence $(\Lambda_n)_{n \in \mathbb N}$. We thus find a constant $C(k)$ and approximate subgroups $\Lambda_n' \subset \Lambda_n^4$  such that  $\Lambda_n$ is covered by $C(k)$ left-translates of $\Lambda_n'$ and $\Lambda_n'$ generates a virtually nilpotent group. We now fix a non-principal ultrafilter $\omega \in \beta \mathbb N \setminus \mathbb N$ and define an approximate subgroup $\Lambda' \subset \Lambda^8$ by
\[
\Lambda' \coloneqq  \{\lambda \in \Lambda^8 \mid \lambda \in (\Lambda_m')^2 \text{ for $\omega$-almost all }m \in \mathbb N\}.
\]
The finite subsets of $\Lambda'$ can be characterized as follows:
\begin{lemma}\label{UltrafilterFiniteSet} Let $F \subset \Lambda^8$ be finite.
\begin{enumerate}[(i)]
\item $F \cap \Lambda' = F \cap (\Lambda_m')^2$ for $\omega$-almost all $m \in \mathbb N$.
\item $F \subset \Lambda'$ if and only if $F \subset (\Lambda_m')^2$ for $\omega$-almost all $m \in \mathbb N$.
\end{enumerate}
\end{lemma}
\begin{proof} (i) For $f \in F$ we denote $\mathbb N_f \coloneqq  \{m\in \mathbb N \mid f \in (\Lambda_m')^2\}$. Then $f \in F \iff \bN_f \in \omega$ and since $\omega$ is an ultrafilter we have
\[
f \not \in F \iff \bN_f \not \in \omega \iff \bN \setminus\bN_f \in \omega.
\]
Moreover, since $\omega$ is a filter it is closed under finite intersections. This implies that
\[
\{m \in \bN \mid F \cap \Lambda' = F \cap (\Lambda_m')^2\} = \bigcap_{f \in F \cap \Lambda'} \underset{\in \omega}{\underbrace{\bN_f}} \cap \bigcap_{f \in F \setminus \Lambda'} \underset{\in \omega}{\underbrace{(\bN \setminus\bN_f)}} \in \omega,\]
which is (i), and (ii) follows from (i).
\end{proof}
The main step in our proof of Theorem \ref{AppBGTMain} is the following proposition:
\begin{proposition}\label{BGTFiniteIndex}
The approximate subgroup $\Lambda' \subset \Lambda^8$ is of finite index.
\end{proposition}
For the proof of the proposition we use the following lemma:
\begin{lemma}\label{BGTLemma2} Let $F \subset \Lambda$ be finite and assume that the sets $f\Lambda'$ are disjoint as $f$ varies over $F$. Then 
\[
|F|\leq k^2 C(k).
\]
\end{lemma}
\begin{proof} Since $F$ is finite, there exists some $n \in \mathbb N$ such that $F \subset \Lambda_n$. Then $F(\Lambda_n^2 \cap \Lambda') \subset \Lambda_n^3$, and since $\Lambda_n$ is a $k$-approximate group we have
\[
k^2|\Lambda_n| \geq |\Lambda_n^3| \geq |F(\Lambda_n^2 \cap \Lambda')| = \left| \bigsqcup_{f \in F} f(\Lambda_n^2 \cap \Lambda')  \right| = |F| \cdot |\Lambda_n^2 \cap \Lambda'|,
\]
and hence
\begin{equation}\label{FiniteSetEstimate1}
|F| \leq \frac{k^2 |\Lambda_n|}{|\Lambda_n^2 \cap \Lambda'|}.
\end{equation}
Now by Lemma \ref{UltrafilterFiniteSet} we have $\Lambda_n^2 \cap \Lambda' = \Lambda_n^2 \cap (\Lambda_m')^2$ for $\omega$-almost all $m \in \mathbb N$. Since $\omega$ is non-principal, the finite set $\{m \in \mathbb N \mid m < n\}$ is not in $\omega$, and hence we can find $m \geq n$ such that
\[
\Lambda_n^2 \cap \Lambda'  = \Lambda_n^2 \cap (\Lambda_m')^2.
\]
Since $m \geq n$ we have $\Lambda_n \subset \Lambda_m$, and hence $\Lambda_n$ is covered by at most $C(k)$ left-translates of $\Lambda_m'$. Applying Corollary \ref{lemmasymmtrick} with $X \coloneqq  \Lambda_n$ and $Y \coloneqq \Lambda_m'$ we thus find a set $F'$ of size $|F'| \leq C(k)$ such that 
\[
\Lambda_n \subset  F'(\Lambda_n^2 \cap (\Lambda_m')^2).
\]
Thus
\[
|\Lambda_n| \leq |F'| |\Lambda_n^2 \cap (\Lambda_m')^2| \leq C(k)|\Lambda_n^2 \cap \Lambda'|.
\]
Plugging this inequality into \eqref{FiniteSetEstimate1} then yields the lemma.
\end{proof}
\begin{proof}[Proof of Proposition \ref{BGTFiniteIndex}]  Take a maximal $F \subset \Lambda$ such that the sets $f\Lambda'$ are pairwise disjoint as $f$ ranges over $F$. By Lemma \ref{BGTLemma2} we have $|F| \leq C(k)k^2$ and by Rusza's Covering Lemma (Lemma \ref{Rusza}) we have
\[
\Lambda \subset F(\Lambda')^2.
\]
Since $(\Lambda')^2$ is of finite index in $\Lambda'$, the proposition follows.
\end{proof}

    \section{The endgame}
We can now complete the proof of Theorem \ref{AppBGTMain}:
\begin{proof}[Proof of Theorem \ref{AppBGTMain}] Let $\Lambda' \subset \Lambda^8$ be the approximate subgroup constructed in the previous section. By assumption, the approximate group $(\Lambda, \Lambda^\infty)$ is geometrically finitely-generated; since $\Lambda'$ is of finite index in $\Lambda$ we have $[\Lambda]_{\mathrm{int}} = [\Lambda']_{\mathrm{int}}$, and hence $(\Lambda', (\Lambda')^\infty)$ is also geometrically finitely-generated, hence algebraically finitely-generated by Corollary \ref{GFGAFG}. Now let $S_0$ be a finite generating set of $(\Lambda')^\infty$. Then every element of $S_0$ can be written as a finite product of elements of $\Lambda'$, hence $(\Lambda')^\infty$ can also be generated by a finite subset 
$S \subset \Lambda'$. By Lemma \ref{UltrafilterFiniteSet} we then have $S \subset \Lambda'_n$ for some $n \in \mathbb N$. This implies that $(\Lambda')^\infty = (\Lambda_n')^\infty$, which by assumption is virtually nilpotent.
\end{proof}
\begin{remark}
Since Theorem \ref{ThmHru} fails for non-finitely-generated groups, some form of finite generation has to be assumed on $\Lambda$ for Theorem \ref{AppBGTMain} to hold. In our proof we have used the assumption that every finite index approximate subgroup of $\Lambda^8$ is algebraically finitely-generated. This property is a priori not implied by $\Lambda$ being \emph{algebraically} finitely-generated, but (as indicated in the proof) it is implied by $\Lambda$ being \emph{geometrically} finitely-generated. We have thus stated the theorem for geometrically finitely-generated approximate groups. It is possible that a weaker finite generation assumption on $\Lambda$ is sufficient for Theorem \ref{AppBGTMain} to hold, but we are not aware of any argument to that effect.
\end{remark}

\end{appendix}

\printindex


\bibliographystyle{alpha}
\bibliography{literature}
\end{document}